\documentclass[10pt]{article}

\usepackage[margin=1.0in]{geometry}
\usepackage{amsthm, amsfonts,amssymb,euscript}
\usepackage{amsmath}
\usepackage{enumerate}
\usepackage{bm}
\usepackage{amssymb}
\usepackage{amsthm}
\usepackage{graphicx}
\usepackage{caption}
\usepackage{subcaption}
\usepackage{amsfonts}
\usepackage{float}
\usepackage{color}
\usepackage{slashed}
\usepackage{mathrsfs}
\usepackage[affil-it]{authblk}

\usepackage[usenames,dvipsnames,svgnames,table]{xcolor}
\usepackage[colorlinks=true]{hyperref}
\hypersetup{linkcolor=BrickRed, urlcolor=green, citecolor=blue, linktoc=page}

\numberwithin{equation}{section}

\newtheorem*{proposition*}{Proposition}
\newtheorem*{theorem*}{Theorem}
\newtheorem*{conjecture*}{Conjecture}
\newtheorem*{claim*}{Claim}
\newtheorem*{lemma*}{Lemma}
\newtheorem*{corollary*}{Corollary}
\newtheorem{theorem}{Theorem}[section]
\newtheorem{proposition}[theorem]{Proposition}

\newtheorem{lemma}[theorem]{Lemma}
\newtheorem{corollary}[theorem]{Corollary}

\newtheorem*{definition*}{Definition}
\newtheorem{definition}{Definition}[section]
\newtheorem*{assumption*}{\mathcal{A}ssumption}

\newtheorem*{remark*}{Remark}
\newtheorem{remark}{Remark}[section]

\newtheorem{mytheo}{Theorem}
\newtheorem{thmx}{Theorem}

\newcommand{\R}{\mathbb{R}}
\newcommand{\s}{\mathbb{S}}
\newcommand{\C}{\mathbb{C}}

\newcommand{\N}{\mathbb{N}}

\newcommand{\snabla}{\slashed{\nabla}}

\newcommand{\Lbar}{\underline{L}}

\allowdisplaybreaks

\begin{document}

\title{A non-degenerate scattering theory for the wave equation \\ on extremal Reissner--Nordstr\"om}

\author[1]{Yannis Angelopoulos \thanks {yannis@caltech.edu}}
\author[2]{Stefanos Aretakis\thanks {aretakis@math.toronto.edu}}
\author[3]{Dejan Gajic \thanks {D.Gajic@dpmms.cam.ac.uk}}
	\affil[1]{\small The Division of Physics, Mathematics and Astronomy, Caltech,
1200 E California Blvd, Pasadena CA 91125, USA}
	\affil[2]{\small Department of Mathematics, University of Toronto, 40 St George Street, Toronto, ON, Canada}
	\affil[3]{\small Centre for Mathematical Sciences, University of Cambridge, Wilberforce Road, Cambridge CB3 0WB, UK}

\date{October 18, 2019}

\maketitle

\begin{abstract}
It is known that sub-extremal black hole backgrounds do not admit a (bijective) non-degenerate scattering theory in the exterior region due to the fact that the redshift effect at the event horizon acts as an unstable blueshift mechanism in the backwards direction in time. In the extremal case, however, the redshift effect degenerates and hence yields a much milder blueshift effect when viewed in the backwards direction. In this paper, we construct a definitive (bijective) non-degenerate scattering theory for the wave equation on extremal Reissner--Nordstr\"{o}m backgrounds. We make use of physical-space energy norms which are non-degenerate both at the event horizon and at null infinity. As an application of our theory we present a construction of a large class of smooth, exponentially decaying modes. We also derive scattering results in the black hole interior region. 

\end{abstract}

\tableofcontents

\section{Introduction}

\subsection{Introduction and background}
\label{sec:Introduction}

Scattering theories for the wave equation
\begin{equation}
\label{eq:waveequation}
\square_g\psi=0
\end{equation}
on black hole backgrounds provide useful insights in studying the evolution of perturbations ``at infinity''. In this article we construct a new scattering theory for scalar perturbations on extremal Reissner--Nordstr\"{o}m. Our theory makes crucial use of the vanishing of the surface gravity on the event horizon and our methods extend those of the horizon instability of extremal black holes in the forward-in-time evolution. In the remainder  of this section we will briefly recall  scattering theories for sub-extremal backgrounds and in the next section we will provide a rough version of the main theorems. 

We will first review the scattering theories of the wave equation \eqref{eq:waveequation}
on Schwarzschild spacetime backgrounds. Let $T$ denote the standard stationary Killing vector field on a Schwarzschild spacetime. Since $T$ is globally causal in the domain of outer communications, the energy flux associated to $T$ is non-negative definite. This property played a crucial role in the work of Dimock and Kay \cite{dimock2, dimock1} where a $T$-scattering theory on Schwarzschild, in the sense of Lax--Phillips \cite{laxphillips}, was developed (Figure \ref{fig:schwscata}). Subsequently, the $T$-scattering theory was understood by Nicolas \cite{nicolas}, following the notion of scattering states by Friedlander \cite{friedlander} (Figure \ref{fig:schwscatb}). 

\begin{figure}[h!]
\centering
\begin{subfigure}[b]{0.4\textwidth}
\centering
\includegraphics[scale=0.45]{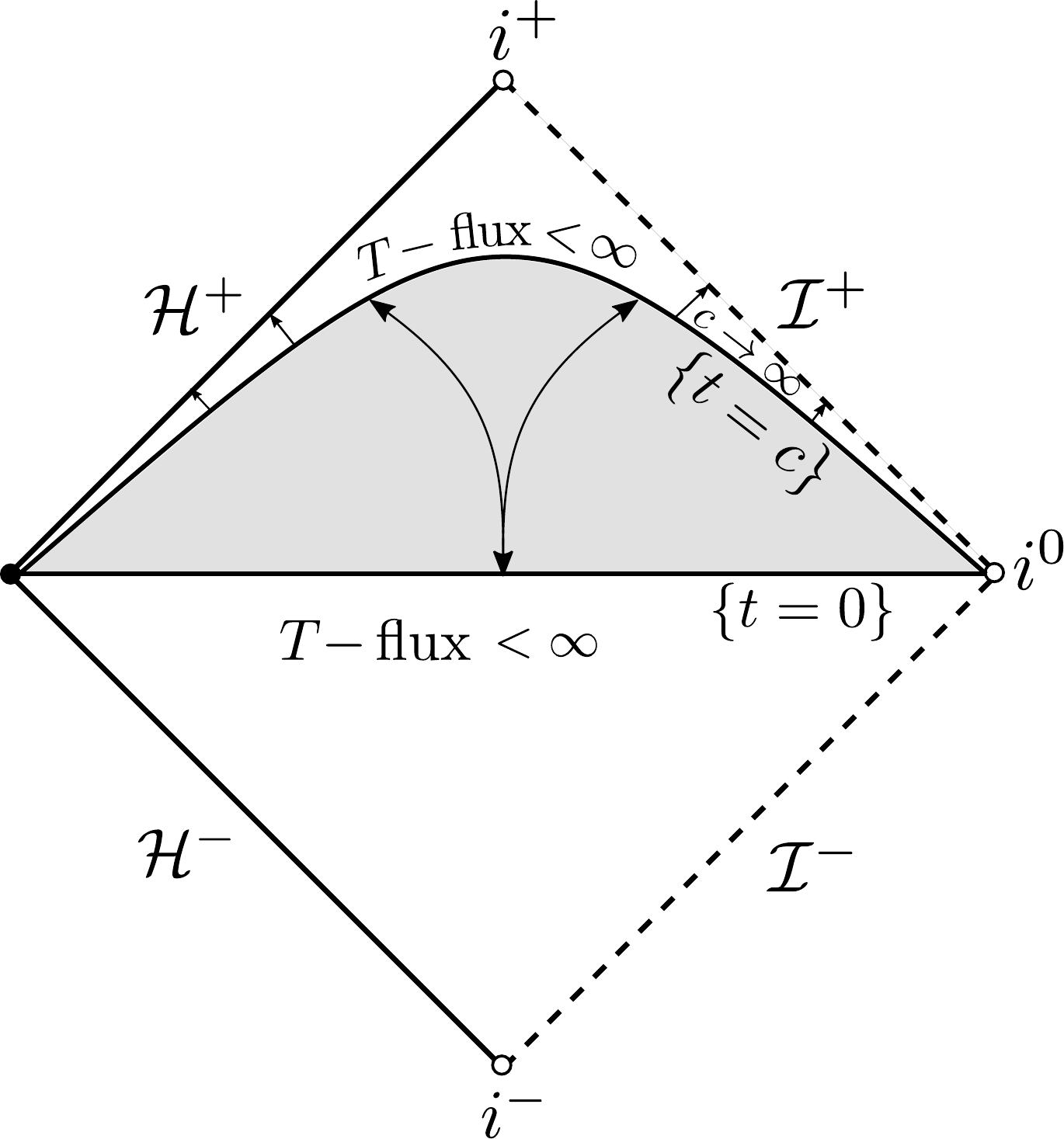}
\caption{The Dimock--Kay $T$-scattering map \\ in the sense of Lax--Phillips. }
\label{fig:schwscata}
\end{subfigure}
\begin{subfigure}[b]{0.4\textwidth}
\centering
\includegraphics[scale=0.45]{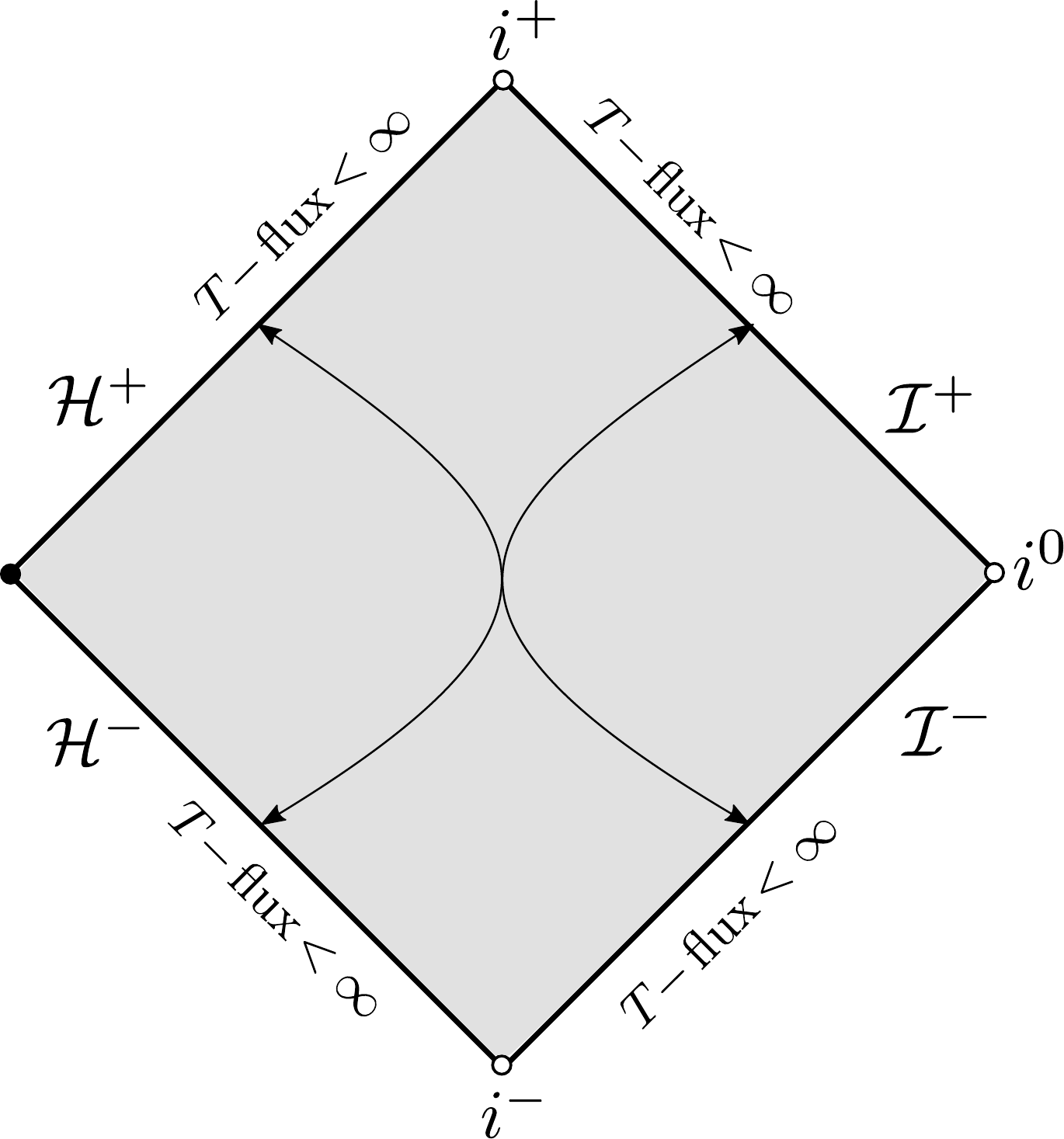}
\caption{The Nicolas $T$-scattering map \\ in the sense of Friedlander. }
\label{fig:schwscatb}
\end{subfigure}
\caption{The $T$-scattering maps on Schwarzschild spacetime.}
\label{fig:schwscat}
\end{figure}

 The $T$-energy scattering theory on Schwarzschild applies also when the standard Schwarzschild time function $t$ is replaced by a time function corresponding to a foliation by hypersurfaces intersecting the future event horizon and terminating at future null infinity (Figure \ref{fig:schwscathypa}).  This is convenient since it allows one to bound energies as measured by local observers. Recall that $T$ is timelike in the black hole exterior and null on the event horizon. For this reason, the $T$-energy flux across an achronal hypersurface intersecting the event horizon is positive-definite away from the horizon and degenerate at the horizon. Hence, the associated norm for the $T$-energy scattering theory is degenerate at the event horizon. On the other hand, it has been shown \cite{linearscattering, DafShl2016} that Schwarzschild does not admit a non-degenerate scattering theory where the norm on the achronal hypersurface is defined in terms of the energy flux associated to a globally timelike vector field $N$ (Figure \ref{fig:schwscathypb}) and the norms on the event horizon and null infinity are also defined in terms of energy flux associated with $N$, but with additional, arbitrarily fast polynomially decaying weights in time. This is due to the celebrated redshift effect which turns into a blueshift instability mechanism when seen from the backwards scattering point of view. 

\begin{figure}[h!]
\centering
\begin{subfigure}[b]{0.4\textwidth}
\centering
\includegraphics[scale=0.45]{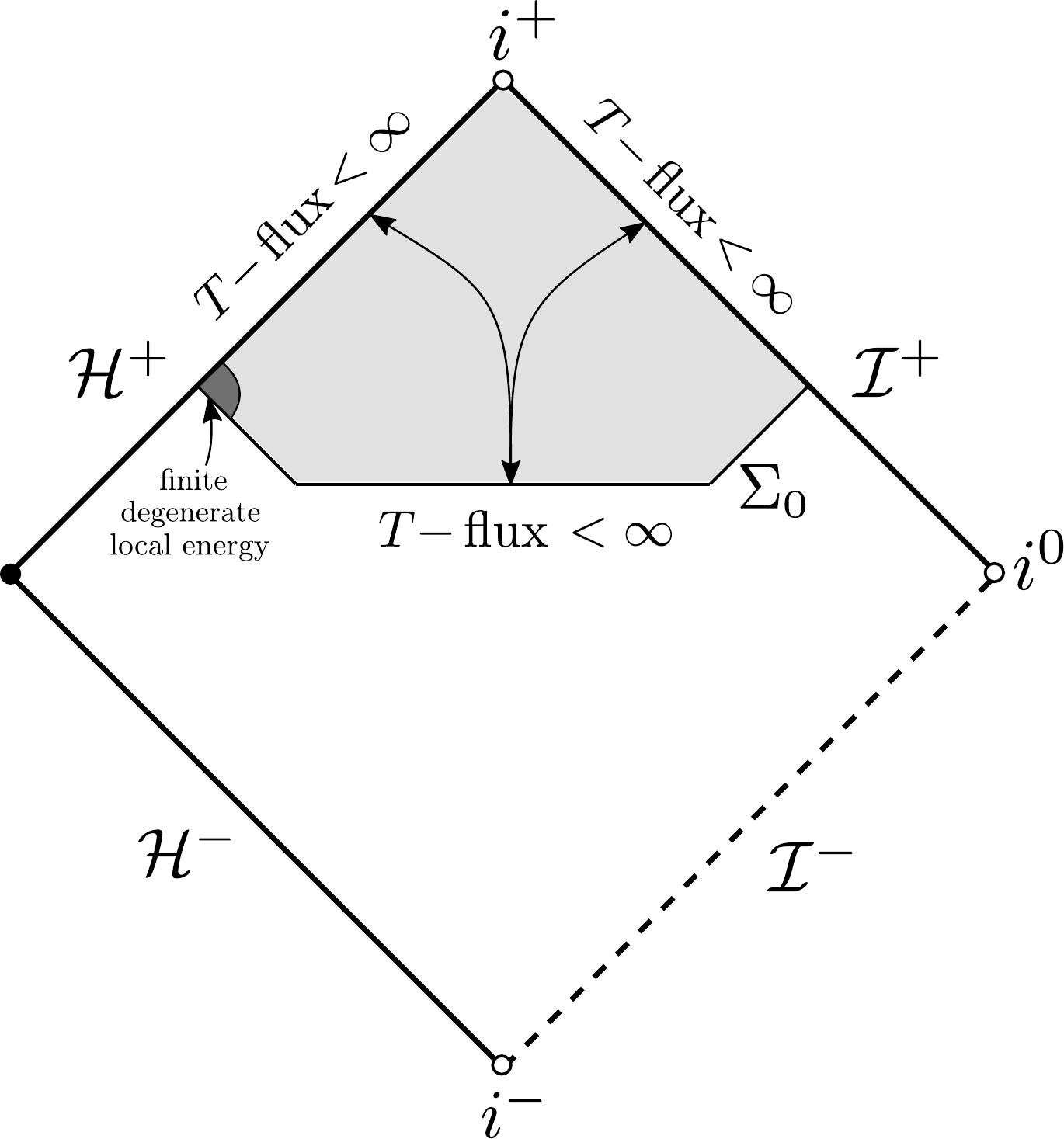}
\caption{The $T$-scattering map.}
\label{fig:schwscathypa}
\end{subfigure}
\begin{subfigure}[b]{0.4\textwidth}
\centering
\includegraphics[scale=0.45]{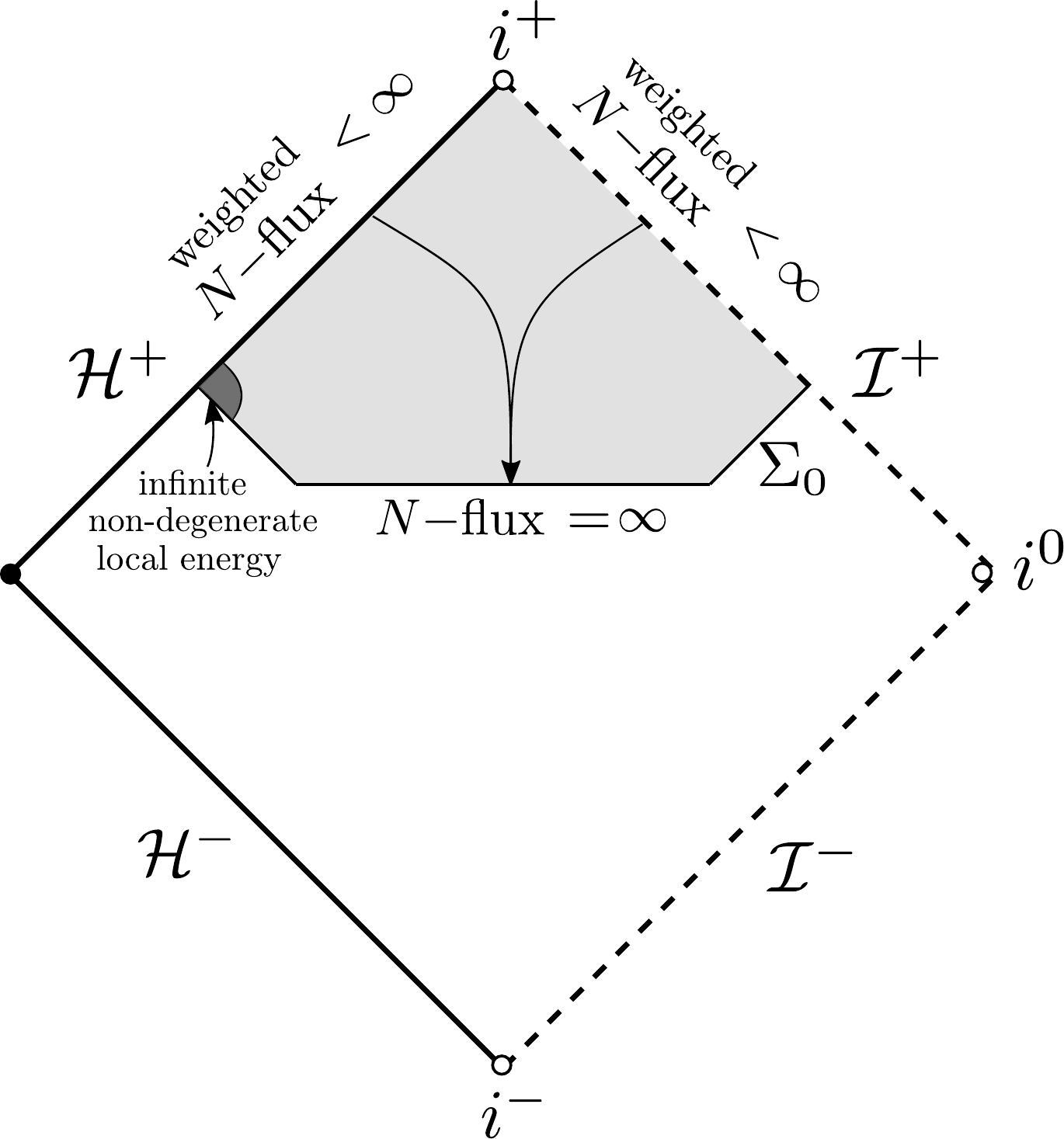}
\caption{The $N$-scattering map\\ fails to be surjective.}
\label{fig:schwscathypb}
\end{subfigure}
\caption{The $T$ and $N$ scattering maps on Schwarzschild.}
\label{fig:schwscathyp}
\end{figure}
It is important to note that one can counter the blue-shift mechanism and define a backwards scattering map for non-degenerate high-regularity norms on an achronal hypersurface if the data on $\mathcal{H}^{+}$ and $\mathcal{I}^{+}$ are sufficiently regular and decay exponentially fast with sufficiently large rate (Figure \ref{fig:higher}). A fully nonlinear version of this statement, in the context of the vacuum Einstein equations, was presented in \cite{scattering}. 
 \begin{figure}[H]
	\begin{center}
\includegraphics[scale=0.45]{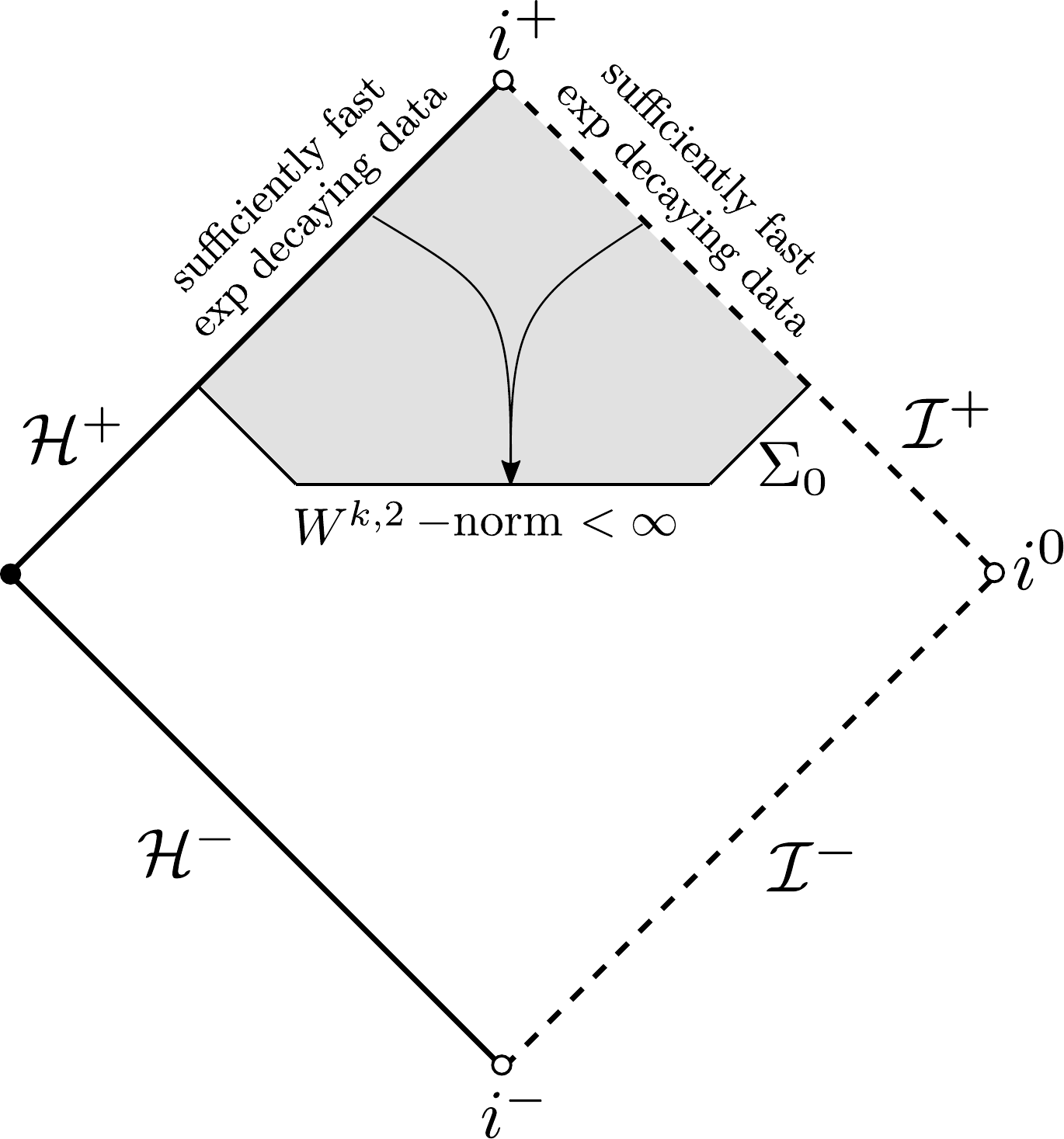}
\end{center}
\vspace{-0.2cm}
\caption{Higher-order non-degenerate 
backwards scattering on Schwarzschild.}
	\label{fig:higher}
\end{figure}
As far as the Kerr family is concerned, Dafermos, Rodnianski and Shlapentokh-Rothman \cite{linearscattering} derived a degenerate scattering theory in terms of the energy flux associated to a globally causal vector field $V$ which is null on the event horizon and timelike in the exterior region. Similarly to the Schwarzschild case, the sub-extremal Kerr backgrounds do not admit a non-degenerate scattering theory in the exterior region. Let us also note that a $T$-energy scattering theory on Oppenheimer--Snyder spacetimes, describing Schwarzschild-like black holes arising from gravitational collapse, was developed in \cite{alf19}.

Finally we present some results regarding the black hole interior region. Luk--Oh \cite{Luk2015} showed that the forward evolution of smooth compactly supported initial data on sub-extremal Reissner--Nordstr\"om (RN) is $W^{1,2}$-singular at the Cauchy horizon (Figure \ref{fig:higher2sub2ern}). 
\begin{figure}[H]
	\begin{center}
\includegraphics[scale=0.40]{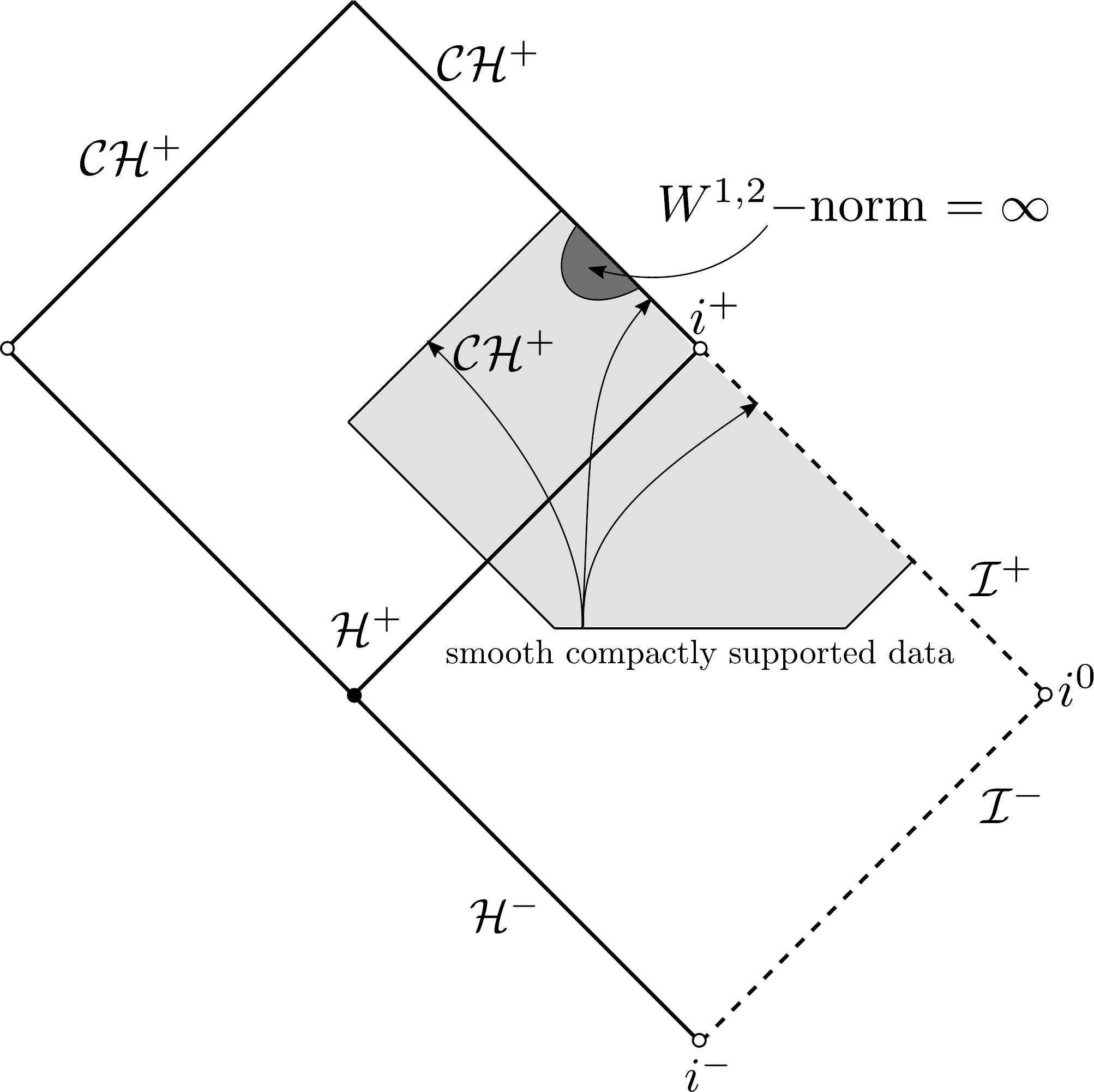}
\end{center}
\vspace{-0.2cm}
\caption{Blow-up of $W^{1,2}$ norm in any neighborhood of the Cauchy horizon.}
	\label{fig:higher2sub2ern}
\end{figure}
Similar instability results for the wave equation on Kerr interiors were presented by Luk--Sbierski \cite{LukSbierski2016} and independently by Dafermos--Shlapentokh-Rothman \cite{DafShl2016} (see also \cite{chr18, hintz17, Franzen2014}). Specifically, in \cite{DafShl2016} the authors assumed trivial data on the past event horizon and arbitrary, non-trivial polynomially decaying data on past null infinity and showed that local (non-degenerate) energies blow up in a neighborhood of any point at the Cauchy horizon (Figure \ref{fig:higher2sub10}). The interior of Schwarzschild was considered by Fournodavlos and Sbierski \cite{gregjan}, who derived asymptotics for the wave equation at the singular boundary $\{r=0\}$.
 \begin{figure}[H]
	\begin{center}
\includegraphics[scale=0.40]{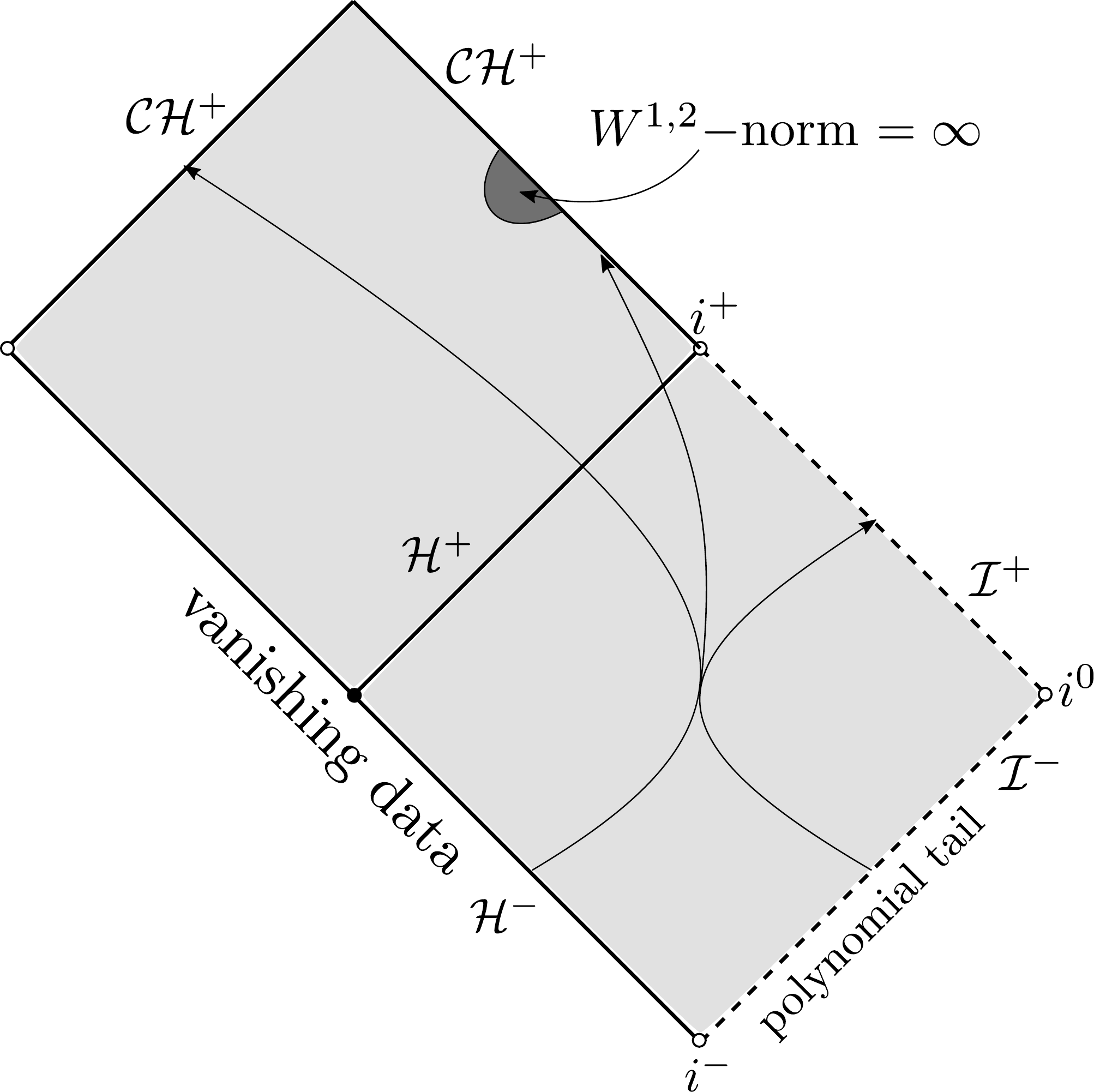}
\end{center}
\vspace{-0.2cm}
\caption{Blow-up of $W^{1,2}$ norm from scattering data on $\mathcal{H}^{-}$ and $\mathcal{I}^{-}$.}
	\label{fig:higher2sub10}
\end{figure}

\subsection{Overview of the main theorems}
\label{sec:OverviewOfTheMainTheorems}

In this section we present a rough version of our main theorems. Theorems \ref{thm:tscatERN} and \ref{thm:Nenergyinfinite} are straightforward extensions of known results, so we will only sketch their proofs, whereas Theorems \ref{theo1}--\ref{theo6} are entirely novel results that require new techniques and whose precise statements of the theorems can be found in Section \ref{sec:thms}. 

First of all, note that the standard stationary Killing vector field $T$ is causal everywhere in the domain of outer communications of ERN. From this, it follows that the $T$-energy scattering theory in Schwarzschild can easily be extended to ERN:
\begin{thmx}
\label{thm:tscatERN}
The $T$-scattering theory in Schwarzschild extends to extremal Reissner--Nordstr\"om.
\end{thmx}
\begin{proof}
Follows by applying the methods in Section 9.6 of \cite{linearscattering} together with the decay estimates derived in \cite{aretakis1}.
\end{proof}
\begin{figure}[H]
\centering
\begin{subfigure}[b]{0.3\textwidth}
\centering
\includegraphics[scale=0.35]{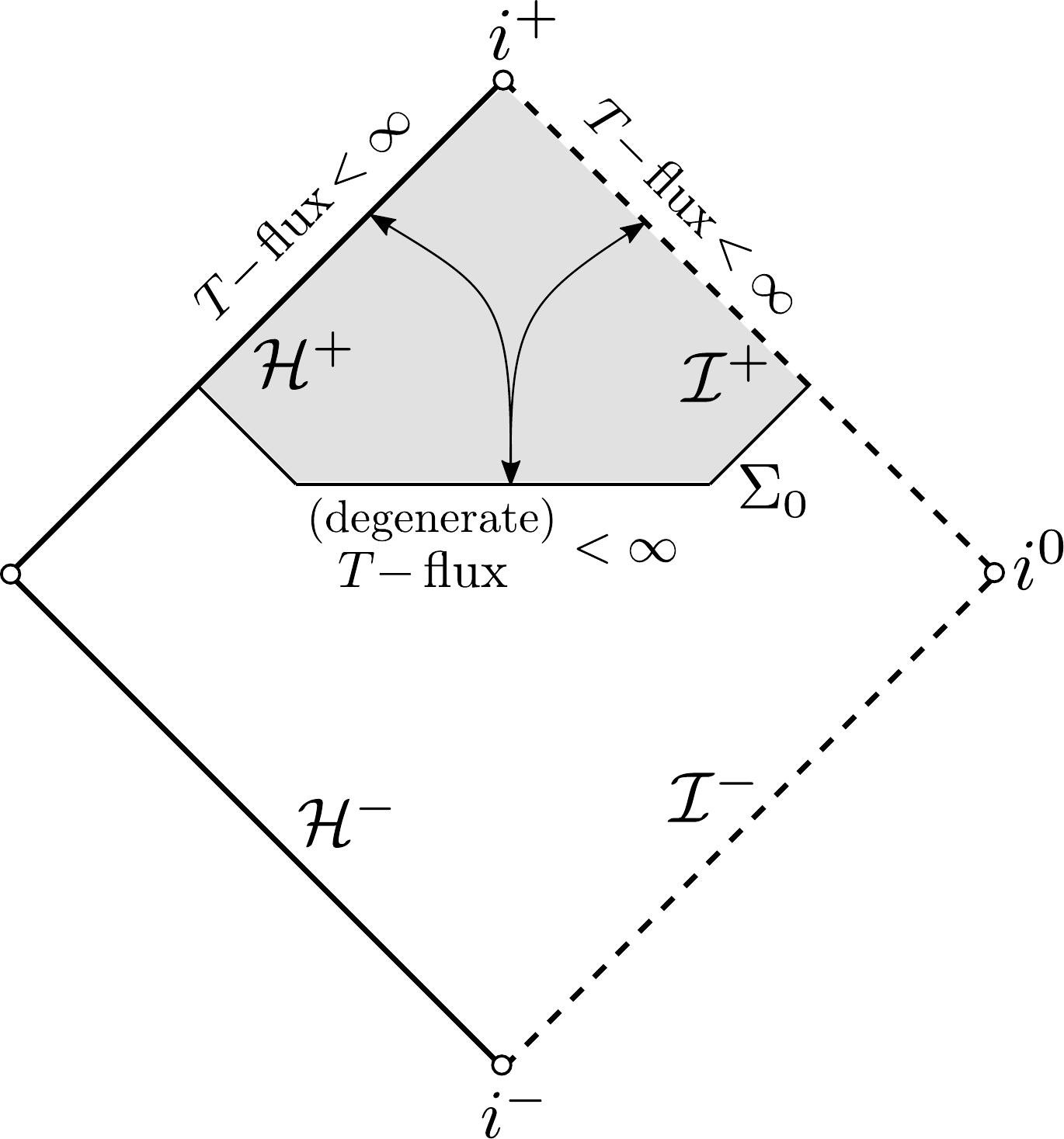}
\end{subfigure}
\centering
\begin{subfigure}[b]{0.3\textwidth}
\centering
\includegraphics[scale=0.35]{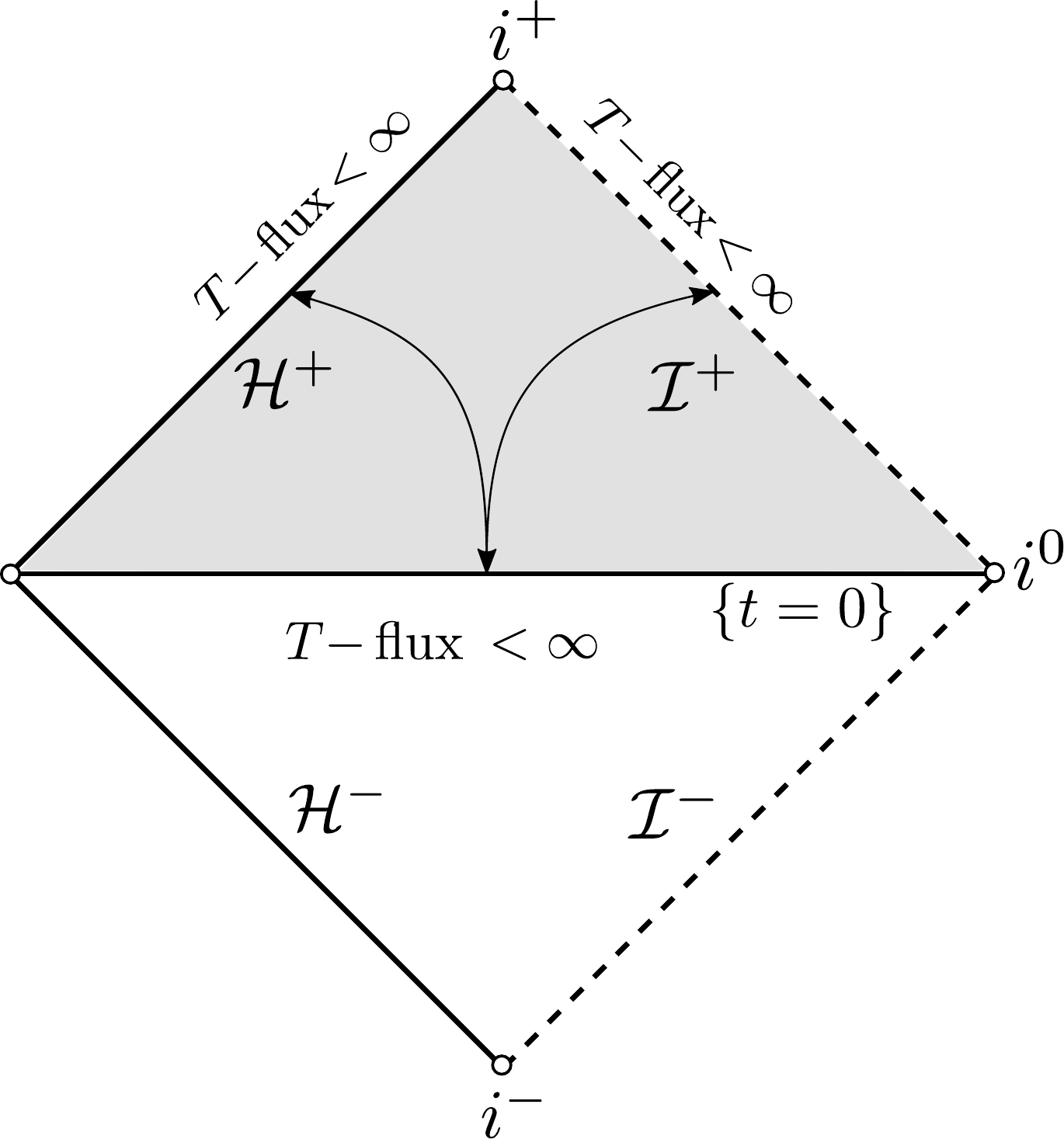}
\end{subfigure}
\begin{subfigure}[b]{0.3\textwidth}
\centering
\includegraphics[scale=0.35]{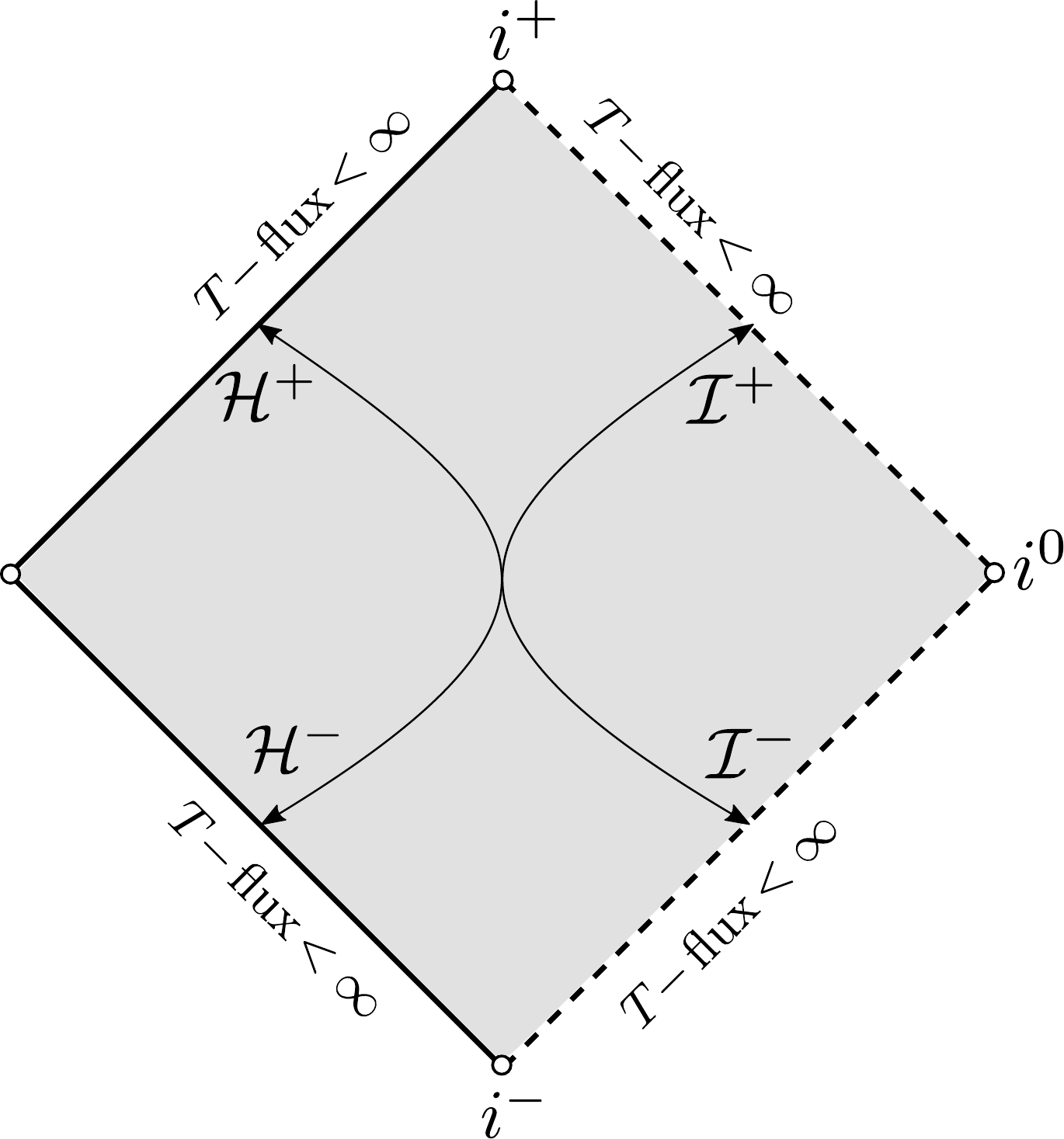}
\end{subfigure}
\caption{The $T$-scattering theory for ERN.}
\label{fig:schwscat1ern}
\end{figure}

In the following theorem, we show that in ERN \textbf{we can in fact go beyond $T$-energy scattering} by providing a bijective scattering theory for weighted and non-degenerate norms on ERN; see Figure \ref{fig:higherFD} for an illustration. Here, $\Sigma_0$ will denote a spacelike-null hypersurface intersecting $\mathcal{H}^{+}$ and terminating at $\mathcal{I}^{+}$.

\begin{mytheo} (Rough version of Theorem \ref{thm:mainthm})
The  scattering maps defined in the black hole exterior of ERN  
 \begin{itemize}
	 \item between weighted energy spaces on  ($\mathcal{H}^{-}$, $\mathcal{I}^{-}$) and ($\mathcal{H}^{+}$, $\mathcal{I}^{+}$),
 \end{itemize}
or 
\begin{itemize}
	\item between a weighted energy space on  ($\mathcal{H}^{+}\cap\mathcal{J}^{+}(\Sigma_0)$, $\mathcal{I}^{+}\cap\mathcal{J}^{+}(\Sigma_0)$) and a \underline{non-degenerate} energy space on $\Sigma_0$  
\end{itemize}
 
are bounded and bijective. 
\label{theo1}
\end{mytheo}

\begin{figure}[H]
\centering
\begin{subfigure}[b]{0.3\textwidth}
\centering
\includegraphics[scale=0.35]{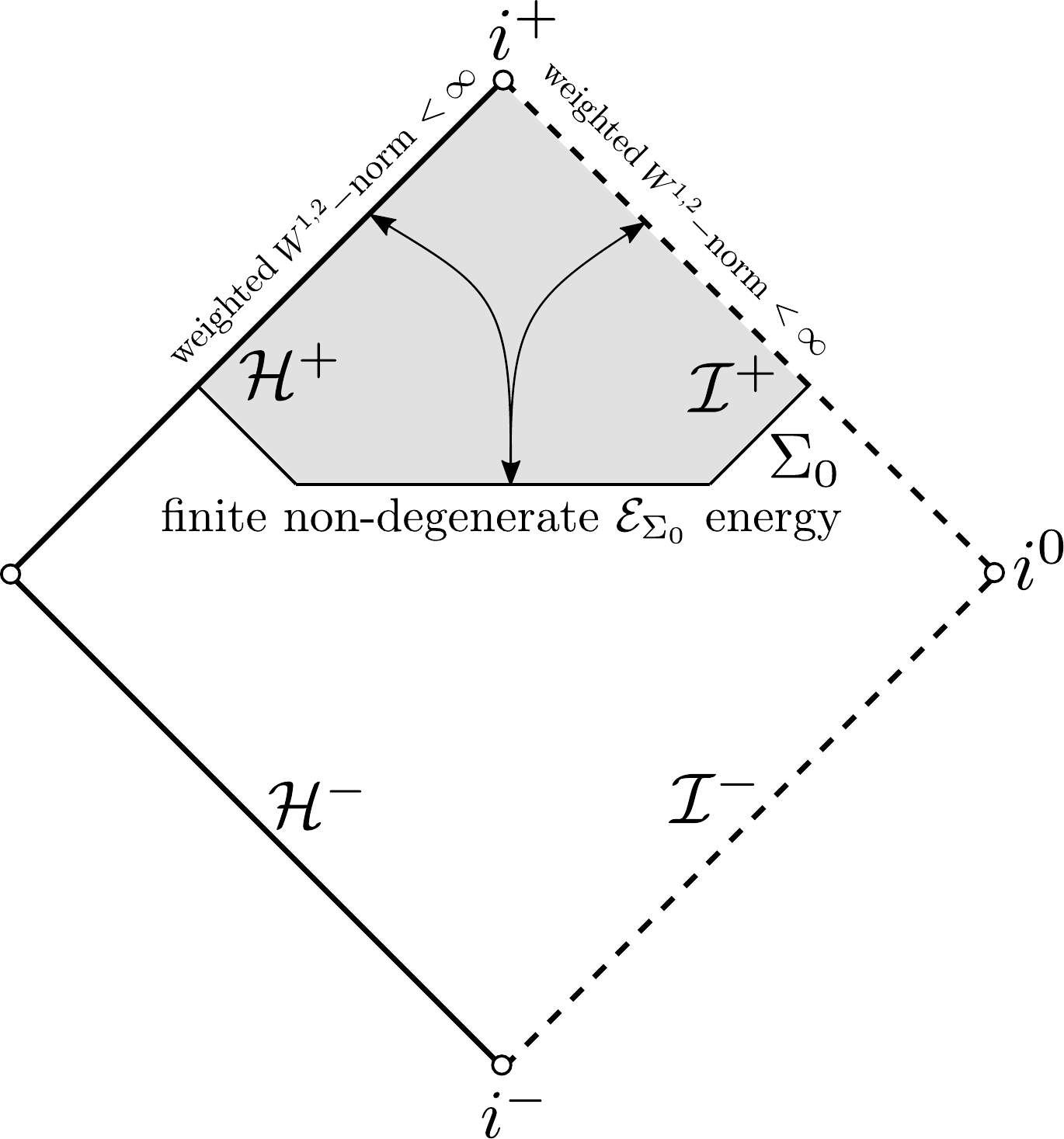}
\end{subfigure}
\centering
\begin{subfigure}[b]{0.3\textwidth}
\centering
\includegraphics[scale=0.35]{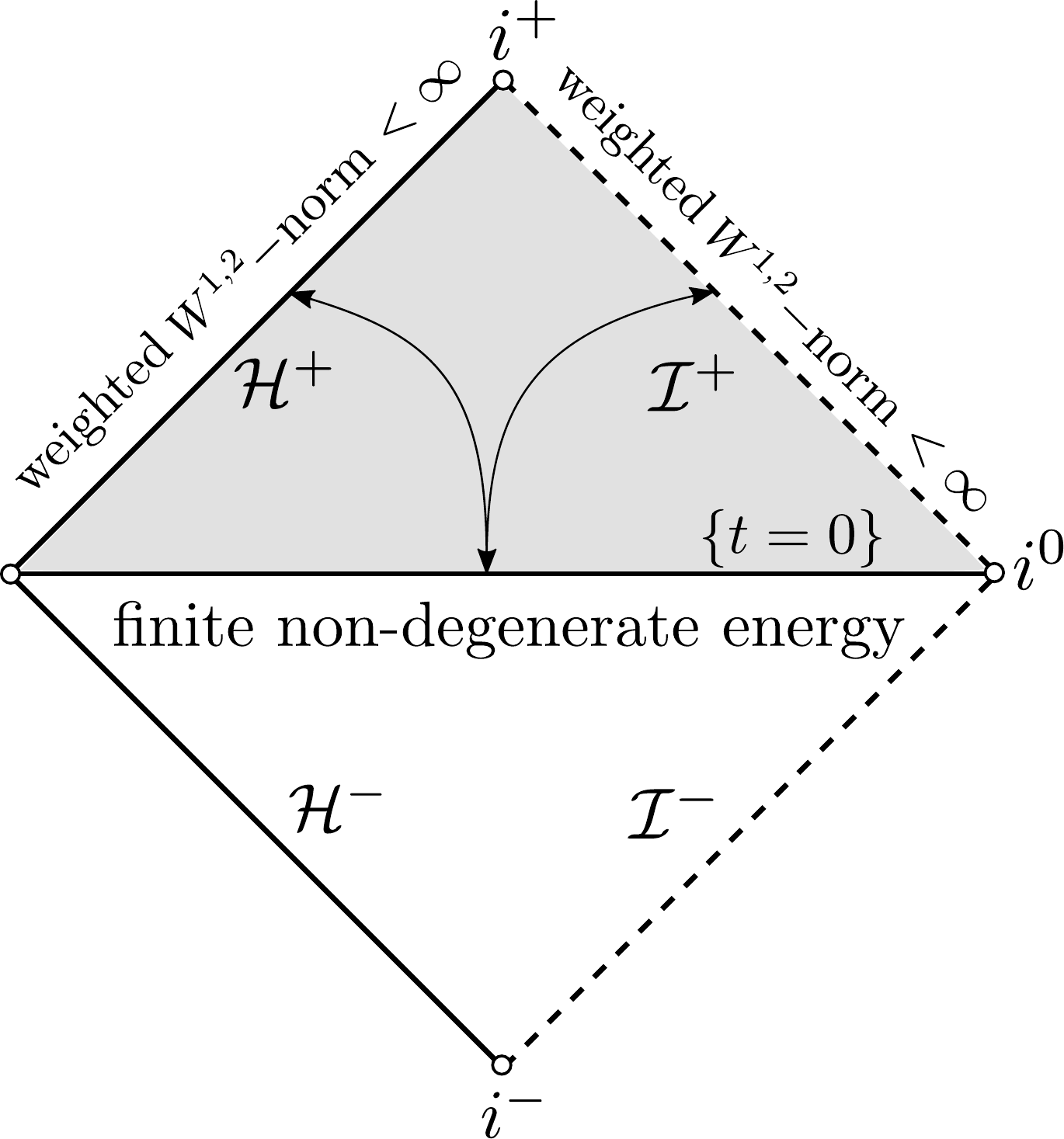}
\end{subfigure}
\begin{subfigure}[b]{0.3\textwidth}
\centering
\includegraphics[scale=0.35]{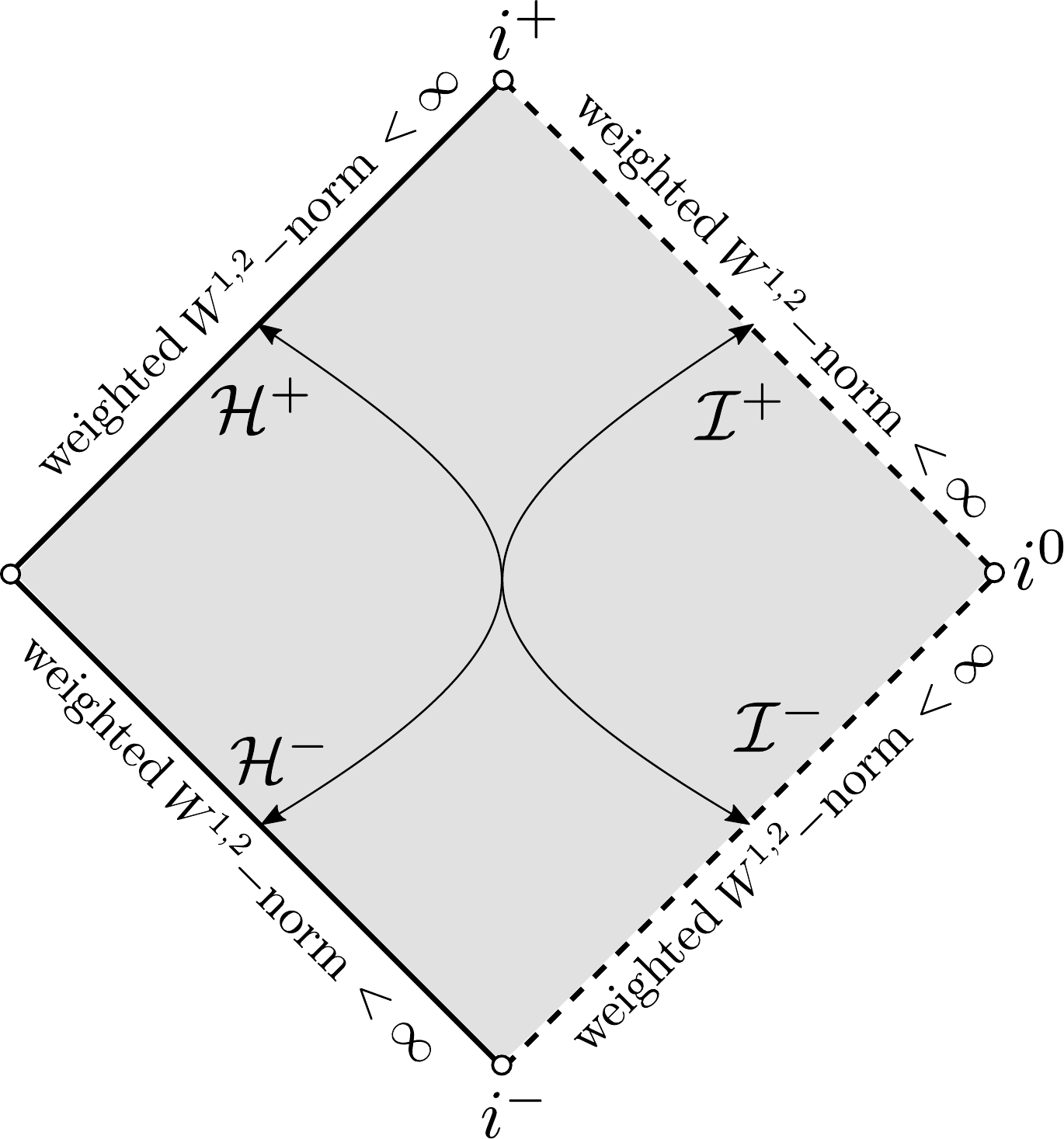}
\end{subfigure}
\caption{A non-degenerate scattering theory on ERN.}
	\label{fig:higherFD}
\end{figure}

 A rough schematic definition of the weighted norms on $\mathcal{H}^{+}$ and $\mathcal{I}^{+}$ is the following
\begin{equation}
\begin{split}
\left\|r\psi\right\|^2_{\mathcal{E}_{\mathcal{H}^{+}}} &= \int_{\mathcal{H}^{+}\cap\mathcal{J}^{+}(\Sigma_0)} (1+v^2)\cdot J^{T}[\psi]+\ldots, \\
\left\|r\psi\right\|^2_{\mathcal{E}_{\mathcal{I}^{+}}} &= \int_{\mathcal{I}^{+}\cap\mathcal{J}^{+}(\Sigma_0)} (1+u^2)\cdot J^{T}[\psi]+\ldots . \\
\end{split}
\label{roughnorms1}
\end{equation}
 A rough schematic definition of non-degenerate energy on $\Sigma_0$ is the following
\begin{equation}
\begin{split}
\left\|(\psi, \partial\psi)\right\|^2_{\mathcal{E}_{\Sigma_0}} &= \int_{\Sigma_0}   J^{N}[\psi]+(\partial_{\rho}(r\psi))^2 +\ldots .
\end{split}
\label{roughnorms2}
\end{equation}
Note that the $\mathcal{E}_{\Sigma_0}-$norm is non-degenerate both at the event horizon and at null infinity (the latter understood in  an appropriate conformal sense; see Section \ref{addnot}). 
The omitted terms involve either smaller weights or extra degenerate factors and additional angular or time derivatives.  Here $J^{T}$ and $J^{N}$ denote the energy fluxes associated to the vector fields $T$ and $N$ and $\partial_{\rho}$ is a tangential to $\Sigma_0$ derivative such that $\partial_{\rho}r=1$. Let $\mathcal{E}_{\mathcal{H}^{+}\cap\mathcal{J}^{+}(\Sigma_0)}, \mathcal{E}_{\mathcal{I}^{+}\cap\mathcal{J}^{+}(\Sigma_0)},\mathcal{E}_{\Sigma_0} $ denote the closure of smooth compactly supported data under the corresponding norms schematically defined above. 

The above theorem is in stark contrast to the sub-extremal case where the backwards evolution is singular at the event horizon (contrast Figure \ref{fig:higherFD} with Figure \ref{fig:schwscathyp}).

By the bijective properties of Theorem \ref{theo1}, we can moreover conclude immediately that \underline{all} scattering data along $\mathcal{H}^+$ and $\mathcal{I}^+$ with \emph{finite} $T$-energy but  with \emph{infinite} weighted norm (as in \eqref{roughnorms1}) will have an \emph{infinite} weighted non-degenerate energy on $\Sigma_0$. The above theorem however does not specify which of  the horizon-localized $N$-energy or  the weighted energy for $\{r>R_0\}$, for some large $R_0>0$, is infinite. The following theorem shows that there are characteristic data for which the solutions specifically have infinite horizon-localized $N$-energy. This immediately implies that the unweighted non-degenerate $N$-energy forward scattering map fails to be invertible, in other words we can find data with finite characteristic $N$-energies but with infinite standard (unweighted) $N$-energy at $\Sigma_0$.

\begin{thmx}
\label{thm:Nenergyinfinite}
There exists solutions $\psi$ to \eqref{eq:waveequation} on ERN that are smooth away from the event horizon $\mathcal{H}^+$ with finite $T$-energy flux along $\mathcal{H}^+$ and future null infinity $\mathcal{I}^+$, such that either:
\begin{enumerate}[\rm (i)]
\item $\psi|_{\mathcal{H}^+}$ vanishes, but $r\psi|_{\mathcal{I}^+}$ satisfies
\begin{equation*}
\int_{\mathcal{I}^{+}\cap\{u\geq 0\}} (1+u)^p(\partial_u(r\psi))^2\,\sin\theta d\theta d\varphi du=\infty\quad \textnormal{if and only if $p\geq 2$}
\end{equation*}
and $\psi$ has infinite unweighted $N$-energy flux along $\Sigma_0\cap \{r\leq r_0\}$, with $r_0>r_+$ arbitrarily close to the horizon radius $r_+$, or
\item
$\psi|_{\mathcal{I}^+}$ vanishes, but $\psi|_{\mathcal{H}^+}$ satisfies
\begin{equation*}
\int_{\mathcal{H}^{+}\cap\{v\geq 0\}} (1+v)^p(\partial_v(r\psi))^2\,\sin\theta d\theta d\varphi dv=\infty \quad \textnormal{if and only if $p\geq 2$}
\end{equation*}
and $\psi$ has infinite weighted $N$-energy flux along $\Sigma\cap \{r\geq R_0\}$ with $R_0>0$ arbitrarily large.
\end{enumerate}
\end{thmx}
\begin{proof}
See Appendix \ref{sec:proofB}.
\end{proof}

The following theorem concerns the scattering of initial data with higher regularity; see Figure \ref{fig:higher21} for an illustration.

\begin{mytheo} (Rough version of Theorem \ref{thm:mainthmho})
The scattering maps defined in the black hole exterior of ERN  
 \begin{itemize}
	 \item between weighted higher-order energy spaces on  ($\mathcal{H}^{-}$, $\mathcal{I}^{-}$) and ($\mathcal{H}^{+}$, $\mathcal{I}^{+}$)
 \end{itemize}
or 
\begin{itemize}
	\item between a weighted higher-order energy space on  ($\mathcal{H}^{+}\cap\mathcal{J}^{+}(\Sigma_0)$, $\mathcal{I}^{+}\cap\mathcal{J}^{+}(\Sigma_0)$) and a \underline{degenerate} higher-order energy space on $\Sigma_0$  
\end{itemize}
 are bounded and bijective.
 \label{theo2} 
\end{mytheo}
\begin{figure}[H]
\centering
\begin{subfigure}[b]{0.3\textwidth}
\centering
\includegraphics[scale=0.35]{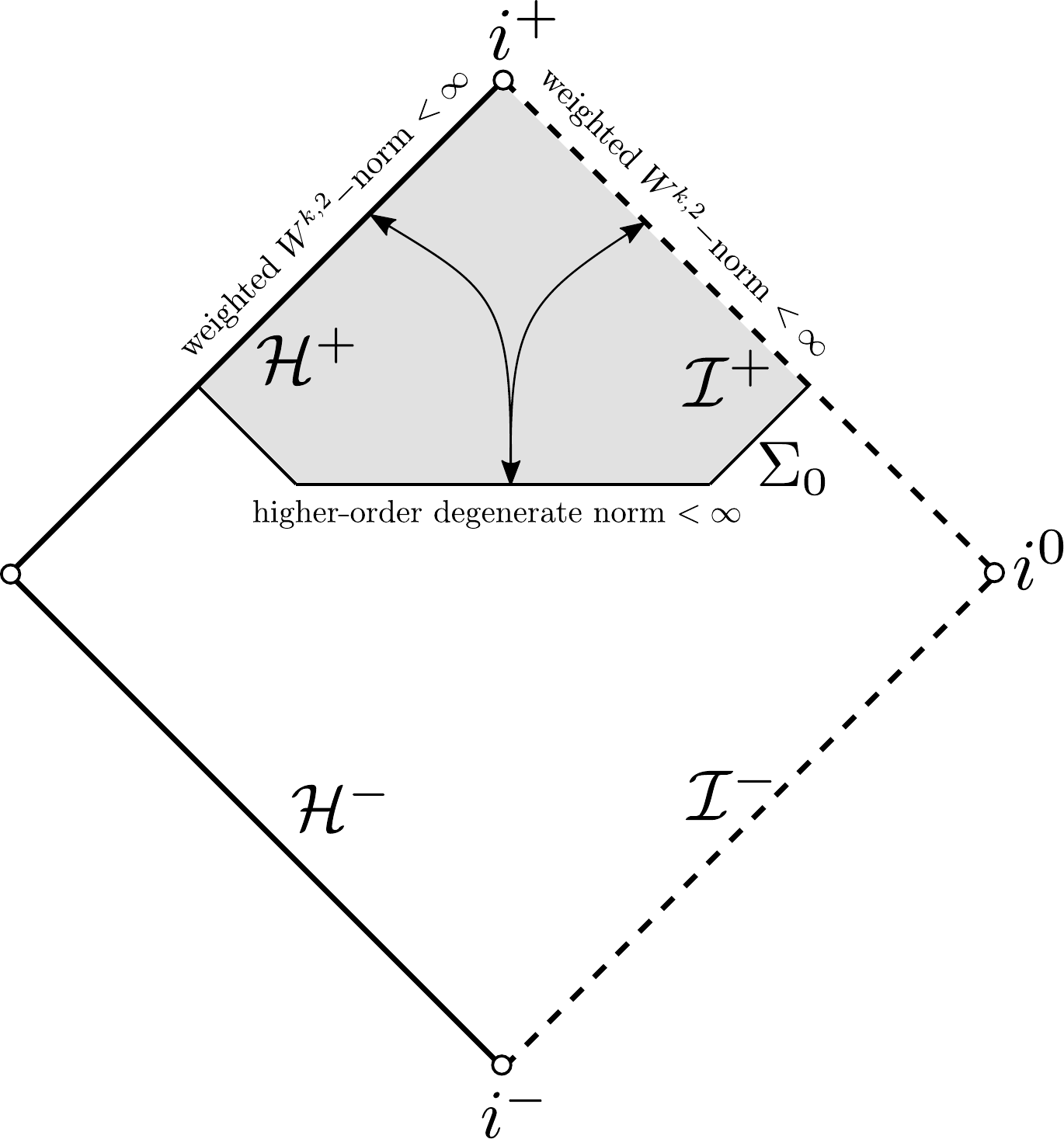}
\end{subfigure}
\centering
\begin{subfigure}[b]{0.3\textwidth}
\centering
\includegraphics[scale=0.35]{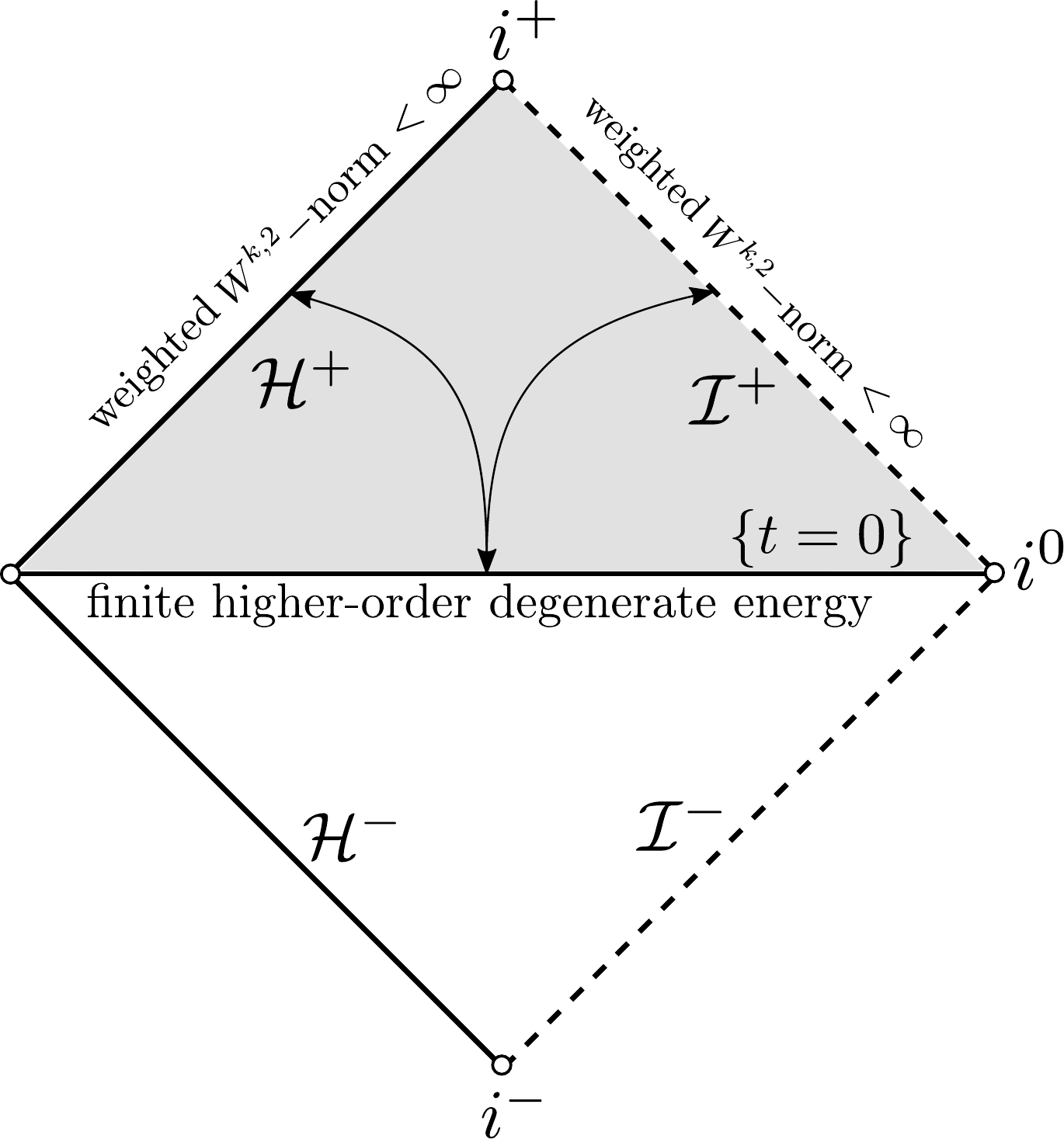}
\end{subfigure}
\begin{subfigure}[b]{0.3\textwidth}
\centering
\includegraphics[scale=0.35]{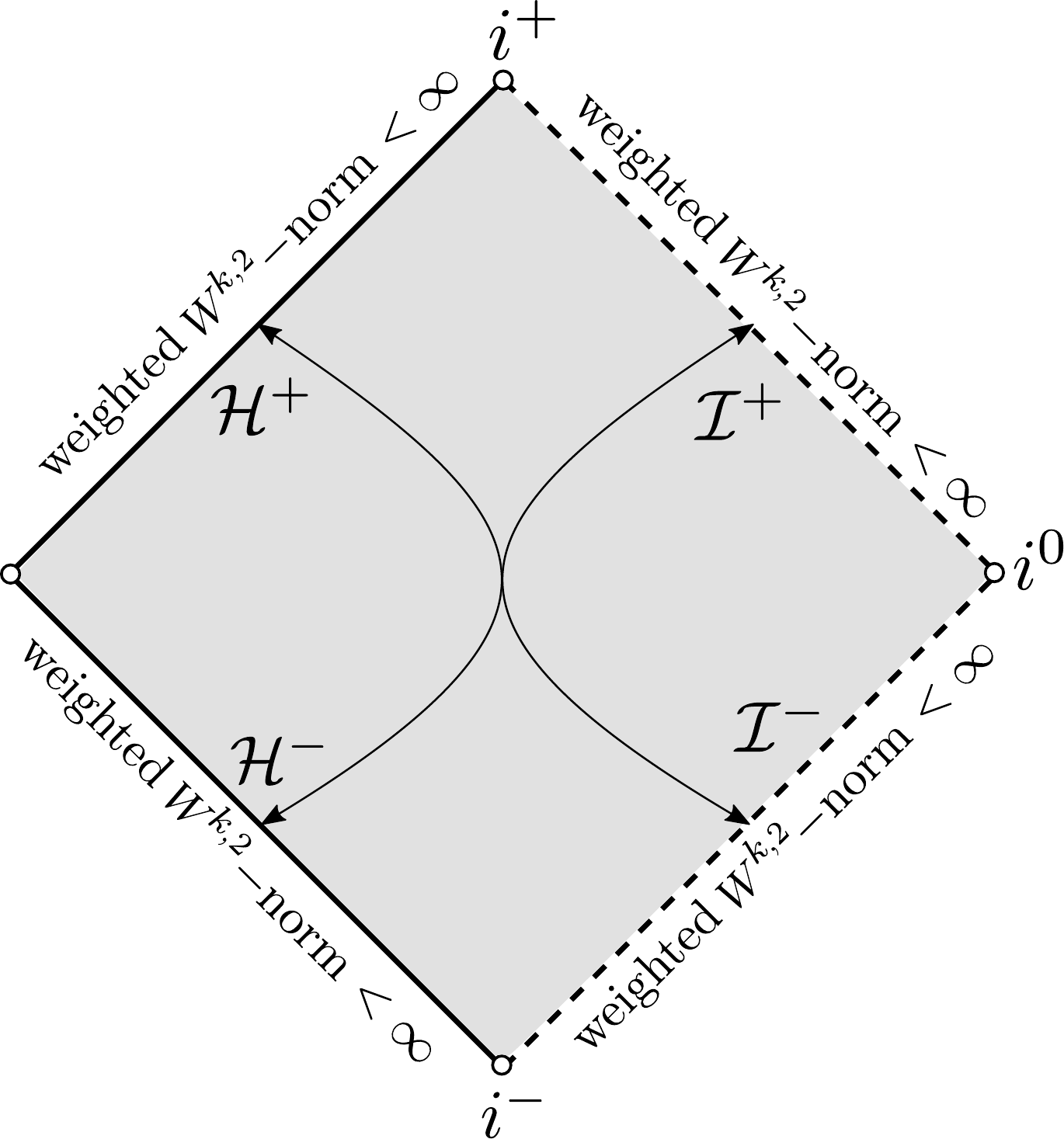}
\end{subfigure}
\caption{Higher-order degenerate scattering theory on ERN.}
	\label{fig:higher21}
\end{figure}
  The above theorem is of particular importance in constructing special solutions with high regularity.  We next present a scattering result for the black hole interior of ERN (Figure \ref{fig:higher210I}) that extends the results derived in \cite{gajic}. 
\begin{mytheo} (Rough version of Theorem \ref{thm:intscat})
The scattering map in the black hole interior of ERN defined between weighted energy spaces is bounded and bijective. 
\end{mytheo}
\begin{figure}[H]
	\begin{center}
\includegraphics[scale=0.5]{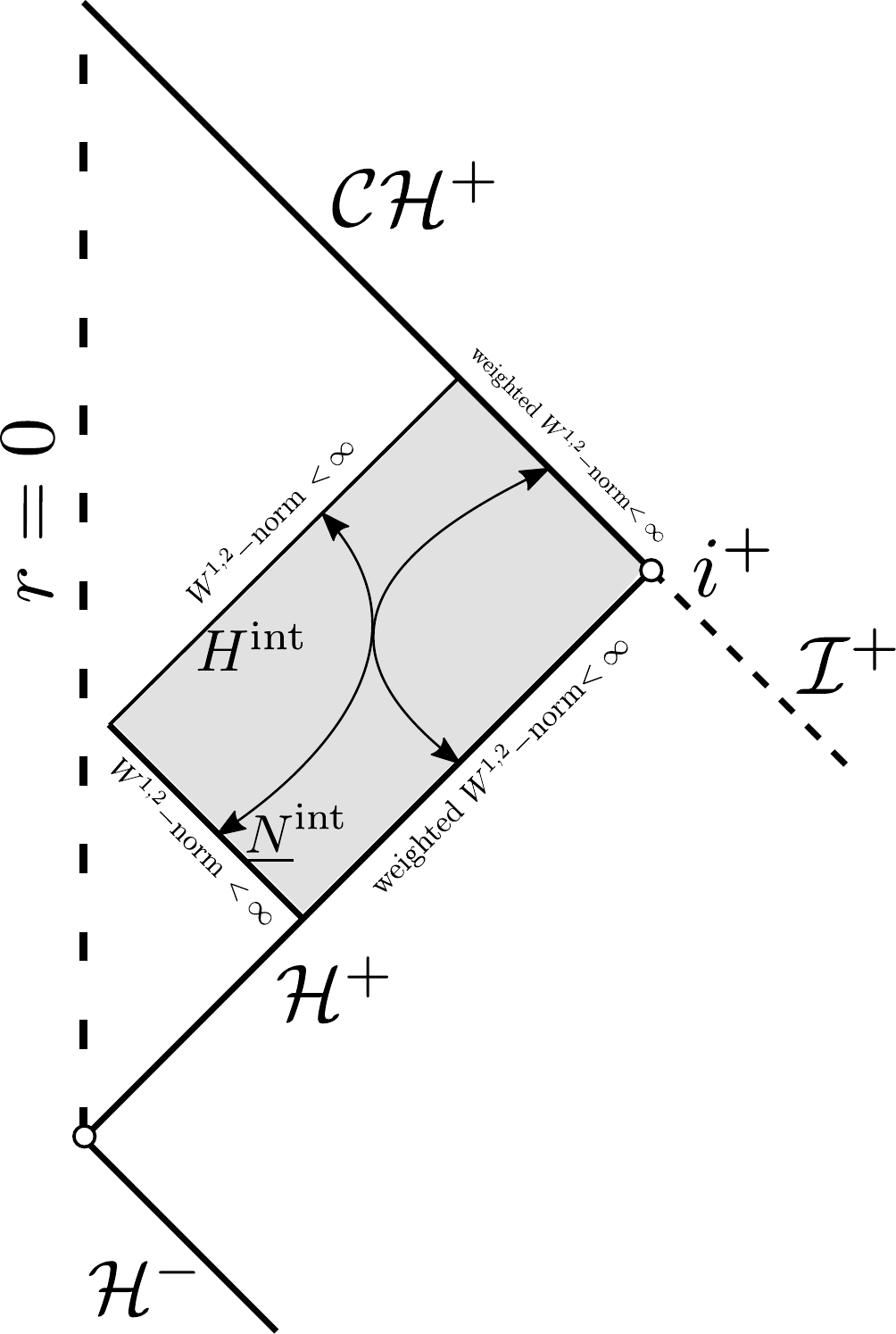}
\end{center}
\vspace{-0.2cm}
\caption{Scattering theory in the black hole interior of ERN.}
	\label{fig:higher210I}
\end{figure}  We will now provide a few applications of the above theorems. The first application has to do with the relation of decay along $\mathcal{H}^{+}$ and $\mathcal{I}^{+}$ and regularity of the data on the hypersurface $\Sigma_0$. 

\begin{mytheo} (Rough version of Theorem \ref{thm:regularity})
Solutions to the wave equation \eqref{eq:waveequation} on ERN with sufficiently fast polynomial decay rates along $\mathcal{H}^{+}$ and $\mathcal{I}^{+}$ have finite $W^{k,2}$ norm in the domain of dependence of $\Sigma_0$. 
\label{theo4} 
\end{mytheo}

 \begin{figure}[H]
	\begin{center}
\includegraphics[scale=0.45]{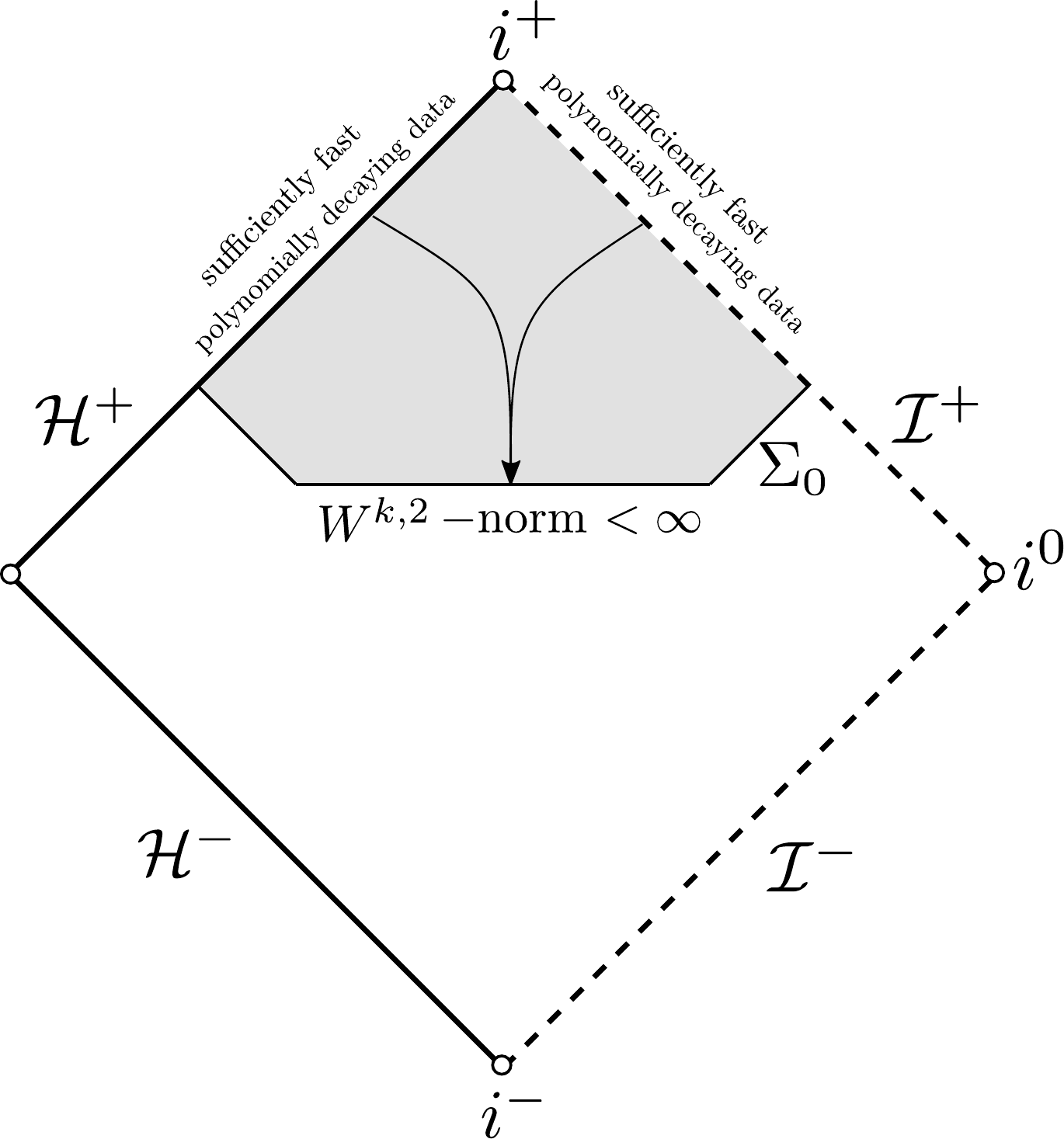}
\end{center}
\vspace{-0.2cm}
\caption{Construction of regular solutions with polynomially decaying scattering data on ERN.  }
	\label{fig:higher212}
\end{figure}
For a precise statement see Theorem \ref{thm:regularity}.
The above theorem relies on a time integral construction and a delicate use of Theorem \ref{theo2}.  Contrast this result with the sub-extremal case where one needs to consider superexponential rates to overcome the (higher-order) blue-shift effect and obtain a similar regularity result in the exterior region up to and including the event horizon. A corollary of this result is the following
\begin{mytheo} (Rough version of Theorem \ref{thm:superpoly})
Consider  smooth scattering data which are exponential in time functions   with identical decay rates on $\mathcal{H}^{+}$ and on $\mathcal{I}^{+}$. There exists a unique exponentially decaying smooth solution to the equation \eqref{eq:waveequation} which admits these data. 
\label{theo5}
\end{mytheo}
We refer to such solutions as mode solutions.  See also Remarks \ref{rm:smoothmodesub} and \ref{remarkquasi} for a discussion about the relation between our modes solutions and the notion of quasinormal mode solutions. 

Finally, we have the following application for the black hole interior of ERN. 
\begin{mytheo} (Rough version of Theorem \ref{thm:interior})
Solutions to the wave equation \eqref{eq:waveequation} on ERN with finite
$\mathcal{E}_{\Sigma_0}$ energy norm on the hypersurface $\Sigma_0$  have finite $W^{1,2}$ norm in the black hole interior region up to and including the Cauchy horizon. 
\label{theo6}
\end{mytheo}
 Contrast Figure \ref{fig:higher2sub1} with Figure \ref{fig:higher2sub2ern} in the sub-extremal case. See also Remark \ref{remarkinterior}. 
\begin{figure}[H]
	\begin{center}
\includegraphics[scale=0.45]{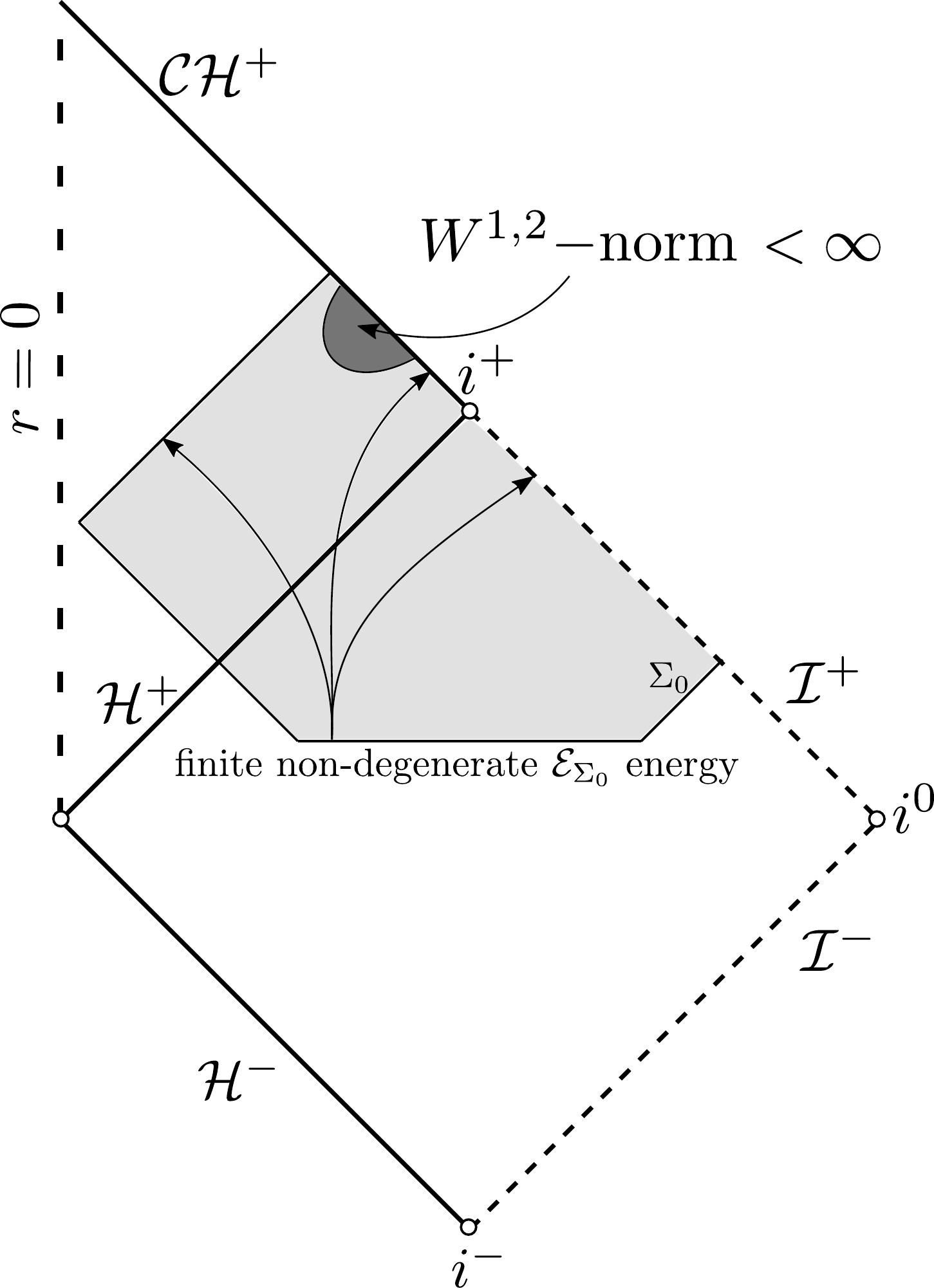}
\end{center}
\vspace{-0.2cm}
\caption{Finiteness of $W^{1,2}$ norm in the black hole interior of ERN.}
	\label{fig:higher2sub1}
\end{figure}

\subsection{Related works}
\label{sec:RelatedWorks}

A closely related topic to the scattering theories on black holes is the black hole stability problem for the forward-in-time evolution. Intense research has been done for both sub-extremal and extremal black holes in this direction. Decay results for the wave equation  on the full sub-extremal Kerr family  were derived in  \cite{part3}. Definitive stability results of the linearized gravity system for Schwarzschild and Reissner--Nordstr\"om were presented in  \cite{Dafermos2016} and \cite{elena1,elena2}, respectively. The non-linear stability of Schwarzschild in a symmetry restricted context was presented in  \cite{klainerman17}. 
The rigorous study of linear waves on extremal black holes was initiated by the second author in \cite{aretakis1, aretakis2, aretakis3, aretakis4, aretakis2012} where it was shown that scalar perturbations are unstable along the event horizon in the sense that higher-order transversal derivatives asymptotically blow up towards the future. The stronger regularity properties of scalar perturbations in the interior of extremal black hole spacetimes compared to sub-extremal black holes was derived by the third author in \cite{gajic,gajic2}.  Precise late-time asymptotics were derived in \cite{paper4}. These asymptotics led to a novel observational signature of ERN \cite{extremal-prl} where it was shown that the horizon instability of ERN is in fact ``observable'' by observers at null infinity. For a detailed study of this signature we refer to the recent \cite{gauravnew}. For works on extremal Kerr spacetimes we refer to the works \cite{hm2012, zimmerman1, harveyeffective}. Extentions of the horizon instability have been presented in various settings \cite{aag1, ori2013,sela, harvey2013, godazgar17, khanna17, cvetic2018}. For a detailed review of scalar perturbations on extremal backgrounds  we refer to  \cite{aretakisbrief}.

\subsection{Discussion on nonlinear problems
}
\label{sec:RelatedWorksdfs}

The methods developed in this article have applications beyond extremal black holes. Indeed, they may be also applied in the construction of non-degenerate scattering theories with weighted energy norms in more general asymptotically flat spacetimes \emph{without} a local redshift effect at the horizon (which acts as a blueshift effect in backwards evolution). One such example would be the Minkowski spacetime; see Section \ref{sec:tech}. Since our methods involve \emph{weighted} and non-degenerate energies, we expect them to be particularly useful for developing a scattering theory for nonlinear wave equations satisfying the classical null condition, as weighted energies need to be controlled in order to obtain global well-posedness for the (forwards) initial value problem \cite{sergiunull}. It would be moreover interesting to explore the generalization of our methods to the setting of perturbations of Minkowski in the context of a scattering problem for the Einstein equations. See also \cite{lindschl17} for work in this direction.

Another interesting direction to explore is the construction of dynamically extremal black holes settling down to extremal Reissner--Nordstr\"om with inverse polynomial rates from initial data along the future event horizon and future null infinity, which would involve a generalization of the backwards evolution estimates in this article to the setting of the Einstein equations. Note that the construction of dynamically extremal black holes settling down \emph{exponentially} follows from an application of the methods of \cite{scattering}. However, whereas it is conjectured in \cite{scattering} that a scattering construction of dynamically sub-extremal black holes settling down inverse polynomially will generically result in spacetimes with a weak null singularity at the event horizon, our methods suggest that the event horizon of dynamically extremal black holes may generically be more regular (with the regularity depending on the assumed polynomial decay rate).

\subsection{Overview of paper}
We provide in this section an overview of the remainder of the paper.
\begin{itemize}
\item In Section \ref{sec:geom}, we introduce the extremal Reissner--Nordstr\"om geometry and spacetime foliations. We also introduce the main notation used throughout the rest of the paper.
\item We introduce in Section \ref{sec:energyspaces} the main Hilbert spaces which appear as domains for our scattering maps.
\item Having introduced the main notation and Hilbert spaces, we subsequently give precise statement of the main theorems of the paper in Section \ref{sec:thms}.
\item In Section \ref{sec:tech}, we outline the main new ideas introduced in the present paper and we provide a sketch of the key proofs.
\item We construct in Section \ref{sec:fowmap} the forwards scattering map $\mathscr{F}$, mapping initial data on a mixed spacelike-null hypersurface to the traces of the radiation field at the future event horizon and future null infinity. We moreover construct restrictions to this map which involve additionally higher-order, degenerate norms.
\item In Section \ref{sec:backwest}, we construct the backwards evolution map $\mathscr{B}$, which send initial data for the radiation field at the future event horizon and future null infinity to the trace of the solution at a mixed spacelike-null hypersurface and is the inverse of $\mathscr{F}	$. Similarly, we construct restrictions of $\mathscr{B}$ involving higher-order, degenerate norms.
\item We prove in Section \ref{sec:spacelikeinfbfsphere} additional energy estimates (in forwards and backwards time direction) that allow us to construct invertible maps $\mathscr{F}_{\pm}$ that send initial data along the asymptotically flat hypersurface $\{t=0\}$ to the future event horizon/null infinity and past event horizon/null infinity, respectively. The composition $\mathscr{S}=\mathscr{F}_+\circ \mathscr{F}_{-}^{-1}$ defines the scattering map, which may be thought of as the key object in our non-degenerate scattering theory.
\item In Section \ref{sec:regcauchyhor} we construct a scattering map $\mathscr{S}^{\rm int}$ in a subset of the black hole interior of extremal Reissner--Nordstr\"om.
\item In the rest of the paper, we provide several applications of the scattering theory developed in the aforementioned sections. In Section \ref{sec:appreg}, we apply the backwards estimates of Section \ref{sec:backwest} to construct arbitrarily regular solutions to \eqref{eq:waveequation} from data along future null infinity and the future event horizon. As a corollary, we construct in Section \ref{sec:appmode} smooth mode solutions from data at infinity and the event horizon.
\end{itemize}

\subsection{Acknowledgements}

The second author (S.A.) acknowledges support through  NSERC grant 502581, an Alfred P. Sloan Fellowship in Mathematics and the Connaught Fellowship 503071.

\section{Geometry and notation}
\label{sec:geom}

\subsection{Black hole exterior}
Consider the 1-parameter family of extremal Reissner-Nordstr\"om spacetimes $(\mathcal{M}^{\rm ext},g_M)$, where\\ $\mathcal{M}^{\rm ext}=\R\times [M,\infty)\times \s$ is a manifold-with-boundary. In $(v,r,\theta,\varphi)$ coordinates, $g$ can be expressed as follows:
\begin{equation}
\label{def:metricext1}
g_M=-Ddv^2+2dvdr+r^2\slashed{g}_{\s^2},
\end{equation}
where $D(r)=(1-Mr^{-1})^2$, with $M>0$ the mass parameter, and $(\theta,\varphi)$ are spherical coordinates on $\s^2$. We denote the boundary as follows $\mathcal{H}^+:=\partial \mathcal{M}^{\rm ext}=\{r=M\}$. We refer to $\mathcal{H}^+$ as the \emph{future event horizon}. The coordinate vector field $T:=\partial_v$ is a Killing vector field that generates the time-translation symmetry of the spacetime.

Consider $u=v-2r_*(r)$, with
\begin{equation*}
r_*(r)=r-M-M^2(r-M)^{-1}+2M \log \left(\frac{r-M}{M}\right).
\end{equation*}
We moreover denote $t=\frac{1}{2}(v+u)$ and we will also employ the notation $u_+:=u$, $u_-:=v$, $v_+:=v$ and $v_-:=u$.

We can change to the coordinate chart $(u,r,\theta, \varphi)$ on the manifold $\mathring{\mathcal{M}}^{\rm ext}=\mathcal{M}^{\rm ext} \setminus \mathcal{H}^+$, in which $g$ can be expressed as follows:
\begin{equation}
\label{def:metricext2}
g_M=-Ddu^2-2dudr+r^2\slashed{g}_{\s^2},
\end{equation}
and $\mathring{\mathcal{M}}^{\rm ext}=\R_u\times (M,\infty)_r\times \s^2$. By employing the coordinate chart $(u,r,\theta,\varphi)$, we can moreover smoothly embed $\mathring{\mathcal{M}}^{\rm ext}$ into a different manifold-with-boundary ${\mathcal{M}'}^{\rm ext}=\R\times [M,\infty)\times \s^2$, where we denote $\mathcal{H}^-:=\partial {\mathcal{M}'}^{\rm ext}=\{r=M\}$. We refer to $\mathcal{H}^-$ as the \emph{past event horizon}. In these coordinates $T=\partial_u$.

Finally, it will also be convenient to employ the Eddington--Finkelstein double null coordinate chart $(u,v,\theta,\varphi)$ in $\mathring{\mathcal{M}}^{\rm ext}$, in which $g$ takes the following form:
\begin{equation}
\label{def:metricext3}
g_M=-Ddudv+r^2\slashed{g}_{\s^2}.
\end{equation}
Here, $(u,v)\in \R\times \R$.

In these coordinates $T=\partial_u+\partial_v$. We moreover introduce the following vector field notation in $(u,v,\theta,\varphi)$ coordinates:
\begin{align*}
L:=&\:\partial_v,\\
\underline{L}:=&\: \partial_u.
\end{align*}
We have that $L(r)=\frac{1}{2}D$ and $\underline{L}(r)=-\frac{1}{2}D$. Note that in $(v,r)$ coordinates, we can express:
\begin{equation*}
\partial_r=2D^{-1}\underline{L}.
\end{equation*}

Let $\slashed{\nabla}$ denote the induced covariant derivative on the spheres of constant $(u,v)$. Then we denote the following rescaled covariant derivative:
\begin{equation*}
\slashed{\nabla}_{\s^2}=r^2\slashed{\nabla}.
\end{equation*}
The rescaled covariant derivative $\slashed{\nabla}_{\s^2}$ is the standard covariant derivative on the unit round sphere.

Consider the following \emph{rescaled} radial coordinate on $\mathring{\mathcal{M}}^{\rm ext}$: $x:=\frac{1}{r}$. The metric $g_M$ takes the following form in $(u,x,\theta,\varphi)$ coordinates:
\begin{equation*}
g_M=-Ddu^2+2r^2dudx+r^2\slashed{g}_{\s^2}.
\end{equation*}

We can then express $\mathring{\mathcal{M}}^{\rm ext}=\R_u\times (0,\frac{1}{M}]_{x}\times \s$. We can embed $\mathcal{M}^{\rm ext}$ into the manifold-with-boundary
\begin{equation*}
 \widehat{\mathcal{M}}^{\rm ext}=\left(\R_u\times \left[0,\frac{1}{M}\right]_{x}\times \s\right)\cup \mathcal{H}^+.
\end{equation*}

We denote $\mathcal{I}^+:=\R_u \times \{0\}_{x} \times \s^2$ and refer to this hypersurface as \emph{future null infinity}. By considering a conformally rescaled metric 
\begin{equation*}
\hat{g}_M=r^{-2}g_M= -Dx^2du^2+2dudx+\slashed{g}_{\s^2}
\end{equation*}
 in $(u,x,\theta, \varphi)$ coordinates, we can extend $\hat{g}_M$ smoothly to $\widehat{\mathcal{M}}^{\rm ext}$ so that $\mathcal{I}^+$ embeds as a genuine null boundary with respect to $\hat{g}_M$. This interpretation, however, will not be necessary for our purposes.

Similarly, we can embed ${\mathcal{M}'^{\rm ext}}=\R_v\times (0,\frac{1}{M}]_{x}\times \s$ into the manifold-with-boundary
\begin{equation*}
{\widehat{\mathcal{M}'}^{\rm ext}}=\R_v\times \left[0,\frac{1}{M}\right]_{x}\times \s
\end{equation*}
and define \emph{past null infinity} as the hypersurface $\mathcal{I}^-:=\R_v \times \{0\}_{x} \times \s^2$, which can be interpreted as a null boundary with respect to a smooth extension of $\hat{g}$.

\subsection{Black hole interior}
By employing $(v,r,\theta,\varphi)$ coordinates it follows immediately that we can smoothly embed $(\mathcal{M}^{\rm ext},g_M)$ into the manifold $\mathcal{M}=\R_v\times (0,\infty)_r\times \s$, where the metric $g$ takes on the form \eqref{def:metricext1}. We will refer to the subset $\mathcal{M}^{\rm int}=\R_v \times (0,M]_r \times \s^2$ as the \emph{black hole interior}. By defining $u=v-2r_*(r)$ in $\mathring{\mathcal{M}}^{\rm int}=\R_v \times (0,M)_r \times \s^2$, with
\begin{equation*}
r_*(r)=r-M+M^2(M-r)^{-1}+2M \log \left(\frac{M-r}{M}\right)
\end{equation*}
we can also introduce $(u,r,\theta,\varphi)$ coordinates on $\mathring{\mathcal{M}}^{\rm int}$, in which the metric takes the expression \eqref{def:metricext2}. In these coordinates, it immediately follows that we can embed $\mathring{\mathcal{M}}^{\rm int}$ into a larger manifold $\widetilde{\mathcal{M}}=\R_u \times (0,\infty)_r\times \s^2$. Let us denote the manifold-with-boundary $\widetilde{\mathcal{M}}^{\rm int}=\R_u \times (0,M]_r \times \s^2$ and the boundary
\begin{equation*}
\mathcal{CH}^+:=\partial \widetilde{\mathcal{M}}^{\rm int}=\{r=M\} \subset \widetilde{\mathcal{M}},
\end{equation*}
which we refer to as the \emph{inner horizon} or the \emph{Cauchy horizon} (the latter terminology follows from the globally hyperbolic spacetime regions considered in Section \ref{sec:foliations}).

Finally, it is also useful to work in Eddington--Finkelstein double-null coordinates $(u,v,\theta,\varphi)$ in $\mathring{\mathcal{M}}^{\rm int}$, in which the metric $g$ takes the form \eqref{def:metricext3}, with $(u,v)\in \{(u',v')\in \R^2\,|\, r(u',v')>0\}$. Furthermore, as in $\mathring{\mathcal{M}}^{\rm ext}$, we have that $L(r)=\frac{1}{2}D(r)$ and $\underline{L}(r)=-\frac{1}{2}D(r)$.\footnote{The positivity of $L(r)$ in $\mathring{\mathcal{M}}^{\rm int}$ illustrates the following characteristic property of extremal Reissner--Nordstr\"om black holes: the spheres foliating the black hole interior are \underline{not} trapped.}

\subsection{Foliations}
\label{sec:foliations}
We introduce the function $v_{\Sigma}(r): [M,\infty) \to \R$, defined as follows: $v_{\Sigma}(M)=v_0>0$ and $\frac{dv_{\Sigma}}{dr}(r)=h(r)$, where $h: [M,\infty) \to \R$ is a piecewise smooth non-negative function satisfying $h(r)=0$, when $r\in [M,r_{\mathcal{H}}]$, with $M<r_{\mathcal{H}}<2M$ and $h=2D^{-1}(r)$ for $r\in [r_{\mathcal{I}},\infty)$, with $2M<r_{\mathcal{I}}<\infty$. Furthermore, in $r_{\mathcal{H}}<r<r_{\mathcal{I}}$, $h$ is smooth and satisfies $0<h<2D^{-1}$. 

Consider the corresponding hypersurface $\Sigma=\{(v,r,\theta,\varphi)\in \mathcal{M}^{\rm ext}\,|\, v=v_{\Sigma}(r)\}$. Then $\underline{N}_{v_0}:=\Sigma|_{\{r\in (M,r_{\mathcal{H}})\}}$ is an ingoing null hypersurface intersecting $\mathcal{H}^+$, tangential to $\underline{L}$ and  ${N}_{u_0}:=\Sigma|_{\{r\in [r_{\mathcal{I}},\infty)\}}$ is an outgoing null hypersurface, tangential to $L$. Furthermore, $\Sigma|_{\{r\in (r_{\mathcal{H}},r_{\mathcal{I}})\}}$ is spacelike. We denote $u_{\Sigma}(r):=v_{\Sigma}(r)-2r_*$ and observe that
\begin{equation*}
u_0:=\lim_{r \to \infty}u_{\Sigma}(r)= u_{r_{\mathcal{H}}}<\infty.
\end{equation*}
Without loss of generality, we can assume that $u_0>0$ (by taking $v_0$ appropriately large for fixed $r_{\mathcal{H}}$ and $r_{\mathcal{I}}$). We will consider the coordinate chart $(\rho:=r|_{\Sigma},\theta,\varphi)$ on $\Sigma$.

We denote with $D^{\pm}(S)$ the future and past domain of dependence, respectively, of a spacelike or mixed spacelike-null hypersurface $S$. Let $\mathcal{R}:=D^+(\Sigma)$. We can foliate $\mathcal{R}$ as follows:
\begin{equation*}
\mathcal{R}=\bigcup_{\tau\in [0,\infty)} \Sigma_{\tau},
\end{equation*}
where $\Sigma_{\tau}$ denote the hypersurfaces induced by flowing $\Sigma$ along $T$, with $\Sigma_0=\Sigma$. 

Denote with ${\underline{N}}_{\tau}=\{v=\tau+v_0,\,M\leq r \leq r_{\mathcal{H}}\}$ the ingoing null part of $\Sigma_{\tau}$ and with\\ ${N}_{\tau}=\{u=\tau+u_0,\, r\geq r_{\mathcal{I}}\}$ the outgoing part.

We can extend $\mathcal{R}$ (with respect to the $(u,x,\theta,\varphi)$ coordinate chart) into the extended manifold-with-boundary $\widehat{\mathcal{M}}^{\rm ext}$ by attaching the boundary $\mathcal{I}^+_{\geq u_0}:=\mathcal{I}^+\cap\{u\geq u_0\}$:
\begin{equation*}
\widehat{\mathcal{R}}:=\mathcal{R} \cup \mathcal{I}^+_{\geq u_0}.
\end{equation*}

Note that we can similarly consider $D^-(\Sigma')$ where $\Sigma'$ is the time-reversed analogue of $\Sigma$ (the roles of $u$ and $v$ reversed) that intersects $\mathcal{H}^-$ and define, with respect to $(v,x,\theta,\varphi)$ coordinates and $v_0'\in \R$ the analogue of $u_0$ and also define $\mathcal{I}^-_{\leq v_0'}:=\mathcal{I}^-\cap\{v\leq v_0'\}$.

The hypersurface $\Sigma$ naturally extends to a hypersurface $\widehat{\Sigma}$ in $\widehat{\mathcal{R}}$, with endpoints on $\mathcal{H}^+$ and $\mathcal{I}^+$, and can be equipped with the coordinate chart $(\chi=x|_{\widehat{\Sigma}},\theta,\varphi)$.

We moreover define $\mathcal{H}^+_{\geq v_0}=\mathcal{H}^+\cap \{v\geq v_0\}$.

Let $u_{\rm int}<0$. We will denote with $N_{v_0}^{\textnormal{int} }$ the hypersurface $\{r(u_{\rm int},v_0)<r<M\,|\,v=v_0\}\subset \mathcal{M}_{\rm int}$. Furthermore, we let
\begin{equation*}
\Sigma^{\rm int}_{0}:=\Sigma_0\cup N_{v_0}^{\textnormal{int} }.
\end{equation*}

We denote furthermore
\begin{align*}
\widetilde{\Sigma}:=&\:\{t=0\},\\
D_{u'}:=&\:D^+\left(\widetilde{\Sigma}\cap \{u\leq u'\}\right),\\
\underline{D}_{v'}:=&\:D^+\left(\widetilde{\Sigma}\cap \{v\leq v'\}\right).
\end{align*}

We foliate the regions $D_{-u_0}$, with $u_0>0$, by outgoing null hypersurfaces that we also denote $N_{u'}$. In this setting $N_{u'}=\{u'=u\:|\: v\geq |u|\}$. It is also useful to consider a foliation by ingoing null hypersurfaces $I_{v'}=\{v=v'\: |\: -v\leq u\leq -u_0\}$.

Similarly, we foliate $\underline{D}_{-v_0}$ by ingoing hypersurfaces $\underline{N}_{v'}=\{v'=v\:|\: u\geq |v|\}$ and outgoing hypersurfaces $H_{u'}=\{u=u'\: |\: -u\leq v\leq -v_0\}$.

We moreover consider the following null hypersurfaces in $D^+(\Sigma_0^{\rm int})\cap \mathring{\mathcal{M}}^{\rm int}$: $\underline{N}_{v'}^{\rm int}=\{v=v'\:|\: |u|\leq |u_{\rm int}|\}$ and $H_{u'}^{\rm int}=\{u=u'\:\: v\geq v_0\}$. We refer to Figure \ref{fig:foliations} for an illustration of the above foliations and hypersurfaces.
 \begin{figure}[H]
	\begin{center}
\includegraphics[scale=0.7]{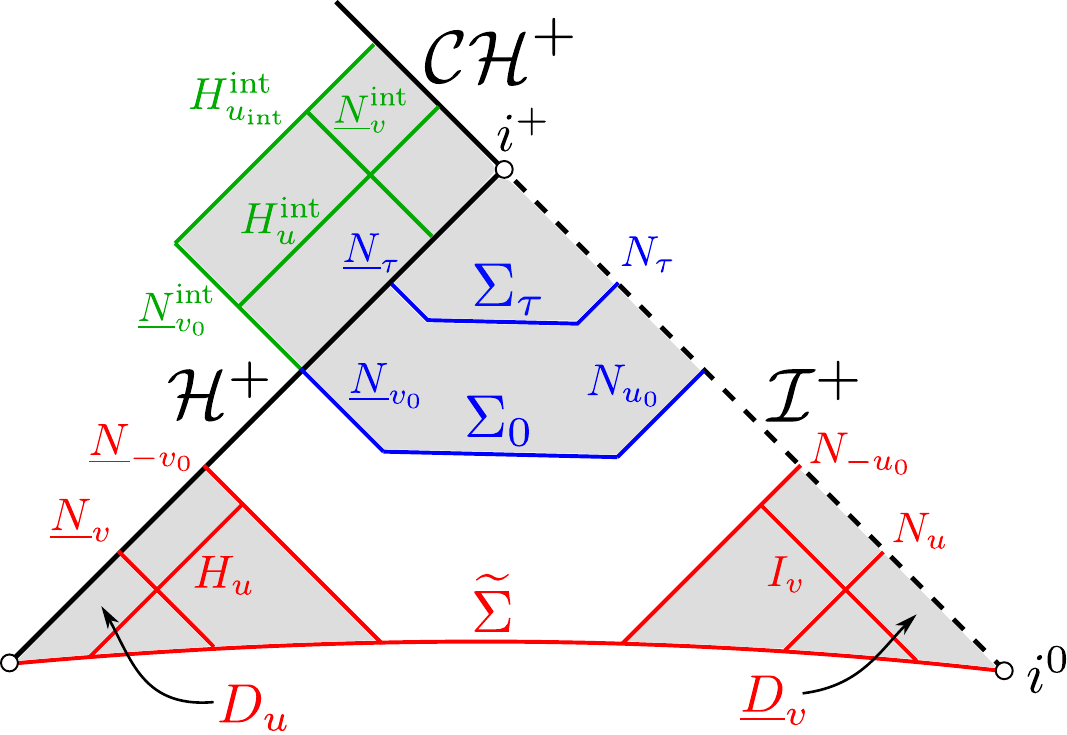}
\end{center}
\vspace{-0.2cm}
\caption{A Penrose diagrammatic representation of the four main spacetime regions (shaded) of the extremal Reissner--Nordstr\"om manifold $\mathcal{M}$ where we derive energy estimates, together with their respective foliations.}
	\label{fig:foliations}
\end{figure}

We use the following notation for the standard volume form on the unit round sphere: $d\omega=\sin \theta d\theta d\varphi$. Let $\mathbf{n}_{\tau}$ and $\mathbf{n}_{\widetilde{\Sigma}}$ be the normal vector fields to $\Sigma_{\tau}$ and $\widetilde{\Sigma}$, respectively. We denote with $d\mu_{\tau}$,  $d\mu_{\widetilde{\Sigma}}$ the induced volume forms on $\Sigma_{\tau}$ and $\widetilde{\Sigma}$ respectively. On the null segments ${N}_{\tau}$ and $\underline{N}_{\tau}$, $\mathbf{n}_{\tau}$ and $d\mu_{\Sigma_{\tau}}$ are not uniquely defined, so we take the following conventions:
\begin{align*}
\mathbf{n}_{\tau}|_{N_{\tau}}= &\:L ,\\
d\mu_{\tau}|_{N_{\tau}}=&\:r^2d\omega dv,\\
\mathbf{n}_{\tau}|_{\underline{N}_{\tau}}= &\:\underline{L}, \\
d\mu_{\tau}|_{\underline{N}_{\tau}}=&\: r^2d\omega du.
\end{align*}

We moreover use the notation $d\mu_{g_M}$ for the natural volume form on $\mathcal{M}_{\rm ext}$ or $\widetilde{\mathcal{M}}_{\rm int}$. Note that in $(u,v,\theta,\varphi)$ coordinates on either $\mathring{\mathcal{M}}_{\rm ext}$ or $\mathring{\mathcal{M}}_{\rm int}$, we can express:
\begin{equation*}
d\mu_{g_M}=Dr^2\, d\omega du dv.
\end{equation*}
We use the notation $d\mu_{\hat{g}_M}$ for the natural volume form on $\widehat{\mathcal{M}}_{\rm ext}$ (corresponding to the metric $\hat{g}_M$). In $(u,x,\theta,\varphi)$ coordinates on $\widehat{\mathcal{M}}_{\rm ext} \setminus \mathcal{H}^+$, we can express:
\begin{equation*}
d\mu_{\hat{g}_M}=d\omega du dx.
\end{equation*}
\subsection{Additional notation}
\label{addnot}

Let $n\in \N_0$. Suppose $K\subset \widehat{\mathcal{R}}$ is compact. Then the Sobolev spaces $W^{n,2}(K)$ are defined in a coordinate-independent way with respect to the following norm:
\begin{equation*}
||f||_{W^{n,2}(K)}^2:=\sum_{0\leq k_1+k_2+k_3\leq n}\int_K |\slashed{\nabla}_{\s^2}^{k_1}(r^2L)^{k_2} (D^{-1}\underline{L})^{k_3}f|^2\,d\mu_{\hat{g}_M}.
\end{equation*}
Recall that we can write in $(v,r,\theta,\varphi)$ coordinates: $2D^{-1}\underline{L}= \partial_r$, which is a regular vector field in $\widehat{\mathcal{R}}$. Furthermore, we can express in $(u,x,\theta,\varphi)$ coordinates:
\begin{equation*}
r^2L=\frac{1}{2}D\partial_x,
\end{equation*}
which implies that $r^2L$ is also regular in $\widehat{\mathcal{R}}$. Hence, $W^{n,2}(K)$ is a natural choice of Sobolev space with respect to the conformal metric $\hat{g}_M$.

If $K^{\rm int}\subset \mathcal{M}_{\rm int}$ is compact, we instead define $W^{n,2}(K^{\rm int})$ in a coordinate-independent way with respect to the following norm:
\begin{equation*}
||f||_{W^{n,2}(K^{\rm int})}^2:=\sum_{0\leq k_1+k_2+k_3\leq n}\int_{K^{\rm int}} |\slashed{\nabla}^{k_1}(D^{-1}L)^{k_2} \underline{L}^{k_3}f|^2\,d\mu_{g_M}
\end{equation*}
In $(u,r,\theta,\varphi)$ coordinates, we can express $2D^{-1}L=\partial_r$, which is a regular vector field in $\widetilde{M}_{\rm int}$. We can also express
\begin{equation*}
\underline{L}=\partial_u-\frac{1}{2}D\partial_r,
\end{equation*}
in $(u,r,\theta,\varphi)$ coordinates, which clearly is also regular $\widetilde{M}_{\rm int}$. We have that $W^{n,2}(K^{\rm int})$ is therefore a natural choice of Sobolev space with respect to $g_M$.

We define the Sobolev spaces $W^{1,2}(\underline{N}^{\textnormal{int}}_{v_0})$ with respect to the following norm:
\begin{equation*}
||f||_{W^{1,2}(\underline{N}^{\textnormal{int}}_{v_0})}^2:=\sum_{0\leq k_1+k_2\leq 1}\int_{r(u_0,v_0)}^{M}\int_{\s^2} |\slashed{\nabla}^{k_1}(D^{-1}\underline{L})^{k_2}f|^2\big|_{\underline{N}^{\textnormal{int}}_{v_0}}\,r^2d\omega dr.
\end{equation*}

Let $f,g$ be positive real-valued functions. We will make use of the notation $f \lesssim g$ when there exists a constant $C>0$ such that $f\leq C \cdot g$. We will denote $f\sim g$ when $f \lesssim g$ and $g \lesssim f$. We will also employ the alternate notation $f \sim_{c,C} g$, with $f,g$ for $0<c\leq C$ positive constants, to indicate:
\begin{equation*}
c \cdot g \leq f \leq C\cdot  g.
\end{equation*}

We use the ``big O'' notation $O((r-M)^p)$ and $O(r^{-p})$, $p\in \R$ to group functions $f$ of $r$ satisfying
\begin{align*}
|f|\lesssim&\: (r-M)^p,\\
|f| \lesssim &\: r^{-p},
\end{align*}
respectively.
\section{Energy spaces}
\label{sec:energyspaces}

\subsection{Main energy spaces}
\label{sec:mainspaces}

In this section, we will introduce the Hilbert spaces on which we will define scattering maps. Before we can do so, we will need existence and uniqueness (in the smooth category) for the Cauchy problem for \eqref{eq:waveequation} on extremal Reissner--Nordstr\"om.

\begin{theorem}
\label{thm:globexuni}
\begin{enumerate}
\item
Consider $(\Psi,\Psi')\in  C^{\infty}(\Sigma_0)\times C^{\infty}(\Sigma_0\cap \{r_{\mathcal{H}}\leq r\leq r_{\mathcal{I}}\})$. Then there exists a unique solution $\psi\in C^{\infty}(D^+(\Sigma_0))$ to \eqref{eq:waveequation} such that $\psi|_{\Sigma_0}=\Psi$ and $\mathbf{n}_{\Sigma_0}\psi|_{\Sigma_0\cap \{r_{\mathcal{H}}\leq r\leq r_{\mathcal{I}}\}}=\Psi'$.
\item Consider $(\Psi,\Psi')\in  C_{c}^{\infty}(\widetilde{\Sigma})\times  C_{c}^{\infty}(\widetilde{\Sigma})$. Then there exists a unique solution $\psi\in C^{\infty}(D^+(\widetilde{\Sigma})\cup \mathcal{H}^+)$ to \eqref{eq:waveequation} such that $\psi|_{\widetilde{\Sigma}}=\Psi$ and $\mathbf{n}_{\widetilde{\Sigma}}\psi|_{\widetilde{\Sigma}}=\Psi'$.
\item Consider characteristic initial data $\Psi\in C^0(\mathcal{H}^+_{\geq v_0}\cup \underline{N}_{v_0}^{\rm int})$, with $\Psi|_{\mathcal{H}^+_{\geq v_0}}\in C^{\infty}(\mathcal{H}^+_{\geq v_0})$ and $\Psi|_{ \underline{N}_{v_0}^{\rm int}} \in  C^{\infty}(\underline{N}_{v_0}^{\rm int})$. Then there exists a unique solution $\psi\in C^{\infty}(D^+(\Sigma_0^{\rm int}) \cap \mathcal{M}^{\rm int})$ to \eqref{eq:waveequation} such that\\ $\psi|_{\mathcal{H}^+_{\geq v_0}\cup \underline{N}_{v_0}^{\rm int}}=\Psi$.
\end{enumerate}
\end{theorem}

We denote with $C^{\infty}(\widehat{\Sigma}_0)$ the space of smooth functions on the hypersurface $\widehat{\Sigma}_0$, with respect to the coordinate chart $(\chi, \theta, \varphi)$ introduced in Section 
\ref{sec:foliations}. We denote with $C^{\infty}(\Sigma_0\cap \{r_{\mathcal{H}}\leq r\leq r_{\mathcal{I}}\})$ the space of smooth function on the restriction $\Sigma_0\cap \{r_{\mathcal{H}}\leq r\leq r_{\mathcal{I}}\}$, with respect to the coordinate chart $(\rho, \theta, \varphi)$.

Let us introduce the \emph{stress-energy tensor} $\mathbf{T}[\psi]$ of \eqref{eq:waveequation}, defined as follows with respect to a coordinate basis:
\begin{equation*}
\mathbf{T}_{\alpha \beta}[\psi]:=\partial_{\alpha}\psi\partial_{\beta}\psi-\frac{1}{2}g_{\alpha \beta}(g^{-1})^{\kappa \lambda}\partial_{\kappa}\psi \partial_{\lambda}\psi.
\end{equation*}
Given a vector field $X$ on $\mathcal{M}$, we define the corresponding $X$-\emph{energy current} $\mathbf{J}^X$ as follows:
\begin{equation*}
(\mathbf{J}^X[\psi])_{\alpha}=\mathbf{T}_{\alpha \beta}[\psi]X^{\beta}.
\end{equation*}

We will denote the \emph{radiation field} of $\psi$ as follows:
\begin{equation*}
\phi:=r \psi.
\end{equation*}

We define the following energy space
\begin{definition}
Define the norm  $||\cdot||_{\mathcal{E}^T_{\Sigma_0}}$  as follows: let $(r\Psi,\Psi')\in C^{\infty}(\widehat{\Sigma}_0)\times C^{\infty}(\Sigma_0\cap \{r_{\mathcal{H}}\leq r\leq r_{\mathcal{I}}\})$, then
\begin{equation*}
\begin{split}
||(\Psi,\Psi')||_{\mathcal{E}^T_{\Sigma_0}}^2:=&\int_{\Sigma_0} {\mathbf{J}}^T[\psi]\cdot \mathbf{n}_{0}\,d\mu_{0}.
\end{split}
\end{equation*}
where $\psi$ denotes the (unique) smooth local extension of $\Psi$ in $\mathcal{R}$ that satisfies $\psi|_{\Sigma_0}=\Psi$ and\\ $\mathbf{n}_{\Sigma_0}\psi|_{\Sigma_0\cap \{r_{\mathcal{H}}\leq r\leq r_{\mathcal{I}}\}}=\Psi'$ and solves \eqref{eq:waveequation} (see Theorem \ref{thm:globexuni}), so that all derivatives of $\psi$ above can be expressed solely in terms of derivatives of $\Psi$ and $\Psi'$.

We also define the norm $||\cdot||_{\mathcal{E}_{\Sigma_0}}$ on $C^{\infty}(\widehat{\Sigma}_0)\times C^{\infty}(\Sigma_0\cap \{r_{\mathcal{H}}\leq r\leq r_{\mathcal{I}}\})$ as follows:
\begin{equation*}
\begin{split}
||(\Psi,\Psi')||_{\mathcal{E}_{\Sigma_0}}^2:=&\: \sum_{j=0}^1 \int_{{\underline{N}}_{v_0}}(r-M)^{-2+j}(\underline{L}T^{j}\phi)^2\,d\omega du+\int_{{N}_{u_0}}r^{2-j}(LT^{j}\phi)^2\,d\omega dv+\sum_{j=0}^2\int_{\Sigma_0}{\mathbf{J}}^T[T^j\psi]\cdot \mathbf{n}_{0}\,d\mu_{0},
\end{split}
\end{equation*}

We denote with $\mathcal{E}^T_{\Sigma_0}$ and $\mathcal{E}_{\Sigma_0}$ the completions of  $C^{\infty}(\widehat{\Sigma}_0)\times C^{\infty}(\Sigma_0\cap \{r_{\mathcal{H}}\leq r\leq r_{\mathcal{I}}\})$ with respect to the norms $||\cdot||_{\mathcal{E}^T_{\Sigma_0}}$ and $||\cdot ||_{\mathcal{E}_{\Sigma_0}}$, respectively. Note that, by construction,
\begin{equation*}
\mathcal{E}_{\Sigma_0} \subset \mathcal{E}^T_{\Sigma_0}.
\end{equation*}
\end{definition}

\begin{definition}
Define the norm  $||\cdot||_{\mathcal{E}^T_{\widetilde{\Sigma}}}$  as follows: let $(\Psi,\Psi')\in C_{c}^{\infty}(\widetilde{\Sigma})\times C_{c}^{\infty}(\widetilde{\Sigma})$, then
\begin{equation*}
\begin{split}
||(\Psi,\Psi')||_{\mathcal{E}^T_{\widetilde{\Sigma}}}^2:=&\int_{\widetilde{\Sigma}} {\mathbf{J}}^T[\psi]\cdot \mathbf{n}_{\widetilde{\Sigma}}\,d\mu_{\widetilde{\Sigma}},
\end{split}
\end{equation*}
where $\psi$ denotes the (unique) smooth local extension of $\Psi$ to $D^+(\widetilde{\Sigma})$ that satisfies $\psi|_{\widetilde{\Sigma}}=\Psi$ and $\mathbf{n}_{\widetilde{\Sigma}}\psi|_{\widetilde{\Sigma}}=\Psi'$ and solves \eqref{eq:waveequation} (see Theorem \ref{thm:globexuni}), so that all derivatives of $\psi$ above can be expressed solely in terms of derivatives of $\Psi$ and $\Psi'$.

We also define the norm $||\cdot||_{\mathcal{E}_{\widetilde{\Sigma}}}$ on $C_{c}^{\infty}(\widetilde{\Sigma})\times C_{c}^{\infty}(\widetilde{\Sigma})$ as follows:
\begin{equation*}
\begin{split}
||(\Psi,\Psi')||_{\mathcal{E}_{\widetilde{\Sigma}}}^2:=&\: \sum_{j=0}^1\int_{\widetilde{\Sigma}\cap\{v\geq v_{r_{\mathcal{I}}}\}} r^{2-j}(LT^j\phi)^2+r^{2-j}(\underline{L}T^j\phi)^2+r^{-j}|\snabla_{\s^2}T^j\phi|^2\,dv\\
&+\int_{\widetilde{\Sigma}\cap\{u\geq u_{r_{\mathcal{H}}}\}} (r-M)^{-2+j}(LT^j\phi)^2+(r-M)^{-2+j}(\underline{L}T^j\phi)^2+(r-M)^j|\snabla_{\s^2}T^j\phi|^2\,dv\\
&+\sum_{m+2|\alpha|\leq 2}\int_{\widetilde{\Sigma}}{\mathbf{J}}^T[T^m\Omega^{\alpha}\psi]\cdot \mathbf{n}_{\widetilde{\Sigma}}\,d\mu_{\widetilde{\Sigma}}.
\end{split}
\end{equation*}

We denote with $\mathcal{E}^T_{\widetilde{\Sigma}}$ and $\mathcal{E}_{\widetilde{\Sigma}}$ the completions of  $C_{c}^{\infty}(\widetilde{\Sigma})\times C_{c}^{\infty}(\widetilde{\Sigma})$ with respect to the norms $||\cdot||_{\mathcal{E}^T_{\widetilde{\Sigma}}}$ and $||\cdot ||_{\mathcal{E}_{\widetilde{\Sigma}}}$, respectively. Note that, by construction,
\begin{equation*}
\mathcal{E}_{\widetilde{\Sigma}}\subset \mathcal{E}^T_{\widetilde{\Sigma}}.
\end{equation*}
\end{definition}

We denote with $C_{c}^{\infty}(\mathcal{H}^+_{\geq v_0})$ and $C_{c}^{\infty}(\mathcal{I}^+_{\geq u_0})$ the spaces of smooth, compactly supported functions on $\mathcal{H}_{\geq v_0}^+$ and $\mathcal{I}^+_{\geq u_0}$, respectively.

\begin{definition}
Let $u_0,v_0>0$. Define the norms $||\cdot||_{\mathcal{E}^T_{\mathcal{H}^+_{\geq v_0}}}$ and $||\cdot||_{\mathcal{E}^T_{\mathcal{I}^+_{\geq u_0}}}$ as follows: let $({\underline{\Phi}},{\Phi})\in C_{c}^{\infty}(\mathcal{H}^+_{\geq v_0})\oplus C_{c}^{\infty}(\mathcal{I}^+_{\geq u_0})$, then
\begin{align*}
||{\underline{\Phi}}||_{\mathcal{E}^T_{\mathcal{H}^+_{\geq v_0}}}^2:=&\:\int_{\mathcal{H}^+_{\geq v_0}}(\partial_v{\underline{\Phi}})^2\,d\omega dv,\\
||{\Phi}||_{\mathcal{E}^T_{\mathcal{I}^+_{\geq u_0}}}^2:=&\int_{\mathcal{I}^+_{\geq u_0}}(\partial_u{\Phi})^2\,d\omega du.
\end{align*}

We also define the norms $||\cdot||_{\mathcal{E}_{\mathcal{H}^+_{\geq v_0}}}$ and $||\cdot||_{\mathcal{E}_{\mathcal{I}^+_{\geq u_0}}}$ as follows: let $({\underline{\Phi}},{\Phi})\in C_{c}^{\infty}(\mathcal{H}^+_{\geq v_0})\oplus C_{c}^{\infty}(\mathcal{I}^+_{\geq u_0})$, then
\begin{align*}
||{\underline{\Phi}}||_{\mathcal{E}_{\mathcal{H}^+_{\geq v_0}}}^2:=&\:\sum_{j=0}^2\int_{\mathcal{H}^+_{\geq v_0}} v^{2-j}(\partial_v^{j+1}{\underline{\Phi}})^2+|\snabla_{\s^2}{\underline{\Phi}}|^2\,d\omega dv,\\
||{\Phi}||_{\mathcal{E}_{\mathcal{I}^+_{\geq u_0}}}^2:=&\sum_{j=0}^2\int_{\mathcal{I}^+_{\geq u_0}} u^{2-j}(\partial_u^{j+1}{\Phi})^2+ |\snabla_{\s^2}{\Phi}|^2\,d\omega du.
\end{align*}

Then we denote with $\mathcal{E}^T_{\mathcal{H}^+_{\geq v_0}}\oplus \mathcal{E}^T_{\mathcal{I}^+_{\geq u_0}}$ and $\mathcal{E}_{\mathcal{H}^+_{\geq v_0}}\oplus \mathcal{E}_{\mathcal{I}^+_{\geq u_0}}$ the completions of  $C_{c}^{\infty}(\mathcal{H}^+_{\geq v_0})\oplus C_{c}^{\infty}(\mathcal{I}^+_{\geq u_0})$ with respect to the product norms associated to $||\cdot||_{\mathcal{E}^T_{\mathcal{H}^+_{\geq v_0}}}$ and $||\cdot||_{\mathcal{E}^T_{\mathcal{I}^+_{\geq u_0}}}$, respectively. 

Note that
\begin{equation*}
\mathcal{E}_{\mathcal{H}^+_{\geq v_0}}\oplus \mathcal{E}_{\mathcal{I}^+_{\geq u_0}} \subset \mathcal{E}^T_{\mathcal{H}^+_{\geq v_0}}\oplus \mathcal{E}^T_{\mathcal{I}^+_{\geq u_0}}.
\end{equation*}
\end{definition}

\begin{definition}
Define the norms $||\cdot||_{\mathcal{E}^T_{\mathcal{H}^{\pm}}}$ and $||\cdot||_{\mathcal{E}^T_{\mathcal{I}^{\pm}}}$ on respectively $\mathcal{H}^{\pm}$ and $\mathcal{I}^{\pm}$ as follows: let $({\underline{\Phi}},{\Phi})\in C_{c}^{\infty}(\mathcal{H}^{\pm})\oplus C_{c}^{\infty}(\mathcal{I}^{\pm})$, then
\begin{align*}
||{\underline{\Phi}}||_{\mathcal{E}^T_{\mathcal{H}^{\pm}}}^2:=&\:\int_{\mathcal{H}^{\pm}}(\partial_{v_{\pm}}{\underline{\Phi}})^2\,d\omega dv_{\pm},\\
||{\Phi}||_{\mathcal{E}^T_{\mathcal{I}^{\pm}}}^2:=&\int_{\mathcal{I}^{\pm}}(\partial_{u_{\pm}}{\Phi})^2\,d\omega du_{\pm},
\end{align*}
with respect to the coordinate charts $(u_{\pm}, v_{\pm},\theta,\varphi)$.

We also define the norms $||\cdot||_{\mathcal{E}_{\mathcal{H}^{\pm}}}$ and $||\cdot||_{\mathcal{E}_{\mathcal{I}^{\pm}}}$ on respectively $\mathcal{H}^{\pm}$ and $\mathcal{I}^{\pm}$ as follows: let $({\underline{\Phi}},{\Phi})\in C_{c}^{\infty}(\mathcal{H}^{\pm})\oplus C_{c}^{\infty}(\mathcal{I}^{\pm})$, then
\begin{align*}
||{\underline{\Phi}}||_{\mathcal{E}_{\mathcal{H}^{\pm}}}^2:=&\:\sum_{j=0}^2\int_{\mathcal{H}^{\pm}} (1+|v_{\pm}|)^{2-j}(\partial_{v_{\pm}}^{j+1}{\underline{\Phi}})^2+|\snabla_{\s^2}{\underline{\Phi}}|^2+|\snabla_{\s^2}\partial_{v_{\pm}}{\underline{\Phi}}|^2\,d\omega dv_{\pm},\\
||{\Phi}||_{\mathcal{E}_{\mathcal{I}^{\pm}}}^2:=&\sum_{j=0}^2\int_{\mathcal{I}^{\pm}} (1+|u_{\pm}|)^{2-j}(\partial_{u_{\pm}}^{j+1}{\Phi})^2+ |\snabla_{\s^2}{\Phi}|^2+ |\snabla_{\s^2}\partial_{u_{\pm}}{\Phi}|^2\,d\omega du_{\pm}.
\end{align*}

Then we denote with $\mathcal{E}^T_{\mathcal{H}^{\pm}}\oplus \mathcal{E}^T_{\mathcal{I}^{\pm}}$ and $\mathcal{E}_{\mathcal{H}^{\pm}}\oplus \mathcal{E}_{\mathcal{I}^{\pm}}$ the \emph{completions} of  $C_{c}^{\infty}(\mathcal{H}^{\pm})\oplus C_{c}^{\infty}(\mathcal{H}^{\pm})$ with respect to the product norms associated to $||\cdot||_{\mathcal{E}^T_{\mathcal{H}^{\pm}}}$ and $||\cdot||_{\mathcal{E}^T_{\mathcal{H}^{\pm}}}$, and $||\cdot||_{\mathcal{E}_{\mathcal{H}^{\pm}}}$ and $||\cdot||_{\mathcal{E}_{\mathcal{H}^{\pm}}}$,respectively. 

Note that
\begin{equation*}
\mathcal{E}_{\mathcal{H}^{\pm}}\oplus \mathcal{E}_{\mathcal{I}^{\pm}} \subset \mathcal{E}^T_{\mathcal{H}^{\pm}}\oplus \mathcal{E}^T_{\mathcal{I}^{\pm}}.
\end{equation*}
\end{definition}

\subsection{Degenerate higher-order energy spaces}
In this section, we will introduce analogues of the Hilbert spaces introduced in Section \ref{sec:mainspaces}, but with norms depending on degenerate higher-order derivatives.
\begin{definition}
Define the norm  $||\cdot||_{\mathcal{E}_{n;\Sigma_0}}$  as follows: let $(r\Psi,\Psi')\in C^{\infty}(\widehat{\Sigma}_0)\times C^{\infty}(\Sigma_0\cap \{r_{\mathcal{H}}\leq r\leq r_{\mathcal{I}}\})$, then
\begin{equation*}
\begin{split}
||(\Psi,\Psi')||_{\mathcal{E}_{n; \Sigma_0}}^2:=&\: \sum_{j=0}^1 \sum_{m+2|\alpha|+2k\leq 2n} \int_{{N}_{u_0}} r^{2+2k-j}(L^{1+k}T^{m+j}\Omega^{\alpha}\phi)^2\,d\omega dv\\
&+ \int_{\underline{N}_{v_0}} (r-M)^{-2-2k+j}(\underline{L}^{1+k}T^{m+j}\Omega^{\alpha}\phi)^2\,d\omega du\\
&+\sum_{\substack{m+2|\alpha|\leq 2n+2\\ |\alpha|\leq n}} \int_{\Sigma_{0}} {\mathbf{J}}^T[T^m \Omega^{\alpha}\psi]\cdot \mathbf{n}_{0}\, d\mu_{0}.
\end{split}
\end{equation*}
We denote with $\mathcal{E}_{n; \Sigma_0}$ the completion of  $C^{\infty}(\widehat{\Sigma}_0)\times C^{\infty}(\Sigma_0\cap \{r_{\mathcal{H}}\leq r\leq r_{\mathcal{I}}\})$ with respect to the norm $||\cdot ||_{\mathcal{E}_{n; \Sigma_0}}$.
\end{definition}

\begin{definition}
Define the norm  $||\cdot||_{\mathcal{E}_{n; \widetilde{\Sigma}}}$  as follows: let $(\Psi,\Psi')\in C_{c}^{\infty}(\widetilde{\Sigma}) \times C_{c}^{\infty}(\widetilde{\Sigma})$, then
\begin{equation*}
\begin{split}
||(\Psi,\Psi')||_{\mathcal{E}_{n; \widetilde{\Sigma}}}^2:=&\sum_{j=0}^2\sum_{m+2|\alpha|+2k\leq 2n}\int_{\widetilde{\Sigma}\cap\{ v\geq v_{r_{\mathcal{I}}}\}} r^{2+2k-j}(L^{k+1}\Omega^{\alpha}T^{j+m}\phi)^2+r^{2k-j}|\snabla_{\s^2}L^{k}\Omega^{\alpha}T^{j+m}\phi|^2 \\ 
&+r^{2k+2-j}(\Lbar^{k+1}\Omega^{\alpha}T^{j+m}\phi)^2+r^{2k-j}|\snabla_{\s^2}\Lbar^{k}\Omega^{\alpha}T^{j+m}\phi|^2\,d\omega dr\\ 
&+\sum_{m+2|\alpha|+2k\leq 2n}\int_{\widetilde{\Sigma}\cap\{v \geq v_{r_{\mathcal{I}}}\}} r^{2k}|\snabla_{\s^2}L^{k+1}\Omega^{\alpha}T^m\phi|^2+r^{2k-2}|\snabla_{\s^2}^2L^{k}\Omega^{\alpha}T^m\phi|^2 \\
&+r^{2k}|\snabla_{\s^2}\Lbar^{k+1}\Omega^{\alpha}T^m\phi|^2+r^{2k-2}|\snabla_{\s^2}^2\Lbar^{k}\Omega^{\alpha}T^m\phi|^2\,d\omega dr\\
&+\sum_{j=0}^2\sum_{m+2|\alpha|+2k\leq 2n}\int_{\widetilde{\Sigma}\cap\{u \geq u_{r_{\mathcal{H}}}\}} (r-M)^{-2k-2+j}(L^{k+1}\Omega^{\alpha}T^{j+m}\phi)^2+(r-M)^{-2k+j}|\snabla_{\s^2}L^{k}\Omega^{\alpha}T^{j+m}\phi|^2 \\ 
&+(r-M)^{-2k-2+j}(\Lbar^{k+1}\Omega^{\alpha}T^{j+m}\phi)^2+(r-M)^{-2k+j}|\snabla_{\s^2}\Lbar^{k}\Omega^{\alpha}T^{j+m}\phi|^2\,d\omega dr_*\\
&+\sum_{m+2|\alpha|+2k\leq 2n}\int_{\widetilde{\Sigma}\cap\{u \geq u_{r_{\mathcal{H}}}\}} (r-M)^{-2k}|\snabla_{\s^2}L^{k+1}\Omega^{\alpha}T^m\phi|^2+(r-M)^{-2k+2}|\snabla_{\s^2}^2L^{k}\Omega^{\alpha}T^m\phi|^2 \\ 
&+(r-M)^{-2k}|\snabla_{\s^2}\Lbar^{k+1}\Omega^{\alpha}T^m\phi|^2+(r-M)^{-2k+2}|\snabla_{\s^2}^2\Lbar^{k}\Omega^{\alpha}T^m\phi|^2\,d\omega dr_*\\
&+\sum_{m+2|\alpha|\leq 2n+2} \int_{\widetilde{\Sigma}} {\mathbf{J}}^T[T^m \Omega^{\alpha}\psi]\cdot \mathbf{n}_{\widetilde{\Sigma}}\, d\mu_{\widetilde{\Sigma}},
\end{split}
\end{equation*}
where $\psi$ denotes the smooth extension of $\Psi$ to $\mathcal{R}$ that satisfies $\psi|_{\widetilde{\Sigma}}=\Psi$ and $\mathbf{n}_{\widetilde{\Sigma}}\psi|_{\widetilde{\Sigma}}=\Psi'$ and solves \eqref{eq:waveequation} (see Theorem \ref{thm:globexuni}), so that all derivatives of $\psi$ above can be expressed solely in terms of derivatives of $\Psi$ and $\Psi'$.

We denote with $\mathcal{E}_{n; \widetilde{\Sigma}}$ the completion of  $C_{c}^{\infty}(\widetilde{\Sigma}) \times C_{c}^{\infty}(\widetilde{\Sigma})$ with respect to the norm $||\cdot ||_{\mathcal{E}_{\widetilde{\Sigma}}}$.
\end{definition}
\begin{definition}
Let $n\in \N_0$ and $u_0,v_0>0$. Define the higher-order norms $||\cdot||_{\mathcal{E}_{n;\mathcal{H}^+_{\geq v_0}}}$ and $||\cdot||_{ \mathcal{E}_{n;\mathcal{I}^+_{\geq u_0}}}$ as follows: let $(\underline{\Phi},\Phi)\in C_{c}^{\infty}(\mathcal{H}^+_{\geq v_0})\oplus C_{c}^{\infty}(\mathcal{I}^+_{\geq u_0})$, then
\begin{align*}
||{\underline{\Phi}}||_{\mathcal{E}_{n;\mathcal{H}^+_{\geq v_0}}}^2:=&\:\sum_{j=0}^2\sum_{m+2k+2|\alpha|\leq 2n}\int_{ \mathcal{H}^+_{\geq v_0}} v^{2k+2-j} (L^{1+k+m+j} \Omega^{\alpha}\underline{\Phi})^2+ |\snabla_{\s^2} L^m\Omega^{\alpha} \underline{\Phi}|^2\,d\omega dv,\\ 
||{\Phi}||_{ \mathcal{E}_{n;\mathcal{I}^+_{\geq u_0}}}^2:=&\:\sum_{j=0}^2\sum_{m+2k+2|\alpha|\leq 2n}\int_{ \mathcal{I}^+_{\geq u_0}} u^{2k+2-j} (\underline{L}^{1+k+m+j} \Omega^{\alpha}\Phi)^2+|\snabla_{\s^2} \underline{L}^m\Omega^{\alpha} \Phi|^2\,d\omega du
\end{align*}
Then we denote with $\mathcal{E}_{n;\mathcal{H}^+_{\geq v_0}}\oplus \mathcal{E}_{n;\mathcal{I}^+_{\geq u_0}}$ the \emph{completion} of  $C_{c}^{\infty}(\mathcal{H}^+_{\geq v_0})\oplus C_{c}^{\infty}(\mathcal{I}^+_{\geq u_0})$ with respect to the norms $||\cdot||_{\mathcal{E}_{2;n}(\mathcal{H}^+_{\geq v_0})}$ and $||\cdot||_{ \mathcal{E}_{n;\mathcal{I}^+_{\geq u_0}}}$. 
\end{definition}

Note that for all $n\in \N_0$,
\begin{equation*}
\mathcal{E}_{n;\mathcal{H}^+_{\geq v_0}}\oplus \mathcal{E}_{n;\mathcal{I}^+_{\geq u_0}} \subset \mathcal{E}_{\mathcal{H}^+_{\geq v_0}}\oplus \mathcal{E}_{\mathcal{I}^+_{\geq u_0}} \subset\mathcal{E}_{\mathcal{H}^+_{\geq v_0}}^T\oplus  \mathcal{E}_{\mathcal{I}^+_{\geq u_0}}^T.
\end{equation*}

\begin{definition}
Let $n\in \N_0$. Define the higher-order norms $||\cdot||_{\mathcal{E}_{n;\mathcal{H}^{\pm}}}$ and $||\cdot||_{\mathcal{E}_{n; \mathcal{I}^{\pm}}}$, as follows: let $(\underline{\Phi},\Phi)\in C_{c}^{\infty}(\mathcal{H}^{\pm})\oplus C_{c}^{\infty}(\mathcal{I}^{\pm})$, then
\begin{align*}
||{\underline{\Phi}}||_{\mathcal{E}_{n;\mathcal{H}^{\pm}}}^2:=&\:\sum_{j=0}^2\sum_{m+2k+2|\alpha|\leq 2n}\int_{\mathcal{H}^{\pm}} (1+|v_{\pm}|)^{2+2k-j}(\partial_{v_{\pm}}^{1+k+m+j}\Omega^{\alpha}{\underline{\Phi}})^2+(1+|v_{\pm}|)^{2k}|\snabla_{\s^2} \partial_{v_{\pm}}^{k+m}\Omega^{\alpha}{\Phi}|^2\\
&+(1+|v_{\pm}|)^{2k}|\snabla_{\s^2} \partial_{v_{\pm}}^{k+1+m}\Omega^{\alpha}{\Phi}|^2 \,d\omega dv_{\pm},\\ 
||{\Phi}||_{\mathcal{E}_{n; \mathcal{I}^{\pm}}}^2:=&\:\sum_{j=0}^2\sum_{m+2k+2|\alpha|\leq 2n}\int_{\mathcal{I}^{\pm}}(1+|u_{\pm}|)^{2+2k+2j}(\partial_{u_{\pm}}^{1+k+m+j}\Omega^{\alpha}T^{-n}{\underline{\Phi}})^2+(1+|u_{\pm}|)^{2k}|\snabla_{\s^2} \partial_{u_{\pm}}^{k+m}\Omega^{\alpha}{\Phi}|^2\\
&+(1+|u_{\pm}|)^{2k}|\snabla_{\s^2} \partial_{u_{\pm}}^{k+1+m}\Omega^{\alpha}{\Phi}|^2\, d\omega du_{\pm},
\end{align*}
with respect to the coordinate charts $(u_{\pm}, v_{\pm},\theta,\varphi)$.

Then we denote with $\mathcal{E}_{n;\mathcal{H}^{\pm}}\oplus \mathcal{E}_{n;\mathcal{I}^{\pm}}$ the \emph{completion} of  $C_{c}^{\infty}(\mathcal{H}^{\pm})\oplus C_{c}^{\infty}(\mathcal{I}^{\pm})$ with respect to the norms $||\cdot||_{\mathcal{E}_{n;\mathcal{H}^{\pm}}}$ and $||\cdot||_{\mathcal{E}_{n; \mathcal{I}^{\pm}}}$. 
\end{definition}

\subsection{Black hole interior energy spaces}
\label{sec:intenspaces}
In this section, we introduce additional energy spaces that play a role in a non-degenerate scattering theory for the extremal Reissner--Nordstr\"om black hole interior.
\begin{definition}
Let $v_0>0$ and $u_{\rm int}<0$. Define the norms $||\cdot||_{\mathcal{E}_{\mathcal{H}^+_{\geq v_0}}}$ and $||\cdot||_{\mathcal{E}_{\mathcal{CH}^+_{\leq u_{\rm int}}}}$ as follows: let $\underline{\Phi} \in C_{c}^{\infty}(\mathcal{H}^+_{\geq v_0})$ and $\Phi \in C_{c}^{\infty}(\mathcal{CH}^+_{\leq u_{\rm int}})$, then
\begin{align*}
||{\underline{\Phi}}||_{\mathcal{E}_{\mathcal{H}^+_{\geq v_0}}}^2:=&\:\int_{\mathcal{H}^+_{\geq v_0} } v^{2}(\partial_v{\underline{\Phi}})^2+|\snabla_{\s^2}{\underline{\Phi}}|^2\,d\omega dv,\\
||\Phi||_{\mathcal{E}_{\mathcal{CH}^+_{\leq u_{\rm int}}}}^2:=&\:\int_{  \mathcal{CH}^+_{ \geq u_{\rm int} }   } u^{2}(\partial_u\Phi)^2+|\snabla_{\s^2} \Phi|^2\,d\omega du.
\end{align*}

Then we denote with $\mathcal{E}^{\rm int}_{\mathcal{H}^+_{\geq v_0}}$ and $\mathcal{E}^{\rm int}_{\mathcal{CH}^+_{\leq u_{\rm int}}}$ the completions of $C_{c}^{\infty}(\mathcal{H}^+_{\geq v_0})$ and $C_{c}^{\infty}(\mathcal{CH}^+_{\leq u_{\rm int}})$ with respect to the norms $||\cdot ||_{\mathcal{E}_{\mathcal{H}^+_{\geq v_0}}}$ and $|| \cdot ||_{\mathcal{E}_{\mathcal{CH}^+_{\leq u_{\rm int}}}}$, respectively. 
\end{definition}

\begin{definition}
Let $v_0>0$ and $u_{\rm int}<0$. Define the norms $||\cdot||_{\mathcal{E}_{\underline{N}^{\rm int}_{v_0}}}$ and $||\cdot||_{\mathcal{E}_{H^{\rm int}_{u_{\rm int}}}}$ as follows: let $\underline{\Phi} \in C^{\infty}(\underline{N}^{\rm int}_{v_0})$ and $\Phi \in C^{\infty}(H^{\rm int}_{u_{\rm int}})$, then
\begin{align*}
||{\underline{\Phi}}||_{\mathcal{E}_{\underline{N}^{\rm int}_{v_0}}}^2:=&\:\int_{\underline{N}^{\rm int}_{v_0} } u^{2}(\partial_u{\underline{\Phi}})^2+D|\snabla_{\s^2}{\underline{\Phi}}|^2\,d\omega du,\\
||\Phi||_{\mathcal{E}_{H^{\rm int}_{ u_{\rm int}}}}^2:=&\:\int_{ H^{\rm int}_{u_{\rm int} }   } v^{2}(\partial_v\Phi)^2+D|\snabla_{\s^2} \Phi|^2\,d\omega dv.
\end{align*}

Then we denote with $\mathcal{E}_{\underline{N}^{\rm int}_{v_0}}$ and $\mathcal{E}_{H^{\rm int}_{ u_{\rm int}}}$ the completions of $C^{\infty}(\underline{N}^{\rm int}_{v_0})$ and $C^{\infty}(H^{\rm int}_{u_{\rm int}})$ with respect to the norms $||\cdot ||_{\mathcal{E}_{\underline{N}^{\rm int}_{v_0}}}$ and $|| \cdot ||_{\mathcal{E}_{H^{\rm int}_{ u_{\rm int}}}}$, respectively. 
\end{definition}

\section{Main theorems}
\label{sec:thms} 
In this section, we give precise statements of the results proved in this paper. We refer to Sections \ref{sec:geom} and \ref{sec:energyspaces} for an introduction to the notation and definitions of the objects appearing in the statements of the theorems.

\subsection{Non-degenerate scattering theory results}
\label{sec:mainthms}
We first state the main theorems that establish a non-degenerate scattering theory in extremal Reissner--Nordstr\"om.
\begin{theorem}
\label{thm:mainthm}
The following linear maps
\begin{align*}
\mathscr{F}: C^{\infty}(\widehat{\Sigma}_0)\times C^{\infty}(\Sigma_0\cap \{r_{\mathcal{H}}\leq r\leq r_{\mathcal{I}}\})&\to \mathcal{E}_{\mathcal{H}^+_{\geq v_0}}\oplus \mathcal{E}_{\mathcal{I}^+_{\geq u_0}},\\
\widetilde{\mathscr{F}}_{\pm}: C_{c}^{\infty}(\widetilde{\Sigma}) \times C_{c}^{\infty}(\widetilde{\Sigma}) &\to \mathcal{E}_{\mathcal{H}^{\pm}}\oplus \mathcal{E}_{\mathcal{I}^{\pm}},
\end{align*}
with $\mathscr{F}(\Psi,\Psi')=(r\psi|_{\mathcal{H}^+_{\geq v_0}},r\psi|_{\mathcal{I}^+_{\geq u_0}})$, $\widetilde{\mathscr{F}}_{\pm}(\Psi,\Psi')=(r\psi|_{\mathcal{H}^{\pm}},r\psi|_{\mathcal{I}^{\pm}})$, are well-defined. Here, $\psi$ denotes the unique solution to \eqref{eq:waveequation} with initial data $(\Psi,\Psi')$ in accordance with statements 2. and 3. of Theorem \ref{thm:globexuni}. 

Furthermore, their unique extensions
\begin{align*}
\mathscr{F}: \mathcal{E}_{\Sigma_0}&\to \mathcal{E}_{\mathcal{H}^+_{\geq v_0}}\oplus \mathcal{E}_{\mathcal{I}^+_{\geq u_0}},\\
\widetilde{\mathscr{F}}_{\pm}: \mathcal{E}_{\widetilde{\Sigma}}&\to \mathcal{E}_{\mathcal{H}^{\pm}}\oplus \mathcal{E}_{\mathcal{I}^{\pm}}
\end{align*}
are bijective and bounded linear operators, and
\begin{equation*}
\mathscr{S}:=\widetilde{\mathscr{F}}_+\circ \widetilde{\mathscr{F}}_-^{-1}: \mathcal{E}_{\mathcal{H}^-}\oplus \mathcal{E}_{\mathcal{I}^-} \to \mathcal{E}_{\mathcal{H}^+}\oplus \mathcal{E}_{\mathcal{I}^+}
\end{equation*}
is also a bijective bounded linear operator.
\end{theorem}

We refer to the maps $\mathscr{F}$ and $\mathscr{F}_{\pm}$ as a \textbf{forwards evolution maps}, $\mathscr{F}^{-1}$ and $\widetilde{\mathscr{F}}_{\pm}^{-1}$ and \textbf{backwards evolution maps} and $\mathscr{S}$ as the \textbf{scattering map}. 

\begin{remark}
An analogous result holds with respect to the degenerate energy spaces $ \mathcal{E}^T_{\Sigma_0}$, $ \mathcal{E}^T_{\widetilde{\Sigma}}$, $\mathcal{E}^T_{\mathcal{H}^+_{\geq v_0}}$, $\mathcal{E}^T_{\mathcal{I}^+_{\geq u_0}}$, $\mathcal{E}^T_{\mathcal{H}^{\pm}}$ and $\mathcal{E}^T_{\mathcal{I}^{\pm}}$. This follows easily from an analogue of Proposition 9.6.1 in \cite{linearscattering} applied to the setting of extremal Reissner--Nordstr\"om; see also Sections \ref{sec:constfevo}, \ref{sec:bacwevomap} and \ref{sec:consscatmat}. They advantage of Theorem \ref{thm:mainthm} is the use of \underline{\emph{non-degenerate}} and \underline{\emph{weighted}} energy norms that also appear when proving global uniform boundedness and decay estimates for solutions to \eqref{eq:waveequation}.
\end{remark}

The following theorem extends Theorem \ref{thm:mainthm} by considering degenerate and weighted higher-order energy spaces.
\begin{theorem}
\label{thm:mainthmho}
Let $n\in\N_0$. We can restrict the codomains of the linear maps $\mathscr{F}$ and $\widetilde{\mathscr{F}}_{\pm}$ defined in Theorem \ref{thm:mainthm}, to arrive at
\begin{align*}
\mathscr{F}_n: C^{\infty}(\widehat{\Sigma}_0)\times C^{\infty}(\Sigma_0\cap \{r_{\mathcal{H}}\leq r\leq r_{\mathcal{I}}\})&\to \mathcal{E}_{n; \mathcal{H}^+_{\geq v_0}}\oplus \mathcal{E}_{n; \mathcal{I}^+_{\geq u_0}},\\
\widetilde{\mathscr{F}}_{n;\pm}: C_{c}^{\infty}(\widetilde{\Sigma}) \times C_{c}^{\infty}(\widetilde{\Sigma})&\to \mathcal{E}_{n; \mathcal{H}^{\pm}}\oplus \mathcal{E}_{n; \mathcal{I}^{\pm}},\\
\end{align*}
which are well-defined. 

Furthermore, the unique extensions
\begin{align*}
\mathscr{F}: \mathcal{E}_{n; \Sigma_0}&\to \mathcal{E}_{n; \mathcal{H}^+_{\geq v_0}}\oplus \mathcal{E}_{n; \mathcal{I}^+_{\geq u_0}},\\
\widetilde{\mathscr{F}}_{n; \pm}: \mathcal{E}_{n; \widetilde{\Sigma}}&\to \mathcal{E}_{n; \mathcal{H}^{\pm}}\oplus \mathcal{E}_{n; \mathcal{I}^{\pm}}
\end{align*}
are bijective and bounded linear operators and
\begin{equation*}
\mathscr{S}_n:=\mathscr{F}_{n;+}\circ \mathscr{F}_{n;-}^{-1}: \mathcal{E}_{n; \mathcal{H}^-}\oplus \mathcal{E}_{n; \mathcal{I}^-} \to \mathcal{E}_{n; \mathcal{H}^+}\oplus \mathcal{E}_{n; \mathcal{I}^+}
\end{equation*}
is also a bijective bounded linear operator.
\end{theorem}

Both Theorem \ref{thm:mainthm} and Theorem \ref{thm:mainthmho} follow by combining Propositions \ref{prop:fowbound} and \ref{prop:backwmapho}, Corollary \ref{cor:bijectivity} and Propositions \ref{prop:towardsscatmat1} and  \ref{prop:towardsscatmat3}.

We additionally construct a scattering map restricted to the black hole interior.

\begin{theorem}
\label{thm:intscat}
Let $u_{\rm int}<0$ with $|u_{\rm int}|$ suitably large. The following linear map:
\begin{equation*}
\mathscr{S}^{\rm int}: C_{c}^{\infty}(\mathcal{H}^+_{\geq v_0})\times C^{\infty}(\underline{N}_{v_0}^{\rm int})\to \mathcal{E}_{\mathcal{CH}^+_{\leq u_{\rm int}}}\oplus \mathcal{E}_{H^{\rm int}_{u_{\rm int}}},
\end{equation*}
with 
\begin{equation*}
\mathscr{S}^{\rm int}(r\psi|_{\mathcal{H}^+_{\geq v_0}}, r\psi|_{\underline{N}_{v_0}^{\rm int}})=(r\psi|_{\mathcal{CH}^+_{\leq u_{\rm int}}},r\psi_{H^{\rm int}_{u_{\rm int}}})
\end{equation*}
is well-defined, with $\psi$ denoting the unique solution to \eqref{eq:waveequation} with initial data $(r\psi|_{\mathcal{H}^+_{\geq v_0}}, r\psi|_{\underline{N}_{v_0}^{\rm int}})$ in accordance with statement 3. of Theorem \ref{thm:globexuni}. 

Furthermore, uniquely as a bijective, bounded linear operator:
\begin{equation*}
\mathscr{S}^{\rm int}: \mathcal{E}_{\mathcal{H}^+_{\geq v_0}}\oplus \mathcal{E}_{\underline{N}_{v_0}^{\rm int}}\to \mathcal{E}_{\mathcal{CH}^+_{\leq u_{\rm int}}}\oplus \mathcal{E}_{H^{\rm int}_{u_{\rm int}}}.
\end{equation*}
\end{theorem}
Theorem \ref{thm:intscat} is a reformulation of Proposition \ref{prop:intscat}.

\subsection{Applications}
In this section, we state some applications of the non-degenerate scattering theory of Section \ref{sec:mainthms}.

In Theorem \ref{thm:regularity} below, we show that we can obtain unique solutions to \eqref{eq:waveequation} with \emph{arbitrary high} Sobolev regularity (with respect to the differentiable structure on $\widehat{\mathcal{R}}$) from suitably regular and polynomially decaying \emph{scattering data} on $\mathcal{H}^+$ and $\mathcal{I}^+$ in an $L^2$-integrated sense.
\begin{theorem}
\label{thm:regularity}
Let $n\in \N_0$ and let $(\underline{\Phi},\Phi)\in C^{\infty}(\mathcal{H}^+_{\geq v_0}) \oplus C^{\infty}(\mathcal{I}^+_{\geq u_0})$. Assume that $\lim_{v \to \infty} v^{n+1}|\underline{\Phi}| (v,\theta,\varphi)<\infty$ and $\lim_{u \to \infty} u^{n+1}|{\Phi}| (u,\theta,\varphi)<\infty$. Define the integral functions $(T^{-n}\underline{\Phi},T^{-n}\Phi)\in C^{\infty}(\mathcal{H}^+_{\geq v_0}) \oplus C^{\infty}(\mathcal{I}^+_{\geq u_0})$ as follows:
\begin{align*}
T^{-n}\underline{\Phi}(v,\theta,\varphi):=\int_{v}^{\infty} \int_{v_1}^{\infty}\ldots \int_{v_{n-1}}^{\infty}\underline{\Phi}(v_{n},\theta,\varphi) \,dv_{n}\ldots dv_1,\\
T^{-n}\Phi(u,\theta,\varphi)=\int_{u}^{\infty} \int_{u_1}^{\infty}\ldots \int_{u_{n-1}}^{\infty}\Phi(u_{n},\theta,\varphi) \,du_{n}\ldots du_1.
\end{align*}
and assume moreover that
\begin{equation}
\label{eq:tinversenorms}
\left|\left|T^{-n}\underline{\Phi}\right|\right|_{\mathcal{E}_{2n; \mathcal{I}^+_{\geq u_0}}}+ \left|\left|T^{-n}\Phi\right|\right|_{ \mathcal{E}_{2n; \mathcal{H}^+_{\geq v_0}}}<\infty.
\end{equation}
Then there exists a unique corresponding solution $\psi$ to \eqref{eq:waveequation} that satisfies $r\psi|_{\mathcal{H}^+_{\geq v_0}}=\underline{\Phi}$, $r\psi|_{\mathcal{I}^+_{\geq u_0}}={\Phi}$ and
\begin{equation*}
r\psi\in W_{\rm loc}^{n+1,2}(\widehat{\mathcal{R}}).
\end{equation*}
\end{theorem}

Theorem \ref{thm:regularity} follows from Proposition \ref{prop:regtimeinv}.

\begin{remark}
Theorem \ref{thm:regularity} illustrates a stark difference in the setting of extremal Reissner--Nordstr\"om with the sub-extremal setting, where generic polynomially decaying data along the future event horizon and future null infinity (with an arbitrarily fast decay rate) lead to blow-up of the non-degenerate energy along $\Sigma_0$; see \cite{linearscattering, DafShl2016}.
\end{remark}
As a corollary of Theorem \ref{thm:regularity}, we can moreover construct smooth solutions and in particular smooth solutions with an exact exponential time dependence.
\begin{theorem}
\label{thm:superpoly}
Let $(\underline{\Phi},\Phi)\in C^{\infty}(\mathcal{H}^+_{\geq v_0}) \oplus C^{\infty}(\mathcal{I}^+_{\geq u_0})$ and assume that $(\underline{\Phi},\Phi)$ and all derivatives up to any order decay \textbf{superpolynomially} in $v$ and $u$, respectively. 
\begin{enumerate}[\rm (i)]
\item Then there exists a corresponding smooth solution $\psi$ to \eqref{eq:waveequation} on $\mathcal{R}$ such that $r\hat{\psi}$ can moreover be smoothly extended to $\hat{\mathcal{R}}$ with respect to the differentiable structure on  $\hat{\mathcal{R}}$.
\item Assume additionally that ${\underline{\Phi}}(v,\theta,\varphi)=f_H(\theta,\varphi)e^{-i\omega v}$ and ${\Phi}(u,\theta,\varphi)=f_I(\theta,\varphi)e^{-i\omega u}$ for $f_H,f_I\in C^{\infty}(\s^2)$, with $\omega \in \C$ such that $\textnormal{Im}\,\omega <0$. Then we can express
\begin{equation*}
r\cdot \psi(\tau,\rho,\theta,\varphi)=f(\rho,\theta,\varphi)e^{-i\omega\cdot \tau},
\end{equation*}
with $f\in C^{\infty}(\hat{\Sigma})$ and
\begin{align*}
\lim_{\rho \downarrow M}f(\rho,\theta,\varphi)=&\: f_H(\theta,\varphi),\\
\lim_{\rho \to \infty }f(\rho,\theta,\varphi)=&\: f_I(\theta,\varphi).
\end{align*}
\end{enumerate}
We refer to $\psi$ as \textbf{mode solutions}.
\end{theorem}

Theorem \ref{thm:superpoly} (i) follows from Corollary \ref{cor:smoothbackw} and Theorem \ref{thm:superpoly} (ii) follows from Proposition \ref{prop:modesol}.
\begin{remark}
\label{rm:smoothmodesub}
Note that in order for an analogous result to Theorem \ref{thm:superpoly} {\rm (i)} to hold in sub-extremal Reissner--Nordstr\"om, one needs to consider scattering data $(\underline{\Phi},\Phi)$ that are \emph{superexponentially} decaying, and hence it cannot be used to prove the analogue of Theorem \ref{thm:superpoly} {\rm (ii)}. Nevertheless, the existence of a more restricted class of smooth solutions that behave exponentially in time with arbitrary $\omega$ such that $\textnormal{Im}\,\omega <0$ in sub-extremal Reissner--Nordstr\"om can be established by restricting to fixed spherical harmonics and applying standard asymptotic ODE analysis.
\end{remark}

\begin{remark}
\label{remarkquasi}
One can apply the results of \cite{paper4} to show that the time translations acting on $L^2$-based Sobolev spaces
\begin{align*}
\mathcal{S}(\tau'):&\: W^{k+1,2}(\widehat{\Sigma})\times W^{k,2}(\widehat{\Sigma}\cap \{r_{\mathcal{H}} <r \leq r_{\mathcal{I}}\} )\to W^{k+1,2}(\widehat{\Sigma})\times W^{k,2}(\widehat{\Sigma}\cap \{r_{\mathcal{H}} <r \leq r_{\mathcal{I}}\} ),\\
(\Psi,\Psi')& \mapsto(\psi|_{\tau=\tau'},T\psi|_{\tau=\tau'}),
\end{align*}
with $\psi$ the solution to \eqref{eq:waveequation} associated to $(\Psi,\Psi')$, form a continuous semi-group, such that $\mathcal{S}(\tau')=e^{\tau' \mathcal{A}}$, with $\mathcal{A}$ the corresponding densely defined infinitesimal generator $\mathcal{A}$ that formally agrees with $T$:
\begin{equation*}
\mathcal{A}(\psi|_{\tau=\tau'},T\psi|_{\tau=\tau'})=(T\psi|_{\tau=0},T^2\psi|_{\tau=0}).
\end{equation*} 

The results of  \cite{warn15} imply that, in the setting of asymptotically de Sitter or anti de Sitter spacetimes, quasi-normal modes or resonances are smooth mode solutions that can be interpreted as eigenfunctions of $\mathcal{A}$ and the corresponding frequencies $\omega$ form a \underline{discrete} set in the complex plane (cf. the normal modes and frequencies of an idealised vibrating string or membrane). 

The smooth mode solutions of Theorem \ref{thm:superpoly} (ii) (and those obtained in the sub-extremal setting by ODE arguments as sketched in Remark \ref{rm:smoothmodesub}) form an \textbf{obstruction} to extending this interpretation to the asymptotically flat setting. Indeed, all the mode solutions of Theorem \ref{thm:superpoly} (ii) are eigenfunctions of $\mathcal{A}$ but the corresponding set of frequencies $\omega$, which is the entire open lower-half complex plane, is certainly \underline{not} discrete. In order to maintain the viewpoint of \cite{warn15}, one has to consider smaller function spaces that \emph{exclude} the smooth mode solutions of Theorem \ref{thm:superpoly} (ii); see \cite{gajwar19a}.
\end{remark}

\begin{theorem}
\label{thm:interior}
Let $u_0$ be suitably large. Then there exists a constant $C=C(M,u_0,v_0)>0$ such that we can estimate in the black hole interior:
\begin{equation*}
\begin{split}
||\psi||_{W^{1,2}_{\rm loc} (D^+(\Sigma_{0}^{\textnormal{int}, u_0})\cap \widetilde{M}^{\rm int})}\leq C\left(||(\psi|_{\Sigma_0},\mathbf{n}_{\Sigma_0}\psi|_{\Sigma_0})||_{\mathcal{E}_{\Sigma_0}}+||\psi||_{W^{1,2} (\underline{N}_{v_0}^{\textnormal{int}})}\right).
\end{split}
\end{equation*}
\end{theorem}

Theorem \ref{thm:interior} follows from Corollary \ref {cor:regint}.
\begin{remark}
\label{remarkinterior}
Theorem \ref{thm:interior} addresses the question of whether $\psi \in W^{1,2}_{\rm loc}$ in the black hole interior of extremal Reissner--Nordstr\"om for localized, low regularity initial data, which was raised as an open problem in \cite{dafshl18}. For \emph{smooth} and localized data, this statement follows from \cite{gajic, paper4}. Indeed, Theorem \ref{thm:interior} demonstrates that boundedness of a non-degenerate energy with weights that grow in $r$ (together with boundedness of energies involving additional derivatives that are tangential to the event horizon) is \emph{sufficient} to establish $\psi \in W^{1,2}_{\rm loc}$. 

Theorem \ref{thm:interior} can straightforwardly be extended to the $\Lambda>0$ setting of extremal Reissner--Nordstr\"om--de Sitter black holes, where there is no need to include $r$-weights in the non-degenerate energy norm that is sufficient to establish $\psi \in W^{1,2}_{\rm loc}$. See also \cite{almakmsc} for the results in the interior of extremal Reissner--Nordstr\"om--de Sitter.
\end{remark}

\section{Overview of techniques and key ideas}
\label{sec:tech}
In this section, we provide an overview of the main techniques that are used in the proofs of the theorems stated in Section \ref{sec:thms}. We will highlight the key new ideas and estimates that are introduced in this paper.

The proof of the main theorems Theorem \ref{thm:mainthm} and Theorem \ref{thm:mainthmho} can roughly be split into four parts:
\begin{enumerate}[\bf1.)]
\item  Showing that the linear maps $\mathscr{F}$, $\mathscr{F}^{-1}$ and $\mathscr{F}_n$, $\mathscr{F}_n^{-1}$ that appear in Theorem \ref{thm:mainthm} and Theorem \ref{thm:mainthmho} are well-defined when considering as a domain spaces of either smooth or smooth and compactly supported functions.
\item Proving uniform boundedness properties of these linear maps with respect to weighted Sobolev norms. This allows one to immediately extend the linear maps to the \emph{completions} of the spaces of smooth (and compactly) supported functions with respect to appropriately weighted Sobolev norms.
\item Constructing the linear maps $\mathscr{S}$ and $\mathscr{S}_n$.
\item Constructing $\mathscr{S}^{\rm int}$ (independently from above).
\end{enumerate}

\textbf{The heart of this paper consists of establishing \textbf{2.)} and \textbf{3.)} by proving \underline{\emph{uniform}} estimates for smooth (and compactly supported) data along $\Sigma_0$, $\widetilde{\Sigma}$ and $\mathcal{H}^{\pm}\cup \mathcal{I}^{\pm}$.} An overview of the corresponding estimates and techniques leading to \textbf{2.)} is given in Section \ref{intro:backwest} -- \ref{intro:hoest}. Part \textbf{3.)} follows by complementing these estimates with additional estimates in $D^{\pm}(\widetilde{\Sigma})$ near the past limit points of $\mathcal{I}^+$ and $\mathcal{H}^+$, which is briefly discussed in Section \ref{intro:spacelikeinf}. We briefly discuss the black hole interior estimates involved in \textbf{4.)} in Section \ref{intro:interior}. 

Part \textbf{1.)} follows from local estimates combined with soft global statements that have already been established in the literature. We give an overview of the logic of the arguments in this section.

The forwards map
\begin{equation*}
\mathscr{F}: C^{\infty}(\widehat{\Sigma}_0)\oplus C^{\infty}(\Sigma_0\cap \{r_{\mathcal{H}}\leq r\leq r_{\mathcal{I}}\})\to \mathcal{E}_{ \mathcal{H}^+_{\geq v_0}}\oplus \mathcal{E}_{\mathcal{I}^+_{\geq u_0}}
\end{equation*}
is well-defined by global existence and uniqueness for \eqref{eq:waveequation} combined with the finiteness (and decay) of the radiation field $r\psi$, see for example the results in \cite{aretakis1,aretakis2, paper4}.

In order to show that the backwards map\footnote{We use the notation $\mathcal{F}^{-1}$ to indicated the backwards map, but we still have to show that this map is indeed a two-sided inverse of $\mathcal{F}$.}
\begin{equation*}
\mathscr{F}^{-1}: C_{c}^{\infty}(\mathcal{H}^+_{\geq v_0})\oplus C_{c}^{\infty}(\mathcal{I}^+_{\geq u_0})\to  \mathcal{E}_{ \Sigma_0}
\end{equation*}
is well-defined, we first need to make sense of the notion of prescribing initial data ``at infinity''; that is to say, we need to show as a preliminary step that we can associate to each pair $(\underline{\Phi},\Phi)\in C_{c}^{\infty}(\mathcal{H}^+_{\geq v_0})\oplus C_{c}^{\infty}(\mathcal{I}^+_{\geq u_0})$ a unique solution $\psi$ to \eqref{eq:waveequation}, such that $r\psi|_{\mathcal{H}^+_{\geq v_0}}=\underline{\Phi}$ and $r\psi|_{\mathcal{I}^+_{\geq u_0}}=\Phi$. This may be viewed as a semi-global problem. We construct $\psi$ as the limit of a sequence of solutions $\psi_i$ arising from a sequence of \emph{local} initial value problems with fixed initial data $(\underline{\Phi},\Phi)$ imposed on the null hypersurfaces $\mathcal{H}^+_{0\leq \tau\leq \tau_*}\cup \{v=v_{i}, u_0\leq u\leq u(\tau_*)\}$ and trivial data on  $\Sigma_{\tau_*}\cap\{v\leq v_i\}$, such that $v_{i}\to \infty$ as $i\to \infty$. A very similar procedure was carried out in the physical space construction of scattering maps on Schwarzschild in Proposition 9.6.1 in \cite{linearscattering}.\footnote{In contrast with [dafrodshl], we take the limit with respect to $L^{\infty}$ norms of the weighted quantities $(r^2\partial_v)^k(r\psi)$ with $k\in \N$.} One could alternatively interpret $\mathcal{I}^+$ as a genuine null hypersurface with respect to the conformally rescaled metric $\hat{g}_M$, which turns the semi-global problem into a local problem.

\subsection{Backwards $r$-weighted estimates}
\label{intro:backwest}
We introduce \textbf{time-reversed analogues} of the $r^p$-weighted estimates of Dafermos--Rodnianski \cite{newmethod} and the $(r-M)^{-p}$-weighted estimates of \cite{paper4}. We first illustrate key aspects of these estimates in the setting of the standard wave equation on Minkowski. We can foliate the causal future of a null cone $C_0$ in Minkowski by outgoing spherical null cones $C_{u}=\{t-r=u\}$, with $t,r$ the standard spherical Minkowski coordinates and $u\geq 0$. Let us denote $\partial_v=\frac{1}{2}(\partial_t+\partial_r)$ and $\partial_u=\frac{1}{2}(\partial_t-\partial_r)$. We consider smooth, compactly supported initial data on $\mathcal{I}^+\cap\{0\leq u\leq u_2\}$, with $u_2>0$ such that $\psi$ vanishes along $C_{u_2}$. 

The $r^p$-weighted estimates applied backwards in time with $p=1$ and $p=2$ give
\begin{align*}
\int_{C_{u_1}} r\cdot (\partial_v(r\psi))^2\,d\omega dv\lesssim& \int_{C_{u_2}} r\cdot (\partial_v(r\psi))^2\,d\omega dv+ \int_{u_1}^{u_2} \int_{C_{u}}\mathbf{T}(\partial_t, \partial_v) r^2\,d\omega dvdu,\\
\int_{C_{u_1}} r^2\cdot (\partial_v(r\psi))^2\,d\omega dv\lesssim& \int_{C_{u_2}}r^2\cdot (\partial_v(r\psi))^2\,d\omega dv+\int_{\mathcal{I}^+\cap\{u'\leq u\leq u_2\}} |\snabla (r\psi)|^2\,d\omega du+ \int_{u_1}^{u_2} \int_{C_{u}}r\cdot (\partial_v(r\psi))^2\,d\omega dvdu,
\end{align*}
for $u_2>u_1>0$.

In contrast with the usual forwards $r^p$-weighted estimates, the spacetime integrals on the right-hand sides above have a \textbf{bad} sign. Hence, in order to obtain control of $r$-weighted energies along $C_{u_1}$, we need to start by controlling
\begin{equation*}
 \int_{u_1}^{u_2} \int_{C_{u}}\mathbf{T}(\partial_t, \partial_v) r^2\,d\omega dr.
\end{equation*}
Note that standard $\partial_t$-energy conservation implies that for any $0<u'<u_2$:
\begin{equation}
\label{eqintro:tenconsv}
\int_{C_{u'}} \mathbf{T}(\partial_t, \partial_v)\,r^2\,d\omega dr=\int_{C_{u_2}} \mathbf{T}(\partial_t, \partial_v) r^2\,d\omega dr+\int_{\mathcal{I}^+\cap\{u'\leq u\leq u_2\}} (\partial_u(r\psi))^2\,d\omega du.
\end{equation}
Hence, using that $\psi$ is vanishing along $C_{u_2}$, we can integrate the above equation in $u'$ to obtain
\begin{equation*}
\int_{u_1}^{u_2} \left[\int_{C_{u}}\mathbf{T}(\partial_t, \partial_v) r^2\,d\omega dr\right]du=\int_{u_1}^{u_2} \left[\int_{\mathcal{I}^+\cap\{u'\leq u\leq u_2\}}(\partial_u(r\psi))^2\,d\omega du\right]du'.
\end{equation*}
We can integrate by parts to convert one $u$-integration into an additional $u$ weight:
\begin{equation}
\label{eqintro:intparts}
\int_{u_1}^{\infty} \left[\int_{\mathcal{I}^+\cap\{u\geq u'\}}(\partial_u(r\psi))^2\,d\omega du\right]du'=\int_{\mathcal{I}^+\cap\{u\geq u_1\}}(u-u_1)\cdot (\partial_u(r\psi))^2\,d\omega du.
\end{equation}
By applying both the $p=1$ and $p=2$ estimates above, and integrating by parts once more along $\mathcal{I}^+$ as in \eqref{eqintro:intparts}, we obtain:
\begin{equation}
\label{eqintro:backwest}
\int_{C_{0}} {r^2}\cdot (\partial_v(r\psi))^2\,d\omega dv+ \int_{C_{0}}\mathbf{T}(\partial_t, \partial_v) r^2\,d\omega dr\lesssim \int_{\mathcal{I}^+\cap\{u\geq 0\}}{(u+1)^2} \cdot (\partial_u(r\psi))^2\,d\omega du
\end{equation}
Comparing \eqref{eqintro:backwest} with \eqref{eqintro:tenconsv} with $u'=0$, we see that we can obtain stronger, weighted uniform control along $C_0$, provided we control an appropriately weighted energy along $\mathcal{I}^+$. One may compare this to the (modified) energy estimate obtained by using the Morawetz conformal vector field $K=u^2\partial_u+v^2\partial_v$, which is the generator of the inverted time translation conformal symmetries, as a vector field multiplier instead of $\partial_t$ \cite{mor1}; see also Section \ref{intro:spacelikeinf}.

The main difference in the setting of extremal Reissner--Nordstr\"om is that the $r^p$-estimates above only apply in the spacetime region where $r\geq r_{\mathcal{I}}$, with $r_{\mathcal{I}}$ suitably large, and they have to be complemented by an analogous hierarchy of $(r-M)^{-p}$ weighted estimates in a region $\{r\leq r_{\mathcal{H}}\}$ near $\mathcal{H}^+$, i.e.\ with $r_{\mathcal{H}}-M$ sufficiently small. Roughly speaking, the analogue of the $p=2$ weighted energy near $\mathcal{H}^+$ corresponds to the restriction of the following \emph{non-degenerate} energy (in $(v,r)$ coordinates):
\begin{equation*}
\int_{\Sigma_0\cap\{r\leq r_{\mathcal{H}}\}} \mathbf{T}(N,\partial_r) r^2\, d\omega dr,
\end{equation*}
where $N$ is a timelike vector field in $\{M\leq r\leq r_{\mathcal{H}}\}$.

It is in controlling the non-degenerate energy in the backwards direction that we make essential use of the \emph{extremality} of extremal Reissner--Nordstr\"om or the degeneracy of the event horizon. Indeed, if we were to consider instead \emph{sub}-extremal Reissner--Nordstr\"om, we would fail to obtain control of a non-degenerate energy near $\mathcal{H}^+$ with polynomially decaying data along $\mathcal{H}^+\cup \mathcal{I}^+$ due to the \emph{blueshift effect} (the time reversed redshift effect); see \cite{linearscattering,  DafShl2016}.\footnote{In this case, one may however assume sufficiently fast \emph{exponentially} decay along $\mathcal{H}^+\cup \mathcal{I}^+$ to beat the blueshift effect and obtain boundedness of the non-degenerate energy, see \cite{scattering}.}

In order to control the boundary terms arising from restricting the $r$-weighted estimates near $\mathcal{I}^+$ and $\mathcal{H}^+$, we apply the Morawetz estimate derived in \cite{aretakis1} in the backwards direction. Note that the presence of trapped null geodesics along the photon sphere at $r=2M$ does \emph{not} lead to a loss of derivatives in the analogue of \eqref{eqintro:backwest}. This is because the backwards estimates, in contrast with the forwards estimates (see Section \ref{intro:fowest}), do not require an application of a Morawetz estimate with non-degenerate control at the photon sphere.

\subsection{Forwards $r$-weighted estimates revisited}
\label{intro:fowest}
We consider again the setting of Minkowski to illustrate the main ideas. In order to construct a bijection from an $r$-weighted energy space on $C_0$ to a $u$-weighted energy space on $\mathcal{I}^+$, we need to complement the backwards estimate \eqref{eqintro:backwest} with the following forwards estimate:
\begin{equation}
\label{eqintro:fowest}
\int_{C_{0}} r^2\cdot (\partial_v(r\psi))^2\,d\omega dv+ \int_{C_{0}}\mathbf{T}(\partial_t, \partial_v) r^2\,d\omega dr\gtrsim \int_{\mathcal{I}^+\cap\{u\geq 0\}} (u+1)^2 \cdot (\partial_u(r\psi))^2\,d\omega du
\end{equation}
Note that a standard application of the $r^p$-weighted estimates (combined with energy conservation \eqref{eqintro:tenconsv} and a Morawetz estimate), see \cite{newmethod}, is the following energy decay statement:
\begin{equation*}
\int_{C_{u}} \mathbf{T}(\partial_t, \partial_v)\,r^2\,d\omega dr\lesssim (1+u)^{-2}\left[\int_{C_{0}}r^2\cdot (\partial_v(r\psi))^2\,d\omega dv+ \int_{C_{0}}\mathbf{T}(\partial_t, \partial_v) r^2\,d\omega dr\right].
\end{equation*}
One can apply this estimate along a suitable dyadic sequence and combine it with energy conservation \eqref{eqintro:tenconsv} to arrive at the estimate
\begin{equation*}
\int_{C_{0}} r^2\cdot (\partial_v(r\psi))^2\,d\omega dv+ \int_{C_{0}}\mathbf{T}(\partial_t, \partial_v) r^2\,d\omega dr\gtrsim \int_{\mathcal{I}^+\cap\{u\geq 0\}} (u+1)^{2-\epsilon} \cdot (\partial_u(r\psi))^2\,d\omega du
\end{equation*}
with $\epsilon>0$. In order to take $\epsilon=0$, we instead revisit the $r^p$-estimates and, rather than deriving energy decay along $C_u$, we observe that the $r^p$-estimates (together with \eqref{eqintro:tenconsv} and a Morawetz estimate) provide \emph{directly} control over
\begin{equation*}
\int_{0}^{\infty}\int_{u}^{\infty}\left[\int_{\mathcal{I}^+\cap\{u\geq u'\}} (\partial_u(r\psi))^2\,d\omega du''\right] du' du.
\end{equation*}
After integrating by parts twice in $u$ as in \eqref{eqintro:intparts}, we arrive at \eqref{eqintro:fowest}.

We arrive at an analogous estimate to \eqref{eqintro:fowest} in the extremal Reissner--Nordstr\"om setting by following the same ideas, both near $\mathcal{I}^+$ and near $\mathcal{H}^+$. The main difference is that whenever we apply a Morawetz estimate, we lose a derivative because of the trapping of null geodesics, which we have to take into account when defining the appropriate energy spaces.

\subsection{Higher-order energies and time integrals}
\label{intro:hoest}
Given suitably regular and suitably decaying scattering data on $\mathcal{H}^+$ and $\mathcal{I}^+$, we can apply Theorem \ref{thm:mainthm} to construct a corresponding solution $\psi\in C^0\cap W^{1,2}_{\rm loc}$ (with respect the differentiable structure on $\widehat{\mathcal{R}}$) to \eqref{eq:waveequation} such that $r\psi$ approaches the scattering data as $r\to M$ or $r\to \infty$. 

In the setting of \eqref{eq:waveequation} on Minkowski with coordinates $(u,x,\theta,\varphi)$, where $x:=\frac{1}{r+1}$ (so that $x \downarrow 0$ as $r\to \infty$ and $x \uparrow 1$ as $r\downarrow 0$), we similarly have that $r\psi\in W^{1,2}([u_1,u_2]_u\times (0,1]_{x}\times \s^2)$ for any $0\leq u_1<u_2<\infty$. In order to show that moreover $r\psi\in W^{2,2}([u_1,u_2]_u\times (0,1)_{x}\times \s^2)$, we first consider $T\psi$. By rearranging and rescaling \eqref{eq:waveequation} in Minkowski, we have that in $(u,x)$ coordinates:
\begin{equation*}
\begin{split}
2\partial_{x}T(r\psi)=\partial_{x}(x^2\partial_{x}(r\psi))+\slashed{\Delta}_{\s^2}(r\psi)
\end{split}
\end{equation*}
with $x^2\partial_{x}(r\psi)=2L(r\psi)$. So, we obtain that
\begin{equation*}
T(r\psi)\in W^{2,2}([u_1,u_2]_u\times (0,1]_{x}\times \s^2)
\end{equation*}
if we can show that
\begin{equation*}
L(r\psi)\in W^{2,2}([u_1,u_2]_u\times (0,1]_{x}\times \s^2)\quad \textnormal{and}\quad \slashed{\Delta}_{\s^2}(r\psi) \in  W^{1,2}([u_1,u_2]_u\times (0,1]_{x}\times \s^2).
\end{equation*}
Since $\slashed{\Delta}_{\s^2}$ commutes with the operator $\square_g$, both in Minkowski and in extremal Reissner--Nordstr\"om, we can immediately obtain $\slashed{\Delta}_{\s^2}(r\psi)\in W^{1,2}$ from Theorem \ref{thm:mainthm} (or its Minkowski analogue). Moreover, $L(r\psi)\in W^{2,2}$ follows from bounding uniformly in $u$ the integral:
\begin{equation*}
\int_{C_u} r^6(L^{3}(r\psi)^2+ r^4(L^2(r\psi)))^2\,d\omega dv.
\end{equation*}
Hence, we have to establish control over \emph{improved} $r$-weighted energies where $r\psi$ is replaced by $L(r\psi)$ and $L^2(r\psi)$. Analogous improved $r$-weighted energies have appeared previously in the setting of forwards estimates in \cite{volker1, paper1, paper4}, see also the related energies in \cite{moschidis1}. The backwards analogues of the corresponding improved $r$-weighted estimates form the core of the proof of Theorem \ref{thm:mainthmho}.

To pass from $T(r\psi) \in W^{2,2}$ to $r\psi \in W^{2,2}$, we apply the above estimates to solutions $\psi^{(1)}$ to \eqref{eq:waveequation}, such that $T\psi^{(1)}=\psi$. Such solutions $\psi^{(1)}$ can easily be constructed by considering initial scattering data that are \emph{time integrals} of the scattering data $\mathcal{H}^+$ in $v$ and $\mathcal{I}^+$ in $u$, assuming moreover that $r\psi^{(1)}|_{\mathcal{H}^+}$ and $r\psi^{(1)}|_{\mathcal{I}^+}$ vanish as $v\to \infty$ and $u\to \infty$, respectively. 

In fact, we can show by an extension of the arguments above that $T^n(r\psi)\in W^{1+n,2}_{\rm loc}$ for all $n\geq 2$, assuming suitably regular and decaying data along  $\mathcal{H}^+$ and $\mathcal{I}^+$, so we can conclude that $\psi\in W^{n+1,2}_{\rm loc}$, provided the scattering data decays suitably fast in time. In order to obtain \emph{more} regularity, we need \emph{faster} polynomial decay along $\mathcal{H}^+\cup \mathcal{I}^+$. This is the content of Theorem \ref{thm:regularity}. By considering smooth and \emph{superpolynomially} decaying data along $\mathcal{H}^+\cup \mathcal{I}^+$ and applying standard Sobolev inequalities, we can in fact take $n$ arbitrarily high and show that $\psi\in C^{\infty}(\widehat{\mathcal{R}})$; see Theorem \ref{thm:superpoly}.

Note that time integrals $\psi^{(1)}$ also play an important role in \cite{paper2, paper4} for spherical symmetric solutions. In that setting, one needs to solve an elliptic PDE (which reduces to an ODE in spherical symmetry) to construct $\psi^{(1)}$, which is contrast with the backwards problem, where the construction is much simpler because we can integrate the scattering data in time to obtain data leading to $\psi^{(1)}$.

\subsection{Estimates near spacelike infinity}
\label{intro:spacelikeinf}
The backwards and forwards estimates sketched in Section \ref{intro:backwest} and Section \ref{intro:fowest} allow us to construct a bijection between weighted energy spaces on $\Sigma_0$ and $\mathcal{H}^+_{\geq v_0}\cup \mathcal{I}^+_{\geq u_0}$. In order to construct the bijection $\mathscr{S}$ between energy spaces on $\mathcal{H}^-\cup \mathcal{I}^-$ and $\mathcal{H}^+\cup \mathcal{I}^+$ we need to additionally construct a bijection between appropriate energy spaces on $\Sigma_0\cup\mathcal{H}^+_{v\leq v_0}\cup \mathcal{I}^+_{u\leq u_0}$ and $\widetilde{\Sigma}=\{t=0\}$. Without loss of generality, we can pick $\Sigma_0$ so that
\begin{equation*}
\Sigma_0\cap \{r_{\mathcal{H}}\leq r\leq r_{\mathcal{I}}\}= \widetilde{\Sigma}\cap \{r_{\mathcal{H}}\leq r\leq r_{\mathcal{I}}\},
\end{equation*}
and we are left with only proving energy estimates in the regions $D_{-u_0}$ and $\underline{D}_{-v_0}$, see Section \ref{sec:foliations} and Figure \ref{fig:foliations}.

While $r$-weighted estimates are still suitable in the forwards direction in $D_{-u_0}$ and $\underline{D}_{-v_0}$, they are \textbf{not} suitable in the backwards direction. We therefore consider energy estimates for the radiation field $r \psi$ with the vector field multiplier $K=u^2\partial_u+v^2\partial_v$, both in $D_{-u_0}$ and $\underline{D}_{-v_0}$ in order to arrive at the analogue of the $p=2$ estimate. In Minkowski space, $K$ corresponds to the generator of a conformal symmetry, the inverted time translations. It is a Killing vector field of the rescaled metric $r^{-2}m$, where $m$ is the Minkowski metric. Hence, $K$ may be thought of as the analogue of $\partial_t$ when considering $r\psi$ instead of $\psi$ and $r^{-2}m$ instead of $m$. In particular, when considering $K$ as a vector field multiplier in a spacetime region of Minkowski, one can obtain a weighted energy conservation law for $r\psi$. Since $r$ is large in $D_{-u_0}$ in extremal Reissner--Nordstr\"om, $K$ may be thought of as an ``approximate Killing vector field'' of the rescaled metric $r^{-2}g$.

Another useful property of $K$ is that it is invariant under the Couch--Torrence conformal symmetry \cite{couch} that maps $D_{-u_0}$ to $\underline{D}_{-v_0}$. It therefore plays the same role when used as a vector field multiplier for the radiation field in $\underline{D}_{-v_0}$ as it does in $D_{-u_0}$. 

In order to obtain the analogue of the $r^p$-weighted estimate with $p=1$ for $T\psi$, we apply instead the vector field multiplier $Y=v\partial_v-u\partial_u$ in $D_{-u_0}$ and $Y=u\partial_u-v\partial_v$ in $\underline{D}_{-v_0}$.

We construct
\begin{equation*}
\mathscr{S}: C_{c}^{\infty}(\mathcal{H}^{-})\oplus C_{c}^{\infty}(\mathcal{I}^{-}) \to \mathcal{E}_{ \mathcal{H}^+}\oplus \mathcal{E}_{\mathcal{I}^+}
\end{equation*}
by first observing that the spacetime is invariant under the map $t\mapsto -t$, so the above discussion on $\mathscr{F}^{-1}$ can be applied to associate to each pair $(\Phi,\underline{\Phi})\in C_{c}^{\infty}(\mathcal{H}^-)\oplus C_{c}^{\infty}(\mathcal{I}^-)$ a solution $\psi \in D^-(\widetilde{\Sigma})$ such that $(\psi|_{\widetilde{\Sigma}},\mathbf{n}_{\widetilde{\Sigma}}\psi|_{\widetilde{\Sigma}\cap \{r_{\mathcal{H}}\leq r\leq r_{\mathcal{I}}\}})$ lie in a suitable energy space. We show that in fact $(\psi|_{\Sigma_0}, \mathbf{n}_{0}\psi|_{\Sigma_0\cap \{r_{\mathcal{H}}\leq r\leq r_{\mathcal{I}}\}})\in \mathcal{E}_{ \Sigma_0}$, so we can apply (the extension of) $\mathscr{F}$ to obtain a pair of radiation fields $(\underline{\Phi}',\underline{\Phi})\in  \mathcal{E}_{ \mathcal{H}^+}\oplus \mathcal{E}_{ \mathcal{I}^+}$.

\subsection{Scattering and regularity in black hole interiors}
\label{intro:interior}
We derive estimates for the radiation field in $\mathcal{M}_{\rm int}$ using once again the vector field $K=u^2\partial_u+v^2\partial_v$. Recall from Section \ref{intro:spacelikeinf} that the favourable properties of $K$ as a vector field multiplier are related to its role as an approximate conformal symmetry generator near infinity and its invariance under the Couch--Torrence conformal symmetry. The equation for the radiation field takes the same form in $\mathcal{M}_{\rm int}$ and $\mathcal{M}_{\rm ext}$ near $\mathcal{H}^+$ if one considers the standard Eddington--Finkelstein double-null coordinates in $\mathring{\mathcal{M}}_{\rm int}$ and in $\mathring{\mathcal{M}}_{\rm ext}$. Therefore, $K$ (now defined with respect to $(u,v)$ coordinates in $\mathring{\mathcal{M}}_{\rm int}$) remains useful in the black hole interior. The usefulness of $K$ in the interior of extremal black holes was already observed in \cite{gajic, gajic2, dejanjon1}.

\section{The forwards evolution map}
\label{sec:fowmap}
In this section, we present the energy estimates in the forwards time direction that are relevant for defining the \emph{forwards evolution map} $\mathscr{F}$ (see Section \ref{sec:constfevo}).

\subsection{Preliminary estimates}
We make use of the following Hardy inequalities:
\begin{lemma}[Hardy inequalities]
\label{lm:hardy}
Let $p\in \R\setminus \{-1\}$ and let $f: [a,b] \to \R$ be a $C^1$ function with $a,b\geq 0$. Then
\begin{align}
\label{eq:hardy1}
\int_{a}^{b} x^pf^2(x)\,dx\leq &\:4(p+1)^{-2} \int_{a}^{b} x^{p+2}\left|\frac{df}{dx}\right|^2\,dx+ 2b^{p+1}f^2(b),\quad \textnormal{for $p>-1$},\\
\label{eq:hardy2}
\int_{a}^{b} x^pf^2(x)\,dx\leq &\:4(p+1)^{-2} \int_{a}^{b} x^{p+2}\left|\frac{df}{dx}\right|^2\,dx+ 2a^{p+1}f^2(a),\quad \textnormal{for $p<-1$}.
\end{align}
 \end{lemma}
\begin{proof}
See the proof of Lemma 2.2 in \cite{paper1}.
\end{proof}
We define the \emph{angular momentum operators} $\Omega_i$, with $i=1,2,3$, as follows:
\begin{align*}
\Omega_1=&\:\sin \varphi\partial_{\theta}+\cot\theta \cos\varphi\partial_{\varphi},\\
\Omega_2=&\:-\cos \varphi \partial_{\theta}+\cot\theta \sin\varphi\partial_{\varphi},\\
\Omega_2=&\:\partial_{\varphi}.
\end{align*}
We denote for $\alpha=(\alpha_1,\alpha_2,\alpha_3)\in \N_0^3$
\begin{equation*}
\Omega^{\alpha}=\Omega_1^{\alpha_1}\Omega_2^{\alpha_2}\Omega_3^{\alpha_3}.
\end{equation*}

We now state the following standard inequalities on $\s^2$:
\begin{lemma}[Angular momentum operator inequalities]
\label{lm:angmom}
Let $f:\s^2\to \R$ be a $C^2$ function. Then we can estimate
\begin{align}
\label{eq:angmomentineq1}
\int_{\s^2} |\snabla_{\s^2}f|^2\,d\omega=&\: \sum_{|\alpha|=1} \int_{\s^2} (\Omega^{\alpha}f)^2\,d\omega,\\
\label{eq:angmomentineq2}
 \int_{\s^2}|\snabla_{\s^2}f|^2+|\snabla_{\s^2}^2f|^2\,d\omega \sim&\: \sum_{|\alpha|=1} \int_{\s^2} |\snabla_{\s^2}\Omega^{\alpha}f|^2\,d\omega \sim \sum_{1\leq |\alpha|\leq 2}\int_{\s^2} (\Omega^{\alpha} f)^2\,d\omega.
\end{align}
\end{lemma}
\begin{lemma}[Degenerate energy conservation]
\label{lm:tenconsv}
Let $\psi$ be a smooth solution to \eqref{eq:waveequation}. Then
\begin{align*}
\int_{\Sigma_{\tau}}&\mathbf{J}^T[\psi]\cdot \mathbf{n}_{\tau} d\mu_{\tau}\\
\sim&\: \int_{\underline{N}_{\tau}} [(\underline{L}\psi)^2  +|\snabla \psi|^2]r^2 d\omega du+\int_{\Sigma_{\tau} \cap \{r_{\mathcal{H}} \leq r \leq r_{\mathcal{I}} \}} (L\psi)^2+(\underline{L}\psi)^2+|\snabla\psi|^2\,d\mu_{\tau}+\int_{{N}_{\tau}} [({L}\psi)^2  +|\snabla \psi|^2]r^2 d\omega dv,\\
\int_{\mathcal{H}^+}& \mathbf{J}^T[\psi]\cdot L\,r^2d\omega dv=\int_{\mathcal{H}^+} (L\phi)^2\,d\omega dv,\\
\int_{\mathcal{I}^+}& \mathbf{J}^T[\psi]\cdot \underline{L}\,r^2d\omega du=\int_{\mathcal{I}^+} (\underline{L}\phi)^2\,d\omega du
\end{align*}
and
\begin{equation*}
\textnormal{div}\, \mathbf{J}^T[\psi]\equiv 0.
\end{equation*}
\end{lemma}
\begin{proof}
See for example \cite{aretakis1,aretakis2}.
\end{proof}
\subsection{Radiation field at null infinity}

We now recall some regularity properties of the radiation field at null infinity, which do not immediately follow from Theorem \ref{thm:globexuni}, and are derived in \cite{paper4}.

\begin{lemma}
\label{lm:maineqhoradfield}
Let $\psi$ be a smooth solution to \eqref{eq:waveequation}. Then for all $n\in \N_0$, we have that
\begin{equation}
\label{eq:maineqhoradfield}
\Lbar L((2D^{-1}r^2L)^n\phi)=[-4nr^{-1}+O(r^{-2})]L((2D^{-1}r^2L)^n\phi)+Dr^{-2}\slashed{\Delta}_{\s^2}((2D^{-1}r^2L)^n\phi)+\sum_{k=0}^{n-1} O(r^{-2}) (2D^{-1}r^2L)^k\phi.
\end{equation}
\begin{proof}
By \eqref{eq:waveequation} we obtain the following equation for $\phi$:
\begin{equation}
\label{eq:maineqradfield}
\Lbar L \phi=-\frac{DD'}{4r}\phi+\frac{D}{4r^2}\slashed{\Delta}_{\s^2}\phi,
\end{equation}
which implies \eqref{eq:maineqhoradfield} with $n=0$. We obtain $n\geq 0$ by induction.
\end{proof}
\end{lemma}
\begin{proposition}
\label{prop:regularityphiinfty}
Let $(\Psi,\Psi')\in C^{\infty}\left(\widehat{\Sigma}_0\right)\oplus C_{c}^{\infty}(\Sigma_0\cap \{r_{\mathcal{H}}\leq r\leq r_{\mathcal{I}}\})$. Then for all $k,l\in \N_0$ and $\alpha\in \N_0^3$,
\begin{equation*}
\lim_{v\to \infty} (r^2L)^kT^l\Omega^{\alpha}\phi(u,v,\theta,\varphi)<\infty.
\end{equation*}
In particular, the limit
\begin{equation*}
r\cdot \psi|_{\mathcal{I}^+}(u,\theta,\varphi):=\lim_{v\to \infty}r\psi(u,v,\theta,\varphi)
\end{equation*}
exists for all $u\geq 0$ and defines a smooth function on $\mathcal{I}^+_{\geq u_0}$.
\end{proposition}
\begin{proof}
The $k\leq 1$ case follows from Section 3 of \cite{paper1} by using \eqref{eq:maineqradfield}. We obtain the $k\geq 2$ case via an induction argument, where in the induction step we simply repeat the argument for $k=1$ using instead the commuted equation \eqref{eq:maineqhoradfield}. See also Proposition 6.2 of \cite{paper4}.
\end{proof}

\subsection{Forwards energy estimates}
\label{sec:estfow}
The two main ingredients for establishing energy decay estimates forwards in time are \textbf{Morawetz estimates} away from $\mathcal{H}^+$ and $\mathcal{I}^+$ (Theorem \ref{thm:morawetzfow} below) and \textbf{ hierarchies of $r^p$- and $(r-M)^{2-p}$-weighted estimates} in a neighbourhood of the event horizon and future null infinity (Theorem \ref{thm:fowhier} below).
\begin{theorem}[Morawetz/integrated local energy decay estimate, \cite{aretakis1}]
\label{thm:morawetzfow}
Let $0\leq \tau_1<\tau_2<\infty$\\ and $M<r_0<r_1<2M<r_2<r_3<\infty$, then for all $k,l\in \N_0$ and $\alpha\in \N_0^3$ there exists a constant $C=C(r_i,M,\Sigma_0,k,l,\alpha)>0$, such that
\begin{equation}
\label{eq:morawetzfow}
\begin{split}
\int_{\tau_1}^{\tau_2}&\left[\int_{\Sigma_{\tau}\cap(\{r_0\leq r\leq r_1\}\cup\{r_2\leq r\leq r_3\})}(\partial_vT^k\partial_r^l\Omega^{\alpha}\psi)^2+|\snabla T^k\partial_r^l\Omega^{\alpha}\psi|^2+(T^k\partial_r^l\Omega^{\alpha}\psi)^2+(\partial_rT^k\partial_r^l\Omega^{\alpha}\psi)^2\,d\mu_{\Sigma_{\tau}}\right]\,d\tau\\
\leq&\:C\sum_{j=0}^{k+l+|\alpha|}\int_{\Sigma_{\tau_1}}{\mathbf{J}}^T[T^j\psi]\cdot \mathbf{n}_{\Sigma_{\tau_1}}\,d\mu_{\Sigma_{\tau_1}}.
\end{split}
\end{equation}
Furthermore, we have that for any $M<r_0<r_1<\infty$:
\begin{equation}
\label{eq:morawetzfowTloss}
\int_{\tau_1}^{\tau_2}\left[\int_{\Sigma_{\tau}\cap\{r_0\leq r\leq r_1\}}\mathbf{J}^T[\psi]\cdot\mathbf{n}_{\Sigma_{\tau}}\,d\mu_{\Sigma_{\tau}}\right]\,d\tau\leq C(r_0,r_1,\Sigma_0)\sum_{j=0}^1\int_{\Sigma_{\tau_1}}{\mathbf{J}}^T[T^j\psi]\cdot \mathbf{n}_{\Sigma_{\tau_1}}\,d\mu_{\Sigma_{\tau_1}}.
\end{equation}
\end{theorem}

\begin{theorem}
\label{thm:fowhier}
Let $\psi$ be a solution to \eqref{eq:waveequation} arising from initial data \\$(\Psi,\Psi')\in C^{\infty}\left(\widehat{\Sigma}_0\right)\oplus C_{c}^{\infty}(\Sigma_0\cap \{r_{\mathcal{H}}\leq r\leq r_{\mathcal{I}}\})$. Let $k\in \N_0$ and $2k\leq p\leq 2+2k$, then we can estimate for all $0\leq \tau_1\leq \tau_2$:
\begin{equation}
\label{eq:fowhierarchy}
\begin{split}
\int_{{\underline{N}}_{\tau_2}}&(r-M)^{-p}(\underline{L}^{k+1}\phi)^2\,d\omega du+\int_{{N}_{\tau_2}}r^{p}(L^{k+1}\phi)^2\,d\omega dv\\
&+\int_{\mathcal{H}^+\cap\{\tau_1\leq \tau\leq \tau_2\}} (r-M)^{2-p}|\snabla_{\s^2}\underline{L}^k\phi|^2\,d\omega dv+\int_{\mathcal{I}^+\cap\{\tau_1\leq \tau\leq \tau_2\}}r^{p-2}|\snabla_{\s^2}L^k\phi|^2\,d\omega du\\
&+\int_{{{\underline{\mathcal{A}}}}^{\tau_2}_{\tau_1}}(r-M)^{1-p}(\underline{L}^{k+1}\phi)^2+(2-p)(r-M)^{3-p}|\snabla_{\s^2} \underline{L}^k\phi|^2\,d\omega du dv\\
+& \int_{{{\mathcal{A}}}^{\tau_2}_{\tau_1}}r^{p-1}(L^{k+1}\phi)^2+(2-p)r^{p-3}|\snabla_{\s^2} L^k\phi|^2\,d\omega dv d\tau\\
\leq&\: C\sum_{j=0}^k\int_{{\underline{N}}_{\tau_1}}(r-M)^{-p+2j}(L^{1+k-j}\phi)^2\,d\omega du+\int_{{N}_{\tau_1}}r^{p-2j}(L^{1+k-j}\phi)^2\,d\omega dv\\
&+C\sum_{j=0}^k\int_{\Sigma_{\tau_1}}{\mathbf{J}}^T[T^j\psi]\cdot \mathbf{n}_{\Sigma_{\tau_1}}\,d\mu_{\Sigma_{\tau_1}}.
\end{split}
\end{equation}
\end{theorem}
\begin{proof}
See Proposition 7.6 of \cite{paper4}.
\end{proof}

By combining Theorem \ref{thm:morawetzfow} and Theorem \ref{thm:fowhier} with Lemma \ref{lm:tenconsv} and applying the mean-value theorem along a dyadic sequence of times (``the pigeonhole principle''), one can obtain energy decay in time along the foliation $\Sigma_{\tau}$; see for example \cite{aretakis1, aretakis2} and \cite{paper4} for an application of this procedure in extremal Reissner--Nordstr\"om. 

In the present article, however, we will not apply the mean-value theorem, bur rather derive uniform boundedness estimates for various time-integrated energies on the left-hand side (see Proposition \ref{prop:integrateddecay}). We will then use these time-integrated energy estimates to obtain estimates for energy fluxes along $\mathcal{H}^+$ and $\mathcal{I}^+$ with growing time weights \emph{inside} the integrals (Corollary \ref{cor:fowdecay}).

\begin{proposition}
\label{prop:integrateddecay}
There exists a constant  $C=C(M,\Sigma_0,r_{\mathcal{H}},r_{\mathcal{I}})>0$ such that
\begin{equation}
\label{eq:integrateddecay1}
\begin{split}
\int_{0}^{\infty}\int_{\tau}^{\infty} \int_{\Sigma_{\tau'}} {\mathbf{J}}^T[\psi]\cdot \mathbf{n}_{\tau'} d\mu_{\tau'} d\tau'd\tau \leq&\: C\Bigg[\sum_{j=0}^1 \int_{{\underline{N}_{v_0}}}(r-M)^{-2+j}(\underline{L}T^{j}\phi)^2\,d\omega du+\int_{{N}_{u_0}}r^{2-j}(LT^{j}\phi)^2\,d\omega dv\\
&+\sum_{j=0}^2\int_{\Sigma_0}{\mathbf{J}}^T[T^j\psi]\cdot \mathbf{n}_{\Sigma_0}\,d\mu_{\Sigma_0}\Bigg].
\end{split}
\end{equation}
and
\begin{equation}
\label{eq:integrateddecay2}
\begin{split}
\int_{0}^{\infty}&\int_{\tau}^{\infty} \left[\int_{\mathcal{I}^+_{\geq \tau'}} {\mathbf{J}}^T[\psi]\cdot L\, r^2d\omega du\right] d\tau'd\tau+\int_{0}^{\infty}\int_{\tau}^{\infty} \left[\int_{\mathcal{H}^+_{\geq \tau'}} {\mathbf{J}}^T[\psi]\cdot \underline{L}\, d\omega dv\right] d\tau'd\tau\\
&+\int_{\mathcal{H}^+_{\geq v_0}} |\snabla_{\s^2}\phi|^2\,d\omega dv+\int_{\mathcal{I}^+_{\geq u_0}}|\snabla_{\s^2}\phi|^2\,d\omega du  \\
\leq&\: C\Bigg[\sum_{j=0}^1 \int_{{\underline{N}_{v_0}}}(r-M)^{-2+j}(\underline{L}T^{j}\phi)^2\,d\omega du+\int_{{N}_{u_0}}r^{2-j}(LT^{j}\phi)^2\,d\omega dv+\sum_{j=0}^2\int_{\Sigma_0}{\mathbf{J}}^T[T^j\psi]\cdot \mathbf{n}_{0}\,d\mu_{0}\Bigg].
\end{split}
\end{equation}
\end{proposition}
\begin{proof}
Note first of all that for all $\tau\geq 0$
\begin{equation}
\label{eq:relationTenpweightI}
\begin{split}
\int_{\tau_1}^{\tau_2}\int_{{N}_{\tau}} {\mathbf{J}}^T[\psi]\cdot L\,r^2d\omega dvd\tau \sim&\: \int_{\tau_1}^{\tau_2}\int_{{N}_{\tau}} r^2(L(r^{-1}\phi))^2+r^{-2}|\snabla_{\s^2}\phi|^2\,d\omega dv d\tau\\
\lesssim &\:  \int_{\tau_1}^{\tau_2}\int_{{N}_{\tau}} r^{-2}\phi^2+(L\phi)^2+r^{-2}|\snabla_{\s^2}\phi|^2\,d\omega dv d\tau\\
\lesssim &\:  \int_{\tau_1}^{\tau_2}\int_{{N}_{\tau}}(L\phi)^2+r^{-2}|\snabla_{\s^2}\phi|^2\,d\omega d\tau+ \int_{\Sigma_{\tau_1}} {\mathbf{J}}^T[\psi]\cdot \mathbf{n}_{\tau_1}\,d\mu_{\tau_1},
\end{split}
\end{equation}
where in the final inequality we applied Lemma \ref{lm:hardy} and \eqref{eq:morawetzfow}, using that $\phi$ attains a finite limit at $\mathcal{I}^+$, by Proposition \ref{prop:regularityphiinfty}.

Similarly, we have that
\begin{equation}
\label{eq:relationTenpweightH}
\begin{split}
\int_{\tau_1}^{\tau_2}\int_{{\underline{N}}_{\tau}} {\mathbf{J}}^T[\psi]\cdot \underline{L}\,r^2d\omega dud\tau \lesssim &\:  \int_{\tau_1}^{\tau_2}\int_{{\underline{N}}_{\tau}}(\underline{L}\phi)^2+(r-M)^2|\snabla_{\s^2}\phi|^2\,d\omega du d\tau+ \int_{\Sigma_{\tau_1}} {\mathbf{J}}^T[\psi]\cdot \mathbf{n}_{\tau_1}\,d\mu_{\tau_1}. 
\end{split}
\end{equation}
We combine \eqref{eq:relationTenpweightI} and \eqref{eq:relationTenpweightH} together with \eqref{eq:morawetzfowTloss} to obtain the estimate:
\begin{equation*}
\begin{split}
\int_{\tau}^{\infty}\int_{\Sigma_{\tau'}} {\mathbf{J}}^T[\psi]\cdot n_{\tau'}\,d\mu_{\tau'}\lesssim&\: \int_{\tau}^{\infty}\int_{{N}_{\tau'}}(L\phi)^2+r^{-2}|\snabla_{\s^2}\phi|^2\,d\omega dv d\tau'+ \int_{\tau}^{\infty}\int_{{\underline{N}}_{\tau'}}(\underline{L}\phi)^2+(r-M)^2|\snabla_{\s^2}\phi|^2\,d\omega du d\tau'\\
&+\sum_{j=0}^1\int_{\Sigma_{\tau}} {\mathbf{J}}^T[T^j\psi]\cdot \mathbf{n}_{\tau}\,d\mu_{\tau}.
\end{split}
\end{equation*}
We now apply \eqref{eq:fowhierarchy} with $k=0$ and $p=1$ to obtain:
\begin{equation}
\label{eq:intp1est}
\begin{split}
\int_{\tau}^{\infty}\int_{\Sigma_{\tau'}} {\mathbf{J}}^T[\psi]\cdot n_{\tau'}\,d\mu_{\tau'}\lesssim&\: \int_{{N}_{\tau}}r(L\phi)^2\,d\omega dv+ \int_{{\underline{N}}_{\tau}}(r-M)^{-1}(\underline{L}\phi)^2\,d\omega du+\sum_{j=0}^1\int_{\Sigma_{\tau}} {\mathbf{J}}^T[T^j\psi]\cdot \mathbf{n}_{\tau}\,d\mu_{\tau}.
\end{split}
\end{equation}

By Lemma \ref{lm:tenconsv} and \eqref{eq:intp1est}, we immediately obtain also
\begin{equation}
\label{eq:intp1estb}
\begin{split}
\int_{\tau}^{\infty}\left[\int_{\mathcal{I}^+_{\geq \tau'}} {\mathbf{J}}^T[\psi]\cdot L\,r^2d\omega\right] du+\int_{\tau}^{\infty}\left[\int_{\mathcal{H}^+_{\geq \tau'}} {\mathbf{J}}^T[\psi]\cdot \underline{L}\,r^2d\omega\right] dv\lesssim&\: \int_{{N}_{\tau}}r(L\phi)^2\,d\omega dv+ \int_{{\underline{N}}_{\tau}}(r-M)^{-1}(\underline{L}\phi)^2\,d\omega du\\
&+\sum_{j=0}^1\int_{\Sigma_{\tau}} {\mathbf{J}}^T[T^j\psi]\cdot \mathbf{n}_{\tau}\,d\mu_{\tau}.
\end{split}
\end{equation}

We integrate once more in $\tau$ and apply \eqref{eq:fowhierarchy} with $k=0$ and $p=2$ to obtain \eqref{eq:integrateddecay1}. Equation \eqref{eq:integrateddecay2} follows from \eqref{eq:integrateddecay1} by applying Lemma \ref{lm:tenconsv} applied in the region $D^+(\Sigma_{\tau'})$, together with \eqref{eq:fowhierarchy} with $p=2$ and $k=0$.
\end{proof}

The following simple lemma is crucial in order to bound energy norms along $\mathcal{H}^+$ and $\mathcal{I}^+$ with time-weights inside the integrals.
\begin{lemma}
\label{lm:integralsandweights}
Let $f\in C^{0}([x_0,\infty))$. Let $ n\in \N$ such that $\lim_{x\to \infty} x^{n+1} |f(x)|=0$. Then 
\begin{equation}
\label{eq:xweightid}
\int_{x_0}^{\infty} (x-x_0)^n f(x)\,dx=n! \int_{x_0}^{\infty} \int_{x_1}^{\infty}\ldots \int_{x_n}^{\infty} f(x_{n+1})\,dx_{n+1}dx_n\ldots dx_1.
\end{equation}
\end{lemma}
\begin{proof}
We integrate the left-hand side of \eqref{eq:xweightid} by parts to obtain
\begin{equation*}
\begin{split}
\int_{x_0}^{\infty} (x-x_0)^n f(x)\,dx=&-\int_{x_0}^{\infty} (x_1-x_0)^n \frac{d}{dx_1} \left[\int_{x_1}^{\infty} f(x_2)\,dx_2\right]\,dx_1\\
=&n \int_{x_0}^{\infty}(x_1-x_0)^{n-1} \int_{x_1}^{\infty} f(x_2)\,dx_2dx_1+  (x_1-x_0)^n\int_{x_1}^{\infty} f(x_2)\,dx_2\Big|^{x=\infty}_{x=x_0}\\
=&n \int_{x_0}^{\infty}(x_1-x_0)^{n-1} \int_{x_1}^{\infty} f(x_2)\,dx_2dx_1+\lim_{x\to \infty} x^n \int_x^{\infty} f(x')\,dx'.
\end{split}
\end{equation*}
Note that for $n\geq 1$:
\begin{equation*}
\lim_{x\to \infty} x^n \left|\int_x^{\infty} f(x')\,dx'\right|=\lim_{x\to \infty} \sup_{x'\geq x} x'^{n+1} |f(x')| x^n\int_x^{\infty}x^{-n-1}\,dx'=0
\end{equation*}
and hence,
\begin{equation*}
\begin{split}
\int_{x_0}^{\infty} (x-x_0)^n f(x)\,dx=n \int_{x_0}^{\infty}(x_1-x_0)^{n-1} \left[\int_{x_1}^{\infty} f(x_2)\,dx_2\right]dx_1.
\end{split}
\end{equation*}

We then keep integrating by parts to arrive \eqref{eq:xweightid}, using that
\begin{equation*}
\lim_{x\to \infty} x^{n-k}  \left|\int_x^{\infty} \int_{x_1}^{\infty}\ldots \int_{x_k}^{\infty} f(x')dx'dx_k\ldots dx_1\right|\leq   \sup_{x'\geq x_0} x^{n+1}|f(x')|\lim_{x\to \infty}\cdot  x^{-1}=0.
\end{equation*}
\end{proof}

\begin{corollary}
\label{cor:fowdecay}
There exists a constant  $C=C(M,\Sigma,r_{\mathcal{H}},r_{\mathcal{I}})>0$ such that
\begin{equation}
\label{eq:backestwithTder}
\begin{split}
\sum_{j=0}^2&\int_{\mathcal{H}^+_{\geq v_0}} v^{2-j}(L^{j+1}\phi)^2+|\snabla_{\s^2}\phi |^2\,d\omega dv+\int_{\mathcal{I}^+_{\geq u_0}} u^{2-j}(\Lbar^{j+1}\phi)^2+|\snabla_{\s^2}\phi |^2\,d\omega dv\\
\leq&\: C\Bigg[\sum_{j=0}^1 \int_{{\underline{N}_{v_0}}}(r-M)^{-2+j}(\underline{L}T^{j}\phi)^2\,d\omega du+\int_{{N}_{u_0}}r^{2-j}(LT^{j}\phi)^2\,d\omega dv\\
&+\sum_{j=0}^2\int_{\Sigma_0}{\mathbf{J}}^T[T^j\psi]\cdot \mathbf{n}_{\Sigma_0}\,d\mu_{\Sigma_0}\Bigg].
\end{split}
\end{equation}
\end{corollary}
\begin{proof}
First of all, by Theorem 5.1 from \cite{paper4} it follows that for $0\leq j\leq 2$ the following qualitative statements hold:\footnote{In fact, Theorem 5.1 from \cite{paper4} provides much stronger, quantitative $L^2(\s^2)$ time-decay estimates, but we do not require those here.}
\begin{align*}
\limsup_{v\to \infty} v^{3-j}\int_{\s^2}(LT^j\phi|_{\mathcal{H}^+})^2\,d\omega<&\:\infty,\\
\limsup_{u\to \infty} u^{3-j}\int_{\s^2}(\Lbar T^j\phi|_{\mathcal{I}^+})^2\,d\omega<&\:\infty.
\end{align*}

We can therefore apply Proposition \ref{prop:integrateddecay} together with Lemma \ref{lm:integralsandweights} with $n=2$ to obtain the desired estimate for the $j=0$ term. The $j=1$ estimate follows by replacing $\phi$ with $T\phi$ and applying \eqref{eq:intp1estb} and Lemma \ref{lm:integralsandweights} with $n=1$. Finally, we obtain the $j=0$ estimate by replacing $\psi$ with $T^2\psi$ and applying Lemma \ref{lm:tenconsv}.
\end{proof}

We will complement \eqref{eq:backestwithTder} in Corollary \ref{cor:fowdecay} with an estimate involving additional angular derivatives. The motivation for this comes from the energy estimates in Section \ref{sec:enestasympflat}.

\begin{corollary}
\label{cor:fowdecayv2}
There exists a constant  $C=C(M,\Sigma,r_{\mathcal{H}},r_{\mathcal{I}})>0$ such that
\begin{equation*}
\begin{split}
\sum_{j=0}^2&\int_{\mathcal{H}^+_{\geq v_0}} v^{2-j}(L^{j+1}\phi)^2+|\snabla_{\s^2}\phi |^2+|\snabla_{\s^2} L \phi|^2\,d\omega dv+\int_{\mathcal{I}^+_{\geq u_0}} u^{2-j}(\Lbar^{j+1}\phi)^2+|\snabla_{\s^2}\phi |^2+|\snabla_{\s^2} \Lbar \phi|^2\,d\omega du\\
\leq&\: C\Bigg[\sum_{j=0}^1 \int_{{\underline{N}_{v_0}}}(r-M)^{-2+j}(\underline{L}T^{j}\phi)^2\,d\omega du+\int_{{N}_{u_0}}r^{2-j}(LT^{j}\phi)^2\,d\omega dv\\
&+\sum_{j=0}^2\int_{\Sigma_0}{\mathbf{J}}^T[T^j\psi]\cdot \mathbf{n}_{\Sigma_0}\,d\mu_{\Sigma_0}+\sum_{|\alpha|=1}\int_{\Sigma_0}{\mathbf{J}}^T[\Omega^{\alpha}\psi]\cdot \mathbf{n}_{\Sigma_0}\,d\mu_{\Sigma_0}\Bigg].
\end{split}
\end{equation*}
\end{corollary}
\begin{proof}
We apply \eqref{eq:backestwithTder} and add the Lemma \ref{lm:tenconsv} estimate applied to $\Omega^{\alpha}\psi$, where $|\alpha|=1$.
\end{proof}

\subsection{Higher-order estimates}
\label{sec:hoestfow}
In this section we will derive the analogue of Corollary \ref{cor:fowdecay} for $T^n\phi$ with $n\geq 1$, but with stronger growing weights in $u$ and $v$ on the left-hand side (depending on $n$).
\begin{proposition}
\label{prop:hoedecay}
Let $n\in \N_0$. Then, there exists a constant $C=C(M,\Sigma,r_{\mathcal{H}},r_{\mathcal{I}},n )>0$, such that
\begin{equation}
\begin{split}
\label{eq:Tnenergyest}
\int_{\tau}^{\infty}\int_{\tau_{2n+1}}^{\infty}&\int_{\tau_{2n}}^{\infty}\ldots \int_{\tau_1}^{\infty} \left[\int_{\Sigma_{\tau'}} {\mathbf{J}}^T[T^n\psi]\cdot \mathbf{n}_{\tau'} d\mu_{\tau'}\right]\,d\tau' d\tau_1\ldots d\tau_{2n+1}\\
&+\int_{\tau }^{\infty} \int_{\tau_{2n}}^{\infty}\ldots \int_{\tau_2}^{\infty}\left[\int_{ \mathcal{H}^+_{\geq \tau_1}} |\snabla_{\s^2}T^{n}\phi|^2\,d\omega dv\right] d\tau_1\ldots d\tau_{2n}\\
&+\int_{\tau }^{\infty} \int_{\tau_{2n}}^{\infty}\ldots \int_{\tau_2}^{\infty}\left[\int_{ \mathcal{I}^+_{\geq \tau_1}} |\snabla_{\s^2}T^{n}\phi|^2\,d\omega du\right] d\tau_1\ldots d\tau_{2n}\\
\leq &\: C \Bigg[\sum_{j=0}^1 \sum_{m+|\alpha|+k\leq n} \int_{{N}_{\tau}} r^{2+2k-j}(L^{1+k}T^{m+j}\Omega^{\alpha}\phi)^2\,d\omega dv+ \int_{{\underline{N}}_{\tau}} (r-M)^{-2-2k+j}(\underline{L}^{1+k}T^{m+j}\Omega^{\alpha}\phi)^2\,d\omega du\\
&+\sum_{k\leq 2n+2} \int_{\Sigma_{\tau}} {\mathbf{J}}^T[T^k\psi]\cdot \mathbf{n}_{\tau} d\mu_{\tau}\Bigg].
\end{split}
\end{equation}
\end{proposition}
\begin{proof}
We will derive \eqref{eq:Tnenergyest} by induction. Observe that the $n=0$ case follows immediately from \eqref{prop:integrateddecay}. Now, suppose $\eqref{eq:Tnenergyest}$ holds for all $n=N$. Then, by replacing $T^N\psi$ with $T^{N+1} \psi$ (using that $T$ commutes with the wave operator $\square_g$) and setting $\tau=\tau_{2N+2}$, we have that
\begin{equation*}
\begin{split}
\int_{\tau_{2N+2}}^{\infty}\int_{\tau_{2N+1}}^{\infty}&\int_{\tau_{2N}}^{\infty}\ldots \int_{\tau_1}^{\infty} \left[\int_{\Sigma_{\tau'}} {\mathbf{J}}^T[T^{N+1}\psi]\cdot \mathbf{n}_{\tau'} d\mu_{\tau'}\right]\,d\tau_1\ldots d\tau_{2N+1} d\tau'\\
\leq &\: C \Bigg[\sum_{j=0}^1 \sum_{m+|\alpha|+k\leq N} \int_{{N}_{\tau_{2N+2}}} r^{2+2k-j}(L^{1+k}T^{m+j+1}\Omega^{\alpha}\phi)^2\,d\omega dv\\
&+ \int_{{\underline{N}}_{\tau_{2N+2}}} (r-M)^{-2-2k+j}(\underline{L}^{1+k}T^{m+j+1}\Omega^{\alpha}\phi)^2\,d\omega du\\
&+\sum_{k\leq 2N+2} \int_{\Sigma_{\tau_{2N+2}}} {\mathbf{J}}^T[T^{k+1}\psi]\cdot \mathbf{n}_{\tau} d\mu_{\tau}\Bigg].
\end{split}
\end{equation*}
Now, we apply the following identities
\begin{align}
\label{eq:TconvLLbar1}
LT^{m+1}\phi=&\:L^2T^m\phi+L\underline{L} T^m\phi=L^2T^m\phi+ \frac{1}{4}D r^{-2} \slashed{\Delta}_{\s^2}T^m\phi+O(r^{-3})T^m\phi,\\
\label{eq:TconvLLbar2}
\underline{L}T^{m+1}\phi=&\:\underline{L}^2T^m\phi+L\underline{L} T^m\phi=\underline{L}^2T^m\phi+ \frac{1}{4}D r^{-2} \slashed{\Delta}_{\s^2}T^N\phi+O((r-M)^{3})T^m\phi,
\end{align}
and we integrate once more in $\tau$ to obtain:
\begin{equation*}
\begin{split}
\int_{\tau_{2N+3}}^{\infty}&\int_{\tau_{2N+2}}^{\infty}\int_{\tau_{2N+1}}^{\infty}\int_{\tau_{2N}}^{\infty}\ldots \int_{\tau_1}^{\infty} \left[\int_{\Sigma_{\tau'}} {\mathbf{J}}^T[T^{N+1}\psi]\cdot \mathbf{n}_{\tau'} d\mu_{\tau'}\right]\,d\tau_1\ldots d\tau_{2N+2} d\tau'\\
\leq &\: C \Bigg[\sum_{j=0}^1 \sum_{m+|\alpha|+k\leq N} \int_{\tau_{2N+3}}^{\infty}\int_{{N}_{\tau_{2N+2}}} r^{2+2k-j}(L^{2+k}T^{m+j}\Omega^{\alpha}\phi)^2+r^{2+2k-j-4}|\slashed{\Delta}_{\s^2} L^{k}T^{m+j}\Omega^{\alpha}\phi|^2\,d\omega dvd\tau'\\
&+  \int_{\tau_{2N+3}}^{\infty}\int_{{\underline{N}}_{\tau_{2N+2}}} (r-M)^{-2-2k+j}(\underline{L}^{2+k}T^{m+j}\Omega^{\alpha}\phi)^2+(r-M)^{-2-2k+j+4}(\slashed{\Delta}_{\s^2}\underline{L}^{k}T^{m+j}\Omega^{\alpha}\phi)^2\,d\omega dud\tau'\\
&+\sum_{k\leq 2N+3}  \int_{\tau_{2N+3}}^{\infty}\int_{\Sigma_{\tau_{2N+2}}} {\mathbf{J}}^T[T^{k}\psi]\cdot \mathbf{n}_{\tau} d\mu_{\tau} d\tau'\Bigg],
\end{split}
\end{equation*}
where we moreover applied Lemma \ref{lm:hardy} (together with a standard averaging argument near the boundaries) and Theorem \ref{thm:morawetzfow} to control the lowest order derivative terms on the right-hand sides of \eqref{eq:TconvLLbar1} and \eqref{eq:TconvLLbar2}.

Now, apply \eqref{eq:fowhierarchy} with $k\leq N+1$ and $p=2k+1$ when $j=0$ and $k\leq N$ and $p=2k$ when $j=1$, together with Lemma \ref{lm:angmom}, to obtain
\begin{equation*}
\begin{split}
\int_{\tau_{2(N+1)+1}}^{\infty}\int_{\tau_{2N+2}}^{\infty}\int_{\tau_{2N+1}}^{\infty}&\int_{\tau_{2N}}^{\infty}\ldots \int_{\tau_1}^{\infty} \left[\int_{\Sigma_{\tau'}} {\mathbf{J}}^T[T^{N+1}\psi]\cdot \mathbf{n}_{\tau'} d\mu_{\tau'}\right]\,d\tau_1\ldots d\tau_{2N+2} d\tau'\\
\leq &\: C \Bigg[\sum_{j=0}^1 \sum_{m+|\alpha|+k\leq N+1} \int_{{N}_{\tau_{2N+3}}} r^{2+2k-j+1}(L^{1+k}T^{m+j}\Omega^{\alpha}\phi)^2\,d\omega dv\\
&+  \int_{{\underline{N}}_{\tau_{2N+3}}} (r-M)^{-2-2k+j-1}(\underline{L}^{1+k}T^{m+j}\Omega^{\alpha}\phi)^2\,d\omega du\\
&+\sum_{k\leq 2N+3}  \int_{\tau_{2N+3}}^{\infty}\left[\int_{\Sigma_{\tau_{2N+2}}} {\mathbf{J}}^T[T^{k}\psi]\cdot \mathbf{n}_{\tau_{2N+2}}\, d\mu_{\tau_{2N+2}}\right] d\tau_{\tau_{2N+2}}\Bigg].
\end{split}
\end{equation*}
Subsequently, apply \eqref{eq:fowhierarchy} again, with $k\leq N+1$ and $p=2k+2$ when $j=0$ and $k\leq N$ and $p=2k+1$ when $j=1$.

Finally, since we are integrating two more times in $\tau$ compared to the $n=N$ estimate, we can also include on the left-hand side of the above estimate the terms
\begin{equation*}
\begin{split}
\int_{\tau }^{\infty} &\int_{\tau_{2N+2}}^{\infty}\ldots \int_{\tau_2}^{\infty}\left[\int_{ \mathcal{H}^+_{\geq \tau_1}} |\snabla_{\s^2}T^{N+1}\phi|^2\,d\omega dv\right] d\tau_1\ldots d\tau_{2N+2}\\
&+\int_{\tau }^{\infty} \int_{\tau_{2N+2}}^{\infty}\ldots \int_{\tau_2}^{\infty}\left[\int_{ \mathcal{I}^+_{\geq \tau_1}} |\snabla_{\s^2}T^{N+1}\phi|^2\,d\omega du\right] d\tau_1\ldots d\tau_{2N+2}
\end{split}
\end{equation*}
 to obtain \eqref{eq:Tnenergyest} with $n=N+1$.
\end{proof}

\begin{corollary}
\label{cor:hoedecayv2}
Let $n\in \N_0$. Then, there exists a constant $C=C(M,\Sigma,r_{\mathcal{H}},r_{\mathcal{I}},n )>0$, such that
\begin{equation}
\begin{split}
\label{eq:Tnenergyestv2}
\sum_{j=0}^2&\sum_{m+2k+2|\alpha|\leq 2n}\int_{ \mathcal{H}^+_{\geq v_0}} v^{2k+2-j} (L^{1+k+m+j} \Omega^{\alpha}\phi)^2+v^{2k} |\snabla_{\s^2}L^{k+m} \Omega^{\alpha}\phi|^2\,d\omega dv\\
&+\int_{ \mathcal{I}^+_{\geq u_0}} u^{2k+2-j}(\Lbar^{1+k+m+j}\Omega^{\alpha}\phi)^2+ u^{2k}|\snabla_{\s^2}\Lbar^{k+m} \Omega^{\alpha}\phi|^2\,d\omega du\\
\leq &\: C \Bigg[\sum_{j=0}^1 \sum_{m+2|\alpha|+2k\leq 2n} \int_{{N}_{v_0}} r^{2+2k-j}(L^{1+k}T^{m+j}\Omega^{\alpha}\phi)^2\,d\omega dv+ \int_{{\underline{N}}_{u_0}} (r-M)^{-2-2k+j}(\underline{L}^{1+k}T^{m+j}\Omega^{\alpha}\phi)^2\,d\omega du\Bigg]\\
&+\sum_{\substack{m+2|\alpha|\leq 2n+2\\ |\alpha|\leq n}} \int_{\Sigma_{0}} {\mathbf{J}}^T[T^m \Omega^{\alpha}\psi]\cdot \mathbf{n}_{0}\, d\mu_{0}.
\end{split}
\end{equation}
\end{corollary}
\begin{proof}
We apply \eqref{eq:Tnenergyest}, with $n$ replaced by $k\leq n$ and $\phi$ replaced by $T^{m}\Omega^{\alpha}\phi$ with $|\alpha|\leq n-k$ and $m\leq 2n-2k-2|\alpha|$ suitably chosen, and combine it with Lemma \ref{lm:tenconsv}, Lemma \ref{lm:integralsandweights} to derive \eqref{eq:Tnenergyestv2}. The decay of $L^{1+k+m+j} \Omega^{\alpha}\phi|_{\mathcal{H}+}$ and $\Lbar^{1+k+m+j} \Omega^{\alpha}\phi|_{\mathcal{I}+}$ that is required in order to be able to apply Lemma \ref{lm:integralsandweights} follows from Theorem 5.1 of \cite{paper4}.
\end{proof}

We will complement \eqref{eq:Tnenergyestv2} in Corollary \ref{cor:hoedecayv2} with an estimate involving additional angular derivatives. The motivation for this comes from the energy estimates in Section \ref{sec:hoestasymflat}.

\begin{corollary}
\label{cor:hoedecayv3}
Let $n\in \N_0$. Then, there exists a constant $C=C(M,\Sigma,r_{\mathcal{H}},r_{\mathcal{I}},n )>0$, such that
\begin{equation}
\begin{split}
\label{eq:Tnenergyestv3}
\sum_{j=0}^2&\sum_{m+2k+2|\alpha|\leq 2n}\int_{ \mathcal{H}^+_{\geq v_0}} v^{2k+2-j} (L^{1+k+m+j} \Omega^{\alpha}\phi)^2+v^{2k} |\snabla_{\s^2}L^{k+m} \Omega^{\alpha}\phi|^2+v^{2k} |\snabla_{\s^2}L^{k+1+m} \Omega^{\alpha}\phi|^2\,d\omega dv\\
&+\int_{ \mathcal{I}^+_{\geq u_0}} u^{2k+2-j}(\Lbar^{1+k+m+j}\Omega^{\alpha}\phi)^2+ u^{2k}|\snabla_{\s^2}\Lbar^{k+m} \Omega^{\alpha}\phi|^2+ u^{2k}|\snabla_{\s^2}\Lbar^{k+1+m} \Omega^{\alpha}\phi|^2\,d\omega du\\
\leq &\:  C\sum_{j=0}^2\sum_{2|\alpha|+2k+m\leq 2n} \Bigg[\int_{N_{u_0}} r^{2k+2-j}(L^{k+1}\Omega^{\alpha}T^{j+m}\phi)^2+r^{2k}|\snabla_{\s^2}L^{k+1}\Omega^{\alpha}T^{m}\phi|^2\,d\omega dv \\
&+ \int_{\underline{N}_{v_0}} (r-M)^{-2k-2+j}(\Lbar^{k+1}\Omega^{\alpha}T^{j+m}\phi)^2+(r-M)^{-2k}|\snabla_{\s^2}\Lbar^{k+1}\Omega^{\alpha}T^{m}\phi|^2\,d\omega du\Bigg]\\
&+C\sum_{2|\alpha|+m\leq 2n+2}\int_{\Sigma_0} \mathbf{J}^T[\Omega^{\alpha}T^m\psi]\cdot \mathbf{n}_0\,d\mu_0.
\end{split}
\end{equation}
\end{corollary}

\subsection{Construction of the forwards evolution map}
\label{sec:constfevo}
In this section, we will use the uniform estimates derived in Section \ref{sec:estfow} and \ref{sec:hoestfow} in order to construct the forward evolution map between suitable weighted energy spaces. 

\begin{proposition}
\label{prop:fowtennormbound}
Let $(\Psi,\Psi')\in  C^{\infty}(\widehat{\Sigma}_0)\oplus C^{\infty}(\Sigma_0\cap \{r_{\mathcal{H}}\leq r\leq r_{\mathcal{I}}\})$. Then the corresponding solution $\psi$ to \eqref{eq:waveequation} satisfies
\begin{equation*}
(r\cdot\psi|_{\mathcal{H}^+_{\geq v_0}},r\cdot\psi|_{\mathcal{I}^+_{\geq u_0}})\in \mathcal{E}^T_{\mathcal{H}^+_{\geq v_0}}\oplus \mathcal{E}^T_{\mathcal{I}^+_{\geq u_0}}.
\end{equation*}
and furthermore,
\begin{equation*}
||r\cdot\psi|_{\mathcal{H}^+_{\geq v_0}}||^2_{ \mathcal{E}^T_{\mathcal{H}^+_{\geq v_0}}}+||r\cdot\psi|_{\mathcal{I}^+_{\geq u_0}}||^2_{\mathcal{E}^T_{\mathcal{I}^+_{\geq u_0}}}=||(\Psi,\Psi')||^2_{\mathcal{E}^T_{\Sigma_0}}.
\end{equation*}
\end{proposition}
\begin{proof}
Follows from Lemma \ref{lm:tenconsv}, \cite{paper4} and Lemma \ref{lm:completionhorinf}.
\end{proof}
\begin{definition}
Define the forwards evolution map $\mathscr{F}: C^{\infty}(\widehat{\Sigma}_0)\oplus C^{\infty}(\Sigma_0\cap \{r_{\mathcal{H}}\leq r\leq r_{\mathcal{I}}\}) \to \mathcal{E}^T_{\mathcal{H}^+_{\geq v_0}}\oplus \mathcal{E}^T_{\mathcal{I}^+_{\geq u_0}}$ as the following linear operator:
\begin{equation*}
\mathscr{F}(\Psi,\Psi')=(r\cdot\psi|_{\mathcal{H}^+_{\geq v_0}},r\cdot\psi|_{\mathcal{I}^+_{\geq u_0}}),
\end{equation*}
where $\psi$ is the unique solution to \eqref{eq:waveequation} with $(\psi|_{\Sigma_0},\mathbf{n}_{\Sigma_0}\psi|_{\Sigma_0\cap \{r_{\mathcal{H}}\leq r\leq r_{\mathcal{I}}\}})=(\Psi,\Psi')$. Then $\mathscr{F}$ extends uniquely to a linear bounded  operator:
\begin{equation*}
\mathscr{F}: \mathcal{E}^T_{\Sigma_0}\to \mathcal{E}^T_{\mathcal{H}^+_{\geq v_0}}\oplus \mathcal{E}^T_{\mathcal{I}^+_{\geq u_0}}.
\end{equation*}
\end{definition}

\begin{proposition}
\label{prop:fowbound}
Let $n\in \N_0$. Then $\mathscr{F}$ is a bounded linear operator from $C^{\infty}(\widehat{\Sigma}_0)$ to $\mathcal{E}_{n; \mathcal{H}^+_{\geq v_0}} \oplus \mathcal{E}_{n; \mathcal{I}^+_{\geq u_0}}$, which can uniquely be extended as as a bounded linear operator 
\begin{equation*}
\mathscr{F}_n: \mathcal{E}_{n; \Sigma_0} \to \mathcal{E}_{n; \mathcal{H}^+_{\geq v_0}} \oplus \mathcal{E}_{n; \mathcal{I}^+_{\geq u_0}} .
\end{equation*}
We moreover have that $\mathscr{F}_n=\mathscr{F}|_{ \mathcal{E}_{n; \Sigma_0}}$.
\end{proposition}
\begin{proof}
First of all, we assume that $(\Psi,\Psi') \in C_{c}^{\infty}(\widehat{\Sigma}_0)\oplus C^{\infty}(\Sigma_0\cap \{r_{\mathcal{H}}\leq r\leq r_{\mathcal{I}}\})$. We apply Proposition \ref{prop:hoedecay} to obtain estimates for the corresponding solution $\psi: D^+(\Sigma_0)\to \R$. By \cite{paper4}, it follows in particular that $\lim_{v\to \infty} \phi|_{\mathcal{H}^+}=0$ and $\lim_{u\to \infty} \phi|_{\mathcal{H}^+}=0$ for all $0\leq k \leq n$. Furthermore, by Corollary \ref{cor:hoedecayv2}, we have that there exists a constant $C>0$ such that
\begin{equation*}
||\phi|_{\mathcal{H}^+}||_{\mathcal{E}_{n;\mathcal{H}^+_{\geq v_0}}}^2+||\phi|_{\mathcal{I}^+}||_{\mathcal{E}_{n;\mathcal{I}^+_{\geq u_0}}}^2\leq C\cdot ||(\Psi,\Psi')||_{\mathcal{E}_{n; \Sigma_0}}^2.
\end{equation*}
Then, by Lemma \ref{lm:completionhorinf} it follows that $\mathscr{F}_n(\Psi,\Psi')=(\phi|_{\mathcal{H}^+},\phi|_{\mathcal{I}^+})\in \mathcal{E}_{n; \mathcal{H}^+_{\geq v_0}} \oplus \mathcal{E}_{n; \mathcal{I}^+_{\geq u_0}}$, so $||\mathscr{F}_n||\leq \sqrt{C}$. Then by a standard functional analytic argument, $\mathscr{F}_n$ extends uniquely to the completion $\mathcal{E}_{n; \mathcal{H}^+_{\geq v_0}} \oplus \mathcal{E}_{n; \mathcal{I}^+_{\geq u_0}}$ and the extension $\mathscr{F}_n$ also satisfies $||\mathscr{F}_n||\leq \sqrt{C}$.
\end{proof}

\section{The backwards evolution map}
\label{sec:backwest}
In this section we will construct a map from suitably weighted energy spaces on $\mathcal{H}^+$ and $\mathcal{I}^+$ to suitably weighted energy spaces on $\Sigma_0$. The construction will proceed in two steps. As a first step, we construct in Section \ref{sec:semiglobalexun} a map with the domain $C_{c}^{\infty}(\mathcal{H}^+_{\geq v_0})\oplus C_{c}^{\infty}(\mathcal{I}^+_{\geq u_0})$.  In other words, we establish \emph{semi-global} existence and uniqueness for the backwards scattering initial value problem.

In the second step, this will be promoted to \emph{global} existence and uniqueness in Section \ref{sec:bacwevomap} by using the global, uniform weighted energy estimates of Section \ref{sec:bacweest} that are valid on the \emph{completion} of $C_{c}^{\infty}(\mathcal{H}^+_{\geq v_0})\oplus C_{c}^{\infty}(\mathcal{I}^+_{\geq u_0})$ with respect to the associated energy norms.

\subsection{Initial value problem with compactly supported scattering data}
\label{sec:semiglobalexun}
In this section we will associate to a pair  $(\underline{\Phi},\Phi)\in C_{c}^{\infty}(\mathcal{H}^+_{\geq v_0})\oplus C_{c}^{\infty}(\mathcal{I}^+_{\geq u_0})$ a unique solution to \eqref{eq:waveequation} in $D^+(\Sigma_0)$ such that $r\cdot \psi|_{\mathcal{H}^+}={\underline{\Phi}}$ and $r\cdot \psi|_{\mathcal{I}^+}={\Phi}$. This association is central to the definition of the backwards evolution map (see Definition \ref{def:backwmap}). 

 \begin{figure}[H]
	\begin{center}
\includegraphics[scale=0.7]{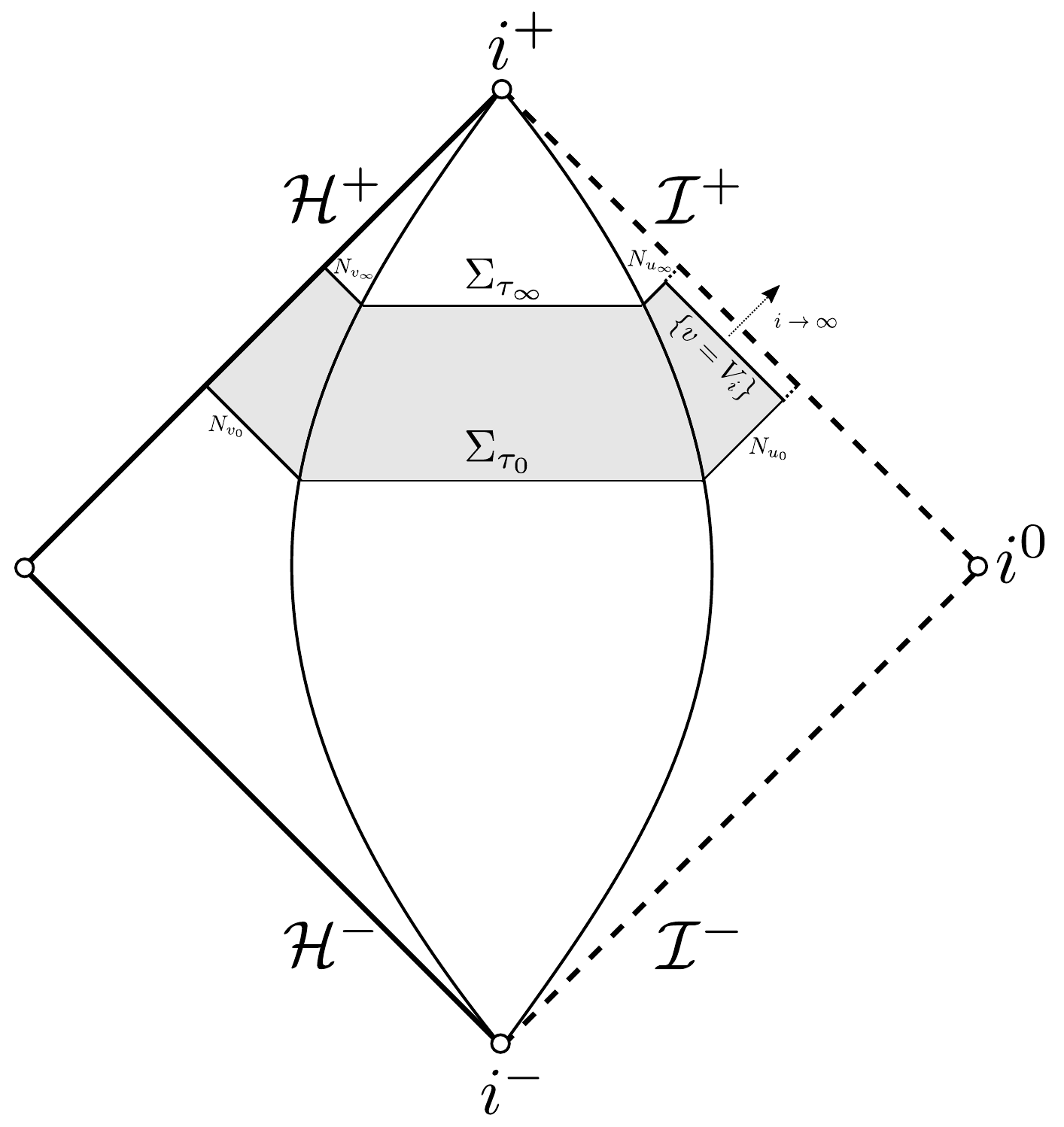}
\end{center}
\vspace{-0.2cm}
\caption{A Penrose diagrammatic representation of the spacetime regions in consideration in Proposition \ref{prop:mainpropbackdef}.}
	\label{fig:backwcons}
\end{figure}

\begin{proposition}
\label{prop:mainpropbackdef}
 Let $\tau_{\infty}>0$ and $-\infty<u_{-\infty},v_{-\infty}\leq u_0,v_0$ and define $u_{\infty}:=u_0+\tau_{\infty}$ and $v_{\infty}:=v_0+\tau_{\infty}$. Let $({\underline{\Phi}},{\Phi})\in C_{c}^{\infty}(\mathcal{H}^+)\oplus C_{c}^{\infty}(\mathcal{I}^+)$ such that ${\textnormal{supp}\, {\underline{\Phi}}}\subset \mathcal{H}^+_{v_{-\infty}<v<v_{\infty}}$ and ${\textnormal{supp}\, {\Phi}}\subset \mathcal{I}^+_{u_{-\infty}<u<u_{\infty}}$. Denote also with ${\Phi}$ a smooth extension to $\widehat{\mathcal{R}}$ of ${\Phi}$ such that ${\Phi}$ vanishes in a neighbourhood of $\widehat{\Sigma}_{\tau_{\infty}}$. Denote with $\psi_{i}$, with $i\in \N$, the unique smooth solution to \eqref{eq:waveequation} in ${D}^+({\Sigma}_{0})\cap \{v\leq V_i:=V\cdot i\}\cap J^-(\Sigma_{\tau_{\infty}})$ such that:
\begin{align*}
r\psi_{i}|_{\mathcal{H}^+_{v_{0}\leq v\leq v_{\infty}}}=&\:{\underline{\Phi}},\\
r\cdot \psi_{i}|_{\{v=V_i\}\cap \{u_{0}\leq u \leq u_{\infty}\}}=&\:{\Phi},\\
(\psi_{i}|_{\Sigma_{\tau_{\infty}}},\mathbf{n}_{{\Sigma}_{\tau_{\infty}}}(\psi_{i})|_{\Sigma_{\tau_{\infty}}\cap\{r_{\mathcal{H}}\leq r \leq r_{\mathcal{I}}\}})=&\:(0,0),
\end{align*}
with $V$ suitably large such that $r(\tau_{\infty},V)>r_{\mathcal{I}}$; see also Figure \ref{fig:backwcons}.
\begin{itemize}
\item[1.)](Semi-global existence) There exists a $\psi \in C^{\infty}({D}^+({\Sigma}_{0})\cap J^-(\Sigma_{\tau_{\infty}}))$ with the following property: let $\widetilde{V}\geq V$ and $n\in \N_0$, then there exists a strictly increasing sequence function $i^{(n)}: \N \to \N$, such that:
\begin{equation*}
\lim_{j\to \infty}||r\psi_{i^{(n)}(j)}-r\psi||_{C^n({D}^+({\Sigma}_{0})\cap \{v\leq \widetilde{V}\}\cap J^-(\Sigma_{\tau_{\infty}}))}=0.
\end{equation*}
In particular, $\square_g\psi=0$. Furthermore,
\begin{align*}
M\psi |_{\mathcal{H}^+_{v_{0}\leq v \leq v_{\infty}}}=&\:{\underline{\Phi}},\\
r\cdot\psi|_{\mathcal{I}^+_{u_{0}\leq u \leq u_{\infty}}}=&\:{\Phi},\\
(\psi|_{\Sigma_{\tau_{\infty}}},\mathbf{n}_{\Sigma_{\tau_{\infty}}}(\psi)|_{\Sigma_{\tau_{\infty}}\cap\{r_{\mathcal{H}}\leq r \leq r_{\mathcal{I}}\}})=&\:(0,0)
\end{align*}
and for any $j,k,l\in \N_0$
\begin{equation}
\label{eq:qualboundhoradfield}
\lim_{v\to \infty}\slashed{\nabla}_{\s^2}^j(r^2L)^k\underline{L}^l(r\cdot \psi)(u,v,\theta,\varphi)<\infty.
\end{equation}
\item[2.)] (Uniqueness)
If $\widetilde{\psi} \in C^{\infty}({D}^+({\Sigma}_{0})\cap J^-(\Sigma_{\tau_{\infty}}))$ is another solution to \eqref{eq:waveequation} that satisfies $$M\widetilde{\psi} |_{\mathcal{H}^+_{v_{0}\leq v \leq v_{\infty}}}={\underline{\Phi}}\quad \textnormal{and} \quad r\widetilde{\psi}|_{\mathcal{I}^+_{u_{0}\leq u \leq u_{\infty}}}={\Phi},$$ then $\widetilde{\psi}=\psi$.
\end{itemize}
\end{proposition}

\begin{remark}
A variant of Proposition \ref{prop:mainpropbackdef} was established in Proposition 9.1.4 of \cite{linearscattering} in the setting of sub-extremal Kerr. Note however that Proposition \ref{prop:mainpropbackdef} establishes in addition qualitative bounds on the radiation field $r\psi$ and weighted higher-order derivatives thereof in the form of the inequality \eqref{eq:qualboundhoradfield}, which will be necessary in the backwards-in-time estimates of Section \ref{sec:bacweest}.
\end{remark}

\begin{proof}[Proof of Proposition \ref{prop:mainpropbackdef}]
Observe first of all that $\psi_i$ is well-defined by local existence and uniqueness with smooth initial data on $\Sigma_{\tau_{\infty}}\cup \{v=V_i\}$.

Apply the divergence theorem with $\mathbf{J}^T$ in the region $\{r\geq r_{\mathcal{I}}\}$ bounded to the past by $I_v=\{v'=v\}\cap\{u_0\leq u\leq u_{\infty}\}$ and ${\Sigma}_{0}$ and to the future by $I_{V_i}:=\{v'=V_i\}\cap\{u_0\leq u\leq u_{\infty}\}$ and $\Sigma_{\tau_{\infty}}$ to obtain:
\begin{equation*}
\int_{I_v}\mathbf{J}^T[\psi_{i}]\cdot \underline{L}\,r^2d\omega du \leq \int_{I_{V_i}}\mathbf{J}^T[\psi_{i}]\cdot \underline{L}\,r^2d\omega du,
\end{equation*}
which is equivalent to
\begin{equation*}
\int_{I_v} r^2 (\underline{L}\psi_{i})^2+\frac{1}{4}D r^{-2}|\slashed{\nabla}_{\s^2}\psi_{i}|^2\,d\omega du \leq \int_{I_{V_i}} r^2 (\underline{L}(r^{-1}{\Phi}))^2+\frac{1}{4}D r^{-2}|\slashed{\nabla}_{\s^2}{\Phi}|^2\,d\omega du.
\end{equation*}
By applying the fundamental theorem of calculus in $u$, integrating from $u'=\tau_{\infty}$ to $u'=u$, together with Cauchy--Schwarz, we therefore obtain
\begin{equation*}
\left[\int_{\s^2}\psi_{i}^2\,d\omega\right](u,v)\lesssim \int_{u}^{\tau_{\infty}}r^{-2}(u',v)\,du'\cdot \int_{I_{V_i}} r^2(\underline{L}\psi_{i})^2\,d\omega du',
\end{equation*}
where we used that $\psi_{i}|_{\Sigma_{\tau_{\infty}}}=0$, from which it follows that
\begin{equation}
\label{eq:tdependentestphi}
\begin{split}
\left[\int_{\s^2}\phi_{i}^2\,d\omega\right](u,v)\lesssim&\: u_{\infty}^2  \int_{I_{v}}(\underline{L} \phi_i)^2\,d\omega du\lesssim u_{\infty}^2  \int_{I_{v}}r^2(\underline{L} \psi_i)^2+\psi_i^2\,d\omega du \\
\lesssim &\:  u_{\infty}^4  \int_{I_{V_i}} r^2 (\underline{L}(r^{-1}{\Phi}))^2+\frac{1}{4}D r^{-2}|\slashed{\nabla}_{\s^2}{\Phi}|^2\,d\omega du'.
\end{split}
\end{equation}

Now, we can use \eqref{eq:tdependentestphi} and  \eqref{eq:maineqhoradfield} with $n=0$ together with the fundamental theorem of calculus in the $u$-direction to obtain
\begin{equation*}
\left[\int_{\s^2}r^4(L\phi_{i})^2\,d\omega\right](u,v)\lesssim C(\tau_{\infty},u_0)\cdot \sum_{|\alpha|\leq 2}\int_{I_{V_i}} r^2 (\underline{L}(r^{-1}{\Omega^{\alpha}\Phi}))^2+\frac{1}{4}D r^{-2}|\slashed{\nabla}_{\s^2}{\Omega^{\alpha}\Phi}|^2\,d\omega du'.
\end{equation*}
Similarly, we can use \eqref{eq:maineqhoradfield} and Lemma \ref{lm:angmom} in a simple induction argument to conclude that for all $n\in \N$ we have in $\{r\geq r_{\mathcal{I}}\}$:
\begin{equation*}
\left[\int_{\s^2}((r^2L)^n\phi_{i})^2\,d\omega\right](u,v)\lesssim C(\tau_{\infty},u_{0})\cdot  \sum_{|\alpha|\leq 2n}\int_{I_{V_i}} r^2 (\underline{L}(r^{-1}{\Omega^{\alpha}\Phi}))^2+\frac{1}{4}D r^{-2}|\slashed{\nabla}_{\s^2}{\Omega^{\alpha}\Phi}|^2\,d\omega du'.
\end{equation*}
We can immediately apply the above argument to $\Omega^{\alpha}\phi$ and $T^k$ for any $\alpha \in \N_0^3$, $k\in \N_0$, together with a standard Sobolev inequality on $\s^2$ to obtain the following $i$-\emph{independent} estimate: for all $k\in \N_0$ and $\alpha\in \N_0^3$, there exists a constant $C(\tau_{\infty},u_0)>0$, such that
\begin{equation}
\label{eq:vinftypointwiseest}
\begin{split}
|(r^2L)^n&T^k\Omega^{\alpha}\phi_{i}|^2(u,v,\theta,\varphi)\\
\leq&\: C(\tau_{\infty},u_{-\infty})\cdot\sum_{|\alpha'|\leq 2n+2}\int_{I_{V_i}}  (\underline{L}T^k\Omega^{\alpha+\alpha'}{\Phi})^2+r^{-2}(T^k\Omega^{\alpha+\alpha'}{\Phi})^2+\frac{1}{4}D r^{-2}|\slashed{\nabla}_{\s^2}T^k\Omega^{\alpha+\alpha'}{\Phi}|^2\,d\omega du.
\end{split}
\end{equation}
We obtain a similar estimate in the region $\{r\leq r_{\mathcal{H}}\}$ by reversing the roles of $u$ and $v$ (integrating in the $v$-direction) and replacing $r$ by $(r-M)^{-1}$:
\begin{equation}
\label{eq:vinftypointwiseesthor}
\begin{split}
|((r-M)^{-2}\underline{L})^n&T^k\Omega^{\alpha}\phi_{i}|^2(u,v,\theta,\varphi)\\
\leq&\: C(\tau_{\infty},v_{-\infty})\cdot\sum_{|\alpha'|\leq 2n+2}\int_{\mathcal{H}^+_{v_{-\infty}\leq v \leq v_{\infty}}} (LT^k\Omega^{\alpha+\alpha'}\underline{\Phi})^2\,d\omega du.
\end{split}
\end{equation}
Given $\widetilde{V}>0$ arbitrarily large and $n\geq N$, we have by \eqref{eq:vinftypointwiseest} and \eqref{eq:vinftypointwiseesthor} that for $I \geq 1$ such that $V_I>\widetilde{V}$, $\phi_i$ is uniformly bounded in $i$ for all $i\geq I$ with respect to the $C^k$ norm on $J^+(\widehat{\Sigma}_0)\cap J^-(\widehat{\Sigma}_{\tau_{\infty}})\cap \{v\leq \widetilde{V}\}\subset \widehat{\mathcal{R}}$ with respect to the differentiable structure on $\widehat{\mathcal{R}}$ and therefore, by Arzel\`a--Ascoli, there exists a subsequence $\{\phi_{i^{(k)}(j)}\}_{j\in \N}$, with $i^{(k)}: \N\to \N$ a strictly increasing function, such that moreover $\{\phi_{i^{(k+1)}(j)}\}_{j\in \N}$ is a subsequence of $\{\phi_{i^{(k)}(j)}\}_{j\in \N}$ for all $k\in \N$, and $\{\phi_{i^{(k)}(j)}\}_{j\in \N}$ converges in $C^k(J^+(\widehat{\Sigma}_0)\cap J^-(\widehat{\Sigma}_{\tau_{\infty}})\cap \{v\leq \widetilde{V}\})$, for any $k\in \N$, to the smooth function $\phi$ on $J^+(\widehat{\Sigma}_0)\cap J^-(\widehat{\Sigma}_{\tau_{\infty}})\cap \{v\leq \widetilde{V}\}$.\footnote{This formulation of Arzel\`a--Ascoli, involving $C^k$-norms, follows straightforwardly by applying the usual formulation (involving convergence with respect to the $C^0$-norm) successively to derivatives of $\phi_i$, using uniform boundedness of all derivatives of $\phi_i$, and taking a further subsequence for each successive derivative. One can further apply a diagonal argument to extract a single subsequence of $\{\phi_i\}$ converging with respect to all $C^k$ norms to $\phi$, but we omit this step.}

We can extend the domain of $\phi$ to $J^+(\widehat{\Sigma}_0)\cap J^-(\widehat{\Sigma}_{\tau_{\infty}})$ as follows: we replace $\widetilde{V}$ above with $\widetilde{V}'>\widetilde{V}$, applying Arzel\`a--Ascoli to the subsequence $\phi_{i_k}$ (starting from $k$ suitably large) in the corresponding larger spacetime region and passing to a further subsequence. By uniqueness of limits, the resulting limit, which we note by $\phi'$ has to agree with $\phi$ when $v\leq \widetilde{V}$.

The above $C^k$ convergence moreover implies that $\square_g\psi=0$, with $\psi=r^{-1}\phi$, $\phi|_{\mathcal{H}^+}={\underline{\Phi}}$ and 
$$(\psi|_{\Sigma_{\tau_{\infty}}},n_{\Sigma_{\tau_{\infty}}}(\psi)|_{\Sigma_{\tau_{\infty}}\cap\{r_{\mathcal{H}}\leq r \leq r_{\mathcal{I}}\}})=(0,0).$$

We also have by \eqref{eq:vinftypointwiseest} that for any $\epsilon>0$, there exist a $V>0$ and $K>0$, such that for all $v\geq V$ and $k>K$ in the region $\{r\geq r_{\mathcal{I}}\}$:
\begin{equation*}
\begin{split}
|r\psi(u,v,\theta,\varphi)-{\Phi}(u,\theta,\varphi)|\leq&\: |r\psi(u,v,\theta,\varphi)-r\psi_k(u,v,\theta,\varphi)|+|r\psi_{i_k}(u,v,\theta,\varphi)-{\Phi}(u,\theta,\varphi)|\\
\leq &\:|r\psi(u,v,\theta,\varphi)-r\psi_{i_k}(u,v,\theta,\varphi)|+ r^{-2}(u,v)\int_{v}^{V i_k} r^2|L \phi_{i_k}|\,dv'\\
\leq&\: \epsilon.
\end{split}
\end{equation*}
for all $u\in (-u_{-\infty},u_{\infty}]$ and $(\theta,\varphi)\in \s^2$. Hence, 
\begin{equation*}
\lim_{v\to \infty} r\psi(u,v,\theta,\varphi)={\Phi}(u,\theta,\varphi).
\end{equation*}

We can analogously use \eqref{eq:vinftypointwiseest} to obtain for all $j,k,l\in \N_0$:
\begin{equation*}
\lim_{v\to \infty}\slashed{\nabla}_{\s^2}^j(r^2\partial_v)^k\partial_u^l(r\cdot \psi)(u,v,\theta,\varphi)<\infty.
\end{equation*}
Furthermore, by replacing $\psi$ by $T^l\Omega^{\alpha}\psi$ we can conclude that with respect to the differentiable structure in $\hat{\mathcal{R}}$, the restriction $r\psi|_{\mathcal{I}^+}$ is a smooth function on $\mathcal{I}^+$, satisfying $r\psi|_{\mathcal{I}^+}={\Phi}$. We can therefore conclude $1.)$ of the proposition.

Now suppose $\widetilde{\psi}$ is another smooth solution to $\square_g \widetilde{\psi}=0$, such that 
\begin{align*}
M\widetilde{\psi} |_{\mathcal{H}^+_{v_{-\infty}\leq v \leq v_{\infty}}}=&\:{\underline{\Phi}},\\
r\cdot\widetilde{\psi}|_{\mathcal{I}^+_{u_{-\infty}\leq u \leq u_{\infty}}}=&\:{\Phi},\\
(\widetilde{\psi}|_{\Sigma_{\tau_{\infty}}},n_{\Sigma_{\tau_{\infty}}}(\widetilde{\psi})|_{\Sigma_{\tau_{\infty}}\cap\{r_{\mathcal{H}}\leq r \leq r_{\mathcal{I}}\}})=&\:(0,0)
\end{align*}
By a global $T$-energy estimate, we have that
\begin{equation*}
\int_{\widetilde{\Sigma}_{0}}\mathbf{J}^T[\tilde{\psi}-\psi]\cdot \mathbf{n}_{\widetilde{\Sigma}_{0}}\,d\mu_{\widetilde{\Sigma}_{0}}=0,
\end{equation*}
so $\widetilde{\psi}=\psi$, which concludes $2.)$ of the proposition.
\end{proof}
\subsection{Backwards energy estimates}
\label{sec:bacweest}
In this section, we will derive estimates for the solutions $\psi$ to \eqref{eq:waveequation} constructed in Proposition \ref{prop:mainpropbackdef} that are uniform in $\tau_{\infty}$. This is crucial for constructing solutions with scattering data that is not compactly suppported.

The main tool we will develop is this section is a hierarchy of $r$-weighted estimates in the backwards time direction. However, we will first state a \emph{backwards} Morawetz estimate that follows immediately from the results in \cite{aretakis1}, i.e.\ an analogue of Theorem \ref{thm:morawetzfow} in the backwards time direction.

In this section, we will always assume that $\psi$ is a solution to \eqref{eq:waveequation} arising from smooth and compactly supported scattering data $({\underline{\Phi}},{\Phi})\in C_{c}^{\infty}(\mathcal{H}^+_{\geq v_0})\oplus C_{c}^{\infty}(\mathcal{I}^+_{\geq u_0})$, as in Proposition \eqref{prop:mainpropbackdef}, i.e. let $\tau_{\infty}>0$ such that $\overline{\textnormal{supp}\, {\underline{\Phi}}}\subset \mathcal{H}^+_{v_0\leq v<\tau_{\infty}+v_0}$ and $\overline{\textnormal{supp}\, {\Phi}}\subset \mathcal{I}^+_{u_0\leq u<\tau_{\infty}+u_0}$. 
\begin{proposition}[Backwards Morawetz/integrated local energy decay estimate,\\ \cite{aretakis1}]
\label{prop:backwmorawetz}
Let $0\leq \tau_1<\tau_2<\infty$\\ and $M<r_0<r_1<2M<r_2<r_3<\infty$, then for all $k,l\in \N_0$ and $\alpha\in \N_0^3$ there exists a constant $C=C(r_i,M,\Sigma_0,k,l,\alpha)>0$, such that
\begin{equation}
\label{eq:morawetz}
\begin{split}
\int_{\tau_1}^{\tau_2}&\left[\int_{\Sigma_{\tau}\cap(\{r_0\leq r\leq r_1\}\cup\{r_2\leq r\leq r_3\})}(\partial_vT^k\partial_r^l\Omega^{\alpha}\psi)^2+|\snabla T^k\partial_r^l\Omega^{\alpha}\psi|^2+(T^k\partial_r^l\Omega^{\alpha}\psi)^2+(\partial_rT^k\partial_r^l\Omega^{\alpha}\psi)^2\,d\mu_{\Sigma_{\tau}}\right]\,d\tau\\
\leq&\:C\sum_{j=0}^{k+l+|\alpha|} \int_{\mathcal{H}^+\cap\{\tau_1\leq v \leq \tau_2\}}{\mathbf{J}}^T[T^j\psi]\cdot L \,d\omega dv+\int_{\mathcal{I}^+\cap\{\tau_1\leq u \leq \tau_2\}}r^2 {\mathbf{J}}^T[T^j\psi]\cdot \underline{L} d\omega du\\
&+C\int_{\Sigma_{\tau_2}}{\mathbf{J}}^T[T^j\psi]\cdot \mathbf{n}_{\Sigma_{\tau_2}}\,d\mu_{\Sigma_{\tau_2}}.
\end{split}
\end{equation}
\end{proposition}
\begin{proof}
The proof of \eqref{eq:morawetz} follows directly from the Morawetz estimates established in \cite{aretakis1}.
\end{proof}

In the propositions below, we derive the ``backwards analogues'' of the hierarchies from Proposition \ref{thm:fowhier}.
\begin{proposition}
\label{prop:backwrp}
Let $0\leq p\leq 2$, then there exists a constant $C(M,\Sigma, r_{\mathcal{I}},r_{\mathcal{H}})>0$, such that for all $0\leq \tau_1\leq \tau_2\leq \tau_{\infty}$:
\begin{equation}
\label{eq:backhierarchy}
\begin{split}
\int_{{\underline{N}}_{\tau_1}}&(r-M)^{-2+p}(\underline{L}\phi)^2\,d\omega du+\int_{{N}_{\tau_1}}r^{p}(L\phi)^2\,d\omega dv\\
\leq&\: C\int_{{{\underline{\mathcal{A}}}}^{\tau_2}_{\tau_1}}(r-M)^{-p+1}(\underline{L}\phi)^2+(2-p)(r-M)^{3-p}|\snabla_{\s^2} \phi|^2\,d\omega du d\tau\\
&+C\int_{{{\mathcal{A}}}^{\tau_2}_{\tau_1}}r^{p-1}(L \phi)^2+(2-p)r^{p-3}|\snabla_{\s^2} \phi|^2\,d\omega dv d\tau\\
&+C\int_{{\underline{N}}_{\tau_2}}(r-M)^{-2+p}(\underline{L}\phi)^2\,d\omega du+C\int_{{N}_{\tau_2}}r^{p}(L\phi)^2\,d\omega dv\\
&+C\int_{\mathcal{H}^+\cap\{\tau_1\leq \tau \leq \tau_2\}}(r-M)^{2-p}|\snabla_{\s^2} \phi|^2\,d\omega dv+C\int_{\mathcal{I}^+\cap\{\tau_1\leq u \leq \tau_2\}}r^{p-2}|\snabla_{\s^2} \phi|^2\,d\omega du\\
&+C\int_{\mathcal{H}^+\cap\{\tau_1\leq \tau \leq \tau_2\}}{\mathbf{J}}^T[\psi]\cdot L \,d\omega dv+C\int_{\mathcal{I}^+\cap\{\tau_1\leq u \leq \tau_2\}}r^2 {\mathbf{J}}^T[\psi]\cdot \underline{L}\, d\omega du\\
&+C\int_{\Sigma_{\tau_2}}{\mathbf{J}}^T[\psi]\cdot \mathbf{n}_{\Sigma_{\tau_2}}\,d\mu_{\Sigma_{\tau_2}},
\end{split}
\end{equation}
\end{proposition}
\begin{proof}
Recall that $\phi$ satisfies the equation:
\begin{equation}
\label{eq:eqphi}
\underline{L}L\phi= \frac{1}{4}DD'r^{-1}\phi+Dr^{-2}\slashed{\Delta}_{\s^2}\phi.
\end{equation}
Therefore,
\begin{equation*}
\begin{split}
\underline{L}(r^p(L\phi)^2)=&-\frac{p}{2}r^{p-1}(L\phi)^2+2r^pL\phi\cdot \underline{L}L\phi\\
=&-\frac{p}{2}r^{p-1}(L\phi)^2+O(r^{p-3})\phi\cdot L\phi+(r^{p-2}+O(r^{p-3}))L\phi\cdot \slashed{\Delta}_{\s^2}\phi\\
=&-\frac{p}{2}r^{p-1}(L\phi)^2+O(r^{p-3})\phi\cdot L\phi+\slashed{\nabla}_{\s^2}(r^{p-2}L\phi\cdot \slashed{\nabla}_{\s^2}\phi)-L\left(\frac{1}{2}r^{p-2}|\snabla_{\s^2}\phi|^2\right)\\\
&-\frac{1}{2}(2-p)r^{p-3}|\snabla_{\s^2}\phi|^2+O(r^{p-3})\slashed{\Delta}_{\s^2}\phi\cdot L\phi\\
=&-\frac{p}{2}r^{p-1}(L\phi)^2+O(r^{p-3})\phi\cdot L\phi+\slashed{\nabla}_{\s^2}(r^{p-2}L\phi\cdot \slashed{\nabla}_{\s^2}\phi)-L\left(\frac{1}{2}r^{p-2}|\snabla_{\s^2}\phi|^2\right)\\\
&-\frac{1}{2}(2-p)r^{p-3}|\snabla_{\s^2}\phi|^2+O(r^{p-1})\underline{L}L\phi\cdot L\phi
\end{split}
\end{equation*}

By reordering the terms, we therefore obtain:
\begin{equation}
\label{eq:maineqrpest}
\begin{split}
\underline{L}((r^p+O(r^{p-1}))(L\phi)^2)&+L\left(\frac{1}{2}r^{p-2}|\snabla_{\s^2}\phi|^2\right)-\slashed{\nabla}_{\s^2}(r^{p-2}L\phi\cdot \slashed{\nabla}_{\s^2}\phi)\\
=&\left(-\frac{p}{2}r^{p-1}+O(r^{p-2})\right)(L\phi)^2-\frac{1}{2}(2-p)r^{p-3}|\snabla_{\s^2}\phi|^2+O(r^{p-3})\phi\cdot L\phi.
\end{split}
\end{equation}
Let $\chi$ denote a cut-off function and consider $\chi \phi$.

We integrate both sides of \eqref{eq:maineqrpest} in spacetime to obtain:
\begin{equation}
\label{eq:maineqrpest2}
\begin{split}
-&\int_{{N}_{\tau_2}}(r^p+O(r^{p-1}))(L(\chi \phi))^2\,d\omega dr+\int_{{N}_{\tau_1}}(r^p+O(r^{p-1}))(L(\chi \phi))^2\,d\omega dr\\
&- \frac{1}{2}\int_{\mathcal{I}^+\cap\{\tau_1\leq u \leq \tau_2\}}r^{p-2}|\snabla_{\s^2}\chi \phi|^2\,d\omega d\tau\\
=&\: \int_{\tau_1}^{\tau_2} \int_{{N}_{\tau}} \left(\frac{p}{2}r^{p-1}+O(r^{p-2})\right)(L(\chi\phi))^2+\frac{1}{2}(2-p)r^{p-3}|\snabla_{\s^2}\chi \phi|^2+O(r^{p-3})\chi\phi\cdot L(\chi \phi)\,d\omega dr d\tau\\
&+\sum_{|\alpha_1|+|\alpha_2|\leq 1}\int_{\tau_1}^{\tau_2}\int_{{N}_{\tau}}R_{\chi}[\partial^{\alpha_1}\phi\cdot \partial^{\alpha_2}\phi]\,d\omega dr d\tau\\
\leq &\: \int_{\tau_1}^{\tau_2} \int_{{N}_{\tau}} Cr^{p-1}(L(\chi\phi))^2+\frac{1}{2}(2-p)r^{p-3}|\snabla_{\s^2}\chi \phi|^2\,d\omega dr d\tau\\
&+ C(r_0,r_1)\int_{\mathcal{H}^+\cap\{\tau_1\leq v \leq \tau_2\}}{\mathbf{J}}^T[\psi]\cdot \partial_v \,d\omega dv+C(r_0,r_1)\int_{\mathcal{I}^+\cap\{\tau_1\leq u \leq \tau_2\}}r^2 {\mathbf{J}}^T[\psi]\cdot \partial_u d\omega du\\
&+C(r_0,r_1)\int_{\Sigma_{\tau_2}}{\mathbf{J}}^T[\psi]\cdot \mathbf{n}_{\Sigma_{\tau_2}}\,d\mu_{\Sigma_{\tau_2}},
\end{split}
\end{equation}
where we applied Lemma \ref{lm:hardy} and \eqref{eq:morawetz} to arrive at the inequality above. See also the derivations in the proof of Lemma 6.3 in \cite{paper4} in the special case $n=0$.

We can repeat the above steps in the region where $r\leq r_{\mathcal{H}}$ by reversing the roles of $L$ and $\underline{L}$ and replacing $r^p$ with $(r-M)^{-p}$; see the proof of Lemma 6.3 in \cite{paper4} for more details.
\end{proof}

We subsequently apply Proposition \ref{prop:backwrp} to arrive at uniform weighted energy estimates along $\Sigma_0$.

\begin{proposition}
\label{prop:mainbackweest}
Then there exists a constant $C(M,\Sigma, r_{\mathcal{I}},r_{\mathcal{H}})>0$, such that
\begin{equation}
\label{eq:backwjzerotau0}
\begin{split}
 \int_{\underline{N}_{v_0}}& (r-M)^{-2} (\underline{L}\phi)^2\,d\omega du +\int_{N_{u_0}} r^2(L\phi)^2\,d\omega dv+\int_{\Sigma_0} {\mathbf{J}}^T[\psi] \cdot \mathbf{n}_0\,d\mu_0\\
 \leq&\: C\int_{\mathcal{H}^+_{\geq v_0}}  v^2(L \phi)^2+|\snabla_{\s^2} \phi|^2\,d\omega dv+C\int_{\mathcal{I}^+_{\geq u_0} }u^2(\Lbar \phi)^2+|\snabla_{\s^2} \phi|^2\,d\omega du.
 \end{split}
\end{equation}
We moreover have that
\begin{equation}
\label{eq:mainbackweest}
\begin{split}
\sum_{j=0}^1 \int_{\underline{N}_{v_0}}& (r-M)^{-2+j} (\underline{L}T^j\phi)^2\,d\omega du +\int_{N_{u_0}} r^{2-j}(LT^j\phi)^2\,d\omega dv+\sum_{j=0}^2\int_{\Sigma_0} {\mathbf{J}}^T[T^j\psi] \cdot \mathbf{n}_0\,d\mu_0\\
\leq&\: C\sum_{j=0}^2\int_{\mathcal{H}^+_{\geq v_0}}  v^{2-j}(LT^j\phi)^2+|\snabla_{\s^2}\phi|^2\,d\omega dv+C\int_{\mathcal{I}^+_{\geq u_0}}  u^{2-j}(\Lbar T^j\phi)^2+|\snabla_{\s^2}\phi|^2\,d\omega du.
 \end{split}
\end{equation}
\end{proposition}
\begin{proof}
By applying Lemma \ref{lm:hardy} and Lemma \ref{lm:tenconsv}, it follows that
\begin{equation}
\label{eq:backwTenbound}
\begin{split}
\int_{\tau}^{\infty}& \int_{N_{\tau}} (L\phi)^2+r^{-2}|\snabla_{\s^2}\phi|^2\,d\omega dv d\tau'+\int_{\tau}^{\infty} \int_{\underline{N}_{\tau}} (\underline{L}\phi)^2+r^{-2}|\snabla_{\s^2}\phi|^2\,d\omega du d\tau' \leq C\int_{\tau}^{\infty} \int_{N_{\tau}} {\mathbf{J}}^T[\psi]\cdot L\,r^2d\omega dv d\tau'\\
&+C\int_{\tau}^{\infty} \int_{\underline{N}_{\tau}} {\mathbf{J}}^T[\psi]\cdot \underline{L}\,r^2d\omega du d\tau'\\
\leq &\:  C\int_{\tau}^{\infty}\int_{\mathcal{H}^+_{\geq \tau'}} {\mathbf{J}}^T[\psi]\cdot L\,r^2d\omega dvd\tau'+ C\int_{\tau}^{\infty}\int_{\mathcal{I}^+_{\geq \tau}} {\mathbf{J}}^T[\psi]\cdot \underline{L}\,r^2d\omega dud\tau'\\
\leq &\:C\int_{\tau}^{\infty} \int_{\mathcal{H}^+_{\geq \tau'}}  (L \phi)^2\,d\omega dvd\tau'+C \int_{\tau}^{\infty}\int_{\mathcal{I}^+_{\geq \tau}}  (\Lbar \phi)^2\,d\omega dud\tau'.
\end{split}
\end{equation}
We now apply \eqref{eq:backhierarchy} with $p=1$, together with \eqref{eq:backwTenbound} to conclude that
\begin{equation*}
\begin{split}
 \int_{\underline{N}_{\tau}}& (r-M)^{-1} (\underline{L}\phi)^2\,d\omega du +\int_{N_{\tau}} r(L\phi)^2\,d\omega dv\leq C\int_{\tau}^{\infty} \int_{\mathcal{H}^+_{\geq \tau_1}}  (L \phi)^2\,d\omega dvd\tau_1\\
 &+C \int_{\tau}^{\infty}\int_{\mathcal{I}^+_{\geq \tau_1}}  (\Lbar \phi)^2\,d\omega dud\tau_1+ C\int_{\mathcal{H}^+_{\geq \tau}}  (L \phi)^2\,d\omega dv+C\int_{\mathcal{I}^+_{\geq \tau}}  (\Lbar \phi)^2\,d\omega du.
 \end{split}
\end{equation*}
Next, apply \eqref{eq:backhierarchy} with $p=2$ to obtain
\begin{equation*}
\begin{split}
 \int_{\underline{N}_{\tau}}& (r-M)^{-2} (\underline{L}\phi)^2\,d\omega du +\int_{N_{\tau}} r^2(L\phi)^2\,d\omega dv\leq C\int_{\tau}^{\infty} \int_{\tau_1}^{\infty}\int_{\mathcal{H}^+_{\geq \tau_2}}  (L \phi)^2\,d\omega dvd\tau_2 d\tau_1\\
 &+C \int_{\tau}^{\infty} \int_{\tau_1}^{\infty}\int_{\mathcal{I}^+_{\geq \tau_2}}  (\Lbar \phi)^2\,d\omega dud\tau_2 d\tau_1+C\int_{\tau}^{\infty} \int_{\mathcal{H}^+_{\geq \tau_1}}  (L \phi)^2\,d\omega dvd\tau_1+C \int_{\tau}^{\infty}\int_{\mathcal{I}^+_{\geq \tau_1}}  (\Lbar \phi)^2\,d\omega dud\tau_1\\
 &+ C\int_{\mathcal{H}^+_{\geq \tau}}  |\snabla_{\s^2} \phi|^2\,d\omega dv+C\int_{\mathcal{I}^+_{\geq \tau}}  |\snabla_{\s^2} \phi|^2\,d\omega du\\
 &+ C\int_{\mathcal{H}^+_{\geq \tau}}  (L \phi)^2\,d\omega dv+C\int_{\mathcal{I}^+_{\geq \tau}}  (\Lbar \phi)^2\,d\omega du.
 \end{split}
\end{equation*}
We apply Lemma \ref{lm:integralsandweights} to rewrite the right-hand side above to arrive at:
\begin{equation}
\label{eq:backwjzero} 
\begin{split}
 \int_{\underline{N}_{\tau}}& (r-M)^{-2} (\underline{L}\phi)^2\,d\omega du +\int_{N_{0}} r^2(L\phi)^2\,d\omega dv\leq C\int_{\mathcal{H}^+_{\geq v_{\mathcal{H}^+} (\tau)}}  v^2(L \phi)^2+|\snabla_{\s^2} \phi|^2\,d\omega dv\\
 &+C\int_{\mathcal{I}^+_{\geq u_{\mathcal{I}^+} (\tau)}}  u^2(\Lbar \phi)^2+|\snabla_{\s^2} \phi|^2\,d\omega du,
 \end{split}
\end{equation}
which leads to  \eqref{eq:backwjzerotau0} when we take $\tau=0$.

By applying the above estimates to $T\psi$ and $T^2\psi$ we moreover obtain:
\begin{equation*}
\begin{split}
\sum_{j=0}^1 \int_{\underline{N}_{0}}& (r-M)^{-2+j} (\underline{L}T^j\phi)^2\,d\omega du +\int_{N_{0}} r^{2-j}(LT^j\phi)^2\,d\omega dv\leq C\sum_{j=0}^1\int_{\mathcal{H}^+_{\geq v_0}}  v^{2-j}(\partial_vT^j\phi)^2\,d\omega dv\\
&+C\int_{\mathcal{I}^+_{\geq u_0}}  u^{2-j}(\partial_uT^j\phi)^2\,d\omega du.
 \end{split}
\end{equation*}
We conclude the proof by combining the above proposition with Lemma \ref{lm:tenconsv} to obtain
\begin{equation*}
\begin{split}
\sum_{j=0}^1 \int_{\underline{N}_{0}}& (r-M)^{-2+j} (\underline{L}T^j\phi)^2\,d\omega du +\int_{N_{0}} r^{2-j}(LT^j\phi)^2\,d\omega dv+\sum_{j=0}^2\int_{\Sigma_0} {\mathbf{J}}^T[T^j\psi] \cdot \mathbf{n}_0\,d\mu_0\\
\leq&\: C\sum_{j=0}^2\int_{\mathcal{H}^+_{\geq v_0}}  v^{2-j}(LT^j\phi)^2+|\snabla_{\s^2} \phi|^2\,d\omega dv+C\int_{\mathcal{I}^+_{\geq u_0}}  u^{2-j}(\Lbar T^j\phi)^2+|\snabla_{\s^2} \phi|^2\,d\omega du.
 \end{split}
\end{equation*}
\end{proof}

\begin{remark}
Note that in contrast with the estimates in Proposition \ref{prop:integrateddecay}, there is no loss of derivatives (caused by the application of \eqref{eq:morawetzfowTloss}) on the right-hand side of \eqref{eq:backwjzerotau0}.
\end{remark}

We will complement \eqref{eq:mainbackweestv2} in Proposition \ref{prop:mainbackweest} with an estimate involving additional angular derivatives. The motivation for this comes from the energy estimates in Section \ref{sec:enestasympflat}.
\begin{corollary}
\label{cor:mainbackweestv2}
Then there exists a constant $C(M,\Sigma, r_{\mathcal{I}},r_{\mathcal{H}})>0$, such that
\begin{equation}
\label{eq:mainbackweestv2}
\begin{split}
\sum_{j=0}^1 \int_{\underline{N}_{v_0}}& (r-M)^{-2+j} (\underline{L}T^j\phi)^2\,d\omega du +\int_{N_{u_0}} r^{2-j}(LT^j\phi)^2\,d\omega dv+\sum_{j=0}^2\int_{\Sigma_0} {\mathbf{J}}^T[T^j\psi] \cdot \mathbf{n}_0\,d\mu_0+\sum_{|\alpha|=1}\int_{\Sigma_0} {\mathbf{J}}^T[\Omega^{\alpha}\psi] \cdot \mathbf{n}_0\,d\mu_0\\
\leq&\: C\sum_{j=0}^2\int_{\mathcal{H}^+_{\geq v_0}}  v^{2-j}(LT^j\phi)^2+|\snabla_{\s^2}\phi|^2+|\snabla_{\s^2} L \phi|^2\,d\omega dv+C\int_{\mathcal{I}^+_{\geq u_0}}  u^{2-j}(\Lbar T^j\phi)^2+|\snabla_{\s^2}\phi|^2+|\snabla_{\s^2} L \phi|^2\,d\omega du.
 \end{split}
\end{equation}
\end{corollary}
\begin{proof}
We apply \eqref{eq:mainbackweestv2} together with Lemma \ref{lm:tenconsv} applied to $\Omega^{\alpha}\psi$, with $|\alpha|=1$.
\end{proof}

\subsection{Higher-order estimates}
\label{sec:backwhoest}

By commuting \eqref{eq:eqphi} with $L^k$, we arrive at
\begin{equation}
\label{eq:eqphicommLk}
\underline{L}L(L^k\phi)=(r^{-2}+O(r^{-3}))\slashed{\Delta}_{\s^2}L^k\phi+\sum_{j=0}^kO(r^{-3-j})L^{k-j}\phi+\sum_{j=1}^{k}O(r^{-2-j})\slashed{\Delta}_{\s^2}L^{k-j}\phi.
\end{equation}
Similarly, we can commute \eqref{eq:eqphi} with $\underline{L}^k$ to obtain:
\begin{equation}
\label{eq:eqphicommLbark}
\underline{L}L(\underline{L}^k\phi)=[M^{-4}(r-M)^2+O((r-M)^{3}]\slashed{\Delta}_{\s^2}\underline{L}^k\phi+\sum_{j=0}^kO((r-M)^{3+j})\underline{L}^{k-j}\phi+\sum_{j=1}^kO((r-M)^{2+j})\slashed{\Delta}_{\s^2}\underline{L}^{k-j}\phi.
\end{equation}

\begin{proposition}
Fix $k\in \N_0$. Let $2k\leq p\leq 2+2k$, then we can estimate for all $0\leq \tau_1\leq \tau_2\leq \tau_{\infty}$:
\begin{equation}
\label{eq:hobackhierarchy}
\begin{split}
\int_{{\underline{N}}_{\tau_1}}&(r-M)^{-p}(\underline{L}^{k+1}\phi)^2\,d\omega dr+\int_{{N}_{\tau_1}}r^{p}(L^{k+1}\phi)^2\,d\omega dv\\
\leq&\: C\sum_{j=0}^k\sum_{|\alpha|\leq j}\int_{{{\underline{\mathcal{A}}}}^{\tau_2}_{\tau_1}}(r-M)^{1-p+2j}(\underline{L}^{k+1-j}\Omega^{\alpha}\phi)^2+(2-p)(r-M)^{3-p+2j}|\snabla_{\s^2} \underline{L}^{k-j}\Omega^{\alpha} \phi|^2\,d\omega du d\tau\\
&+C\sum_{j=0}^k\sum_{|\alpha|\leq j}\int_{{{\mathcal{A}}}^{\tau_2}_{\tau_1}}r^{p-1-2j}(L^{k+1-j}\Omega^{\alpha}\phi)^2+(2-p)r^{p-3-2j}|\snabla_{\s^2} L^{k-j}\Omega^{\alpha}\phi|^2\,d\omega dv d\tau\\
&+C\sum_{j=0}^k\sum_{|\alpha|\leq j}\int_{{\underline{N}}_{\tau_2}}(r-M)^{-p+2j}(\underline{L}^{k+1-j}\Omega^{\alpha}\phi)^2\,d\omega du+C\sum_{j=0}^k\sum_{|\alpha|\leq j}\int_{{N}_{\tau_2}}r^{p-2j}(L^{k+1-j}\Omega^{\alpha}\phi)^2\,d\omega dv\\
&+C\sum_{j=0}^k\sum_{|\alpha|\leq j}\int_{\mathcal{H}^+\cap\{\tau_1\leq \tau \leq \tau_2\}}(r-M)^{2-p-2j}|\snabla_{\s^2} \underline{L}^{k-j}\Omega^{\alpha}\phi|^2\,d\omega dv\\
&+C\sum_{j=0}^k\sum_{|\alpha|\leq j}\int_{\mathcal{I}^+\cap\{\tau_1\leq \tau \leq \tau_2\}}r^{p-2-2j}|\snabla_{\s^2} L^{k-j}\Omega^{\alpha}\phi|^2\,d\omega du\\
&+C\sum_{j=0}^k\sum_{|\alpha|\leq j}\int_{\mathcal{H}^+\cap\{\tau_1\leq \tau \leq \tau_2\}}{\mathbf{J}}^T[T^{k-j}\Omega^{\alpha}\psi]\cdot L \,d\omega dv+C\int_{\mathcal{I}^+\cap\{\tau_1\leq \tau \leq \tau_2\}}r^2 {\mathbf{J}}^T[T^{k-j}\Omega^{\alpha} \psi]\cdot \underline{L}\, d\omega du\\
&+C\sum_{j=0}^k\int_{\Sigma_{\tau_2}}{\mathbf{J}}^T[T^j\psi]\cdot \mathbf{n}_{\Sigma_{\tau_2}}\,d\mu_{\Sigma_{\tau_2}}.
\end{split}
\end{equation}
\end{proposition}
\begin{proof}
The proof is a straightforward generalisation of the proof of Proposition \ref{prop:backwrp}: we repeat the steps in the proof of Proposition \ref{prop:backwrp}, but we replace $\phi$ with either $L^k\phi$ (when $\{r\geq r_{\mathcal{I}}\}$) or $\underline{L}^k\phi$ (when $\{r\leq r_{\mathcal{H}}\}$), and we use \eqref{eq:eqphicommLk} and \eqref{eq:eqphicommLbark}.
\end{proof}
\begin{proposition}
\label{prop:mainbackweestho}
Let $n\in \N_0$ and let $\psi$ be a solution to \eqref{eq:waveequation} such that $\psi|_{\Sigma_{\tau_{\infty}}}=0$ and  $n_{\tau_{\infty}}\psi|_{\Sigma_{\tau_{\infty}}}=0$ for some $\tau_{\infty}<\infty$. Then there exists a constant $C(M,\Sigma, r_{\mathcal{I}},r_{\mathcal{H}},n)>0$ such that
\begin{equation}
\label{eq:mainbackweestho}
\begin{split}
\sum_{j=0}^1&\sum_{m+k+|\alpha|\leq n}\int_{{N}_{\tau}} r^{2+2k} (L^k T^{m+j}\Omega^{\alpha} \phi)^2\,d\omega dv+ \int_{{\underline{N}}_{\tau}} (r-M)^{-2-2k}(\underline{L}^kT^{m+j}\Omega^{\alpha}\phi)^2\,d\omega du\\
\leq &\:C\sum_{j=0}^1\sum_{|\alpha|+m\leq n}\int_{\mathcal{I}^+_{\geq u_{\mathcal{I}^+}(\tau)} } u^{2-j+2m}(\partial_u T^{m+j} \Omega^{\alpha} \phi)^2+|\snabla_{\s^2}\Omega^{\alpha}T^m\phi|^2\,d\omega du\\
&+ C\sum_{j=0}^1\sum_{|\alpha|+m\leq n}\int_{\mathcal{H}^+_{\geq v_{\mathcal{H}^+}(\tau)} } v^{2-j+2m}(\partial_v T^{m+j} \Omega^{\alpha} \phi)^2+|\snabla_{\s^2}\Omega^{\alpha}T^m\phi|^2\,d\omega dv.
 \end{split}
\end{equation}
\end{proposition}
\begin{proof}
We first consider the $n=1$ case. Note that by \eqref{eq:backhierarchy} with $k=1$ and $p=3$:
\begin{equation}
\label{eq:backwn1p3}
\begin{split}
\int_{{N}_{\tau}} r^3 (L^2 \phi)^2\,d\omega dv+ \int_{{\underline{N}}_{\tau}}& (r-M)^{-3}(\underline{L}^2\phi)^2\,d\omega du\leq C\sum_{|\alpha|\leq 1}\int_{\tau}^{\infty} \int_{{N}_{\tau'}} r^2(L^2\phi)^2+ (L \Omega^{\alpha} \phi)^2+ r^{-2}|\snabla_{\s^2} \Omega^{\alpha}\phi|^2\, d\omega dvdu\\
&+C\sum_{|\alpha|\leq 1}\int_{\tau}^{\infty} \int_{\underline{N}_{\tau'}} (r-M)^{-2}(\underline{L}^2\phi)^2+ (\underline{L} \Omega^{\alpha} \phi)^2+ (r-M)^2|\snabla_{\s^2} \Omega^{\alpha}\phi|^2\, d\omega dudv\\
&+ C \sum_{|\alpha|+m\leq 1}\int_{\mathcal{I}^+_{\geq u_{\mathcal{I}^+}(\tau)} } (\partial_u T^m \Omega^{\alpha} \phi)^2\,du+ C\sum_{|\alpha|+m\leq 1}\int_{\mathcal{H}^+_{\geq v_{\mathcal{H}^+}(\tau)} } (\partial_v T^m \Omega^{\alpha} \phi)^2\, d\omega dv\\
\stackrel{\textnormal{\eqref{eq:TconvLLbar1} and \eqref{eq:TconvLLbar2}}}{\leq}&\: C\sum_{|\alpha|\leq 1}\int_{\tau}^{\infty} \int_{{N}_{\tau'}} r^2(LT\phi)^2+ (L \Omega^{\alpha} \phi)^2+ r^{-2}|\snabla_{\s^2} \Omega^{\alpha}\phi|^2+r^{-4}\phi^2\, d\omega dvdu\\
+C\sum_{|\alpha|\leq 1}\int_{\tau}^{\infty}& \int_{\underline{N}_{\tau'}} (r-M)^{-2}(\underline{L}T\phi)^2+ (\underline{L} \Omega^{\alpha} \phi)^2+ (r-M)^2|\snabla_{\s^2} \Omega^{\alpha}\phi|^2+(r-M)^4 \phi^2\, d\omega dudv\\
&+ C\sum_{|\alpha|+m\leq 1}\int_{\mathcal{I}^+_{\geq u_{\mathcal{I}^+}(\tau)} } (\partial_u T^m \Omega^{\alpha} \phi)^2\,du+ \sum_{|\alpha|+m\leq 1}\int_{\mathcal{H}^+_{\geq v_{\mathcal{H}^+}(\tau)} } (\partial_v T^m \Omega^{\alpha} \phi)^2\, d\omega dv\\
\stackrel{\textnormal{\eqref{eq:backwjzero}, Lemma \ref{lm:integralsandweights} and  Lemma \ref{lm:tenconsv}}}{\leq}&\: C\sum_{|\alpha|+m\leq 1}\int_{\mathcal{I}^+_{\geq u_{\mathcal{I}^+}(\tau)} } u^{1+2m}(\partial_u T^m \Omega^{\alpha} \phi)^2\,du+ C\sum_{|\alpha|+m\leq 1}\int_{\mathcal{H}^+_{\geq v_{\mathcal{H}^+}(\tau)} } v^{1+2m}(\partial_v T^m \Omega^{\alpha} \phi)^2\, d\omega dv.
\end{split}
\end{equation}

Now, we apply  \eqref{eq:backhierarchy} with $k=1$ and $p=4$:
\begin{equation*}
\begin{split}
\int_{{N}_{\tau}}& r^4 (L^2 \phi)^2\,d\omega dv+ \int_{{\underline{N}}_{\tau}} (r-M)^{-4}(\underline{L}^2\phi)^2\,d\omega du\leq C\sum_{|\alpha|\leq 1}\int_{\tau}^{\infty} \int_{{N}_{\tau'}} r^3(L^2\phi)^2+ r (L \Omega^{\alpha} \phi)^2\, d\omega dvdu\\
&+C\sum_{|\alpha|\leq 1}\int_{\tau}^{\infty} \int_{\underline{N}_{\tau'}} (r-M)^{-3}(\underline{L}^2\phi)^2+ (r-M)^{-1} (\underline{L} \Omega^{\alpha} \phi)^2\, d\omega dudv\\
&+ C\sum_{|\alpha|+m\leq 1}\int_{\mathcal{I}^+_{\geq u_{\mathcal{I}^+}(\tau)} } (\partial_u T^m \Omega^{\alpha} \phi)^2\,du+ C\sum_{|\alpha|+m\leq 1}\int_{\mathcal{H}^+_{\geq v_{\mathcal{H}^+}(\tau)} } (\partial_v T^m \Omega^{\alpha} \phi)^2\, d\omega dv\\
\stackrel{\eqref{eq:backwn1p3}}{\leq}&\:C\sum_{|\alpha|+m\leq 1}\int_{\mathcal{I}^+_{\geq u_{\mathcal{I}^+}(\tau)} } u^{2+2m}(\partial_u T^m \Omega^{\alpha} \phi)^2\,du+ C\sum_{|\alpha|+m\leq 1}\int_{\mathcal{H}^+_{\geq v_{\mathcal{H}^+}(\tau)} } v^{2+2m}(\partial_v T^m \Omega^{\alpha} \phi)^2\, d\omega dv.
\end{split}
\end{equation*}
By replacing $\phi$ on the left-hand side of \eqref{eq:backwn1p3} with $T^j\phi$ and applying Proposition \ref{prop:mainbackweest} to $T^m \Omega^{\alpha}\phi$, we therefore obtain:
\begin{equation*}
\begin{split}
\sum_{j=0}^1&\sum_{m+k+|\alpha|\leq 1}\int_{{N}_{\tau}} r^{2+2k-j} (L^{k+1} T^{m+j}\Omega^{\alpha} \phi)^2\,d\omega dv+ \int_{{\underline{N}}_{\tau}} (r-M)^{-2-2k+j}(\underline{L}^{k+1}T^{m+j}\Omega^{\alpha}\phi)^2\,d\omega du\\
\leq &\:C\sum_{j=0}^1\sum_{|\alpha|+m\leq 1}\int_{\mathcal{I}^+_{\geq u_{\mathcal{I}^+}(\tau)} } u^{2-j+2m}(\partial_u T^{m+j} \Omega^{\alpha} \phi)^2+|\snabla_{\s^2}\Omega^{\alpha}\phi|^2\, d\omega du\\
&+ C\sum_{j=0}^1\sum_{|\alpha|+m\leq 1}\int_{\mathcal{H}^+_{\geq v_{\mathcal{H}^+}(\tau)} } v^{2-j+2m}(\partial_v T^{m+j} \Omega^{\alpha} \phi)^2+|\snabla_{\s^2}\Omega^{\alpha}\phi|^2\, d\omega dv,
\end{split}
\end{equation*}
where we applied Proposition \ref{prop:mainbackweest} and Lemma \ref{lm:integralsandweights} to arrive at the final inequality.

The general $n$ case now follows easily via an inductive argument, where we apply \eqref{eq:backhierarchy} with $k=n$ and $p=2n+1$ and $p=2n+2$.
\end{proof}

Proposition \ref{prop:mainbackweestho} combined with Lemma \ref{lm:tenconsv} immediately implies the following:
\begin{corollary}
\label{cor:mainbackweesthov2}
Let $n\in \N_0$. Then there exists a constant $C(M,r_{\mathcal{I}},r_{\mathcal{H}},n)>0$ such that
\begin{equation}
\label{eq:mainbackweesthov2}
\begin{split}
\sum_{j=0}^1&\sum_{m+2k+2|\alpha|\leq 2n}\int_{{N}_{u_0}} r^{2+2k-j} (L^{k+1} T^{m+j}\Omega^{\alpha} \phi)^2\,d\omega dv+ \int_{{\underline{N}}_{v_0}} (r-M)^{-2-2k+j}(\underline{L}^{k+1}T^{m+j}\Omega^{\alpha}\phi)^2\,d\omega du\\
&+\sum_{\substack{m+2|\alpha|\leq 2n+2\\|\alpha|\leq n}} \int_{\Sigma_{0}} {\mathbf{J}}^T[T^m \Omega^{\alpha}\psi]\cdot \mathbf{n}_{0}\, d\mu_{0}\\
\leq &\:C\Bigg[\sum_{j=0}^2\sum_{2|\alpha|+2k+m\leq 2n}\int_{\mathcal{I}^+_{\geq u_0 }} u^{2+2k-j}(\Lbar T^{m+j} \Omega^{\alpha} \phi)^2+|\snabla_{\s^2}\Omega^{\alpha}T^m\phi|^2\,du\\
&+ \int_{\mathcal{H}^+_{\geq v_0}} v^{2+2k-j}(L T^{m+j} \Omega^{\alpha} \phi)^2+|\snabla_{\s^2}\Omega^{\alpha}T^m\phi|^2\,dv\Bigg].
 \end{split}
\end{equation}
\end{corollary}

We will complement \eqref{eq:mainbackweesthov2} in Corollary \ref{cor:mainbackweesthov2} with an estimate involving additional angular derivatives. The motivation for this comes from the energy estimates in Section \ref{sec:hoestasymflat}.

\begin{corollary}
\label{cor:mainbackweesthov3}
Let $n\in \N_0$. Then there exists a constant $C(M,r_{\mathcal{I}},r_{\mathcal{H}},n)>0$ such that
\begin{equation}
\begin{split}
\label{eq:mainbackweesthov3}
\sum_{j=0}^2&\sum_{2|\alpha|+2k+m\leq 2n} \Bigg[\int_{N_{u_0}} r^{2k+2-j}(L^{k+1}\Omega^{\alpha}T^{j+m}\phi)^2+r^{2k}|\snabla_{\s^2}L^{k+1}\Omega^{\alpha}T^{m}\phi|^2\,d\omega dv \\
&+ \int_{\underline{N}_{v_0}} (r-M)^{-2k-2+j}(\Lbar^{k+1}\Omega^{\alpha}T^{j+m}\phi)^2+(r-M)^{-2k}|\snabla_{\s^2}\Lbar^{k+1}\Omega^{\alpha}T^{m}\phi|^2\,d\omega du\Bigg]\\
&+\sum_{2|\alpha|+m\leq 2n+2}\int_{\Sigma_0} \mathbf{J}^T[\Omega^{\alpha}T^m\psi]\cdot \mathbf{n}_0\,d\mu_0\\
\leq &\:  C \sum_{j=0}^2\sum_{m+2k+2|\alpha|\leq 2n}\int_{ \mathcal{H}^+_{\geq v_0}} v^{2k+2-j} (L^{1+k+m+j} \Omega^{\alpha}\phi)^2+v^{2k} |\snabla_{\s^2}L^{k+m} \Omega^{\alpha}\phi|^2+v^{2k} |\snabla_{\s^2}L^{k+1+m} \Omega^{\alpha}\phi|^2\,d\omega dv\\
&+C\int_{ \mathcal{I}^+_{\geq u_0}} u^{2k+2-j}(\Lbar^{1+k+m+j}\Omega^{\alpha}\phi)^2+ u^{2k}|\snabla_{\s^2}\Lbar^{k+m} \Omega^{\alpha}\phi|^2+ u^{2k}|\snabla_{\s^2}\Lbar^{k+1+m} \Omega^{\alpha}\phi|^2\,d\omega du.
\end{split}
\end{equation}
\end{corollary}

\subsection{Construction of the backwards evolution map}
\label{sec:bacwevomap}
In this section, we apply the uniform estimates derived in Sections \ref{sec:bacweest} and \ref{sec:backwhoest} to construct the backwards evolution maps $\mathscr{B}$ on appropriate energy spaces.

\begin{proposition}
\label{prop:BTenergy}
Let $({\underline{\Phi}},{\Phi}) \in C_{c}^{\infty}(\mathcal{H}^+_{\geq v_0})\oplus C_{c}^{\infty}(\mathcal{I}^+_{\geq u_0})$, then the corresponding solution $\psi$ to \eqref{eq:waveequation} satisfies
\begin{equation*}
(\psi|_{\Sigma_0},\mathbf{n}_{\Sigma_0}\psi|_{\Sigma_0\cap\{r_{\mathcal{H}}\leq r\leq r_{\mathcal{I}}\}})\in  \mathcal{E}^T_{\Sigma_0}
\end{equation*}
and
\begin{equation*}
||(\psi|_{\Sigma_0},\mathbf{n}_{\Sigma_0}\psi|_{\Sigma_0\cap\{r_{\mathcal{H}}\leq r\leq r_{\mathcal{I}}\}})||^2_{\mathcal{E}^T_{\Sigma_0}}=||\underline{\Phi}||_{\mathcal{E}^T_{\mathcal{H}^+_{\geq v_0}}}^2+||\Phi||_{\mathcal{E}^T_{\mathcal{I}^+_{\geq u_0}}}^2.
\end{equation*}
\end{proposition}
\begin{proof}
From Proposition \ref{prop:mainpropbackdef} it follows that $\psi|_{\Sigma_0}\in C^{\infty}(\overline{\Sigma})$ and $\mathbf{n}_{\Sigma_0}\psi|_{\Sigma_0\cap\{r_{\mathcal{H}}\leq r\leq r_{\mathcal{I}}\}}\in C^{\infty}(\Sigma_0\cap\{r_{\mathcal{H}}\leq r\leq r_{\mathcal{I}}\})$. The remaining statment follows from Lemma \ref{lm:tenconsv}.
\end{proof}

Using Proposition \ref{prop:BTenergy}, together with the standard general construction of the unique extensions of bounded linear operators to the completion of their domains, we can define the \emph{backwards evolution map} as follows:

\begin{definition}
\label{def:backwmap}
The backwards evolution map is the map $\mathscr{B}: C_{c}^{\infty}(\mathcal{H}^+_{\geq v_0})\oplus C_{c}^{\infty}(\mathcal{I}^+_{\geq u_0}) \to \mathcal{E}_T(\Sigma_0)$, such that
\begin{equation*}
\mathscr{B}({\underline{\Phi}},{\Phi})=(\psi|_{\Sigma_0},\mathbf{n}_{\Sigma_0}\psi|_{\Sigma_0\cap\{r_0\leq r\leq r_1\}}),
\end{equation*}
where $\psi$ is the unique solution to $\square_g\psi=0$ with $(M\psi|_{\mathcal{H}^+_{\geq v_0}},r\psi|_{\mathcal{I}^+_{\geq u_0}})=({\underline{\Phi}},{\Phi})$. The map $\mathscr{B}$ uniquely extends to a unitary linear operator, which we will also denote with $\mathscr{B}$:
\begin{equation*}
\mathscr{B}: \mathcal{E}^T_{\mathcal{H}^+}\oplus \mathcal{E}^T_{\mathcal{I}^+} \to\mathcal{E}^T_{\Sigma_0}.
\end{equation*}
\end{definition}

In the proposition below, we show that we can consider restriction of $\mathscr{B}$ to suitably weighted energy spaces.
\begin{proposition}
\label{prop:backwmapho}
Let $n\in \N_0$. The backwards evolution map $\mathscr{B}$ is a bounded linear operator from $C_{c}^{\infty}(\mathcal{H}^+_{\geq v_0})\oplus C_{c}^{\infty}(\mathcal{I}^+_{\geq u_0})$ to $\mathcal{E}_{n; \Sigma_0}$, which can uniquely be extended as as the bounded linear operator 
\begin{equation*}
\mathscr{B}_n: \mathcal{E}_{n; \mathcal{H}^+_{\geq v_0}}\oplus \mathcal{E}_{n; \mathcal{I}^+_{\geq u_0}}\to \mathcal{E}_{n; \Sigma_0}.
\end{equation*}
We moreover have that $\mathscr{B}_n=\mathscr{B}|_{ \mathcal{E}_{n; \mathcal{H}^+_{\geq v_0}}\oplus \mathcal{E}_{n; \mathcal{I}^+_{\geq u_0}}}$.
\end{proposition}
\begin{proof}
By Proposition \ref{prop:mainpropbackdef} it follows that the solution $\psi$ corresponding to $(\underline{\Phi},\Phi)\in C_{c}^{\infty}(\mathcal{H}^+_{\geq v_0})\oplus C_{c}^{\infty}(\mathcal{I}^+_{\geq u_0})$ satisfies $\phi|_{\Sigma_0}\in C^{\infty}(\widehat{\Sigma}_0)$ and $n_{\Sigma_0}\psi|_{\Sigma_0}\in C^{\infty}(\Sigma_0\cap\{r_{\mathcal{H}}\leq r\leq r_{\mathcal{I}}\})$. By Corollary \ref{cor:mainbackweesthov2} it follows moreover that
\begin{equation*}
||(\psi|_{\Sigma_0},\mathbf{n}_{\Sigma_0}\psi|_{\Sigma_0})||_{\mathcal{E}_{n; \Sigma_0}}^2\leq C ||\underline{\Phi}||_{\mathcal{E}_{n; \mathcal{H}^+_{\geq v_0}}}^2+C||\underline{\Phi}||_{\mathcal{E}_{n; \mathcal{I}^+_{\geq v_0}}}^2,
\end{equation*}
so $||\mathscr{B}||\leq C$. We can infer that, in particular, $(\psi|_{\Sigma_0},\mathbf{n}_{\Sigma_0}\psi|_{\Sigma_0})\in \mathcal{E}_{n; \Sigma_0}$. The map $\mathscr{B}$ extends uniquely to the completion $\mathcal{E}_{\mathcal{H}^+_{\geq v_0}}\oplus \mathcal{E}_{\mathcal{I}^+_{\geq u_0}}$ and satisfies $||\mathscr{B}||\leq \sqrt{C}$.
\end{proof}

\begin{corollary}
\label{cor:bijectivity}
The map  $\mathscr{F}: \mathcal{E}^T_{\Sigma_0} \to \mathcal{E}^T_{\mathcal{H}^+_{\geq v_0}}\oplus \mathcal{E}^T_{\mathcal{I}^+_{\geq u_0}}$ is a bijection with inverse $\mathscr{B}=\mathscr{F}^{-1}$ and for each $n\in \N_0$, the restrictions  $\mathscr{F}_n: \mathcal{E}_{n;\Sigma_0} \to \mathcal{E}_{n;\mathcal{H}^+_{\geq v_0}}\oplus \mathcal{E}_{n;\mathcal{I}^+_{\geq u_0}}$ are also bijections with inverses $\mathscr{B}_n=\mathscr{F}_n^{-1}$.
\end{corollary}
\begin{proof}

Let $(\underline{\Phi},\Phi)\in C_{c}^{\infty}(\mathcal{H}^+_{\geq v_0})\oplus C_{c}^{\infty}(\mathcal{I}^+_{\geq u_0})$, then the corresponding solution $\psi$ to \eqref{eq:waveequation} satisfies $\phi|_{\Sigma_0}\in C^{\infty}(\widehat{\Sigma}_0)$ and $n_{\Sigma_0}\psi|_{\Sigma_0}\in C^{\infty}(\Sigma_0\cap\{r_{\mathcal{H}}\leq r\leq r_{\mathcal{I}}\})$, and hence $\mathscr{F}(\phi|_{\Sigma_0}, n_{\Sigma_0}\psi|_{\Sigma_0})=(\phi|_{\mathcal{H}^+,}\phi|_{\mathcal{I}^+})$ is well-defined and $(\phi|_{\mathcal{H}^+,}\phi|_{\mathcal{I}^+})=(\underline{\Phi},\Phi)$. We conclude that $\mathscr{F}\circ \mathscr{B}=\textnormal{id}$ on a dense subset. By boundedness of $\mathscr{F}\circ \mathscr{B}$ we can conclude that  $\mathscr{F}\circ \mathscr{B}=\textnormal{id}$ on the full domain. Hence, $\mathscr{F}$ must be surjective and in fact bijective (we have already established injectivity). It immediately follows then that $\mathscr{B}\circ \mathscr{F}=\textnormal{id}$. The above argument can also be applied to $\mathscr{F}_n$ and $\mathscr{B}_n$.
\end{proof}

\section{The scattering map}
\label{sec:spacelikeinfbfsphere}
The aim of this section is to extend the estimates of Section \ref{sec:fowmap} and \ref{sec:backwest} from the hypersurface $\Sigma_0$ to the hypersurface $\widetilde{\Sigma}$. This will allow us to construct the scattering map $\mathcal{S}$, a bijective map between (time-weighted) energy spaces on $\mathcal{H}^-\cup \mathcal{I}^-$ and $\mathcal{H}^+\cup \mathcal{I}^+$. The estimates in this section will therefore concern the ``triangular'' regions bounded to the future by the null hypersurfaces ${N}_0$ and $\underline{N}_0$ and to the past by $\widetilde{\Sigma}=\{t=0\}$.

\subsection{Weighted energy estimates near spacelike infinity}
\label{sec:enestasympflat}

In the proposition below we derive energy estimates with respect to the \emph{vector field multiplier}  $K=v^2L+u^2\underline{L}$, which is commonly referred to as the \emph{Morawetz conformal vector field}.\footnote{The geometric significance of $K$ is that it generates the inverted translation conformal symmetry on the Minkowski spacetime.} The main purpose of $K$ is to derive backwards energy estimates along $\widetilde{\Sigma}$ with $r$-weighted initial data along ${N}_{-u_0}$ and $\underline{N}_{-v_0}$ which are analogous to the $r$-weighted boundary terms in the estimates in Proposition \ref{prop:backwrp} with $p=2$.

\begin{proposition}
\label{prop:estasympflat}
Let $u_{-\infty}, v_{-\infty}<0$, with $|u_{-\infty}|, |v_{-\infty}|$ arbitrarily large. There exist constants $C,c=C,c(M, r_{\mathcal{I}},r_{\mathcal{H}},u_0,v_0)>0$, such that
\begin{align}
\label{eq:asymflatestback1}
\int_{\widetilde{\Sigma}\cap\{u_0\leq v\leq -u_{-\infty}\}}& r^2(L\phi)^2+r^2(\Lbar \phi)^2+|\snabla_{\s^2}\phi|^2\,dv\\ \nonumber
\sim_{C,c}&\: \int_{N_{-u_0}} r^2(L\phi)^2+r^{-2}|\snabla_{\s^2}\phi|^2\,d\omega dv+\int_{\mathcal{I}^+_{\leq -u_0}} u^2(\Lbar\phi)^2+|\snabla_{\s^2}\phi|^2\,d\omega du,\\
\label{eq:asymflatestback2}
\int_{\widetilde{\Sigma}\cap\{v_0\leq u\leq -v_{-\infty}\}}&(r-M)^{-2}(L\phi)^2+(r-M)^{-2}(\Lbar \phi)^2+|\snabla_{\s^2}\phi|^2\,du\\ \nonumber
\sim_{C,c}&\:  \int_{\underline{N}_{-v_0}} D^{-1}(\Lbar \phi)^2+D|\snabla_{\s^2} \phi|^2\,d\omega du+\int_{\mathcal{H}^+_{\leq -v_0}} v^2(L\phi)^2+ |\snabla_{\s^2}\phi|^2\,d\omega dv.
\end{align}
\end{proposition}
\begin{proof}
By \eqref{eq:maineqradfield} it follows that
\begin{equation*}
\begin{split}
L(u^2(\Lbar\phi)^2)+\Lbar(v^2(L\phi)^2)=&\:2u^2\Lbar L\phi\cdot\Lbar\phi+2v^2\Lbar L \phi\cdot L\phi\\
=&\: \frac{1}{2}\frac{u^2}{r^2}D \slashed{\Delta}_{\s^2}\phi\cdot \Lbar\phi+\frac{1}{2}\frac{v^2}{r^2}D \slashed{\Delta}_{\s^2}\phi\cdot L\phi+(u^2\Lbar \phi+v^2L\phi)O(r^{-3})\phi.
\end{split}
\end{equation*}
After integrating by parts on $\s^2$, we therefore obtain:
\begin{equation}
\label{eq:keyidasymflat1}
\begin{split}
\int_{\s^2}&L(u^2(\Lbar \phi)^2)+\Lbar(v^2(L\phi)^2)+\frac{1}{4}L(v^2r^{-2}D|\slashed{\nabla}_{\s^2}\phi|^2)+\frac{1}{4}\Lbar(u^2r^{-2}D|\slashed{\nabla}_{\s^2}\phi|^2)\,d\omega\\
=&\:\int_{\s^2}\frac{1}{4}\left(\Lbar\left(\frac{u^2D}{r^2}\right)+L\left(\frac{v^2D}{r^2}\right)\right)|\snabla_{\s^2}\phi|^2\,d\omega+\int_{\s^2}(u^2\Lbar\phi+v^2L\phi)O(r^{-3})\phi \,d\omega.
\end{split}
\end{equation}
We first consider estimates in the backwards time direction. We integrate \eqref{eq:keyidasymflat1} in spacetime and we use the following identity:
\begin{equation}
\label{eq:keycancelinf}
\Lbar\left(\frac{u^2D}{r^2}\right)+L\left(\frac{v^2D}{r^2}\right)=O(r^{-2})\log r
\end{equation}
to estimate
\begin{equation}
\label{eq:auxestasympflat}
\begin{split}
\sup_{u} &\int_{N_u} v^2(L\phi)^2+\frac{1}{4}u^2r^{-2}D|\snabla_{\s^2}\phi|^2\,d\omega dv+ \sup_{v} \int_{I_v} u^2(\Lbar\phi)^2+\frac{1}{4}v^2r^{-2}D|\snabla_{\s^2}\phi|^2\,d\omega du\\
\leq&\: \int_{N_{-u_0}} v^2(L\phi)^2+\frac{1}{4}u^2r^{-2}D|\snabla_{\s^2}\phi|^2\,d\omega dv+\int_{\mathcal{I}^+_{u_{-\infty}\leq u  \leq -u_0}} u^2(\Lbar\phi)^2+\frac{1}{4}v^2r^{-2}D|\snabla_{\s^2}\phi|^2\,d\omega du\\
&+C\int_{u_{0}}^{\infty} \int^{-u_0}_{-\min\{v,|u_{-\infty}|\}}r^{-3}|\phi|\cdot (u^2|\Lbar \phi|+v^2|L\phi|)\,d\omega du dv\\
&+C\int_{u_{0}}^{\infty} \int^{-u_0}_{-\min\{v,|u_{-\infty}|\}}r^{-3} \log r|\snabla_{\s^2}\phi|^2 \,d\omega du dv.
\end{split}
\end{equation}
Using that $r\sim v+|u|\lesssim v$ in the integration region, we can further estimate:
\begin{equation*}
\begin{split}
\int_{u_0}^{\infty} \int^{-u_0}_{-\min\{v,|u_{-\infty}|\}} r^{-2} \log r|\snabla_{\s^2}\phi|^2 \,d\omega du dv\leq&\:C \int_{u_0}^{\infty}  v^{-2} \log v\,dv\cdot  \sup_{v} \int_{I_v}v^2r^{-2}D|\snabla_{\s^2}\phi|^2\,d\omega dv\\
\leq&\:C \epsilon \sup_{v} \int_{I_v}v^2r^{-2}D|\snabla_{\s^2}\phi|^2\,d\omega dv
\end{split}
\end{equation*}
for $\epsilon>0$ arbitrarily small given $r_{\mathcal{I}}>0$ suitably large (and $v^{-1}\lesssim r^{-1}$ in the integration region). Note that we can absorb the very right-hand side above into the left-hand side of \eqref{eq:auxestasympflat} when $\epsilon>0$ is suitably small.

We apply Young's inequality to estimate
\begin{equation*}
\begin{split}
r^{-3}|\phi| (u^2|\Lbar\phi|+v^2|L\phi|)\leq&\: r^{-1-\eta}  (u^2(\Lbar \phi)^2+v^2(L\phi)^2)+r^{-5+\eta}(u^2+v^2)\phi^2.
\end{split}
\end{equation*}
We absorb the spacetime integral of $(\Lbar \phi)^2$ and $(L \phi)^2$ to the left-hand side of \eqref{eq:auxestasympflat}, using that $r$ is suitably large and $(v+|u|)\lesssim r$ in the integration region. In order to absorb the $\phi^2$ term, we first observe that by assumption, we are considering $\phi$ such that $\phi|_{\mathcal{I}^+}$ is well-defined and is compactly supported in $u>u_{-\infty}$, so
\begin{equation*}
\lim_{v\to \infty} \phi(u_{-\infty},v)=0.
\end{equation*}
Therefore, by Cauchy--Schwarz, we can estimate
\begin{equation*}
\begin{split}
\left[\int_{\s^2}(\phi-\phi|_{\mathcal{I}^+})^2\,d\omega\right](u,v)\leq&\: \int_{\s^2}\left(\int_v^{\infty}(L\phi)^2\,dv'\right)^2d\omega\leq v^{-1} \int_{v}^{\infty} \int_{\s^2}v'^2(L\phi)^2\,d\omega dv'\\
\leq&\: v^{-1} \sup_u \int_{N_u}v'^2(L\phi)^2\,d\omega dv'.
\end{split}
\end{equation*}
Furthermore, similarly we have that 
\begin{equation*}
\begin{split}
\left[\int_{\s^2}\phi|_{\mathcal{I}^+}^2\,d\omega\right](u)\leq&\: u^{-1} \int^u_{u_{-\infty}}\int_{\s^2} u'^2(\Lbar \phi)^2\,d\omega du'\\
\leq&\: u^{-1} \sup_v \int_{I_v} u'^2(\Lbar \phi)^2\,d\omega du'.
\end{split}
\end{equation*}
Hence,
\begin{equation*}
\left[\int_{\s^2}\phi^2\,d\omega\right]\leq (u^{-1}+v^{-1}) \left[\sup_u \int_{N_u}v'^2(L\phi)^2\,d\omega dv'+\sup_v \int_{I_v} u'^2(\Lbar \phi)^2\,d\omega du'\right],
\end{equation*}
so we can estimate:
\begin{equation*}
\int_{u_0}^{\infty} \int^{-u_0}_{-\min\{v,|u_{-\infty}|\}} \int_{\s^2}(u^2+v^2)r^{-5+\eta}\phi^2\,d\omega du dv \lesssim \epsilon \left[\sup_u \int_{N_u}v'^2(L\phi)^2\,d\omega dv'+\sup_v \int_{I_v} u'^2(\Lbar\phi)^2\,d\omega du'\right],
\end{equation*}
with $\epsilon>0$ suitably small given ${r}_{\mathcal{I}}$ suitably large. As a result, we obtain
\begin{equation}
\label{eq:mainestKcomm}
\begin{split}
\sup_{u} &\int_{N_u} v^2(L\phi)^2+\frac{1}{4}u^2r^{-2}D|\snabla_{\s^2}\phi|^2\,d\omega dv+ \sup_{v} \int_{I_v} u^2(\Lbar \phi)^2+\frac{1}{4}v^2r^{-2}D|\snabla_{\s^2}\phi|^2\,d\omega du\\
\leq&\: C\int_{N_{-u_0}} v^2(L\phi)^2+\frac{1}{4}u^2r^{-2}D|\snabla_{\s^2}\phi|^2\,d\omega dv+C\int_{\mathcal{I}^+_{\leq -u_0}} u^2(\Lbar \phi)^2+\frac{1}{4}v^2r^{-2}D|\snabla_{\s^2}\phi|^2\,d\omega du.
\end{split}
\end{equation}
We integrate \eqref{eq:keyidasymflat1} and apply \eqref{eq:mainestKcomm} to obtain:
\begin{equation*}
\begin{split}
\int_{\widetilde{\Sigma}\cap\{v_{r_{\mathcal{I}}}\leq v\leq -u_{-\infty}\}}& r^2(L \phi)^2+r^2(\Lbar \phi)^2+|\snabla_{\s^2}\phi|^2\,dv\\
\leq&\: C\int_{N_{-u_0}} v^2(L \phi)^2+u^2r^{-2}D|\snabla_{\s^2}\phi|^2\,d\omega dv+\int_{\mathcal{I}^+_{\leq -u_0}} u^2(\Lbar \phi)^2+v^2r^{-2}D|\snabla_{\s^2}\phi|^2\,d\omega du\\
&+C\int_{v_{r_{\mathcal{I}}}}^{\infty} \int^{-u_0}_{-\min\{v,|u_{-\infty}|\}} r^{-3}|\phi|\cdot (u^2|\Lbar \phi|+v^2|L \phi|)\,d\omega du dv\\
&+C\int_{v_{r_{\mathcal{I}}}}^{\infty} \int^{-u_0}_{-\min\{v,|u_{-\infty}|\}} r^{-2} \log r|\snabla_{\s^2}\phi|^2 \,d\omega du dv\\
\leq &\: C\int_{N_{-u_0}} v^2(L \phi)^2+u^2r^{-2}D|\snabla_{\s^2}\phi|^2\,d\omega dv+\int_{\mathcal{I}^+_{\leq -u_0}} u^2(\Lbar \phi)^2+v^2r^{-2}D|\snabla_{\s^2}\phi|^2\,d\omega du.
\end{split}
\end{equation*}

Analogously, we have that
\begin{equation}
\label{eq:keyidasymflat2}
\begin{split}
\int_{\s^2}&L(u^2(\Lbar \phi)^2)+\Lbar(v^2(L \phi)^2)+\frac{1}{4}L(v^2r^{-2}D|\slashed{\nabla}_{\s^2}\phi|^2)+\frac{1}{4}\Lbar (u^2r^{-2}D|\slashed{\nabla}_{\s^2}\phi|^2)\,d\omega\\
=&\:\int_{\s^2}\frac{1}{4}\left(\Lbar \left(\frac{u^2D}{r^2}\right)+L\left(\frac{v^2D}{r^2}\right)\right)|\snabla_{\s^2}\phi|^2\,d\omega+\int_{\s^2}(u^2\Lbar \phi+v^2L \phi)O((r-M)^{3})\phi \,d\omega.
\end{split}
\end{equation}
and
\begin{equation}
\label{eq:keycancelhor}
\Lbar \left(\frac{u^2D}{r^2}\right)+L\left(\frac{v^2D}{r^2}\right)=O((r-M)^2)|\log (r-M)|,
\end{equation}
so that we can estimate
\begin{equation}
\label{eq:auxestasympflathor}
\begin{split}
\sup_{v} &\int_{\underline{N}_v} u^2(\Lbar \phi)^2+\frac{1}{4}v^2r^{-2}D|\snabla_{\s^2}\phi|^2\,d\omega du+ \sup_{u} \int_{H_u} v^2(L \phi)^2+\frac{1}{4}u^2r^{-2}D|\snabla_{\s^2}\phi|^2\,d\omega dv\\
\leq&\: \int_{\underline{N}_{v_0}}u^2(\Lbar \phi)^2+\frac{1}{4}v^2r^{-2}D|\snabla_{\s^2}\phi|^2\,d\omega du+\int_{\mathcal{H}^+_{\leq -v_{0}}}  v^2(L \phi)^2+\frac{1}{4}u^2r^{-2}D|\snabla_{\s^2}\phi|^2\,d\omega dv\\
&+C\int_{u_{r_{\mathcal{H}}}}^{\infty} \int^{-v_0}_{-\min\{u,|v_{-\infty}|\}}(r-M)^{3}|\phi|\cdot (u^2|\Lbar \phi|+v^2|L \phi|)\,d\omega du dv\\
&+C\int_{u_{r_{\mathcal{H}}}}^{\infty} \int^{-v_0}_{-\min\{u,|v_{-\infty}|\}}\log((r-M)^{-1}) D|\snabla_{\s^2}\phi|^2 \,d\omega du dv.
\end{split}
\end{equation}

Using that $(r-M)^{-1}\sim u+|v|\lesssim u$, we estimate further:
\begin{equation*}
\begin{split}
\int_{u_{r_{\mathcal{H}}}}^{\infty} \int^{-v_0}_{-\min\{u,|v_{-\infty}|\}}\log((r-M)^{-1}) D|\snabla_{\s^2}\phi|^2 \,d\omega du dv\leq &\: C\int_{u_{r_{\mathcal{H}}}}^{\infty} u^{-2}\log u\,du\cdot \sup_u \int_{H_u} u^2r^{-2}D|\snabla_{\s^2}\phi|^2\,d\omega dv\\
\leq &\: \epsilon \sup_u \int_{H_u} u^2r^{-2}D|\snabla_{\s^2}\phi|^2\,d\omega dv
\end{split}
\end{equation*}
for $\epsilon>0$ arbitrarily small given $r_{\mathcal{H}}-M>0$ suitably small. Note that we can absorb the very right-hand side above into the left-hand side of \eqref{eq:auxestasympflathor} when $\epsilon>0$ is suitably small.

We apply Young's inequality to estimate
\begin{equation*}
\begin{split}
(r-M)^{3}|\phi|(u^2|\Lbar \phi|+v^2|L \phi|)\leq (r-M)^{1+\eta}  (u^2(\Lbar \phi)^2+v^2(L \phi)^2)+(r-M)^{5-\eta} (u^2+v^2) \phi^2
\end{split}
\end{equation*}
and absorb the corresponding spacetime integral to the left-hand side of \eqref{eq:auxestasympflathor}, using that
\begin{equation*}
\left[\int_{\s^2}\phi^2\,d\omega\right]\leq (u^{-1}+v^{-1}) \left[\sup_u \int_{\underline{N}_u}u'^2(\Lbar \phi)^2\,d\omega du'+\sup_v \int_{\underline{H}_u} v'^2(L \phi)^2\,d\omega dv'\right],
\end{equation*}
which follows from Cauchy--Schwarz combined with the assumption that $\phi|_{\mathcal{H}^+}(v)=0$ for $v\leq v_{-\infty}$.
We are left with
\begin{equation}
\label{eq:mainestKcomm2}
\begin{split}
\sup_{v} &\int_{\underline{N}_v} u^2(\Lbar \phi)^2+\frac{1}{4}v^2r^{-2}D|\snabla_{\s^2}\phi|^2\,d\omega du+ \sup_{u} \int_{H_u} v^2(L \phi)^2+\frac{1}{4}u^2r^{-2}D|\snabla_{\s^2}\phi|^2\,d\omega dv\\
\leq&\: C\int_{\underline{N}_{v_0}}u^2(\Lbar \phi)^2+\frac{1}{4}v^2r^{-2}D|\snabla_{\s^2}\phi|^2\,d\omega du+C\int_{\mathcal{H}^+_{\leq -v_0}}  v^2(L \phi)^2+\frac{1}{4}u^2r^{-2}D|\snabla_{\s^2}\phi|^2\,d\omega dv
\end{split}
\end{equation}
and hence,
\begin{equation*}
\begin{split}
\int_{\widetilde{\Sigma}\cap\{u_{r_{\mathcal{H}}}\leq u\leq -v_{-\infty}\}}& (r-M)^{-2}(L \phi)^2+(r-M)^{-2}(\Lbar \phi)^2+|\snabla_{\s^2}\phi|^2\,du \leq C \int_{\mathcal{H}^+_{\leq -v_0}} v^2(L \phi)^2+ |\snabla_{\s^2}\phi|^2\,d\omega dv\\ 
&+C \int_{\underline{N}_{-v_0}} D^{-1}(\Lbar \phi)^2+D|\snabla_{\s^2} \phi|^2\,d\omega du.
\end{split}
\end{equation*}

We now consider the forwards time direction. First of all, we are assuming compact support on $\widetilde{\Sigma}_0\cap\{v_{r_{\mathcal{I}}}\leq v\leq -u_{-\infty}\}$, so for $|u_{-\infty}|,|v_{-\infty}|$ suitably large, we have that $\phi$ vanishes along $N_{-u_{-\infty}}$, $\underline{N}_{-v_{-\infty}}$, $\mathcal{I}^+\cap\{{u\leq u_{-\infty}\}}$ and $\mathcal{H}^+\cap\{{v\leq v_{-\infty}\}}$, by the domain of dependence property of the wave equation. 

We then apply the estimates \eqref{eq:mainestKcomm} and \eqref{eq:mainestKcomm2} to obtain:
\begin{align*}
\int_{N_{-u_0}}& v^2(L \phi)^2+\frac{1}{4}u^2r^{-2}D|\snabla_{\s^2}\phi|^2\,d\omega dv+\int_{\mathcal{I}^+_{\leq -u_0}} u^2(\Lbar \phi)^2+\frac{1}{4}v^2r^{-2}D|\snabla_{\s^2}\phi|^2\,d\omega du\\
\leq &\:\frac{1}{c}\int_{\widetilde{\Sigma}\cap\{v_{r_{\mathcal{I}}}\leq v\leq -u_{-\infty}\}} r^2(L \phi)^2+r^2(\Lbar \phi)^2+|\snabla_{\s^2}\phi|^2\,dv,\\
\int_{\underline{N}_{v_0}}&u^2(\Lbar \phi)^2+\frac{1}{4}v^2r^{-2}D|\snabla_{\s^2}\phi|^2\,d\omega du+\int_{\mathcal{H}^+_{\leq -v_{\infty}}}  v^2(L \phi)^2+\frac{1}{4}u^2r^{-2}D|\snabla_{\s^2}\phi|^2\,d\omega dv\\
\leq &\: \frac{1}{c}\int_{\widetilde{\Sigma}\cap\{u_{r_{\mathcal{H}}}\leq u\leq -v_{-\infty}\}}(r-M)^{-2}(L \phi)^2+(r-M)^{-2}(\Lbar \phi)^2+|\snabla_{\s^2}\phi|^2\,d\omega du,
\end{align*}
for a suitably small positive constant $c>0$.
\end{proof}

We complement Proposition \ref{prop:estasympflat} with estimates involving lower weights in $r$, $u$ and $v$, applied to $T\phi$ rather than $\phi$. The $r$-weighted energies along $N_{-u_0}$ and $\underline{N}_{-v_0}$ appearing in the proposition below appear as energy flux terms in Proposition \ref{prop:backwrp} with $p=1$. 
\begin{proposition}
\label{prop:p1estspacelikeinf}
Let $u_{-\infty}, v_{-\infty}<0$, with $|u_{-\infty}|, |v_{-\infty}|$ arbitrarily large. There exists constants\\ $C,c=C,c(M, r_{\mathcal{I}},r_{\mathcal{H}},u_0,v_0)>0$, such that
\begin{align}
\label{eq:p1estnearspacelikeinf1}
\int_{\widetilde{\Sigma}\cap\{v_{r_{\mathcal{I}}}\leq v\leq -u_{-\infty}\}}& r (\Lbar T \phi)^2+r (LT\phi)^2+r^{-1}|\snabla_{\s^2}T\phi|^2\,d\omega dr\\ \nonumber
&+\sum_{|\alpha|\leq 1}\int_{\widetilde{\Sigma}\cap\{v_{r_{\mathcal{I}}}\leq v\leq -u_{-\infty}\}}  (\Lbar \Omega^{\alpha} \phi)^2+(L\Omega^{\alpha}\phi)^2+r^{-2}|\snabla_{\s^2}\Omega^{\alpha}\phi|^2\,d\omega dr\\ \nonumber
\sim_{c,C}&\: \sum_{|\alpha|\leq 1}\int_{N_{-u_0}} v (LT \phi)^2+r^{-2}|\snabla_{\s^2} T\phi|^2+(L\Omega^{\alpha}\phi)^2+r^{-2}|\snabla_{\s^2}\Omega^{\alpha}\phi|^2\,d\omega dv\\ \nonumber
&+\sum_{j=0}^1\int_{\mathcal{I}^+_{\leq -u_0}} |u|^j(\Lbar T^j\phi)^2+|\snabla_{\s^2}\Lbar  \phi|^2\,d\omega du,\\
\label{eq:p1estnearspacelikeinf2}
\int_{\widetilde{\Sigma}\cap\{u_{r_{\mathcal{H}}}\leq u\leq -v_{-\infty}\}}& (r-M)^{-1}(L T\phi)^2+(r-M)^{-1}(\Lbar T\phi)^2+(r-M)|\snabla_{\s^2}T\phi|^2\,du \\ \nonumber
&+\sum_{|\alpha|\leq 1}\int_{\widetilde{\Sigma}\cap\{u_{r_{\mathcal{H}}}\leq u\leq -v_{-\infty}\}} (L \Omega^{\alpha}\phi)^2+(\Lbar  \Omega^{\alpha}\phi)^2+(r-M)^2|\snabla_{\s^2}\Omega^{\alpha}\phi|^2\,du \\ \nonumber
\sim_{c,C}&\:  \sum_{|\alpha|\leq 1}\int_{\underline{N}_{-v_0}} u (\Lbar T \phi)^2+(r-M)^{2}|\snabla_{\s^2} T\phi|^2+(\Lbar \Omega^{\alpha}\phi)^2+(r-M)^2|\snabla_{\s^2}\Omega^{\alpha}\phi|^2\,d\omega du\\ \nonumber
&+\sum_{j=0}^1\int_{\mathcal{H}^+_ {\leq -v_0}} |v|^j(LT^j\phi)^2+|\snabla_{\s^2}L \phi|^2\,  d\omega dv.
\end{align}
\end{proposition}

\begin{remark}
The energy estimates \eqref{eq:p1estnearspacelikeinf1} and \eqref{eq:p1estnearspacelikeinf2} are associated to the vector field multiplier $Y=v\partial_v-u\partial_u$ near infinity and $Y=-v\partial_v+u\partial_u$ near the horizon. In contrast with the vector field $K$ that plays a role in Proposition \ref{prop:estasympflat}, $Y$ does not correspond to a (conformal) symmetry generator in Minkowski.
\end{remark}
\begin{proof}[Proof of Proposition \ref{prop:p1estspacelikeinf}]
First of all, we have immediately that by Lemma \ref{lm:tenconsv} and [Hardy]
\begin{align}
\label{eq:tenconsvradfield1}
\int_{\widetilde{\Sigma}\cap\{v_{r_{\mathcal{I}}}\leq v\leq -u_{-\infty}\}}&  (\Lbar \phi)^2+(L \phi)^2+r^{-2}|\snabla_{\s^2}\phi|^2\,d\omega dv\\  \nonumber
&\sim_{c,C} \int_{\underline{N}_{u_0}}  (L \phi)^2+r^{-2}|\snabla_{\s^2}\phi|^2  d\omega dv+\int_{\mathcal{I}^+_ {\leq -u_0}}(\Lbar \phi)^2\, d\omega du,\\
\label{eq:tenconsvradfield2}
\int_{\widetilde{\Sigma}\cap\{u_{r_{\mathcal{H}}}\leq u\leq -v_{-\infty}\}}& (L \phi)^2+(\Lbar \phi)^2+(r-M)^2|\snabla_{\s^2}\phi|^2\, d\omega du \\ \nonumber
\sim_{c,C}&\:  \int_{\underline{N}_{-v_0}}(\Lbar \phi)^2+(r-M)^2|\snabla_{\s^2}\phi|^2\,d\omega du+ \int_{\mathcal{H}^+_{\leq -v_0}} (L \phi)^2\,d\omega dv.
\end{align}
We can moreover replace $\phi$ with $\Omega^{\alpha}\phi$ in the above estimates, with $|\alpha|\leq 1$, due to the commutation properties of $\Omega_i$ and $\square_g$.

By \eqref{eq:maineqradfield} it follows that
\begin{equation*}
\begin{split}
L (|u|(\Lbar T\phi)^2)+\Lbar (v(LT\phi)^2)=&\:2|u|\Lbar L T\phi\cdot \Lbar T\phi+2v\Lbar L T\phi\cdot LT\phi\\
=&\: \frac{1}{2}\frac{|u|}{r^2}D \slashed{\Delta}_{\s^2}T\phi\cdot \Lbar T\phi+\frac{1}{2}\frac{v}{r^2}D \slashed{\Delta}_{\s^2}T\phi\cdot L T\phi+(|u|\Lbar T\phi+vL T\phi)O(r^{-3})T\phi.
\end{split}
\end{equation*}
After integrating by parts on $\s^2$, we therefore obtain:
\begin{equation}
\label{eq:keyidasymflatp11}
\begin{split}
\int_{\s^2}& L(|u|(\Lbar T\phi)^2)+\Lbar(v(LT\phi)^2)+\frac{1}{4} L(v r^{-2}D|\slashed{\nabla}_{\s^2}T\phi|^2)+\frac{1}{4}\Lbar (|u|r^{-2}D|\slashed{\nabla}_{\s^2}T\phi|^2)\,d\omega\\
=&\:\int_{\s^2}\left(L\left(\frac{vD}{4r^2}\right)-\Lbar \left(\frac{uD}{4r^2}\right)\right)|\snabla_{\s^2}T\phi|^2\,d\omega+\int_{\s^2}(|u|\Lbar T\phi+vLT\phi)O(r^{-3})T\phi \,d\omega.
\end{split}
\end{equation}
Note that
\begin{equation*}
L\left(\frac{vD}{4r^2}\right)-\Lbar\left(\frac{uD}{4r^2}\right)=\frac{1}{8}(v+u)D\frac{d}{dr}(Dr^{-2})=-\frac{t}{2}(r^{-3}+O(r^{-4}))
\end{equation*}
Hence, after integrating \eqref{eq:keyidasymflatp11} in spacetime, the $|\snabla_{\s^2}T\phi|^2$ term on the right-hand side  will have a good sign if we consider forwards-in-time estimates and a bad sign if we consider backwards-in-time estimates.

In the backwards-in-time case, we use that $T=\partial_u+\partial_v$ and $t=\frac{1}{2}(v-|u|)$ and  $|u|+v\lesssim r$ in the integration region, together with Lemma \ref{lm:angmom} to estimate:
\begin{equation*}
\begin{split}
\int_{v_{r_{\mathcal{I}}}}^{-u_{\infty}} \int_{-u_0}^{-v}t\cdot r^{-3} |\snabla_{\s^2}T\phi|^2\,d\omega du dv\lesssim&\: \sum_{|\alpha|=1}\int_{v_{r_{\mathcal{I}}}}^{-u_{-\infty} } \int_{-u_0}^{-v}r^{-2}\left[ r^2(\Lbar\Omega^{\alpha}\psi)^2+r^2(L \Omega^{\alpha}\psi)^2\right]\,d\omega du dv\\
\lesssim&\: \sum_{|\alpha|=1} \sup_{v} \int_{I_v} r^2(\Lbar \Omega^{\alpha}\psi)^2\,d\omega du+ \sup_{u} \int_{N_u} r^2(L\Omega^{\alpha}\psi)^2\,d\omega dv\\
\lesssim&\: \sum_{|\alpha|\leq 1} \sup_{v} \int_{I_v} {\mathbf{J}}^T[\Omega^{\alpha}\psi]\cdot \underline{L}\,d\omega du+ \sup_{u} \int_{N_u} {\mathbf{J}}^T[\Omega^{\alpha}\psi]\cdot L\,d\omega dv\\
\lesssim&\: \sum_{|\alpha|\leq 1} \int_{\underline{N}_{-u_0}}  {\mathbf{J}}^T[\Omega^{\alpha}\psi]\cdot L\,   d\omega dv+\int_{\mathcal{I}^+_ {\leq -u_0}} {\mathbf{J}}^T[\Omega^{\alpha}\psi]\cdot \underline{L}\,  d\omega du,
\end{split}
\end{equation*}
where we arrived at the last inequality by applying Lemma \ref{lm:tenconsv}. \textbf{Note that in this step we needed to use that our solution to \eqref{eq:waveequation} a time derivative, i.e. it is of the form $T\psi$!}

We moreover apply Young's inequality to estimate
\begin{equation*}
\begin{split}
r^{-3}|T\phi| (|u||\Lbar T\phi|+v|LT\phi|)\leq&\: r^{-1-\eta}  (|u|(\Lbar T\phi)^2+v(LT\phi)^2)+ r^{-5+\eta}(|u|+v)(T\phi)^2\\
\leq&\: r^{-1-\eta}  (|u|(\Lbar T\phi)^2+v(LT\phi)^2)+ r^{-5+\eta}(|u|+v)((\Lbar \phi)^2+(L \phi)^2).
\end{split}
\end{equation*}
We can absorb the spacetime integrals of the terms on the very right-hand side into the following flux terms:
\begin{equation*}
\sup_{u} \int_{N_u} v(LT\phi)^2+\frac{1}{4}|u|r^{-2}D|\snabla_{\s^2}T\phi|^2\,d\omega dv+ \sup_{v} \int_{I_v} |u|(\Lbar T\phi)^2+\frac{1}{4}vr^{-2}D|\snabla_{\s^2}T\phi|^2\,d\omega du
\end{equation*}
and
\begin{equation*}
\sup_{u} \int_{N_u} \mathbf{J}^T[\psi]\cdot {L}\,d\omega dv+ \sup_{v} \int_{I_v}{\mathbf{J}}^T[\psi]\cdot \underline{L}\,d\omega du.
\end{equation*}

Integrating the identity \eqref{eq:keyidasymflatp11} in $u$ and $v$ and applying the above estimates therefore gives the following inequality:
\begin{equation}
\label{eq:fluxbound}
\begin{split}
\sup_{u} &\int_{N_u} v(LT\phi)^2+\frac{1}{4}|u|r^{-2}D|\snabla_{\s^2}T\phi|^2\,d\omega dv+ \sup_{v} \int_{I_v} |u|(\Lbar T\phi)^2+\frac{1}{4}vr^{-2}D|\snabla_{\s^2}T\phi|^2\,d\omega du\\
\leq&\: C\int_{N_{-u_0}} v(L T\phi)^2+\frac{1}{4}|u|r^{-2}D|\snabla_{\s^2}T\phi|^2\,d\omega dv+C\int_{\mathcal{I}^+_{\leq -u_0}} |u|(\Lbar T\phi)^2+\frac{1}{4}vr^{-2}D|\snabla_{\s^2}T\phi|^2\,d\omega du\\
&+ C\sum_{|\alpha|\leq 1} \int_{\underline{N}_{-u_0}}  {\mathbf{J}}^T[\Omega^{\alpha}\psi]\cdot L  d\omega dv+\int_{\mathcal{I}^+_ {\leq -u_0}} {\mathbf{J}}^T[\Omega^{\alpha}\psi]\cdot \underline{L}  d\omega du
\end{split}
\end{equation}
and hence, using \eqref{eq:keyidasymflatp11} and the above estimate once more, now in combination with \eqref{eq:fluxbound}, we arrive at 
\begin{equation*}
\begin{split}
\int_{\widetilde{\Sigma}\cap\{v_{r_{\mathcal{I}}}\leq v\leq -u_{-\infty}\}}& r (\Lbar T\phi)^2+r^2(L T\phi)^2+r^{-1}|\snabla_{\s^2}T\phi|^2\,d\omega dv\\
\leq&\: C\int_{N_{-u_0}} v(LT\phi)^2+\frac{1}{4}|u|r^{-2}D|\snabla_{\s^2}T\phi|^2\,d\omega dv+C\int_{\mathcal{I}^+_{\leq -u_0}} |u|(\Lbar T\phi)^2+\frac{1}{4}vr^{-2}D|\snabla_{\s^2}T\phi|^2\,d\omega du\\
&+ C\sum_{|\alpha|\leq 1} \int_{{N}_{-u_0}}  {\mathbf{J}}^T[\Omega^{\alpha}\psi]\cdot L  d\omega dv+\int_{\mathcal{I}^+_ {\leq -u_0}} (\Lbar \Omega^{\alpha}\phi)^2 d\omega du.
\end{split}
\end{equation*}

We repeat the above arguments near $\mathcal{H}^+$ by considering
\begin{equation*}
\Lbar(|v|(LT\phi)^2)+L(u(\Lbar T \phi)^2)
\end{equation*}
and reversing the roles of $u$ and $v$ and $L$ and $\underline{L}$, in order to obtain the near-horizon estimate in the backwards time direction. We omit further details of this step.

Now, we consider the forwards time direction. By repeating the arguments above in the forwards time direction, using that the $\psi$ and $\mathbf{n}_{\widetilde{\Sigma}}\psi$ are initially compactly supported and taking $|u_{-\infty}|$ and $|v_{-\infty}|$ appropriately large, we obtain moreover that
\begin{equation*}
\begin{split}
\sup_{u} &\int_{N_u} v(LT\phi)^2+\frac{1}{4}|u|r^{-2}D|\snabla_{\s^2}T\phi|^2\,d\omega dv+ \sup_{v} \int_{I_v} |u|(\Lbar T\phi)^2+\frac{1}{4}vr^{-2}D|\snabla_{\s^2}T\phi|^2\,d\omega dv\\
\leq&\: C\int_{\widetilde{\Sigma}\cap\{v_{r_{\mathcal{I}}}\leq v\leq -u_{-\infty}\}} r (\Lbar T\phi)^2+r^2(LT\phi)^2+r^{-1}|\snabla_{\s^2}T\phi|^2\,d\omega dv\\
&+C\int_{\widetilde{\Sigma}\cap\{v_{r_{\mathcal{I}}}\leq v\leq -u_{-\infty}\}}  {\mathbf{J}}^T[\psi]\cdot \mathbf{n}_{\widetilde{\Sigma}}\,  d\mu_{\widetilde{\Sigma}}.
\end{split}
\end{equation*}
Note that, in contrast with the backwards-in-time estimates, there is no need for an additional angular derivative in the $T$-energy term on the right hand side. The analogous estimate near $\mathcal{H}^+$ proceeds by repeating the above arguments, interchanging the roles of $u$ and $v$ and replacing $r$ by $(r-M)^{-1}$.
\end{proof}

\begin{corollary}
\label{cor:mainspacelikeinf}
Let $u_{-\infty}, v_{-\infty}<0$, with $|u_{-\infty}|, |v_{-\infty}|$ arbitrarily large. There exists constants $C,c=C,c(M, r_{\mathcal{I}},r_{\mathcal{H}},u_0,v_0)>0$, such that
\begin{align}
\label{eq:fullestnearspacelikeinf1}
\sum_{j=0}^2\int_{\widetilde{\Sigma}\cap\{v_{r_{\mathcal{I}}}\leq v\leq -u_{-\infty}\}}& r^{2-j} (\Lbar T^j \phi)^2+r^{2-j}(LT^j\phi)^2+r^{2-j}|\snabla_{\s^2}T^j\phi|^2\,d\omega dr\\ \nonumber
&+\int_{\widetilde{\Sigma}\cap\{v_{r_{\mathcal{I}}}\leq v\leq -u_{-\infty}\}}  |\snabla_{\s^2}\Lbar  \phi|^2+ |\snabla_{\s^2}L\phi|^2+r^{-2}|\snabla_{\s^2}^2 \phi|^2\,d\omega dr\\ \nonumber
\sim_{c,C}&\: \sum_{j=0}^2\int_{N_{-u_0}} r^{2-j} (LT^j \phi)^2+r^{-2}|\snabla_{\s^2} T^j\phi|^2+|\snabla_{\s^2}L\phi|^2+r^{-2}|\snabla_{\s^2}^2\phi|^2\,d\omega dv\\ \nonumber
&+\sum_{j=0}^2\int_{\mathcal{I}^+_{\leq -u_0}} (1+|u|)^{2-j}(\Lbar T^j\phi)^2+|\snabla_{\s^2} \phi|^2+|\snabla_{\s^2}\Lbar  \phi|^2\,d\omega du,\\
\label{eq:fullestnearspacelikeinf2}
\sum_{j=0}^2\int_{\widetilde{\Sigma}\cap\{u_{r_{\mathcal{H}}}\leq u\leq -v_{-\infty}\}}& (r-M)^{-2+j}(L T^j\phi)^2+(r-M)^{-2+j}(\Lbar T^j\phi)^2+(r-M)^{j}|\snabla_{\s^2}T^j\phi|^2\,du \\ \nonumber
&+\int_{\widetilde{\Sigma}\cap\{u_{r_{\mathcal{H}}}\leq u\leq -v_{-\infty}\}}  |\snabla_{\s^2}L \phi|^2+ |\snabla_{\s^2}\Lbar \phi|^2+(r-M)^2|\snabla_{\s^2}^2\phi|^2\,du \\ \nonumber
\sim_{c,C}&\:\sum_{j=0}^2\int_{\underline{N}_{-v_0}} (r-M)^{-2+j} (\Lbar T^j \phi)^2+(r-M)^{2}|\snabla_{\s^2} T^j\phi|^2+ |\snabla_{\s^2}\Lbar\phi|^2+(r-M)^{2}|\snabla_{\s^2}^2 \phi|^2\,d\omega du\\ \nonumber
&+\sum_{j=0}^2\int_{\mathcal{H}^+_ {\leq -v_0}} (1+|v|)^{2-j}(LT^j\phi)^2+|\snabla_{\s^2} \phi|^2+|\snabla_{\s^2}L  \phi|^2\,  d\omega dv.
\end{align}
\end{corollary}
\begin{proof}
We combine \eqref{eq:tenconsvradfield2}, \eqref{eq:tenconsvradfield2} (and apply it to $T^2\phi$), Proposition \ref{prop:estasympflat} and Proposition \ref{prop:p1estspacelikeinf}. We moreover apply Lemma \ref{lm:angmom}.
\end{proof}

\subsection{Higher-order estimates}
\label{sec:hoestasymflat}
The aim of this section is to derive analogues of the estimates in Proposition \ref{prop:estasympflat} for higher-order derivatives of $\psi$ (with additional growing weights). The key vector field that plays a role in this step is $S=u\Lbar+vL$. This vector field is also called the \emph{scaling vector field} because it generates the scaling conformal symmetry in Minkowski. Even though the exact symmetry property is lost in extremal Reissner--Nordstr\"om, we will see below that the vector field still has favourable commutation properties with the operator $L\Lbar$.
\begin{lemma}
\label{lm:equationSnphi}
Let $n \in \N_0$ and $S=u\Lbar+vL$. Then
\begin{align}
\label{eq:commSn1}
\Lbar L(S^n\phi)=&\frac{1}{4}Dr^{-2}\slashed{\Delta}_{\s^2}(S^n\phi)-\frac{DD'}{4r}S^n\phi+ n\sum_{k=0}^{\max\{n-1,0\}} O(r^{-3})S^k\phi+ O(r^{-3})\log r\slashed{\Delta}_{\s^2}S^k\phi,\\
\label{eq:commSn2}
\Lbar L(S^n\phi)=&\frac{1}{4}Dr^{-2}\slashed{\Delta}_{\s^2}(S^n\phi)-\frac{DD'}{4r}S^n\phi+ n\sum_{k=0}^{\max\{n-1,0\}} O((r-M)^{3})S^k\phi+ O((r-M)^{3})\log ((r-M)^{-1})\slashed{\Delta}_{\s^2}S^k\phi.
\end{align}
\end{lemma}
\begin{proof}
We will derive \eqref{eq:commSn1} and \eqref{eq:commSn2} inductively. Note that \eqref{eq:commSn1} and \eqref{eq:commSn2} hold for $n=0$ by \eqref{eq:maineqradfield}. Now assume  \eqref{eq:commSn1} and \eqref{eq:commSn2} hold for $n=N$ with $N\geq 0$.

Note first of all that for an arbitrary $C^2$ function $f$:
\begin{equation*}
\begin{split}
\Lbar L(Sf)=&\:\Lbar L(u\Lbar f+vLf)\\
=&\: (u\Lbar+vL+ 2)(\Lbar L f).
\end{split}
\end{equation*}
For any $p\geq 0$ we have that:
\begin{align*}
S(O(r^{-p}))= &O(r^{-p}),\\
S(O((r-M)^{p}))=&O((r-M)^{p}).
\end{align*}
Furthermore, we can expand
\begin{align*}
Dr^{-2}=&\:\frac{4}{(v-u)^2}+O(r^{-3})\log r,\\
Dr^{-2}=&\:\frac{4}{(v-u)^2}+O((r-M)^{3})\log ((r-M)^{-1}).
\end{align*}
Hence,
\begin{align*}
S(Dr^{-2})=&\:-\frac{8}{(v-u)^2}+O(r^{-3})\log r=-2Dr^{-2}+O(r^{-3})\log r,\\
S(Dr^{-2})=&\:-\frac{8}{(v-u)^2}+O((r-M)^{3})\log ((r-M)^{-1})=-2Dr^{-2}+O((r-M)^{3})\log ((r-M)^{-1}),
\end{align*}
and we obtain, using the above observations and applying \eqref{eq:commSn1} with $n=N$:
\begin{equation*}
\begin{split}
\Lbar L(S^{N+1}\phi)=&\:\frac{1}{4}Dr^{-2}\slashed{\Delta}_{\s^2}(S^{N+1}\phi)-\frac{DD'}{4r}S^{N+1}\phi+ N\sum_{k=0}^{\max\{N-1,0\}} O(r^{-3})S^{k+1}\phi+ O(r^{-3})\log r\slashed{\Delta}_{\s^2}S^{k+1}\phi\\
&+\frac{1}{4}(S(Dr^{-2})+2Dr^{-2})\slashed{\Delta}_{\s^2}S^N\phi+ \sum_{k=0}^{N} O(r^{-3})S^k\phi+ O(r^{-3})\log r\slashed{\Delta}_{\s^2}S^k\phi\\
=&\:\frac{1}{4}Dr^{-2}\slashed{\Delta}_{\s^2}(S^{N+1}\phi)-\frac{DD'}{4r}S^{N+1}\phi+ (N+1)\sum_{k=0}^{N} O(r^{-3})S^{k+1}\phi+ O(r^{-3})\log r\slashed{\Delta}_{\s^2}S^{k+1}\phi.
\end{split}
\end{equation*}
Hence, we can conclude that \eqref{eq:commSn1} must hold for all $n\in \N_0$. It follows analogously that \eqref{eq:commSn2} must hold for all $n\in \N_0$.
\end{proof}

Since the vector field $S$ does not commute with $\square_g$, we do not immediately obtain Lemma \ref{lm:tenconsv} for $S^n\psi$ replacing $\psi$, with $n\in \N$. However, we show in Proposition \ref{prop:TenergyboundSkphi} that, when considering $\phi$ instead of $\psi$, an equivalent energy boundedness statement holds.

\begin{proposition}
\label{prop:TenergyboundSkphi}
Let $n\in \N_0$. There exists constants $c,C=c,C(M,\widetilde{\Sigma}, r_{\mathcal{I}},r_{\mathcal{H}},,u_0,v_0,n)>0$, such that
\begin{align}
\label{est:p0skcomm1}
\sum_{0\leq k\leq n}\int_{\widetilde{\Sigma}\cap\{v_{r_{\mathcal{I}}}\leq v\leq -u_{-\infty}\}}& (L S^k\phi)^2+(\Lbar S^k\phi)^2+r^{-2}|\snabla_{\s^2}S^k\phi|^2\,d\omega dv\\ \nonumber
\sim_{c,C}&\: \sum_{0\leq k\leq n}\left[\int_{N_{-u_0}} (L S^k\phi)^2+r^{-2}|\snabla_{\s^2}S^k\phi|^2\,d\omega dv+\int_{\mathcal{I}^+_{\leq -u_0}} (\Lbar S^k\phi)^2\,d\omega du\right],\\
\sum_{0\leq k\leq n} \int_{\widetilde{\Sigma}\cap\{u_{r_{\mathcal{H}}}\leq u\leq -v_{-\infty}\}}& (LS^k\phi)^2+(\Lbar S^k\phi)^2+D|\snabla_{\s^2}S^k\phi|^2\,du\\ \nonumber
\sim_{c,C}&\:  \sum_{0\leq k\leq n} \left[\int_{\underline{N}_{ -v_0}} (\Lbar S^k\phi)^2+D|\snabla_{\s^2} S^k\phi|^2\,d\omega du+\int_{\mathcal{H}^+_{\leq -v_0}} (LS^k\phi)^2\,d\omega dv\right].
\end{align}
\end{proposition}
\begin{proof}
We establish the estimate \eqref{est:p0skcomm1} inductively. We prove the $n=0$ case first and then assume that \eqref{est:p0skcomm1} holds for $0\leq k\leq n-1$ in order to prove the $k=n$ case. We will in fact do both of these steps at the same time in the argument below. By Lemma \ref{lm:equationSnphi}, we have that
\begin{equation}
\label{eq:p0skcomm}
\begin{split}
L((\Lbar S^n\phi)^2)+\Lbar (v^2(LS^n\phi)^2)=&\:2\Lbar L S^n\phi\cdot \Lbar S^n\phi+2\Lbar L S^n\phi\cdot LS^n\phi\\
=&\: \frac{1}{2}\frac{1}{r^2}D \slashed{\Delta}_{\s^2}S^n\phi\cdot \Lbar S^n\phi+\frac{1}{2}\frac{1}{r^2}D \slashed{\Delta}_{\s^2}S^n\phi\cdot L S^n\phi\\
&+(\Lbar S^n\phi+L S^n\phi)\left(\sum_{k=0}^nO(r^{-3})S^k\phi+\sum_{k=0}^{n-1}\log r O(r^{-3})\slashed{\Delta}_{\s^2}S^k\phi\right).
\end{split}
\end{equation}
Furthermore,
\begin{equation*}
\frac{1}{2}\frac{1}{r^2}D \slashed{\Delta}_{\s^2}S^n\phi\cdot \Lbar S^n\phi+\frac{1}{2}\frac{1}{r^2}D \slashed{\Delta}_{\s^2}S^n\phi\cdot L S^n\phi= -\Lbar \left(\frac{1}{4r^2}D |\snabla_{\s^2}S^n\phi|^2\right)-L\left(\frac{1}{4r^2}D |\snabla_{\s^2}S^n\phi|^2\right).
\end{equation*}
We subsequently integrate both sides of \eqref{eq:p0skcomm} in $u$, $v$ and $\s^2$ and we apply Young's inequality to absorb all the spacetime integrals either into the corresponding boundary integrals as in the proof of Proposition \ref{prop:estasympflat}, or (if $n\geq 1$) also into the left-hand sides of the estimates contained in \eqref{est:p0skcomm1} with $0\leq k\leq n-1$. 
\end{proof}

\begin{proposition}
\label{prop:p2asympflatest}
Let $n\in \N_0$. There exists constants $c,C=c,C(M, r_{\mathcal{I}},r_{\mathcal{H}},n,u_0,v_0)>0$, such that
\begin{align}
\label{eq:asymflatestback1ho}
\sum_{0\leq k\leq n}&\sum_{|\alpha|\leq n-k}\int_{\widetilde{\Sigma}\cap\{v_{r_{\mathcal{I}}}\leq v\leq -u_{-\infty}\}} r^{2+2k}(L^{k+1}\Omega^{\alpha}\phi)^2+r^{2+2k}(\Lbar ^{k+1}\Omega^{\alpha}\phi)^2\\ \nonumber
&+r^{2k}(|\snabla_{\s^2}L^k\Omega^{\alpha}\phi|^2+|\snabla_{\s^2}\Lbar^k\Omega^{\alpha}\phi|^2)\,d\omega dv\\ \nonumber
\sim_{c,C}&\: \sum_{0\leq k\leq n}\sum_{|\alpha|\leq n-k}\int_{N_{-u_0}} r^{2+2k}(L^{k+1}\Omega^{\alpha}\phi)^2+r^{-2+2k}|\snabla_{\s^2}L^k\Omega^{\alpha}\phi|^2\,d\omega dv\\ \nonumber
&+\int_{\mathcal{I}^+_{\leq -u_0}} u^{2+2k}(\Lbar^k\Omega^{\alpha}\phi)^2+u^{2k}|\snabla_{\s^2}\Lbar^k\Omega^{\alpha}\phi|^2\,d\omega du,\\
\label{eq:asymflatestback2ho}
\sum_{0\leq k\leq n}&\sum_{|\alpha|\leq n-k}\int_{\widetilde{\Sigma}\cap\{u_{r_{\mathcal{H}}}\leq u\leq -v_{-\infty}\}} (r-M)^{-2-2k}(L^{k+1}\Omega^{\alpha}\phi)^2+(r-M)^{-2-2k}(\Lbar^{k+1}\Omega^{\alpha}\phi)^2\\ \nonumber
&+(r-M)^{-2k}(|\snabla_{\s^2}L^k\Omega^{\alpha}\phi|^2+|\snabla_{\s^2}\Lbar^k\Omega^{\alpha}\phi|^2)\,d\omega du \\ \nonumber
\sim_{c,C} &\:  \sum_{0\leq k\leq n}\sum_{|\alpha|\leq n-k} \int_{\underline{N}_{-v_0}} (r-M)^{-2-2k}(\Lbar^{k+1}\Omega^{\alpha}\phi)^2+(r-M)^{2-2k}|\snabla_{\s^2}\Lbar^k\Omega^{\alpha}\phi|^2\,d\omega du\\ \nonumber
&+ \int_{\mathcal{H}^+_{\leq -v_0}} v^{2+2k}(L^k\Omega^{\alpha}\phi)^2+v^{2k}|\snabla_{\s^2}L^k\Omega^{\alpha}\phi|^2\,d\omega dv.
\end{align}
\end{proposition}
\begin{proof}
We can apply the same arguments as in Proposition \ref{prop:estasympflat}, replacing $\phi$ by $S^k\phi$, with $0\leq k\leq n$ and applying the more general equations \eqref{eq:commSn1} and \eqref{eq:commSn2} instead of  \eqref{eq:maineqradfield} to obtain:
\begin{align*}
\sum_{0\leq k\leq n}\int_{\widetilde{\Sigma}\cap\{v_{r_{\mathcal{I}}}\leq v\leq -u_{-\infty}\}}& r^2(LS^k\phi)^2+r^2(\Lbar S^k\phi)^2+|\snabla_{\s^2}S^k\phi|^2\,d\omega dv\\ 
\sim_{c,C}&\: \sum_{0\leq k\leq n}\left[\int_{N_{-u_0}} r^2(LS^k\phi)^2+r^{-2}|\snabla_{\s^2}S^k\phi|^2\,d\omega dv+\int_{\mathcal{I}^+_{\leq -u_0}} u^2(\Lbar S^k\phi)^2+|\snabla_{\s^2}S^k\phi|^2\,d\omega du\right],\\
\sum_{0\leq k\leq n} \int_{\widetilde{\Sigma}\cap\{u_{r_{\mathcal{H}}}\leq u\leq -v_{-\infty}\}}& (r-M)^{-2}(LS^k\phi)^2+(r-M)^{-2}(\Lbar S^k\phi)^2+|\snabla_{\s^2}S^k\phi|^2\,du\\
\sim_{c,C}&\:  \sum_{0\leq k\leq n} \Bigg[\int_{\underline{N}_{ -v_0}} (r-M)^{-2}(\Lbar S^k\phi)^2+(r-M)^2|\snabla_{\s^2} S^k\phi|^2\,d\omega du\\
&+\int_{\mathcal{H}^+_{\leq -v_0}} v^2(LS^k\phi)^2+ |\snabla_{\s^2}S^k\phi|^2\,d\omega dv\Bigg].
\end{align*}

We conclude the proof by rewriting $S^k\phi$ in terms of $u$ and $v$ derivatives and we moreover apply Lemma \ref{lm:equationSnphi} to rewrite all mixed $u$ and $v$ derivatives. Furthermore, we apply Lemma \ref{lm:angmom} to replace the angular derivatives by derivatives of the form $\Omega^{\alpha}$.
\end{proof}

\begin{proposition}
\label{prop:hop1asympflatest}
Let $n\in \N_0$. Then there exists constants $c,C=c,C(M, r_{\mathcal{I}},r_{\mathcal{H}},u_0,v_0,n)>0$, such that
\begin{align}
\label{eq:asymflatestbackp11ho}
\sum_{j=0}^1&\sum_{0\leq k\leq n}\sum_{|\alpha|\leq n-k+1-j}\int_{\widetilde{\Sigma}\cap\{v_{r_{\mathcal{I}}}\leq v\leq -u_{-\infty}\}} r^{2k+j}(L^{k+1}\Omega^{\alpha}T^{j}\phi)^2+r^{2k-2+j}|\snabla_{\s^2}L^{k}\Omega^{\alpha}T^{j}\phi|^2 \\ \nonumber
&+r^{2k+j}(\Lbar^{k+1}\Omega^{\alpha}T^{j}\phi)^2+r^{2k-2+j}|\snabla_{\s^2}\Lbar^{k}\Omega^{\alpha}T^{j}\phi|^2\,d\omega dr\\ \nonumber
&+\sum_{|\alpha|\leq n+1}\int_{\widetilde{\Sigma}\cap\{v_{r_{\mathcal{I}}}\leq v\leq -u_{-\infty}\}}   \mathbf{J}^T[\Omega^{\alpha} \psi]\cdot \mathbf{n}_{\widetilde{\Sigma}}\,d\mu_{\widetilde{\Sigma}} \\ \nonumber
\sim_{c,C}&\: \sum_{j=0}^1\sum_{0\leq k\leq n}\sum_{|\alpha|\leq n-k+1-j}\int_{N_{-u_0}} r^{2k+j}(L^{k+1}\Omega^{\alpha}T^{j}\phi)^2+r^{2k-2+j}|\snabla_{\s^2}L^{k}\Omega^{\alpha}T^{j}\phi|^2\,d\omega dv\\ \nonumber
&+\sum_{0\leq k\leq n}\sum_{|\alpha|\leq n-k}\int_{\mathcal{I}^+_{\leq -u_0}} |u|^{2k+1}(\Lbar^{k+1}\Omega^{\alpha}T\phi)^2+ |u|^{2k}(\Lbar^{k+1}\Omega^{\alpha}\phi)^2\,d\omega du\\ \nonumber
&+ \sum_{|\alpha|\leq n+1}\int_{{N}_{-u_0}} \mathbf{J}^T[\Omega^{\alpha}\psi]\cdot {L}\,d\omega dv+ \sum_{|\alpha|\leq n+1}\int_{\mathcal{I}^+_{\leq -u_0}} (\Lbar \Omega^{\alpha}\phi)^2\,d\omega du ,\\
\label{eq:asymflatestbackp12ho}
\sum_{j=0}^1&\sum_{0\leq k\leq n}\sum_{|\alpha|\leq n-k+1-j}\int_{\widetilde{\Sigma}\cap\{u_{r_{\mathcal{H}}}\leq u\leq -v_{-\infty}\}} (r-M)^{-2k-j}(L^{k+1}\Omega^{\alpha}T^{j}\phi)^2+(r-M)^{-2k+2-j}|\snabla_{\s^2}L^{k}\Omega^{\alpha}T^{j}\phi|^2 \\ \nonumber
&+(r-M)^{-2k-j}(\Lbar^{k+1}\Omega^{\alpha}T^{j}\phi)^2+(r-M)^{-2k+2-j}|\snabla_{\s^2}\Lbar^{k}\Omega^{\alpha}T^{j}\phi|^2\,d\omega dr_*\\ \nonumber
&+\sum_{|\alpha|\leq n+1}\int_{\widetilde{\Sigma}\cap\{u_{r_{\mathcal{H}}}\leq u\leq -v_{-\infty}\}}   \mathbf{J}^T[\Omega^{\alpha} \psi]\cdot \mathbf{n}_{\widetilde{\Sigma}}\,d\mu_{\widetilde{\Sigma}}\\ \nonumber
\sim_{c,C}&\: \sum_{j=0}^1 \sum_{0\leq k\leq n}\sum_{|\alpha|\leq n-k+1-j}\int_{\underline{N}_{-v_0}} (r-M)^{-2k-j}(\Lbar^{k+1}\Omega^{\alpha}T^{j}\phi)^2+(r-M)^{-2k+2}|\snabla_{\s^2}\Lbar^{k}\Omega^{\alpha}T^{j}\phi|^2\,d\omega du\\ \nonumber
&+\sum_{0\leq k\leq n}\sum_{|\alpha|\leq n-k}\int_{\mathcal{H}^+_{\leq -v_0}} |v|^{2k+1}(L^{k+1}\Omega^{\alpha}T\phi)^2+|v|^{2k}(L^{k+1}\Omega^{\alpha}\phi)^2+|v|^{2k}|\snabla_{\s^2} L^{k}\Omega^{\alpha} \phi|^2\,d\omega dv\\ \nonumber
&+ \sum_{|\alpha|\leq n+1}\int_{\underline{N}_{-v_0}} \mathbf{J}^T[\Omega^{\alpha}\psi]\cdot \underline{L}\,d\omega du+\sum_{|\alpha|\leq n+1}\int_{\mathcal{H}^+_{\leq -v_0}} (L\Omega^{\alpha}\phi)^2\,d\omega dv.
\end{align}
\end{proposition}
\begin{proof}
We repeat the arguments in the proof of Proposition \ref{prop:p1estspacelikeinf}, applying the equations in Lemma \ref{lm:equationSnphi} that introduce additional terms, which can be absorbed straightforwardly. Furthermore, rather than using Lemma \ref{lm:tenconsv}, we apply Proposition \ref{prop:TenergyboundSkphi} where necessary. We then obtain:
\begin{align*}
\sum_{0\leq k\leq n}\int_{\widetilde{\Sigma}\cap\{v_{r_{\mathcal{I}}}\leq v\leq -u_{-\infty}\}}& r (\Lbar T S^k\phi)^2+r (L TS^k\phi)^2+r^{-1}|\snabla_{\s^2}TS^k\phi|^2\,d\omega dr\\ \nonumber
&+\sum_{|\alpha|\leq 1}\int_{\widetilde{\Sigma}\cap\{v_{r_{\mathcal{I}}}\leq v\leq -u_{-\infty}\}}  (\Lbar \Omega^{\alpha} S^k\phi)^2+(L \Omega^{\alpha}\phi)^2+r^{-2}|\snabla_{\s^2}\Omega^{\alpha}S^k\phi|^2\,d\omega dr\\ \nonumber
\sim_{c,C}&\: \sum_{0\leq k\leq n}\sum_{|\alpha|\leq 1}\int_{N_{-u_0}} r (L T S^k\phi)^2+r^{-2}|\snabla_{\s^2} TS^k\phi|^2+(L\Omega^{\alpha}S^k\phi)^2+r^{-2}|\snabla_{\s^2}\Omega^{\alpha}S^k\phi|^2\,d\omega dv\\ \nonumber
&+\sum_{j=0}^1\int_{\mathcal{I}^+_{\leq -u_0}} |u|^j(\Lbar T^jS^k\phi)^2+|\snabla_{\s^2}\Lbar  S^k\phi|^2\,d\omega du,\\
\sum_{0\leq k\leq n}\int_{\widetilde{\Sigma}\cap\{u_{r_{\mathcal{H}}}\leq u\leq -v_{-\infty}\}}& (r-M)^{-1}(LTS^k\phi)^2+(r-M)^{-1}(\Lbar TS^k\phi)^2+(r-M)|\snabla_{\s^2}TS^k\phi|^2\,du \\ \nonumber
&+\sum_{|\alpha|\leq 1}\int_{\widetilde{\Sigma}\cap\{u_{r_{\mathcal{H}}}\leq u\leq -v_{-\infty}\}} (L \Omega^{\alpha}S^k\phi)^2+(\Lbar \Omega^{\alpha}S^k\phi)^2+(r-M)^2|\snabla_{\s^2}\Omega^{\alpha}S^k\phi|^2\,du \\ \nonumber
\sim_{c,C}&\:  \sum_{0\leq k\leq n}\sum_{|\alpha|\leq 1}\int_{\underline{N}_{-v_0}} (r-M)^{-1} (\Lbar T S^k\phi)^2+(r-M)^{2}|\snabla_{\s^2} TS^k\phi|^2+(\Lbar \Omega^{\alpha}S^k\phi)^2\\
&+(r-M)^2|\snabla_{\s^2}\Omega^{\alpha}S^k\phi|^2\,d\omega du\\ \nonumber
&+\sum_{j=0}^1\int_{\mathcal{H}^+_ {\leq -v_0}} |v|^j(L T^jS^k\phi)^2+|\snabla_{\s^2}L S^k\phi|^2\,  d\omega dv.
\end{align*}
We conclude the proof by replacing the $S^k$ derivatives by $u$ and $v$ derivatives with weights in $|u|$ and $|v|$, and moreover applying Lemma \ref{lm:equationSnphi} to rewrite all mixed $u$ and $v$ derivatives in terms of pure $u$ or $v$ derivatives, angular derivatives and lower-order derivatives.
\end{proof}

\begin{corollary}
Let $n\in \N_0$. Then there exists constants $c,C=c,C(M, r_{\mathcal{I}},r_{\mathcal{H}},u_0,v_0,n)>0$, such that
\begin{equation}
\label{eq:asymflatestbacktot1ho}
\begin{split}
\sum_{j=0}^2&\sum_{|\alpha|+k\leq n}\Bigg[\int_{\widetilde{\Sigma}\cap\{v_{r_{\mathcal{I}}}\leq v\leq -u_{-\infty}\}} r^{2+2k-j}(L^{k+1}\Omega^{\alpha}T^{j}\phi)^2+r^{2k-j}|\snabla_{\s^2}L^{k}\Omega^{\alpha}T^{j}\phi|^2 \\ 
&+ r^{2k}|\snabla_{\s^2}L^{k+1}\Omega^{\alpha}\phi|^2+r^{2k-2}|\snabla_{\s^2}^2L^{k}\Omega^{\alpha}\phi|^2+r^{2k+2-j}(\Lbar^{k+1}\Omega^{\alpha}T^{j}\phi)^2 \\
&+r^{2k-j}|\snabla_{\s^2}\Lbar^{k}\Omega^{\alpha}T^{j}\phi|^2+r^{2k}|\snabla_{\s^2}\Lbar^{k+1}\Omega^{\alpha}\phi|^2+r^{2k-2}|\snabla_{\s^2}^2\Lbar^{k}\Omega^{\alpha}\phi|^2\,d\omega dr\Bigg]  \\
\sim_{c,C}&\: \sum_{j=0}^2\sum_{|\alpha|+k\leq n} \Bigg[\int_{N_{-u_0}} r^{2k+2-j}(L^{k+1}\Omega^{\alpha}T^{j}\phi)^2+r^{2k}|\snabla_{\s^2}L^{k+1}\Omega^{\alpha}\phi|^2\,d\omega dv \\
&+\int_{\mathcal{I}^+_{\leq -u_0}} |u|^{2k+2-j}(\Lbar^{k+1}\Omega^{\alpha}T^{j}\phi)^2+ |u|^{2k}|\snabla_{\s^2}\Lbar^{k+1}\Omega^{\alpha}\phi|^2+|u|^{2k}|\snabla_{\s^2}\Lbar^{k}\Omega^{\alpha}\phi|^2\,d\omega du\Bigg]\\ 
&+\sum_{|\alpha|\leq n+1}\int_{{N}_{-u_0}} \mathbf{J}^T[\Omega^{\alpha}\psi]\cdot {L}\,d\omega dv+\sum_{|\alpha|\leq n+1} \int_{\mathcal{I}^+_{\leq -u_0}} (\Lbar \Omega^{\alpha} \phi)^2\,d\omega du,
\end{split}
\end{equation}
and
\begin{equation}
\label{eq:asymflatestbacktot2ho}
\begin{split}
\sum_{j=0}^2&\sum_{|\alpha|+k\leq n}\Bigg[\int_{\widetilde{\Sigma}\cap\{u_{r_{\mathcal{H}}}\leq u\leq -v_{-\infty}\}} (r-M)^{-2k-2+j}(L^{k+1}\Omega^{\alpha}T^{j}\phi)^2+(r-M)^{-2k+j}|\snabla_{\s^2}L^{k}\Omega^{\alpha}T^{j}\phi|^2 \\ 
&+(r-M)^{-2k}|\snabla_{\s^2}L^{k+1}\Omega^{\alpha}\phi|^2+(r-M)^{-2k+2}|\snabla_{\s^2}^2L^{k}\Omega^{\alpha}\phi|^2\\
&+(r-M)^{-2k-2+j}(\Lbar^{k+1}\Omega^{\alpha}T^{j}\phi)^2+(r-M)^{-2k+j}|\snabla_{\s^2}\Lbar^{k}\Omega^{\alpha}T^{j}\phi|^2\\
&+(r-M)^{-2k}|\snabla_{\s^2}\Lbar^{k+1}\Omega^{\alpha}\phi|^2+(r-M)^{-2k+2}|\snabla_{\s^2}^2\Lbar^{k}\Omega^{\alpha}\phi|^2\,d\omega dr_*\Bigg]\\
\sim_{c,C}&\: \sum_{j=0}^2 \sum_{|\alpha|+k\leq n}\Bigg[\int_{\underline{N}_{-v_0}} (r-M)^{-2k-2+j}(\Lbar^{k+1}\Omega^{\alpha}T^{j}\phi)^2+(r-M)^{-2k}|\snabla_{\s^2}\Lbar^{k+1}\Omega^{\alpha}\phi|^2\,d\omega du\\ 
&+\int_{\mathcal{H}^+_{\leq -v_0}} |v|^{2k+2-j}(L^{k+1}\Omega^{\alpha}T^{j}\phi)^2+|v|^{2k}|\snabla_{\s^2}L^{k+1}\Omega^{\alpha}\phi|^2+|v|^{2k}|\snabla_{\s^2} L^{k}\Omega^{\alpha} \phi|^2\,d\omega dv\Bigg] \\
&+ \sum_{|\alpha|\leq n+1}\int_{\underline{N}_{-v_0}} \mathbf{J}^T[\Omega^{\alpha}\psi]\cdot \underline{L}\,d\omega du+\sum_{|\alpha|\leq n+1} \int_{\mathcal{H}^+_{\leq -v_0}} (L\Omega^{\alpha} \phi)^2\,d\omega dv.
\end{split}
\end{equation}
\end{corollary}
\begin{proof}
Follows immediately after combining the results of Proposition \ref{prop:p2asympflatest} and Proposition \ref{prop:hop1asympflatest}.
\end{proof}

By commuting $\square_g$ additionally with $T$ and applying Lemma \ref{lm:tenconsv}, we arrive at energy estimates along $N_{u_0}$ and $\underline{N}_{v_0}$ (rather than $N_{-u_0}$ and $\underline{N}_{-v_0}$) with the same weights and number of derivatives as the energy fluxes that appear in Corollary \ref{cor:hoedecayv2} and Corollary \ref{cor:mainbackweesthov2}.

\begin{corollary}
\label{cor:hoestspinf}
Let $n\in \N_0$. Then there exists constants $c,C=c,C(M, r_{\mathcal{I}},r_{\mathcal{H}},u_0,v_0,n)>0$, such that
\begin{equation}
\label{eq:asymflatestbacktot1hov2}
\begin{split}
\sum_{j=0}^2&\sum_{2|\alpha|+2k+m\leq 2n}\Bigg[\int_{\widetilde{\Sigma}\cap\{v_{r_{\mathcal{I}}}\leq v\leq -u_{-\infty}\}} r^{2+2k-j}(L^{k+1}\Omega^{\alpha}T^{j+m}\phi)^2+r^{2k-j}|\snabla_{\s^2}L^{k}\Omega^{\alpha}T^{j+m}\phi|^2 \\ 
&+ r^{2k}|\snabla_{\s^2}L^{k+1}\Omega^{\alpha}T^m\phi|^2+r^{2k-2}|\snabla_{\s^2}^2L^{k}\Omega^{\alpha}T^m\phi|^2+r^{2k+2-j}(\Lbar^{k+1}\Omega^{\alpha}T^{j+m}\phi)^2 \\
&+r^{2k-j}|\snabla_{\s^2}\Lbar^{k}\Omega^{\alpha}T^{j+m}\phi|^2+r^{2k}|\snabla_{\s^2}\Lbar^{k+1}\Omega^{\alpha}T^{m}\phi|^2+r^{2k-2}|\snabla_{\s^2}^2\Lbar^{k}\Omega^{\alpha}T^{m}\phi|^2\,d\omega dr\Bigg]  \\
\sim_{c,C}&\: \sum_{j=0}^2\sum_{2|\alpha|+2k+m\leq 2n} \Bigg[\int_{N_{-u_0}} r^{2k+2-j}(L^{k+1}\Omega^{\alpha}T^{j+m}\phi)^2+r^{2k}|\snabla_{\s^2}L^{k+1}\Omega^{\alpha}T^{m}\phi|^2\,d\omega dv \\
&+\int_{\mathcal{I}^+_{\leq -u_0}} (1+|u|)^{2k+2-j}(\Lbar^{k+1}\Omega^{\alpha}T^{j+m}\phi)^2+ (1+|u|)^{2k}|\snabla_{\s^2}\Lbar^{k+1}\Omega^{\alpha}T^{m}\phi|^2+(1+|u|)^{2k}|\snabla_{\s^2}\Lbar^{k}\Omega^{\alpha}\phi|^2\,d\omega du\Bigg]\\ 
&+\sum_{2|\alpha|+m\leq 2n+2}\int_{{N}_{-u_0}} \mathbf{J}^T[\Omega^{\alpha}T^m\psi]\cdot {L}\,d\omega dv+\sum_{2|\alpha|+m\leq 2n+2} \int_{\mathcal{I}^+_{\leq -u_0}} (\Lbar \Omega^{\alpha} T^m\phi)^2\,d\omega du,
\end{split}
\end{equation}
and
\begin{equation}
\label{eq:asymflatestbacktot2hov2}
\begin{split}
\sum_{j=0}^2&\sum_{2|\alpha|+2k+m\leq 2n}\Bigg[\int_{\widetilde{\Sigma}\cap\{u_{r_{\mathcal{H}}}\leq u\leq -v_{-\infty}\}} (r-M)^{-2k-2+j}(L^{k+1}\Omega^{\alpha}T^{j+m}\phi)^2+(r-M)^{-2k+j}|\snabla_{\s^2}L^{k}\Omega^{\alpha}T^{j+m}\phi|^2 \\ 
&+(r-M)^{-2k}|\snabla_{\s^2}L^{k+1}\Omega^{\alpha}T^{m}\phi|^2+(r-M)^{-2k+2}|\snabla_{\s^2}^2L^{k}\Omega^{\alpha}T^{m}\phi|^2\\
&+(r-M)^{-2k-2+j}(\Lbar^{k+1}\Omega^{\alpha}T^{j+m}\phi)^2+(r-M)^{-2k+j}|\snabla_{\s^2}\Lbar^{k}\Omega^{\alpha}T^{j+m}\phi|^2\\
&+(r-M)^{-2k}|\snabla_{\s^2}\Lbar^{k+1}\Omega^{\alpha}T^{m}\phi|^2+(r-M)^{-2k+2}|\snabla_{\s^2}^2\Lbar^{k}\Omega^{\alpha}T^{m}\phi|^2\,d\omega dr_*\Bigg]\\ 
\sim_{c,C}&\: \sum_{j=0}^2 \sum_{2|\alpha|+2k+m\leq 2n}\Bigg[\int_{\underline{N}_{-v_0}} (r-M)^{-2k-2+j}(\Lbar^{k+1}\Omega^{\alpha}T^{j+m}\phi)^2+(r-M)^{-2k}|\snabla_{\s^2}\Lbar^{k+1}\Omega^{\alpha}T^{m}\phi|^2\,d\omega du\\ 
&+\int_{\mathcal{H}^+_{\leq -v_0}} (1+|v|)^{2k+2-j}(L^{k+1}\Omega^{\alpha}T^{j+m}\phi)^2+(1+|v|)^{2k}|\snabla_{\s^2}L^{k+1}\Omega^{\alpha}T^m\phi|^2+(1+|v|)^{2k}|\snabla_{\s^2} L^{k}\Omega^{\alpha} \phi|^2\,d\omega dv\Bigg] \\
&+ \sum_{2|\alpha|+m\leq 2n+2}\int_{\underline{N}_{-v_0}} \mathbf{J}^T[\Omega^{\alpha}T^m\psi]\cdot \underline{L}\,d\omega du+\sum_{2|\alpha|+m\leq 2n+2} \int_{\mathcal{H}^+_{\leq -v_0}} (L\Omega^{\alpha} T^m \phi)^2\,d\omega dv.
\end{split}
\end{equation}
\end{corollary}

\subsection{Construction of the scattering map}
\label{sec:consscatmat}
In this section we will construct the scattering map, which is a map from energy spaces on $\mathcal{I}^-$ and $\mathcal{H}^-$ to energy spaces on $\mathcal{I}^+$ and $\mathcal{H}^+$. First, we need to define what we mean by the solution to \eqref{eq:waveequation} in $J^+(\widetilde{\Sigma})$ arising from scattering data along $\mathcal{H}^+\cup \mathcal{I}^+$.

We introduce the following hypersurface: let $s<0$, then
\begin{equation*}
\widetilde{\Sigma}_s:=\widetilde{\Sigma}\cap\{s<r_*<|s|\}\cup N_{s}\cup \underline{N}_{s}.
\end{equation*}
\begin{definition}
\label{def:backwsolspacelikeinf}
Let $s_2<s_1<0$ and define the solutions $\psi_{s_i}: D^+(\widetilde{\Sigma}_{s_i})\to\R$ as the unique smooth solutions to \eqref{eq:waveequation} corresponding to scattering data $(\underline{\Phi},\Phi)\in C_{c}^{\infty}(\mathcal{H}^+)\oplus C_{c}^{\infty}(\mathcal{I}^+)$ in accordance with Proposition \ref{prop:mainpropbackdef}. Then, by uniqueness,
\begin{equation*}
\psi_{s_2}|_{D^+(\widetilde{\Sigma}_{s_1})}=\psi_{s_1},
\end{equation*}
so we can define the function $\psi: D^+(\widetilde{\Sigma})\to \R$ as follows: let $p\in D^+(\widetilde{\Sigma})$, then there exists an $s_*>0$ such that $p\in  D^+(\widetilde{\Sigma}_{s_*})$. Let
\begin{equation*}
\psi(p)=\psi_{s_*}(p).
\end{equation*}
It follows immediately that $\psi$ is a uniquely determined smooth solution to \eqref{eq:waveequation}, such that $\lim_{v\to \infty} r\psi(u,v,\theta,\varphi)=\Phi(u,\theta,\varphi)$ and $M\psi|_{\mathcal{H}^+}=\underline{\Phi}$.
\end{definition}

\begin{proposition}
Let $(\Psi,\Psi')\in  (C_{c}^{\infty}(\widetilde{\Sigma}))^2$. Then the corresponding solution $\psi$ to \eqref{eq:waveequation} satisfies
\begin{equation*}
(r\cdot\psi|_{\mathcal{H}^{\pm}},r\cdot\psi|_{\mathcal{I}^{\pm}})\in \mathcal{E}^T_{\mathcal{H}^{\pm}}\oplus \mathcal{E}^T_{\mathcal{I}^{\pm}}.
\end{equation*}
and furthermore, the following identity holds
\begin{equation*}
||r\cdot\psi|_{\mathcal{H}^{\pm} }||^2_{ \mathcal{E}^T_{\mathcal{H}^{\pm} }}+||r\cdot\psi|_{\mathcal{I}^{\pm}}||^2_{\mathcal{E}^T_{\mathcal{I}^{\pm}}}= ||(\Psi,\Psi')||_{\mathcal{E}^T_{\widetilde{\Sigma}} }^2.
\end{equation*}
\end{proposition}
\begin{proof}
Follows from Lemma \ref{lm:tenconsv} and Proposition \ref{prop:fowtennormbound} (combined with an analogue of Proposition \ref{prop:fowtennormbound} in the past-direction, making use of the time-symmetry of the spacetime).
\end{proof}
\begin{definition}
Define the evolution maps $\widetilde{\mathscr{F}}_{\pm}: (C_{c}^{\infty}(\widetilde{\Sigma}))^2 \to \mathcal{E}^T_{\mathcal{H}^{\pm}}\oplus \mathcal{E}^T_{\mathcal{I}^{\pm}}$ as the following linear operator:
\begin{equation*}
\widetilde{\mathscr{F}}_{\pm}(\Psi,\Psi')=(r\cdot\psi|_{\mathcal{H}^{\pm}},r\cdot\psi|_{\mathcal{I}^{\pm}}),
\end{equation*}
where $\psi$ is the unique solution to \eqref{eq:waveequation} with $(\psi|_{\widetilde{\Sigma}},\mathbf{n}_{\widetilde{\Sigma}}\psi|_{\widetilde{\Sigma}})=(\Psi,\Psi')$. Then $\widetilde{\mathscr{F}}_{\pm}$ extends uniquely to a linear bounded  operator, also denoted $\widetilde{\mathscr{F}}_{\pm}$:
\begin{equation*}
\widetilde{\mathscr{F}}_{\pm}: \mathcal{E}^T_{\widetilde{\Sigma}}\to \mathcal{E}^T_{\mathcal{H}^{\pm}}\oplus \mathcal{E}^T_{\mathcal{I}^{\pm}}.
\end{equation*}
\end{definition}

\begin{proposition}
\label{prop:towardsscatmat1}
Let $n\in \N_0$. Then for all $n\in \N_0$
\begin{equation}
\label{eq:fowincltilde}
\widetilde{\mathscr{F}}_{\pm}(C_{c}^{\infty}(\widetilde{\Sigma}))^2)\subseteq \mathcal{E}_{n; \mathcal{H}^{\pm}}\oplus \mathcal{E}_{n; \mathcal{I}^{\pm}}, 
\end{equation}
and $\widetilde{\mathscr{F}}_{\pm}$ can uniquely be extended as as the following bounded linear operator
\begin{equation*}
\widetilde{\mathscr{F}}_{n; \pm}: \mathcal{E}_{n; \widetilde{\Sigma}}\to \mathcal{E}_{n; \mathcal{H}^{\pm}}\oplus \mathcal{E}_{n; \mathcal{I}^{\pm}}.
\end{equation*}
We moreover have that $\widetilde{\mathscr{F}}_{n; \pm}=\widetilde{\mathscr{F}}_{\pm}|_{  \mathcal{E}_{n; \widetilde{\Sigma}}}$.
\end{proposition}
\begin{proof}
Without loss of generality, we restrict our considerations to $\widetilde{\mathscr{F}}_{+}$. We choose $\Sigma_0$ so that
\begin{equation*}
\Sigma_0\cap \{r_{\mathcal{H}}\leq r\leq r_{\mathcal{I}}\}=\widetilde{\Sigma}\cap  \{r_{\mathcal{H}}\leq r\leq r_{\mathcal{I}}\}.
\end{equation*}
Let $\psi$ denote the solution to \eqref{eq:waveequation} corresponding to initial data $(\Psi,\Psi')\in C_{c}^{\infty}(\widetilde{\Sigma}))^2$. We apply Corollary \ref{cor:hoestspinf} to conclude that
\begin{equation*}
||(\psi|_{\Sigma_0}, \mathbf{n}_{\Sigma_0}\psi|_{\Sigma_0 \cap \{r_{\mathcal{H}}\leq r\leq r_{\mathcal{I}}\}})||_{ \mathcal{E}_{n; \Sigma_0}}\leq C||(\Psi,\Psi')||_{\mathcal{E}_{n; \widetilde{\Sigma}}}.
\end{equation*}
We then apply the bounded operator $\mathscr{F}_n$ from Corollary \ref{prop:fowbound} to arrive at \eqref{eq:fowincltilde}. The extension property follows immediately from the uniform boundedness of $\widetilde{\mathscr{F}}_{+}$ with respect to the desired norms.
\end{proof}
\begin{proposition}
Let $(\underline{\Phi},\Phi)\in C_{c}^{\infty}(\mathcal{H}^{\pm})\oplus C_{c}^{\infty}(\mathcal{I}^{\pm})$. Then the corresponding solution $\psi$ according to Definition \ref{def:backwsolspacelikeinf} satisfies $\psi|_{\widetilde{\Sigma}}(r,\theta,\varphi)\to 0$ as $r\to \infty$ and $r\downarrow M$ and
\begin{equation*}
||(\psi|_{\widetilde{\Sigma}},\mathbf{n}_{\widetilde{\Sigma}} \psi|_{\widetilde{\Sigma}})||_{\mathcal{E}^T_{\widetilde{\Sigma}} }^2=||\underline{\Phi}||^2_{ \mathcal{E}^T_{\mathcal{H}^{\pm} }}+||\Phi||^2_{\mathcal{E}^T_{\mathcal{I}^{\pm}}}.
\end{equation*}
\end{proposition}
\begin{proof}
By applying the fundamental theorem of calculus, we have that for suitably large $r_*>0$
\begin{equation*}
\psi^2(0,r_*,\theta,\varphi)\leq \frac{1}{r} \int_{N_{-r_*}} \mathbf{T}(\partial_t,L)\, r^2d\omega dv\leq \int_{\mathcal{H}^+} (L\phi)^2\,d\omega dv+ \int_{\mathcal{I}^+} (\Lbar \phi)^2\,d\omega du.
\end{equation*}
so $\psi|_{\widetilde{\Sigma}}(r,\theta,\varphi)\to 0$ as $r\to \infty$. By considering $r_*<0$ with $|r_*|$ suitably large, we can conclude analogously that $\psi|_{\widetilde{\Sigma}}(r,\theta,\varphi)\to 0$ as $r\to \infty$ and $r\downarrow M$.

The energy conservation statement simply follows from applying Lemma \ref{lm:tenconsv}. 
\end{proof}

\begin{definition}
Define the backwards evolution maps $\widetilde{\mathscr{B}}_{\pm}:C_{c}^{\infty}(\mathcal{H}^{\pm})\oplus C_{c}^{\infty}(\mathcal{I}^{\pm})\to \mathcal{E}^T_{\widetilde{\Sigma}}$ as the following linear operator:
\begin{equation*}
\widetilde{\mathscr{B}}_{\pm}(\underline{\Phi},\Phi)=(\psi|_{\widetilde{\Sigma}},\mathbf{n}_{\widetilde{\Sigma}} \psi|_{\widetilde{\Sigma}}),
\end{equation*}
where $\psi$ is the corresponding unique solution to \eqref{eq:waveequation} as defined in Definition \ref{def:backwsolspacelikeinf}. Then $\widetilde{\mathscr{B}}_{\pm}$ extends uniquely to a linear bounded  operator, also denoted $\widetilde{\mathscr{B}}_{\pm}$:
\begin{equation*}
\widetilde{\mathscr{B}}_{\pm}: \mathcal{E}^T_{\mathcal{H}^{\pm}}\oplus \mathcal{E}^T_{\mathcal{I}^{\pm}}\to \mathcal{E}^T_{\widetilde{\Sigma}}.
\end{equation*}
\end{definition}

\begin{proposition}
\label{prop:towardsscatmat2}
The linear operator $\widetilde{\mathscr{F}}_{\pm}: \mathcal{E}^T_{\widetilde{\Sigma}} \to \mathcal{E}^T_{\mathcal{H}^{\pm}}\oplus \mathcal{E}^T_{\mathcal{I}^{\pm}}$ is bijective with $\widetilde{\mathscr{B}}_{\pm}=\widetilde{\mathscr{F}}_{\pm}^{-1}$.
\end{proposition}
\begin{proof}
Follows by the same arguments as in the proof of Proposition \ref{cor:bijectivity}.
\end{proof}

\begin{proposition}
\label{prop:towardsscatmat3}
Let $n\in \N_0$. Then for all $n\in \N_0$
\begin{equation}
\label{eq:bacwincltilde}
\widetilde{\mathscr{B}}_{\pm}(C_{c}^{\infty}(\mathcal{H}^{\pm})\oplus C_{c}^{\infty}(\mathcal{I}^{\pm}))\subseteq  \mathcal{E}_{n; \widetilde{\Sigma}}, 
\end{equation}
and $\widetilde{\mathscr{B}}_{\pm}$ can uniquely be extended as as the following bounded linear operator
\begin{equation*}
\widetilde{\mathscr{B}}_{n; \pm}:  \mathcal{E}_{n; \mathcal{H}^{\pm}}\oplus \mathcal{E}_{n; \mathcal{I}^{\pm}} \to \mathcal{E}_{n; \widetilde{\Sigma}}.
\end{equation*}
We moreover have that $\widetilde{\mathscr{B}}_{n; \pm}=\widetilde{\mathscr{B}}_{\pm}|_{  \mathcal{E}_{n; \widetilde{\Sigma}}}$ and $\widetilde{\mathscr{B}}_{n; \pm}= \widetilde{\mathscr{F}}_{n;\pm}^{-1}$.
\end{proposition}
\begin{proof}
Without loss of generality, we consider $\widetilde{\mathscr{B}}_{+}$. We choose $\Sigma_0$ so that
\begin{equation*}
\Sigma_0\cap \{r_{\mathcal{H}}\leq r\leq r_{\mathcal{I}}\}=\widetilde{\Sigma}\cap  \{r_{\mathcal{H}}\leq r\leq r_{\mathcal{I}}\}.
\end{equation*}
We apply $\mathscr{B}_n$ from Proposition \ref{prop:backwmapho} to conclude that the $\psi$ corresponding to initial data $(\underline{\Phi},\Phi)\in C_{c}^{\infty}(\mathcal{H}^{+})\oplus C_{c}^{\infty}(\mathcal{I}^{+})$ satisfies
\begin{equation*}
||(\psi|_{\Sigma_0}, \mathbf{n}_{\Sigma_0}\psi|_{\Sigma_0 \cap \{r_{\mathcal{H}}\leq r\leq r_{\mathcal{I}}\}})||_{ \mathcal{E}_{n; \Sigma_0}}^2\leq C( ||\underline{\Phi}||^2_{ \mathcal{E}_{n; \mathcal{H}^{+}}}+||\Phi||^2_{ \mathcal{E}_{n; \mathcal{I}^{+}}}).
\end{equation*}
Hence, we can apply Corollary \ref{cor:hoestspinf} to obtain \eqref{eq:bacwincltilde}. The extension property then follows from the uniformity of all estimates involved. The inversion follows by repeating the arguments in the proof of Proposition \ref{cor:bijectivity}.
\end{proof}

\begin{definition}
\label{def:scatmat}
We define the scattering matrix $\mathscr{S}: \mathcal{E}^T_{\mathcal{H}^{-}} \oplus \mathcal{E}^T_{\mathcal{I}^{-}}\to \mathcal{E}^T_{\mathcal{H}^{+}} \oplus \mathcal{E}^T_{\mathcal{I}^{+}}$ as the following bounded linear operator:
\begin{equation*}
\mathscr{S}:= \widetilde{\mathscr{F}}_{+} \circ \widetilde{\mathscr{B}}_{-}.
\end{equation*}
Let $n\in \N_0$. Then we define the restricted scattering matrix $\mathscr{S}_n:  \mathcal{E}_{n; \mathcal{H}^{-}}\oplus \mathcal{E}_{n; \mathcal{I}^{-}}\to  \mathcal{E}_{n; \mathcal{H}^{+}}\oplus \mathcal{E}_{n; \mathcal{I}^{+}}$ as the following bounded linear operator:
\begin{equation*}
\mathscr{S}_n:= \widetilde{\mathscr{F}}_{n;+} \circ \widetilde{\mathscr{B}}_{n; -}.
\end{equation*}
\end{definition}

\section{Scattering in the black hole interior}
\label{sec:regcauchyhor}
In this section, we obtain some additional estimates in the black hole interior, which allow use to construct a non-degenerate interior scattering map.
\begin{proposition}
\label{prop:interiorestimates}
Let $u_{\rm int}<0$ with $|u_{\rm int}|$ suitably large. Then there exist constants $c,C=c,C(M,u_0,v_0)>0$ such that
\begin{equation}
\label{eq:intest}
\begin{split}
\int_{{H}^{\rm int}_{u_{\rm int}}}& v^2(L \phi)^2+ D |\snabla_{\s^2}\phi|^2\,dv + \int_{\mathcal{CH}^+_{\leq u_{\rm int}}} u^2(\Lbar \phi)^2 + |\snabla_{\s^2}\phi|^2 \, d\omega du\\
 \sim_{c,C}&\: \int_{\underline{N}^{\rm int}_{v_0}} u^2 (\Lbar \phi)^2+ D |\snabla_{\s^2}\phi|^2\,d\omega du+\int_{\mathcal{H}^+_{\geq v_0}} v^2(L \phi)^2+ |\snabla_{\s^2}\phi|^2\,dv.
 \end{split}
\end{equation}
\end{proposition}
\begin{proof}
Observe first that \eqref{eq:keyidasymflat2} and \eqref{eq:auxestasympflathor} hold also in $\mathcal{M}^{\rm int}$, with respect to the Eddington--Finkelstein double-null coordinates $(u,v)$. Hence, we can estimate, for $\psi$ arising from data along $\mathcal{H}^+_{\geq v_0}$ and $\underline{N}^{\rm int}_{v_0}$:
\begin{equation}
\label{eq:auxestasympflathorint}
\begin{split}
\sup_{v} &\int_{\underline{N}^{\rm int}_v} u^2(\Lbar \phi)^2+\frac{1}{4}v^2r^{-2}D|\snabla_{\s^2}\phi|^2\,d\omega du+ \sup_{u} \int_{H^{\rm int}_u} v^2(L \phi)^2+\frac{1}{4}u^2r^{-2}D|\snabla_{\s^2}\phi|^2\,d\omega dv\\
\leq&\: \int_{\underline{N}^{\rm int}_{v_0}}u^2(\Lbar \phi)^2+\frac{1}{4}v^2r^{-2}D|\snabla_{\s^2}\phi|^2\,d\omega du+\int_{\mathcal{H}^+_{\geq v_{0}}}  v^2(L \phi)^2+\frac{1}{4}u^2r^{-2}D|\snabla_{\s^2}\phi|^2\,d\omega dv\\
&+Cv_0^{-\epsilon}\int_{-\infty }^{u_0} \int_{v_0}^{\infty}(r-M)^{3}|\phi|\cdot (u^2|\Lbar \phi|+v^2|L \phi|)\,d\omega du dv\\
&+Cv_0^{-\epsilon}\int_{-\infty }^{u_0} \int_{v_0}^{\infty} \log((r-M)^{-1}) D|\snabla_{\s^2}\phi|^2 \,d\omega du dv.
\end{split}
\end{equation}

Using that $(r-M)^{-1}\sim  v+|u|$ in $\mathcal{M}^{\rm int}\cap D^+(\Sigma_0\cup \underline{N}^{\rm int}_{v_0})$, we can absorb the last two integrals on the right-hand side into the left-hand side for $|u_0|$ suitably large, in order to obtain
\begin{equation*}
\begin{split}
\sup_{v} &\int_{\underline{N}^{\rm int}_v} u^2(\Lbar \phi)^2+\frac{1}{4}v^2r^{-2}D|\snabla_{\s^2}\phi|^2\,d\omega du+ \sup_{u} \int_{H^{\rm int}_u} v^2(L \phi)^2+\frac{1}{4}u^2r^{-2}D|\snabla_{\s^2}\phi|^2\,d\omega dv\\
\leq&\: C\int_{\underline{N}^{\rm int}_{v_0}}u^2(\Lbar \phi)^2+\frac{1}{4}v^2r^{-2}D|\snabla_{\s^2}\phi|^2\,d\omega du+C\int_{\mathcal{H}^+_{\geq v_{0}}}  v^2(L \phi)^2+\frac{1}{4}u^2r^{-2}D|\snabla_{\s^2}\phi|^2\,d\omega dv.
\end{split}
\end{equation*}
From the above estimate it moreover follows that for any increasing sequence $\{v_k\}$ we can bound for any $n>m\geq 0$:
\begin{equation*}
\begin{split}
\Bigg|\int_{\underline{N}^{\rm int}_{v_n}}& u^2(\Lbar \phi)^2+\frac{1}{4}v^2r^{-2}D|\snabla_{\s^2}\phi|^2\,d\omega du-\int_{\underline{N}^{\rm int}_{v_m}} u^2(\Lbar \phi)^2+\frac{1}{4}v^2r^{-2}D|\snabla_{\s^2}\phi|^2\,d\omega du \Bigg|\\
\leq&\: C(v_m)^{-\epsilon}\left[ \int_{\underline{N}^{\rm int}_{v_0}}u^2(\Lbar \phi)^2+\frac{1}{4}v^2r^{-2}D|\snabla_{\s^2}\phi|^2\,d\omega du+\int_{\mathcal{H}^+_{\geq v_{0}}}  v^2(L \phi)^2+\frac{1}{4}u^2r^{-2}D|\snabla_{\s^2}\phi|^2\,d\omega dv\right]\\
&+\int_{\mathcal{H}^+_{\geq v_{m}}}  v^2(L \phi)^2+\frac{1}{4}u^2r^{-2}D|\snabla_{\s^2}\phi|^2\,d\omega dv.
\end{split}
\end{equation*}
So we can conclude that
\begin{equation*}
 \int_{\underline{N}^{\rm int}_{v_k}} u^2(\Lbar \phi)^2+\frac{1}{4}v^2r^{-2}D|\snabla_{\s^2}\phi|^2\,d\omega du
\end{equation*}
is a Cauchy sequence, so it must converge as $k\to \infty$. Furthermore, the limit is independent of the choice of sequence. Hence, 
\begin{equation*}
 \int_{\mathcal{CH}^+_{\leq u_{\rm int}}} u^2(\Lbar \phi)^2 + \frac{1}{4}v^2r^{-2}D|\snabla_{\s^2}\phi|^2 \, d\omega du=: \lim_{v\to \infty }\int_{\underline{N}^{\rm int}_{v}} u^2(\Lbar \phi)^2+\frac{1}{4}v^2r^{-2}D|\snabla_{\s^2}\phi|^2\,d\omega du
\end{equation*}
is well-defined.

Similarly, if we take $\psi$ to arise from data along $\mathcal{CH}^+_{\leq u_{\rm int}}$ and $H_{u_{\rm int}}^{ \rm int}$, we can apply \eqref{eq:keyidasymflat2} and \eqref{eq:auxestasympflathor} to show that
\begin{equation*}
\begin{split}
\sup_{v} &\int_{\underline{N}^{\rm int}_v} u^2(\Lbar \phi)^2+\frac{1}{4}v^2r^{-2}D|\snabla_{\s^2}\phi|^2\,d\omega du+ \sup_{u} \int_{H^{\rm int}_u} v^2(L \phi)^2+\frac{1}{4}u^2r^{-2}D|\snabla_{\s^2}\phi|^2\,d\omega dv\\
\leq&\: \frac{1}{c}\int_{{H}^{\rm int}_{u_{\rm int}}} v^2(L \phi)^2+ D |\snabla_{\s^2}\phi|^2\,dv+\frac{1}{c}\int_{\mathcal{CH}^+_{\leq u_{\rm int}}} u^2(\Lbar \phi)^2 + |\snabla_{\s^2}\phi|^2 \, d\omega du.
\end{split}
\end{equation*}
and it follows analogously that 
\begin{equation*}
 \int_{\mathcal{H}^+_{\geq v_0}} v^2(L \phi)^2 + \frac{1}{4}u^2r^{-2}D|\snabla_{\s^2}\phi|^2 \, d\omega du=: \lim_{u\to -\infty }\int_{{N}^{\rm int}_{u}} v^2(L \phi)^2+\frac{1}{4}u^2r^{-2}D|\snabla_{\s^2}\phi|^2\,d\omega dv
\end{equation*}
is well-defined.

The estimate \eqref{eq:intest} then follows by combining the above estimates.
\end{proof}

\begin{proposition}
\label{prop:intscat}
Let $u_{\rm int}<0$ with $|u_{\rm int}|$ suitably large. Let $\mathscr{S}^{\rm int}: C_{c}^{\infty}(\mathcal{H}^+_{\geq v_0})\times C^{\infty}(\underline{N}_{v_0}^{\rm int})\to \mathcal{E}_{\mathcal{CH}^+_{\leq u_{\rm int}}}\oplus \mathcal{E}_{H^{\rm int}_{u_{\rm int}}}$ be defined as follows:
\begin{equation*}
\mathscr{S}^{\rm int}(r\psi|_{\mathcal{H}^+_{\geq v_0}}, r\psi|_{\underline{N}_{v_0}^{\rm int}})=(r\psi|_{\mathcal{CH}^+_{\leq u_{\rm int}}},r\psi_{H^{\rm int}_{u_{\rm int}}}).
\end{equation*}
Then $\mathscr{S}^{\rm int}$ extends uniquely as a bijective, bounded linear operator:
\begin{equation*}
\mathscr{S}^{\rm int}: \mathcal{E}_{\mathcal{H}^+_{\geq v_0}}\oplus \mathcal{E}_{\underline{N}_{v_0}^{\rm int}}\to \mathcal{E}_{\mathcal{CH}^+_{\leq u_{\rm int}}}\oplus \mathcal{E}_{H^{\rm int}_{u_{\rm int}}}.
\end{equation*}
\end{proposition}
\begin{proof}
The construction of $\mathscr{S}^{\rm int}$ and its inverse, on a domain of smooth, compactly supported functions, follow immediately from the estimates in the proof of Proposition \ref{prop:interiorestimates}, where $r\psi|_{\mathcal{CH}^+_{\leq u_{\rm int}}}$ (in the forwards direction) and $r\psi|_{\mathcal{H}^+_{\geq v_0}}$ (in the backwards direction) can understood in a limiting sense, as in Proposition \ref{prop:interiorestimates}, and it follows that $r\psi|_{\mathcal{CH}^+_{\leq u_{\rm int}}}\in \mathcal{E}_{\mathcal{CH}^+_{\leq u_{\rm int}}}$ and $r\psi|_{\mathcal{H}^+_{\geq v_0}} \in \mathcal{E}_{\mathcal{H}^+_{\geq v_0}}$ by the fundamental theorem of calculus:
\begin{align*}
\left(\int_{\s^2}\phi^2 \,d\omega\right)(u,v)\leq |u|^{-1}\int_{\underline{N}^{\rm int}_v} u'^2(\Lbar \phi)^2\,d\omega du'\quad \textnormal{in the forwards direction and}\\
\left(\int_{\s^2}\phi^2 \,d\omega\right)(u,v)\leq |v|^{-1}\int_{H^{\rm int}_u} v'^2(L\phi)^2\,d\omega dv' \quad \textnormal{in the backwards direction},
\end{align*}
and (a straightforward variation of) 2.) of Lemma \ref{lm:completionhorinf}. The extendibility follows moreover from the uniformity of the estimates in Proposition \ref{prop:interiorestimates}.
\end{proof}

\begin{corollary}
\label{cor:regint}
Let $u_{\rm int}<0$ with $|u_{\rm int}|$ suitably large. Let $u_1<u_{\rm int}$, with $|u_1|$ arbitrarily large. Then there exist a constant $C=C(M,u_{\rm int}, u_1,v_0)>0$ such that we can estimate with respect to $(u,r)$ coordinates:
\begin{equation}
\label{eq:intest1}
\begin{split}
\sup_{u_1\leq u\leq u_{\rm int}}&\int_{{H}^{\rm int}_{u_{\rm int}}} (\partial_r\phi)^2+ |\snabla_{\s^2}\phi|^2\,d\omega dr+\sup_{v_0\leq v\leq \infty}\int_{\underline{N}^{\rm int}_v\cap \{u_1\leq u\leq u_{\rm int}\}} (\Lbar \phi)^2+|\snabla_{\s^2}\phi|^2\,d\omega du\\
\leq&\: C \Bigg[\int_{\underline{N}^{\rm int}_{v_0}} u^2 (\partial_u \phi)^2+ D |\snabla_{\s^2} \phi|^2\,d\omega du+\int_{\mathcal{H}^+_{\geq v_0}} v^2(\partial_v \phi)^2+ |\snabla_{\s^2} \phi|^2\,d\omega dv\Bigg]
\end{split}
\end{equation}
Furthermore,
\begin{equation}
\label{eq:intest2}
\begin{split}
\sup_{u_1\leq u\leq u_{\rm int}}&\int_{{H}^{\rm int}_{u_{\rm int}}} (\partial_r\phi)^2+ |\snabla_{\s^2}\phi|^2\,d\omega dr+\sup_{v_0\leq v\leq \infty}\int_{\underline{N}^{\rm int}_v\cap \{u_1\leq u\leq u_{\rm int}\}} (\Lbar \phi)^2+|\snabla_{\s^2}\phi|^2\,d\omega du\\
\leq&\: C \Bigg[  \int_{\underline{N}^{\rm int}_{v_0}} u^2 (\partial_u  \phi)^2+ D |\snabla_{\s^2}\phi|^2\,d\omega du+||(\psi|_{\Sigma_0},\mathbf{n}_{0} \psi|_{\Sigma_0})||_{\mathcal{E}_{\Sigma_0}}^2\Bigg]
\end{split}
\end{equation}
\end{corollary}
\begin{proof}
We use that $\partial_v r|_{{H}^{\rm int}_{u_{\rm int}}}\sim v^{-2}$, together with
\begin{equation*}
\left(\int_{\s^2}\phi^2 \,d\omega\right)(u,v)\leq |u|^{-1}\int_{\underline{N}^{\rm int}_v} u'^2(\Lbar \phi)^2\,d\omega du'
\end{equation*}
and we apply the estimates of Proposition \ref{prop:interiorestimates}, replacing $\psi$ with $T^j \psi$, $j=0,1$, to arrive at \eqref{eq:intest1}. We obtain \eqref{eq:intest2} by appealing additionally to Corollary \ref{cor:fowdecay}.
\end{proof}

\begin{remark}
One can easily extend the estimate in Corollary \ref{cor:regint} to smaller values of $|u_{\rm int}|$ (provided $r>r_{\rm min}>0$ in the spacetime region under consideration), by applying a standard Gr\"onwall inequality.
\end{remark}

\section{Application 1 : Regularity at the event horizon and null infinity}
\label{sec:appreg}
As an application of the maps $\mathscr{B}_n$ constructed in Proposition \ref{prop:backwmapho}, we can show that we can associate arbitrarily regular solutions to suitably polynomially decaying scattering data along $\mathcal{H}^+$ and $\mathcal{I}^+$. First of all, we will show that by considering $T^k\psi$, rather than $\psi$, we obtain higher-regularity near $\mathcal{H}^+$ and $\mathcal{I}^+$.

Before we address these regularity properties, we will relate the differential operators $(r^2L)^k$ and $((r-M)^{-2}\underline{L})^k$ to $r^{2k}L^{k}$ and $(r-M)^{-2k}\underline{L}^k$.
\begin{lemma}
\label{lm:equationreghoderphi}
Let $\psi$ be a solution to \eqref{eq:waveequation}. Then we can express for all $k\in \N_0$:
\begin{equation*}
r^2L\underline{L}((r^2L)^k\phi)=O(r)L((r^2L)^k\phi)+\sum_{j=0}^k O(1)(r^2L)^j\phi+O(1)(r^2L)^j\slashed{\Delta}_{\s^2}\phi.
\end{equation*}
and
\begin{equation*}
(r-M)^{-2}L\underline{L}(((r-M)^{-2}\underline{L})^k\phi)=O((r-M)^{-1})\underline{L}((r^2\underline{L})^k\phi)+\sum_{j=0}^k O(1)(r^2\underline{L})^j\phi+O(1)(r^2\underline{L})^j\slashed{\Delta}_{\s^2}\phi.
\end{equation*}
\end{lemma}
\begin{proof}
The identities can be obtained inductively by applying \eqref{eq:eqphi} and commuting $L \underline{L}$ with $r^2L$ and $r^2\underline{L}$. See Lemma 6.1 in \cite{paper4} for more details.
\end{proof}
\begin{proposition}
\label{prop:regderphi}
Let $\psi$ be a solution to \eqref{eq:waveequation}. For all $k\geq 1$ we have that:
\begin{equation}
\label{eq:regderphi1}
(r^2LT)^{k}(\phi)=\sum_{\substack{m+i+2s\leq 2k\\ 0\leq s\leq k, m\geq 0 \\ 0\leq i\leq  k-1}} \sum_{j=0}^{m} O(r^{m-j})L^{m-j}\slashed{\Delta}_{\s^2}^sT^i\phi.
\end{equation}
and
\begin{equation}
\label{eq:regderphi2}
((r-M)^{-2}\underline{L}T)^{k}(\phi)=\sum_{\substack{m+i+2s\leq 2k\\ 0\leq s\leq k, m\geq 0 \\ 0\leq i\leq k-1}} \sum_{j=0}^{m} O((r-M)^{-m+j})\underline{L}^{m-j}\slashed{\Delta}_{\s^2}^sT^i\phi.
\end{equation}
\end{proposition}
\begin{proof}
We will do a proof by induction. We have that  \eqref{eq:regderphi1} and \eqref{eq:regderphi2} hold for $k=1$. Suppose \eqref{eq:regderphi1} and \eqref{eq:regderphi2} hold for $1\leq k\leq n$. We will show below that \eqref{eq:regderphi1} and \eqref{eq:regderphi2}  also hold for $k=n+1$.

Writing $T=L+\underline{L}$, we can express
\begin{equation*}
(r^2LT)^{n+1}\phi=r^2L^2(r^2LT)^n\phi+r^2L\underline{L}((r^2LT)^n\phi)
\end{equation*}
and apply Lemma \ref{lm:equationreghoderphi} to obtain
\begin{equation*}
(r^2LT)^{n+1}\phi=r^2L^2(r^2LT)^n\phi+O(r)L((r^2LT)^n\phi)+\sum_{k=0}^n O(1)(r^2LT)^kT^{n-k}\phi+O(1)(r^2LT)^k\slashed{\Delta}_{\s^2}T^{n-k}\phi.
\end{equation*}
Now, we take apply \eqref{eq:regderphi1} for $0\leq k\leq n$ (taking appropriate derivatives on both sides of the equation) to obtain:
\begin{equation*}
\begin{split}
(r^2LT)^{n+1}\phi=&\:\sum_{\substack{m+i+2s\leq 2n+2\\ 0\leq s\leq n, m\geq 0 \\ 0\leq i\leq n-1}} \sum_{j=0}^{m} O(r^{m-j})L^{m-j}\slashed{\Delta}_{\s^2}^sT^i\phi\\
&+\sum_{k=1}^n\sum_{\substack{m+i+2s\leq 2k+2\\ 0\leq s\leq k+1, m\geq 0 \\ n-k\leq i\leq k-1+(n-k)}} \sum_{j=0}^{m} O(r^{m-j})L^{m-j}\slashed{\Delta}_{\s^2}^sT^i\phi\\
&+O(1)T^n\phi+O(1)\slashed{\Delta}_{\s^2}T^n\phi\\
=&\:\sum_{\substack{m+i+2s\leq 2(n+1)\\ 0\leq s\leq n+1, m\geq 0 \\ 0\leq i\leq  n}} \sum_{j=0}^{m} O(r^{m-j})L^{m-j}\slashed{\Delta}_{\s^2}^sT^i\phi.
\end{split}
\end{equation*}
We apply an analogous argument, using that $\underline{L}(O((r-M)^{p})=O((r-M)^{p+1})$, to also conclude that \eqref{eq:regderphi2} holds for $k=n+1$.
\end{proof}
\begin{proposition}
\label{prop:regtimeder}
Let $n\in \N_0$. Suppose that $(\Phi,\underline{\Phi})\in \mathcal{E}_{2n; \mathcal{I}^+_{\geq u_0}}\oplus  \mathcal{E}_{2n; \mathcal{H}^+_{\geq v_0}}$. Then we have that the corresponding solution $\psi$ to \eqref{eq:waveequation} satisfies
\begin{equation*}
T^n(r\psi)\in W_{\rm loc}^{n+1,2}(\widehat{\mathcal{R}}).
\end{equation*}
\end{proposition}
\begin{proof}
By Proposition \ref{prop:backwmapho}, we have that $\mathscr{B}_{2n}(\Phi,\underline{\Phi})\in \mathcal{E}_{2n; \Sigma_0}$. Hence,
\begin{equation*}
\begin{split}
\sum_{j=0}^1&\sum_{m+2k+2|\alpha|\leq 4n}\int_{{N}_{u_0}} r^{2+2k-j} (L^{k+1} T^{m+j}\Omega^{\alpha} \phi)^2\,d\omega dv+ \int_{{\underline{N}}_{v_0}} (r-M)^{-2-2k+j}(\underline{L}^{k+1}T^{m+j}\Omega^{\alpha}\phi)^2\,d\omega du\\
&+\sum_{\substack{m+2|\alpha|\leq 4n+2\\|\alpha|\leq 2n}} \int_{\Sigma_{0}} {\mathbf{J}}^T[T^m \Omega^{\alpha}\psi]\cdot \mathbf{n}_{0}\, d\mu_{0}< \infty.
 \end{split}
\end{equation*}
We subsequently apply Proposition \ref{prop:regderphi}  to obtain in $(v,r)$ coordinates:
\begin{equation*}
\begin{split}
\sum_{k+m+\alpha\leq n}&\int_{{N}_{v_0}} r^2(L(r^2L)^{k}T^m \Omega^{\alpha}(T^{n}\phi))^2\,d\omega dv+ \int_{{\underline{N}}_{u_0}} (\partial_r^{k}T^m \Omega^{\alpha}(T^{n}\phi))^2\,d\omega dr\\
&+\sum_{\substack{m+2|\alpha|\leq 3n+2\\|\alpha|\leq 2n}}\int_{\Sigma_{0}} {\mathbf{J}}^T[T^m \Omega^{\alpha}(T^n\psi)]\cdot \mathbf{n}_{0}\, d\mu_{0}<\infty.
\end{split}
\end{equation*}
We conclude the proof by integrating the above norm locally in $\tau$.
\end{proof}

\begin{definition}
\label{def:timeinv}
Consider $(\underline{\Phi},\Phi) \in C^{\infty}(\mathcal{H}^+_{\geq v_0}) \oplus C^{\infty}(\mathcal{I}^+_{\geq u_0})$ such that
\begin{align*}
\int_{v_0}^{\infty} |\underline{\Phi}|\,dv<&\:\infty,\\
\int_{u_0}^{\infty} |{\Phi}|\,du<&\:\infty.
\end{align*}
Then we define the time-integrals $T^{-1}\underline{\Phi}$ and $T^{-1}\Phi$ of $\underline{\Phi}$ and $\Phi$ as follows:
\begin{align*}
T^{-1}\underline{\Phi}(v,\theta,\varphi)=&-\int_{v}^{\infty} \underline{\Phi}(v',\theta,\varphi)\,dv',\\
T^{-1}{\Phi}(u,\theta,\varphi)=&-\int_{u}^{\infty}{\Phi}(u',\theta,\varphi)\,du'
\end{align*}
Let $n\geq1$ and $\delta>0$ and suppose that $\lim_{v \to \infty} v^{n+\delta}|\underline{\Phi}| (v,\theta,\varphi)<\infty$ and $\lim_{u \to \infty} u^{n+\delta}|{\Phi}| (u,\theta,\varphi)<\infty$. Then we define the $n$-{th} order time-integrals $T^{-n}\underline{\Phi}$ and $T^{-n}\Phi$ of $\underline{\Phi}$ and $\Phi$ inductively as follows:
\begin{align*}
T^{-n}\underline{\Phi}(v,\theta,\varphi)=&-\int_{v}^{\infty} T^{-(n-1)}\underline{\Phi}(v',\theta,\varphi)\,dv',\\
T^{-n}{\Phi}(u,\theta,\varphi)=&-\int_{u}^{\infty}T^{-(n-1)}{\Phi}(u',\theta,\varphi)\,du',
\end{align*}
with $T^0 \underline{\Phi}:=\underline{\Phi}$ and $T^0 \Phi:= \Phi$.
\end{definition}

\begin{lemma}
\label{lm:timeinv}
Let $n\in \N_0$ and let $(\underline{\Phi},\Phi)\in C^{\infty}(\mathcal{H}^+_{\geq v_0}) \oplus C^{\infty}(\mathcal{I}^+_{\geq u_0})$. Assume that $\lim_{v \to \infty} v^{n+\delta}|\underline{\Phi}| (v,\theta,\varphi)<\infty$ and $\lim_{u \to \infty} u^{n+\delta}|{\Phi}| (u,\theta,\varphi)<\infty$ for some $\delta>0$ and assume moreover that
\begin{equation}
\label{eq:normtminnphi2}
||T^{-n}\Phi||_{\mathcal{E}_{n; \mathcal{I}^+_{\geq u_0}}}+ ||T^{-n}\underline{\Phi}||_{ \mathcal{E}_{n; \mathcal{H}^+_{\geq v_0}}}<\infty.
\end{equation}
Then
\begin{equation*}
T^n(T^{-n}\psi)=\psi,
\end{equation*}
with $\psi$ the solution associated to $(\underline{\Phi},\Phi)$ and $T^{-n}\psi$ the solution associated to $(T^{-n}\Phi, T^{-n}\underline{\Phi})$.
\end{lemma}
\begin{proof}
By \eqref{eq:normtminnphi2}, we can conclude that
\begin{align*}
\lim_{u\to \infty}T^n(rT^{-n}\psi)(u,v,\theta,\varphi)=L^n(T^{-n}\underline{\Phi})=\underline{\Phi},\\
\lim_{v\to \infty}T^n(rT^{-n}\psi)(u,v,\theta,\varphi)=\underline{L}^n(T^{-n}\underline{\Phi})=\underline{\Phi}
\end{align*}
Hence, by uniqueness of $\psi$ given $(\underline{\Phi},\Phi)$, we conclude that 
\begin{equation*}
T^n(T^{-n}\psi)=\psi.
\end{equation*}
\end{proof}

\begin{proposition}
\label{prop:regtimeinv}
Let $n\in \N_0$ and let $(\underline{\Phi},\Phi)\in C^{\infty}(\mathcal{H}^+_{\geq v_0}) \oplus C^{\infty}(\mathcal{I}^+_{\geq u_0})$. Assume that\\ $\lim_{v \to \infty} v^{n+\delta}|\underline{\Phi}| (v,\theta,\varphi)<\infty$ and $\lim_{u \to \infty} u^{n+\delta}|{\Phi}| (u,\theta,\varphi)<\infty$ for some $\delta>0$ and assume moreover that
\begin{equation}
\label{eq:normtminnphi}
||T^{-n}\Phi||_{\mathcal{E}_{2n; \mathcal{I}^+_{\geq u_0}}}+ ||T^{-n}\underline{\Phi}||_{ \mathcal{E}_{2n; \mathcal{H}^+_{\geq v_0}}}<\infty.
\end{equation}
Then
\begin{equation*}
r\psi\in W_{\rm loc}^{n+1,2}(\widehat{\mathcal{R}}).
\end{equation*}
\end{proposition}
\begin{proof}
By the assumptions on the limiting behaviour of $\Phi$ and $\underline{\Phi}$, together with \eqref{eq:normtminnphi}, we can apply Proposition \ref{lm:completionhorinf} to conclude that $(T^{-n}\underline{\Phi},T^{-n}\Phi)\in \mathcal{E}_{2n; \mathcal{H}^+_{\geq v_0}}\oplus \mathcal{E}_{2n; \mathcal{I}^+_{\geq u_0}}$. Then we can apply Proposition \ref{prop:regtimeder} together with Lemma \ref{lm:timeinv} to conclude the proof.
\end{proof}

\section{Application 2: A scattering construction of smooth solutions}
\label{sec:appmode}
We make use of the results in Section \ref{sec:appreg} to construct smooth solutions from scattering data.

\begin{corollary}
\label{cor:smoothbackw}
Let  $(\underline{\Phi},\Phi)\in C^{\infty}(\mathcal{H}^+_{\geq v_0}) \oplus C^{\infty}(\mathcal{I}^+_{\geq u_0})$ such that 
\begin{align*}
\lim_{v \to \infty} v^{p}|L^k\Omega^{\alpha}\underline{\Phi}| (v,\theta,\varphi)=&\:0,\\
\lim_{u \to \infty} u^{p}|\underline{L}^k\Omega^{\alpha}{\Phi}| (u,\theta,\varphi)=&\: 0,
\end{align*}
for all $p\in \R$, $k\in \N_0$ and $\alpha\in \N_0^3$. Then
\begin{equation*}
r\psi\in C^{\infty}(\widehat{\mathcal{R}}).
\end{equation*}
\end{corollary}
\begin{proof}
By the initial data assumptions, we have that $T^{-n}\Phi$ and $T^{-n}\underline{\Phi}$ are well-defined and satisfy \eqref{eq:normtminnphi} for all $n\in \N_0$. Hence we arrive at the desired statement by applying Proposition \ref{prop:regtimeinv} together with standard Sobolev embeddings.
\end{proof}
Corollary \ref{cor:smoothbackw} allows us to construct smooth ``mode solutions'' with an \emph{arbitrary} frequency $\omega$ with postitive imaginary part:
\begin{proposition}
\label{prop:modesol}
Let $\omega \in \C$ with $\textnormal{Im}\,\omega <0$. Let ${\underline{\Phi}}(v,\theta,\varphi)=f_H(\theta,\varphi)e^{-i\omega v}$ and ${\Phi}(u,\theta,\varphi)=f_I(\theta,\varphi)e^{-i\omega u}$ for $f_H,f_I\in C^{\infty}(\s^2)$. Then there exists a unique smooth solution $\psi$ to \eqref{eq:waveequation} on $\hat{\mathcal{R}}$, such that
\begin{equation*}
r\cdot \psi(\tau,\rho,\theta,\varphi)=f(\rho,\theta,\varphi)e^{-i\omega\cdot \tau},
\end{equation*}
with $f\in C^{\infty}(\hat{\Sigma})$ and
\begin{align*}
\lim_{\rho \downarrow M}f(\rho,\theta,\varphi)=&\: f_H(\theta,\varphi),\\
\lim_{\rho \to \infty }f(\rho,\theta,\varphi)=&\: f_I(\theta,\varphi).
\end{align*}
\end{proposition}
\begin{proof}
The initial data satisfy the assumptions of Corollary \ref{cor:smoothbackw}, so we have that $r\psi\in C^{\infty}(\widehat{\mathcal{R}})$ and
\begin{align*}
\lim_{u\to \infty} rT\psi(u,v,\theta,\varphi)=&\:T\underline{\Phi}(v,\theta,\varphi),\\
\lim_{v\to \infty} rT\psi(u,v,\theta,\varphi)=&\:T{\Phi}(u,\theta,\varphi).
\end{align*}
Furthermore, the specific choice of $(\underline{\Phi},\Phi)$ ensures that
\begin{align*}
T \underline{\Phi}+i\omega \underline{\Phi}=&\:0,\\
T\Phi+i\omega\Phi=&\:0.
\end{align*}
Hence, by uniqueness of the associated solution to \eqref{eq:waveequation}, linearity and Lemma \ref{lm:timeinv}, we have that $T\psi+i\omega \psi=0$ so $\psi(\tau,\rho,\theta,\varphi)=f(\rho,\theta,\varphi)e^{-i\omega\cdot \tau}$ for some $f\in C^{\infty}(\hat{\Sigma})$.
\end{proof}

\appendix

\section{Proof of Theorem \ref{thm:Nenergyinfinite}}
\label{sec:proofB}
We consider first (i). The argument consists of an application of the propositions proved in Section 6 of \cite{DafShl2016}, which apply directly to the setting of ERN provided the following key assumptions are verified:
\begin{enumerate}[1)]
\item We can associate to smooth radiation fields $r\psi|_{\mathcal{H}^+}$ and $r \psi|_{\mathcal{I}^+}$ in $\bigcap_{s=1}^{\infty} \dot{H}^s(\R\times \s^2)$ a corresponding unique solution $\psi$ to \eqref{eq:waveequation} that is smooth away from $\mathcal{H}^-$ and has finite $T$-energy along $\mathcal{H}^-$ and $\mathcal{I}^-$.
\item If we consider spherically symmetric, smooth, compactly supported $r\psi|_{\mathcal{I}^+}$ and vanishing $\psi|_{\mathcal{H}^+}$, then $T\psi$ must be non-vanishing on $\mathcal{H}^-$.
\end{enumerate}
Assumption 1) holds in ERN by Theorem \ref{thm:tscatERN}, together with commutation with $T$ and standard elliptic estimates.

In order to verify that assumption 2) also holds, we first conclude by a domain of dependence argument that the solution is vanishing in the region $u\geq u_*$ for sufficiently large $u_*$, where $u_*$ depends on the size of the support of $\phi|_{\mathcal{I}^+}$.

Now suppose $T\psi$ is vanishing along $\mathcal{H}^+$. Consider the following identity for $\psi$:
\begin{equation*}
-\partial_u\left(\frac{1}{2}r^2(\partial_v\psi)^2\right)+\partial_v\left(\frac{1}{2}r^2(\partial_u\psi)^2\right)=Dr^2\partial_u\psi\partial_v\psi,
\end{equation*}
which follows immediately from \eqref{eq:waveequation} and spherical symmetry of $\psi$. Integrating the above equation in the rectangle $[u_1,u_*]\times [-\infty,v_1]$, with $u<u_*$ and $v_1\in \R$ arbitrary, we obtain:
\begin{equation*}
\begin{split}
\sup_{-\infty<v'\leq v_1}& \int_{\{v=v'\}} r^2(\partial_u\psi)^2\,du+\sup_{u_1<u'\leq u_*} \int_{\{u=u'\}} r^2(\partial_v\psi)^2\,dv\leq \int_{\{v=-\infty\}} r^2(\partial_u\psi)^2\,du'+ \int_{\{u=u_*\}} r^2(\partial_v\psi)^2\,dv'\\
&+\int_{u_1}^{u_*} \int_{-\infty}^{v_1} r^2 D |\partial_u\psi||\partial_v\psi|\,dudv'\\
\leq&\: \int_{\{v=-\infty\}} r^2(\partial_u\psi)^2\,du'+ \int_{\{u=u_*\}} r^2(\partial_v\psi)^2\,dv'+\frac{1}{2}\int_{-\infty}^{v_1}(1+|v'|)^{-2} \left[\sup_{-\infty \leq v''\leq v'}\int_{\{v=v''\}} r^2(\partial_u\psi)^2\,du'\right]\,dv'\\
&+\frac{1}{2} \int_{u_1}^{u_*} \sup_{-\infty <v' \leq v_1} (D^2(u,v') (1+|v'|)^{2}) \left[\sup_{-\infty \leq u''\leq u'}\int_{\{u=u''\}} r^2(\partial_v\psi)^2\,dv'\right]\,du'.
\end{split}
\end{equation*}
Since $\sup_{-\infty <v' \leq v_1} D^2(u,v') (1+|v'|)^{2}$ is bounded, we can apply a Gr\"onwall inequality (see for example Lemma 4.1 of \cite{gajic}) to conclude that: for all $v_1\in \R$ and $u_0<u_*$
\begin{equation*}
\int_{\{v=v_1\}} r^2(\partial_u\psi)^2\,du+ \int_{\{u=u_1\}} r^2(\partial_v\psi)^2\,dv\leq C_{u_1,v_1}\left[\int_{\{v=-\infty\}} r^2(\partial_u\psi)^2\,du'+ \int_{\{u=u_*\}} r^2(\partial_v\psi)^2\,dv'\right].
\end{equation*}
Since $\partial_u\psi$ vanishes along $\mathcal{H}^-=\{v=-\infty\}$ and $\partial_v\psi$ vanishes along $\{u=u_*\}$, we can conclude that $\psi$ must vanish in $[u_1,u_*]\times [-\infty,v_1]$ and therefore in the full spacetime. This is a contradiction, since $r\psi$ has non-trivial support on $\mathcal{I}^+$. Hence, assumption 2) is indeed verified.

Then, we can apply Proposition 6.1 of \cite{DafShl2016} with $p=1$, to construct smooth, spherically symmetric data along $\mathcal{I}^+$ with a polynomial tail that propagates to the past event horizon $\mathcal{H}^-$, and then apply $T$-energy conservation to further propagate to $\Sigma_0$ in order to conclude that the corresponding solution $\psi$ to \eqref{eq:waveequation} has infinite $N$-energy flux in a neighbourhood of $\mathcal{H}^+$.

The argument above can be applied \emph{mutatis mutandis} by reversing the role of $\mathcal{I}^+$ and $\mathcal{H}^+$ above to prove (ii).

\section{Basic estimates}

\begin{lemma}
\label{lm:completionhorinf}
Let $(f,g)\in C^{\infty}(\mathcal{H}^+_{\geq v_0})\oplus  C^{\infty}(\mathcal{I}^+_{\geq u_0})$.
\begin{itemize}
\item[1.)]We have that $(f,g)\in \mathcal{E}_{ \mathcal{H}^+_{\geq v_0}}^T\oplus \mathcal{E}_{ \mathcal{I}^+_{\geq u_0}}^T$ if 
\begin{equation*}
||f||_{\mathcal{E}_{\mathcal{H}^+_{\geq v_0}}^T}+ ||g||_{\mathcal{E}_{ \mathcal{I}^+_{\geq u_0}}^T}<\infty.
\end{equation*}
 and $\lim_{v\to \infty} f<\infty$, $\lim_{u\to \infty}  g<\infty$.
\item[2.)]Let $n\in \N_0$. Then $(f,g)\in \mathcal{E}_{n; \mathcal{H}^+_{\geq v_0}}\oplus \mathcal{E}_{n; \mathcal{I}^+_{\geq u_0}}$ if 
\begin{equation*}
||f||_{\mathcal{E}_{n; \mathcal{H}^+_{\geq v_0}}}+ ||g||_{\mathcal{E}_{n; \mathcal{I}^+_{\geq u_0}}}<\infty
\end{equation*}
 and $\lim_{v\to \infty} f=0$, $\lim_{u\to \infty}  g=0$.
 \end{itemize}
\end{lemma}
\begin{proof}
We will prove $2.)$. The proof of $1.)$ proceeds very similarly. Without loss of generality, we can restrict to $f$. The estimates for $g$ proceed entirely analogously. We introduce a smooth cut-off $\chi:[v_0,\infty)\to \R$ such that $\chi(v)=1$ for all $v\leq 2v_0$ and $\chi=0$ for all $v\geq 4v_0$. Then $|\chi^{(k)}|\leq C_k$. Rescale $\chi$ by defining $\chi_i:[v_0,\infty)\to \R$, with $i\in \N$, as follows: $\chi_i(v):=\chi(\frac{v}{i})$. Then $|\chi_i^{(k)}|\leq i^{-k}C_k$ for all $0\leq k\leq n$ and hence $|\chi_i^{(k)}|\leq C_kv^{-k}$.

Now, define $f_i=\chi_i\cdot f$, then $f_i\in C_{c}^{\infty}(\mathcal{H}^+_{\geq v_0})$. Furthermore, by applying the Leibniz rule successively and using that $|\chi_i^{(k)}|\leq C_kv^{-k}$ for $k\geq 1$, we obtain
\begin{equation*}
\begin{split}
||f-f_i||_{\mathcal{E}_{n; \mathcal{H}^+_{\geq v_0}}}^2\leq&\: C \sum_{j=0}^2\sum_{m+2k+2|\alpha|\leq 2n}\int_{ \mathcal{H}^+_{\geq  2v_0 i}} v^{2k+2-j} (\partial_v^{1+k+m+j} \Omega^{\alpha} f)^2+v^{2k} |\snabla_{\s^2}\partial_v^{k+m} \Omega^{\alpha}f|^2\,d\omega dv\\
&+C\sum_{j=0}^2\sum_{m+2k \leq 2n}\int_{ \mathcal{H}^+_{2v_0 i\leq v\leq 4v_0 i}} v^{2+2k-j}|\chi^{(1+k+m+j)}|^2 f^2\,d\omega dv.
\end{split}
\end{equation*}
and
\begin{equation*}
\begin{split}
\sum_{j=0}^2\sum_{m+2k \leq 2n}\int_{ \mathcal{H}^+_{2v_0 i\leq v\leq 4v_0 i}} v^{2+2k-j}|\chi^{(1+k+m+j)}|^2 f^2\,d\omega dv\leq&\: C\int_{ \mathcal{H}^+_{2v_0 i\leq v\leq 4v_0 i}}f^2\,d\omega dv.
\end{split}
\end{equation*}
Note that since $f(v,\theta,\varphi)\to 0$ as $v\to \infty$, we can estimate
\begin{equation*}
f^2(v,\theta,\varphi)\leq v^{-1}\int_v^{\infty} {v'}^2(\partial_vf)^2(v',\theta,\varphi)\,dv,
\end{equation*}
so
\begin{equation*}
\int_{ \mathcal{H}^+_{2v_0 i\leq v\leq 4v_0 i}} f^2\,d\omega dv\leq \int_{ \mathcal{H}^+_{ v\geq 2v_0 i}} v^2(\partial_vf)^2\,d\omega dv.
\end{equation*}
Hence, $||f-f_i||_{\mathcal{E}_{n; \mathcal{H}^+_{\geq v_0}}}\to 0$ as $i\to \infty$ and we can conclude that $f$ lies in the completion of $C_{c}^{\infty}(\mathcal{H}^+_{\geq v_0})$ with respect to the norm $||\cdot ||_{\mathcal{E}_{n; \mathcal{H}^+_{\geq v_0}}}$.
\end{proof}

\begin{lemma}
\label{lm:completiontildeSigma}
Let $(f,g)\in (C^{\infty}(\widetilde{\Sigma}))^2$.
\begin{itemize}
\item[1.)] Then $(f,g)\in \mathcal{E}_{n; \widetilde{\Sigma}}$ if 
\begin{equation*}
||(f,g)||_{\mathcal{E}_{ \widetilde{\Sigma}}^T}<\infty
\end{equation*}
and $\lim_{r\to \infty} f=0$.
\item[2.)]Let $n\in \N_0$. Then $(f,g)\in \mathcal{E}_{n; \widetilde{\Sigma}}$ if 
\begin{equation*}
||(f,g)||_{\mathcal{E}_{n; \widetilde{\Sigma}}}<\infty
\end{equation*}
 and $\lim_{r\to \infty} rf=0$.
 \end{itemize}
\end{lemma}
\begin{proof}
We will prove $2.)$. The proof of $1.)$ proceeds very similarly. The proof proceeds analogously to the proof of Lemma \ref{lm:completionhorinf}. We first introduce a cut-off $\chi: (-\infty,\infty)\to \R$ such that $\chi(r_*)=1$ for all $|r_*|\leq 2r_0$, with $r_0>0$, and  $\chi(r_*)=0$ for all $|r_*|\geq 4r_0$. Then we define  $\chi_i: (-\infty,\infty)\to \R$ as follows: $\chi_i(r_*)=\chi(\frac{r_*}{i})$. Observe that $|\chi_i^{(k)}|\leq C_k(1+|r_*|)^{-k}$. Define $f_i:=\chi_i\cdot f$ and $g_i=\chi_i\cdot f$, then $(f_i,g_i)\in C_{c}^{\infty}(\widetilde{\Sigma})\times C_{c}^{\infty}(\widetilde{\Sigma})$.

Furthermore, if $\psi$ denotes the solution corresponding to the initial data $(f,g)$, we can estimate
\begin{equation*}
\begin{split}
||(f-f_i,g-g_i)||_{\mathcal{E}_{n; \widetilde{\Sigma}}}\leq&\: C\sum_{j=0}^1\sum_{m+2k+2|\alpha|\leq 2n}\int_{\widetilde{\Sigma}\cap\{r_*\geq 2r_0 i\}} r^{2+2k-j}(\partial_v^{k+1}\Omega^{\alpha}T^{j+m}\phi)^2+r^{2k-j}|\snabla_{\s^2}\partial_v^{k}\Omega^{\alpha}T^{j+m}\phi|^2 \\
&+r^{2k+2-j}(\partial_u^{k+1}\Omega^{\alpha}T^{j+m}\phi)^2+r^{2k-j}|\snabla_{\s^2}\partial_u^{k}\Omega^{\alpha}T^{j+m}\phi|^2\,d\omega dr_*\\
&+C\sum_{j=0}^1\sum_{m+2k+2|\alpha|\leq 2n}\int_{\widetilde{\Sigma}\cap\{r_*\leq -2r_0i\}} (r-M)^{-2k-2+j}(\partial_v^{k+1}\Omega^{\alpha}T^{j+m}\phi)^2\\
&+(r-M)^{-2k+j}|\snabla_{\s^2}\partial_v^{k}\Omega^{\alpha}T^{j+m}\phi|^2 \\
&+(r-M)^{-2k-2+j}(\partial_u^{k+1}\Omega^{\alpha}T^{j+m}\phi)^2+(r-M)^{-2k+j}|\snabla_{\s^2}\partial_u^{k}\Omega^{\alpha}T^{j+m}\phi|^2\,d\omega dr_*\\
&+C\sum_{m\leq 2n+2} \int_{\widetilde{\Sigma} \cap\{|r_*|\geq 2r_0 i\}} {\mathbf{J}}^T[T^m\psi]\cdot \mathbf{n}_{\widetilde{\Sigma}} \,d\mu_{\widetilde{\Sigma}}\\
&+C\int_{ \widetilde{\Sigma}\cap\{2r_0 i\leq |r_*|\leq 4r_0 i\} }\phi^2\,d\omega dr_*.
\end{split}
\end{equation*}
Furthermore, using that $\lim_{|r_*|\to \infty} \phi=0$,
\begin{equation*}
\int_{ \widetilde{\Sigma}\cap\{2r_0 i\leq |r_*|\leq 4r_0 i\} }\phi^2\,d\omega dr_* \leq C\int_{ \widetilde{\Sigma}\cap\{ |r_*|\geq 2r_0 i\} }r^2(L \phi)^2\,d\omega dr_*.
\end{equation*}
Hence, $||(f-f_i,g-g_i)||_{\mathcal{E}_{n; \widetilde{\Sigma}}} \to 0$ as $i\to \infty$, so $(f,g)$ are in the completion of $(C_{c}^{\infty}(\widetilde{\Sigma}))^2$ with respect to the norm $|| \cdot ||_{\mathcal{E}_{n; \widetilde{\Sigma}}}$.
\end{proof}

\bibliographystyle{plain}

\end{document}